\documentclass{article}

\usepackage{amsthm}
\usepackage{amsmath}
\usepackage{amssymb}
\usepackage{mathrsfs}
\usepackage{graphicx}
\usepackage{hyperref}
\usepackage{bm}
\usepackage{dsfont}
\usepackage{enumitem}
\usepackage{amsfonts}
\usepackage{textcomp}
\usepackage{MnSymbol}
\usepackage{geometry}
\usepackage[english]{babel}
\usepackage[T1]{fontenc}

\usepackage{caption}
\usepackage{graphicx}
\graphicspath{ {./images/} }
\usepackage{pst-knot}
\usepackage{tikz}
\usetikzlibrary{patterns}
\usepackage[utf8]{inputenc}

\newtheorem{theo}{Theorem}[section]

\newtheorem{lem}[theo]{Lemma}

\newtheorem{definition}[theo]{Definition}

\newtheorem{proposition}[theo]{Proposition}

\newtheorem{corollary}[theo]{Corollary}

\newtheorem{conjecture}[theo]{Conjecture}

\newtheorem{ex}[theo]{Example}

\title{Strong bolicity and the Baum-Connes conjecture for relatively hyperbolic groups}
\author{Hermès Lajoinie-Dodel \footnote{Hermès Lajoinie-Dodel, lajoinie@math.uni-bielefeld.de, Universität Bielefeld, Germany}}
\date{December 2025}

\begin{document}
\maketitle

\begin{center}
\begin{minipage}{0.8\textwidth}
\textsc{Abstract.}  We construct a \emph{strongly bolic metric} for a certain class of \emph{relatively hyperbolic groups}, which includes those with \emph{$CAT(0)$ parabolics and virtually abelian parabolics}. If we further assume that the parabolics satisfy $(RD)$, applying a theorem of Lafforgue, we deduce the Baum-Connes conjecture for these groups. One of the key ingredients in our construction is the use of random coset representatives called \emph{masks}, developed by Chatterji and Dahmani.

\end{minipage}
\end{center}

\let\thefootnote\relax

\section{Introduction} \label{section : Intro }

Originally formulated by Baum and Connes in 1982 in a paper that was eventually published in 2000 in \cite{BaumConnes1}, and later revisited by Baum, Connes, and Higson in 1994 in \cite{BaumConnes2Nigel}, the Baum–Connes conjecture establishes a bridge between commutative and noncommutative geometry. More precisely, it states that for any locally compact group 
$G$, the $G$-equivariant $K$-homology of the classifying space for proper $G$-actions is isomorphic to the $K$-theory of the group reduced $C^{\ast}$-algebra. The particular case of discrete groups is already very interesting and highly nontrivial.

Proving the Baum–Connes conjecture for a given group $G$ implies several famous conjectures in topology, geometry, and functional analysis. For instance, the injectivity of the Baum–Connes map implies the celebrated Novikov conjecture \cite{LandInjBcImpliesNovikov}, while its surjectivity implies according to \cite{AparicioJulgValetteBaumConnesConjecture} the Kaplansky–Kadison conjecture, which states that when 
$G$ is torsion-free, its reduced $C^{\ast}$-algebra contains no nontrivial idempotent.\\

According to Higson and Kasparov in \cite{HigsonKasparovHaagerupimpliqueBaum-Connes}, the Haagerup property implies the Baum-Connes conjecture. Consequently, the conjecture holds for groups that act properly on a tree, or more generally on a median metric space \cite{Tmedianviewpoint}. However, this criterion does not allow one to prove it for infinite groups having Kazhdan Property $(T)$. The first examples of infinite groups with Property $(T)$ satisfying the Baum–Connes conjecture are due to the remarkable work of Vincent Lafforgue, notably through the following theorem.
\begin{theo}\cite{LafforgueBaumConnesConjecture}\label{theo : Baum-Connes=Rapid decay + Strongly Bolic}

Let $G$ be a finitely generated group satisfying the following.

\begin{itemize}
    \item $G$ has the rapid decay property.

    \item $G$ acts properly by isometries on a uniformly locally finite, strongly bolic, weakly geodesic metric space.
\end{itemize}

Then $G$ satisfies the Baum-Connes conjecture.

\end{theo}

The $(RD)$ property is a property of the group reduced $C^{\ast}$-algebra. It holds when the operator norm can be controlled by a simpler norm. In particular, this property is satisfied, among others, by hyperbolic groups according to results due to de la Harpe \cite{delaHarpeRD} and Jolissaint \cite{JolissaintRD}. In our work, we focus on the second point of the theorem.

Bolic metric spaces were defined by Kasparov and Skandalis in \cite{KasparovSkandalisBolic1}, \cite{KasparovSkandalisBolic2} in connection with the Novikov conjecture, and were later strengthened by Lafforgue in the context of Theorem \ref{theo : Baum-Connes=Rapid decay + Strongly Bolic}. A metric space is said to be strongly bolic when its balls satisfy a smoothness and convexity condition. In most cases, strongly bolic metrics take non-integer values, which prevents them from being realized as graph metrics. Examples of such metric spaces include $CAT(0)$ spaces as well as $L^{p}$
 spaces for $1<p<\infty$.

Using Theorem \ref{theo : Baum-Connes=Rapid decay + Strongly Bolic}, Vincent Lafforgue showed that “classical” hyperbolic groups, i.e. uniform lattices in simple Lie groups of rank 1 satisfy the Baum–Connes conjecture. Therefore, the uniform lattices in
$Sp(n,1)$ constitute the first example of infinite groups with Property $(T)$ satisfying the Baum–Connes conjecture.

This result was later generalized by Mineyev and Yu in \cite{MineyevYuStrongBolicityhypergroups} to all hyperbolic groups, proving that they satisfy the Baum–Connes conjecture, although it is unknown whether these groups are $CAT(0)$.\\

A natural question then is to ask to which generalizations of hyperbolic groups this result can be extended. In this paper, we focus on relatively hyperbolic groups. Relative hyperbolicity was defined by Gromov in 1987 in \cite{Gromovhypgroups}.  It is a geometric generalization to hyperbolic groups for a larger class of groups.  The rough idea is that a group $G$ is hyperbolic relative to a family of subgroups $\mathcal{P}$ if the geometry of $G$ is hyperbolic outside of $\mathcal{P}$. Among the many definitions of these groups, we will here use only Bowditch’s definition \cite{Bowditchrelhyp}, which involves the coned-off graph and requires that this graph be hyperbolic and fine.

According to \cite{DrutuSapirRdrelivementHyp}, the $(RD)$ property is well understood for relatively hyperbolic groups. Indeed, a group that is hyperbolic relative to groups satisfying the 
$(RD)$ property also satisfies 
$(RD)$. Therefore, we focus on the second point of Theorem \ref{theo : Baum-Connes=Rapid decay + Strongly Bolic}, namely the construction of a strongly bolic metric. The strongest result we can hope for is for groups hyperbolic relatively to subgroups with a strongly bolic action.

In this article, we provide a partial answer to this question in the case where the parabolic subgroups are $CAT(0)$.

\begin{theo}\label{theo: main theo}

Let 
$G$ be a group hyperbolic relative to $CAT(0)$ subgroups. Then 
$G$ admits a metric $\hat{d}$ satisfying the following properties:

\begin{itemize}
\item 
$\hat{d}$ is invariant under the action of $G$, i.e., 
$\hat{d}(g.x,g.y)=\hat{d}(x,y)$ for all $g\in G$ and $x,y \in G$,

\item $\hat{d}$ is uniformly locally finite and quasi-isometric to the word metric,

\item the metric space $(G,\hat{d})$ is weakly geodesic and strongly bolic.

\end{itemize}
    
\end{theo}

As a consequence, we get the following corollary.

\begin{corollary}
Let $G$ be a group hyperbolic relatively to $CAT(0)$ subgroups with the $(RD)$ property. Then $G$ satisfies the Baum-Connes conjecture.
    
\end{corollary}

More precisely, we obtain the following corollary for families of $CAT(0)$ groups that have $(RD)$ property.

\begin{corollary}
Groups hyperbolic relatively to following groups:
\begin{itemize}
    \item virtually abelian,

    \item cocompaclty cubulated,

    \item Coxeter groups,
\end{itemize}

and their subgroups satisfy the Baum-Connes conjecture.

\end{corollary}

This result applies in particular to the fundamental groups of finite-volume real hyperbolic manifolds. In this way, we re-establish that these groups satisfy the Baum–Connes conjecture.

In our article, we provide a new proof that hyperbolic groups admit a strongly bolic metric. Our metric has the advantage of being relatively simple to define. It uses a notion of angle for hyperbolic spaces introduced in \cite{anglehyperbolicHaettelChatterjiDahmaniLecureux} (see Proposition \ref{Proposition: angle for hyperbolic metric space}), which generalizes the visual angles of $CAT(-1)$ spaces. In the relatively hyperbolic case, our metric builds on the construction for hyperbolic groups presented in Section \ref{section : bolicité forte pour gp hyperbolique}, as well as on a notion of random representatives, called masks, introduced by Chatterji and Dahmani in \cite{ChatterjiDahmani}.\\

It is important to note that independently and simultaneously in a recent preprint, Nishikawa and Petrosyan obtained in \cite{NishikawaPetrosyanBC}, by different methods, a stronger version of the Baum–Connes conjecture for a very significant class of relatively hyperbolic groups. This class includes groups all lattices in rank one Lie groups, which are not covered by our theorem. We nevertheless believe that our work is of interest in its geometric nature.\\

\textbf{Structure of the article.} Section \ref{section: Baum-Connes} is a quick overview on the Baum-Connes conjecture with an emphasis on strong bolicity.\\
Section \ref{section : groupe relativement hyperbolic} is dedicated to hyperbolic geometry and relatively hyperbolic groups in particular.\\
In Section \ref{section : bolicité forte pour gp hyperbolique}, we give a new proof of the fact that hyperbolic groups admit a strongly bolic metric.\\
To conclude, in Section \ref{section : bolicité forte pour gp relativement hyperbolique}, we prove Theorem \ref{theo: main theo}. As explained above, the main ingredient is the notion of masks, developed by Dahmani and Chatterji in \cite{ChatterjiDahmani}, which we present here.\\

\textbf{Acknowledgements.}
This article is part of the author's PhD thesis. The author would like to thank Rémi Coulon and Mikael de la Salle, as thesis reviewers, for their detailed reading of the article, which greatly contributed to its improvement. The author also sends warm thanks to Moulay Tahar Benammeur for engaging and insightful discussions, as well as for agreeing to serve on the thesis jury. The author is grateful to Indira Chatterji and François Dahmani for accepting to be members of the jury. The author also thanks Ashot Minasyan for pointing out reference \cite{minasyan}. Finally, the author warmly thanks Thomas Haettel for his support, availability, and numerous careful readings of this text.

\section{Baum-Connes conjecture and strong bolicity}\label{section: Baum-Connes}

In this section, we provide a brief introduction to this conjecture. For more general references, we refer to \cite{ValetteIntroBaumConnes} and \cite{AparicioJulgValetteBaumConnesConjecture}.

\subsection{Quick presentation of the conjecture}
The Baum-Connes conjecture is part of Alain Connes' ''noncommutative geometry'' program.
 The Baum-Connes conjecture was formulated in \cite{BaumConnes1} in 1982 in an article finally published in 2000. The modern version comes from the article \cite{BaumConnes2Nigel} of Baum, Connes and Higson. This conjecture builds a bridge between geometry/topology on one hand and analysis on the other hand.

In fact, for every locally compact group $G$, there is a Baum-Connes conjecture.
For every locally compact group $G$, we can always associate four abelian groups $K_{\ast}^{top}(G)$ and $K_{\ast}(C_{r}^{\ast}(G))$ (with $\ast=0,1$). We can also construct two group morphisms:
$$\mu_{r}: K_{\ast}^{top}(G) \rightarrow K_{\ast}(C_{r}^{\ast}(G)).$$

The Baum-Connes conjecture states that $\mu_{r}$ is an isomorphism for $\ast=0,1$.

\begin{conjecture}\label{conjecture: Baum-Connes}(Baum-Connes conjecture)

Let $G$ be a locally compact group.
The assembly map:
$$ \mu_{r}: K_{\ast}^{top}(G) \rightarrow K_{\ast}(C_{r}^{\ast}(G))$$
is an isomorphism for $\ast=0,1$.
    
\end{conjecture}

The left-hand side $K_{\ast}^{top}(G)$ is mentioned as the \emph{geometric, topological} or the \emph{commutative} side. It involves a space $\underline{EG}$, the classifying space for proper action of $G$, and $K_{\ast}^{top}(G)$ is the $G$-equivariant $K$-homology of $\underline{EG}$.

The right-hand side, $K_{\ast}(C_{r}^{\ast}(G))$ is mentioned as the \emph{analytic} side or the \emph{non-commutative} side. $C_{r}^{\ast}(G)$ is \emph{the reduced $C^{\ast}$-algebra} of $G$, which is the $C^{ \ast}$-algebra generated by $G$ in its left regular representation on the Hilbert space $\ell^{2}(G)$.

The assembly map $\mu_{r}$ is defined with Kasparov $KK$-theory and is related to index theory.\\

The Baum-Connes Conjecture implies various other famous conjectures. We cite these conjectures here without providing precise statements. For a more precise discussion about the consequences of the Baum-Connes conjecture, see \cite[Section 4.5.]{AparicioJulgValetteBaumConnesConjecture}

It is often said that the surjectivity of $\mu_{r}$ has implication in analysis, while injectivity has implication in topology.

The injectivity of the assembly map implies the Novikov conjecture according to \cite{LandInjBcImpliesNovikov}. We refer to \cite{SurveyNovikov} for a survey on this conjecture.

The injectivity of the assembly map implies also the Gromov-Lawson-Rosenberg conjecture, this fact is proved in \cite{BaumConnes2Nigel}, see \cite{RosenbergStolz} for a discussion on the Gromov-Lawson-Rosenberg conjecture.

The surjectivity implies the Kadison-Kaplansky conjecture, which we state now.
\begin{conjecture}\label{conjecture: Kaplansky-Kadison}
If $G$ is torsion-free, then $C_{r}^{\ast}(G)$ has no idempotents except $0$ and $1$.
\end{conjecture}
A weaker conjecture, formulated by Kaplansky, states that the group algebra does not have trivial idempotents. For a proof of the fact that the surjectivity of the assembly map implies the Kaplansky-Kadison conjecture, see \cite[Proposition 4.20.]{AparicioJulgValetteBaumConnesConjecture}.

The surjectivity of the assembly map implies also the vanishing of a topological Whitehead group. We refer to \cite{Whiteheadgroups} for the definitions and a proof of this fact.

\subsection{Status of the conjecture}\label{subsection: status of the conjecture}

To begin, let us remark that no-counter example of the Baum-Connes conjecture is known.

We give here two ways to prove the Baum-Connes conjecture.

The first way is the Haagerup Property.

\begin{theo}\cite{HigsonKasparovHaagerupimpliqueBaum-Connes}

Discrete groups with the Haagerup Property satisfy the Baum-Connes conjecture.
\end{theo}

The Baum-Connes conjecture is therefore satisfied by amenable groups \cite{BekkaCherixValetteHaagamenable}, $SO(n,1)$ \cite{VershikGel'fandGraevHaagerupgroupesrelhyp1}, $SU(n,1)$ and their lattices \cite{VershikGel'fandGraevHaagerupgroupesrelhyp2}.
The conjecture is also verified for groups acting properly on trees and, more generally, on median spaces according to \cite{Tmedianviewpoint}
.
It is for this reason that Property $(T)$ is identified as a difficulty in proving the Baum-Connes conjecture.
The first examples of infinite groups with Property $(T)$ that satisfy the Baum-Connes conjecture were given by Vincent Lafforgue using the following theorem.

\begin{theo}\cite{LafforgueBaumConnesConjecture}
Let $G$ be a finitely generated group satisfying the following.

\begin{itemize}
    \item $G$ has the rapid decay property.

    \item $G$ acts properly by isometries on a uniformly locally finite, strongly bolic, weakly geodesic metric space.
\end{itemize}

Then $G$ satisfies the Baum-Connes conjecture.

\end{theo}

We will discuss more specifically later on property $(RD)$ and strongly bolic metric spaces.

Thanks to Theorem \ref{theo : Baum-Connes=Rapid decay + Strongly Bolic}, Lafforgue proves the Baum-Connes conjecture for cocompact lattices in simple Lie groups of rank $1$. Cocompact lattices of $Sp(n,1)$ provide the first examples of infinite Property $(T)$ groups which satisfies the Baum-Connes conjecture.
This result was then generalized to all hyperbolic groups according to de la Harpe \cite{delaHarpeRD} and Mineyev and Yu \cite{MineyevYuStrongBolicityhypergroups}.

Using different methods, Vincent Lafforgue proved in \cite{LafforgueBaumConnesàcoeff} a stronger version of the Baum-Connes conjecture for hyperbolic groups, the Baum-Connes conjecture with coefficients.

It should be noted that the Baum-Connes conjecture remains open for the non-uniform lattices of $Sp(n,1)$ until the work of \cite{NishikawaPetrosyanBC}.

Vincent Lafforgue identified in \cite{LafforgueStrongTetBaumConnes} the strong Property $(T)$ as a very important difficulty for proving the Baum-Connes conjecture, more specifically for proving the surjectivity of the assembly map. The only results are known for cocompact lattices in $SL_{3}(\mathbb{R}),SL_{3}(\mathbb{C}),SL_{3}(\mathbb{H})$ and $E_{6}(-26)$. For example, the surjectivity is open for $SL_{3}(\mathbb{Z})$.\\

Although it remains an open problem whether the Baum–Connes conjecture is inherited by subgroups in general, Lafforgue Theorem \ref{theo : Baum-Connes=Rapid decay + Strongly Bolic} pass to subgroups, since the first point of the theorem clearly does so, and property $(RD)$ is also stable under passage to subgroups.\\

According to Oyono \cite{OyonoBcetarbres}, the Baum-Connes conjecture is stable under free or amalgamated products and HNN extensions.

\subsection{Rapid decay Property}\label{subsection: Rd property}

We will be very brief regarding the rapid decay property 
$(RD)$. Introductions to this topic are \cite{ChatterjiRd} and \cite{SapirRd}.\\

The $(RD)$ property is a property of the reduced group $C^{\ast}$-algebra. It holds when the operator norm can be bounded in terms of a simpler, more manageable norm.
This property was studied by Jolissaint in \cite{JolissaintRD}, before being used by Connes and Moscovici to prove the Novikov conjecture for hyperbolic groups \cite{ConnesMoscoviciNovikovConjHypGroup}. The recent interest in this property is Theorem \ref{theo : Baum-Connes=Rapid decay + Strongly Bolic} due to Lafforgue in \cite{LafforgueBaumConnesConjecture}, showing that it is one of the two conditions that allow to prove the Baum-Connes conjecture.\\

According to \cite{JolissaintRD} and \cite{delaHarpeRD}, hyperbolic groups satisfy also the Rapid decay property. For relatively hyperbolic groups, it was first shown by Chatterji and Ruane in \cite{ChatterjiRuaneRdreseauxderangUn} for rank-1 lattices, before being generalized by Drutu and Sapir to groups that are hyperbolic relatively to subgroups with property (RD) \cite{DrutuSapirRdrelivementHyp}. For a more comprehensive list of groups that do or do not satisfy the (RD) property, we refer to \cite{ChatterjiRd}.\\

More generally property $(RD)$ is satisfied by groups admitting a nice definition of quasi-center for triangles, this property is called the \emph{centroid property} in \cite{SapirRd}.

\subsection{Strong bolicity}\label{subsection: strongly bolic}

\paragraph{Strong bolicity}

Kasparov and Skandalis defined bolicity in their work on the Novikov conjecture \cite{KasparovSkandalisBolic1}, \cite{KasparovSkandalisBolic2}. Vincent Lafforgue strengthened this notion in connection with the Baum-Connes conjecture and Theorem \ref{theo : Baum-Connes=Rapid decay + Strongly Bolic}. For a more general introduction to this notion and related properties, see \cite{BucherKarlssonDefBolicSapces}.

\begin{definition}\label{definition: faiblement geodesique}
Let $\eta \geq 0$. A metric space $(X,d)$ is said to be \emph{$\eta$-weakly geodesic}, if for all $x,y \in X$, for all real $t \in [0,d(x,y)]$, there exists $z \in X$ such that $d(x,z)\leq t+\eta$ and $d(z,y) \leq d(x,y)-t+\eta$.

For all $x,y \in X$, $z$ belongs to the $\eta$-geodesic interval between $x$ and $y$ if:
$$d(x,z)+d(z,y)\leq 2\eta.$$

For all $x,y \in X$, a $\eta$-weak-geodesic between $x$ and $y$ is a map $\sigma :[0,d(x,y)]\rightarrow [0,1]$ such that for all $t \in [0,d(x,y)]$:
\begin{itemize}
    \item $d(x,\sigma(t))\leq t+\eta$,
    \item $d(\sigma(t),y)\leq d(x,y)-t+\eta$.
\end{itemize}

$(X,d)$ is said to be  \emph{weakly-geodesic} if there exists $\eta \geq 0$ such that $(X,d)$ is \emph{$\eta$-weakly geodesic}.

\end{definition}

Geodesic metric spaces are $0$-weakly geodesic. A geodesic between two points is a $0$-weak-geodesic between those two points.

Let $(X,d)$, for $x \in X, r \in \mathbb{R}_{+}$, $B(x,r)$ will denote the closed ball centered at $x$ of radius $r$.

\begin{definition}\label{definition: uniformly locally finite}
A metric space is said to be \emph{uniformly locally finite} if for all $r \in \mathbb{R}_{+}$, there exists $K \in \mathbb{N}$ such that, for all $x \in X$, $B(x,r)$ contains at most $K$ points.
\end{definition}

Simplicial graphs of bounded valency are particular examples of uniformly locally finite graphs.

\begin{definition} (Strong bolicity)
A metric space $(X,d)$ is called \emph{strongly bolic} if the following two conditions hold :
\begin{itemize}
    \item \emph{strongly}-$B1$, if for all $\eta,r>0$, there exists $R=R(\eta,r)\geq 0$ such that for all $x_{1}, x_{2}, y_{1}, y_{2} \in X$, with $d(x_{1},y_{1}),d(x_{1},y_{2}),d(x_{2},y_{1}),d(x_{2},y_{2}) \geq R$ and  \newline $d(x_{1},x_{2}),d(y_{1},y_{2}) \leq r$, we have:
$$ |d(x_{1},y_{1})+d(x_{2},y_{2})-d(x_{1},y_{2})-d(x_{2},y_{1})| \leq \eta .$$

    \item  \emph{weakly}-$B2'$, if there exists $\eta>0$ and a map $m: X \times X \rightarrow X$, such that:

    \begin{itemize}

    \item $m$ is a $\eta$-middle point map, that is for all $x,y \in X$,
    $$|2d(x,m(x,y))-d(x,y)|\leq 2 \eta,$$
$$|2d(y,m(x,y))-d(x,y)|\leq 2 \eta,$$

    \item for all $x,y,z \in X$,
    $$ d(m(x,y),z) \le \max(d(x,z),d(y,z)) + 2 \eta, $$

    \item for every $p \in \mathbb{R}_{+}$, there exists $N(p) \in \mathbb{R}_{+}$ such that for all $N \geq N(p) $, with $d(x,z) \leq N$, $d(y,z) \leq N$ and $d(x,y)>N$, we have:
    $$ d(m(x,y),z) \leq N-p. $$

    \end{itemize}

\end{itemize}
\end{definition}

The $B1$-condition is a smoothness condition of the distance see Figure \ref{fig: B1 condition }, whereas the $B2$-condition refers to convexity see Figure \ref{fig: B2 condition }.

\begin{figure}[!ht]
    \centering
   \includegraphics[scale=0.4]{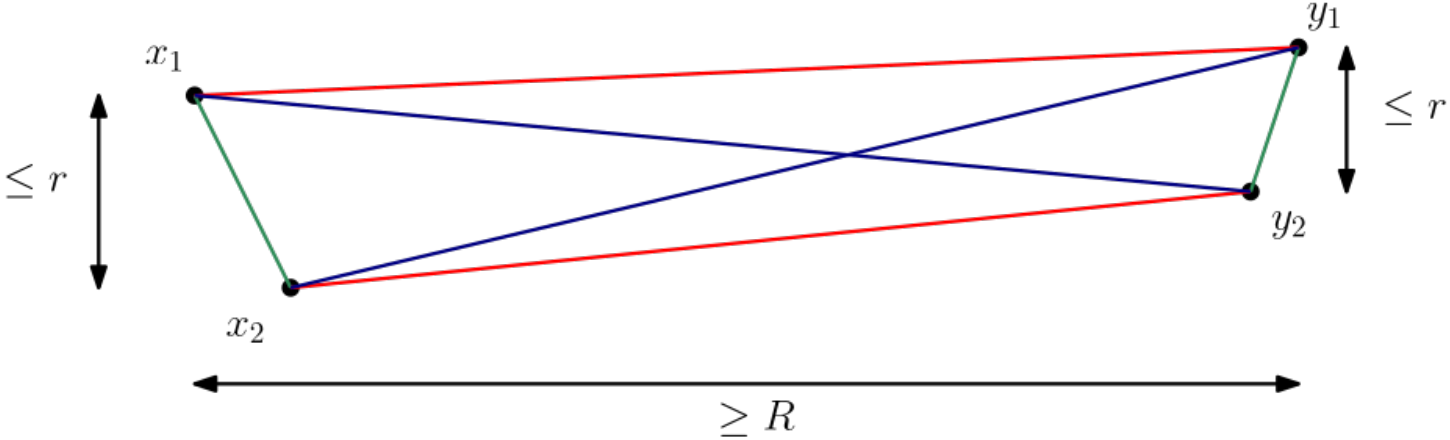}
  \caption{$B1$ condition: smoothness}
   \label{fig: B1 condition }
\end{figure}

\begin{figure}[!ht]
   \centering
   \includegraphics[scale=0.4]{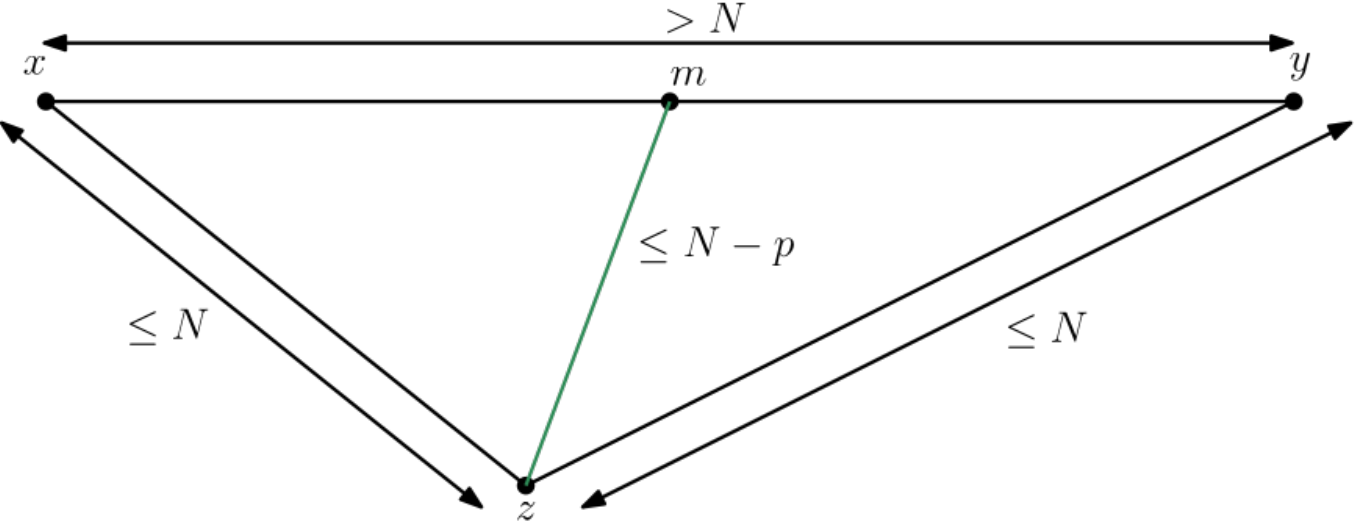}
   \caption{Part of $B2'$ condition: convexity}
   \label{fig: B2 condition }
 \end{figure}

Examples of strongly bolic metric spaces include $L^{p}$ spaces for $1<p< \infty$ as well as $CAT(0)$ spaces. More recently, Haettel, Hoda and Petyt showed in \cite[Theorem 4.18]{ThomasHodaPetytlpcomplexes} that under certain conditions, a cell complex equipped with the piecewise $\ell^{p}$ metric, $1<p< \infty$, is strongly bolic.

Using the fact that $CAT(0)$ spaces are strongly bolic, we know that $CAT(0)$ groups admit a strongly bolic action. As explained in \cite[Chapter 2]{ValetteIntroBaumConnes}, this class of groups contains, in particular, groups, which act properly, cocompactly, isometrically on either a Euclidiean building, or Riemannian symmetric space. Thus, Vincent Lafforgue noted that 'classical' hyperbolic groups, i.e cocompact lattices of rank-one simple Lie groups ($SO(n,1),SU(n,1),Sp(n,1), F_{4(-20)})$ admit a strongly bolic action.

To apply Theorem \ref{theo : Baum-Connes=Rapid decay + Strongly Bolic}, we need to find  a strongly bolic, weakly geodesic metric. Let us note here that being weakly geodesic implies the existence of quasi-midpoints.

\begin{proposition}\label{proposition : midpoint}
Let $\eta \geq 0$, let $(X,d)$ be a $\eta$-weakly geodesic metric space, then $(X,d)$ admits $\eta$-midpoints.

\end{proposition}

\begin{proof}
 For every $x,y \in X$, there exists $z$ such that:
$$|\hat{d}(x,z)-\frac{\hat{d}(x,y)}{2}|\leq \eta,$$
$$|\hat{d}(z,y)-\frac{\hat{d}(x,y)}{2}|\leq \eta ,$$

this implies directly that $(X,\hat{d})$ admits a $\eta$-midpoint map.

\end{proof}

\paragraph{Strong bolicity for hyperbolic groups.} A very interesting class of groups admitting a strongly bolic action are hyperbolic groups. Mineyev and Yu \cite{MineyevYuStrongBolicityhypergroups} showed that for any hyperbolic group $G$ and any of its Cayley graphs $X$, $X$ can be equipped with a strongly bolic $G$-equivariant metric. Moreover, this metric is quasi-isometric to the graph metric of $X$. This result implies the Baum-Connes conjecture for hyperbolic groups and their subgroups.
A simpler proof of this fact was given by Haïssinsky and Matthieu \cite{HaissinskyMatthieu}. Their proof uses the Green function associated to a random walk on the Cayley graph and the fact that for this metric, the Gromov boundary and the horofunction bordification are the same. Another proof of this fact is given by Lafforgue in \cite{LafforgueBaumConnesàcoeff} where he proves that hyperbolic groups satisfy a stronger version of the Baum-Connes conjecture, the Baum-Connes conjecture with coefficients.

In Section \ref{section : bolicité forte pour gp hyperbolique}, we will give a new proof of the fact that hyperbolic groups admit a strongly hyperbolic metric using Proposition \ref{Proposition: angle for hyperbolic metric space}. Then, we will generalize this method to relatively hyperbolic groups whose parabolic subgroups are $CAT(0)$.

\paragraph{Obstruction for actions on strongly bolic metric spaces.}

By studying, once again, the relation between the Gromov boundary and the horofunction bordification, Haettel, Hoda and Petyt identified in \cite{ThomasHodaPetytlpcomplexes} an obstruction to the existence of a semi-simple action on a strongly bolic metric space. This is a splitting result that generalizes a well-known property of $CAT(0)$-spaces \cite[Theorem 6.12]{BridsonHaefliger}. 

\begin{theo}\label{theo: obstructionactionbolic}

Let $G$ be a finitely generated group acting by semi-simple isometries on $X$, a strongly bolic metric space.
If $z$ is a central element in $G$, then some finite index subgroup of $G$ admits $\langle z \rangle$ as a direct factor. 
\end{theo}

As a corollary \cite[Corollary 4.20.]{ThomasHodaPetytlpcomplexes}, mapping class groups of surfaces of genus at least $3$ do not act by semi-simple isometries on strongly bolic metric spaces.

In the same way, nilpotent groups do not act by semi-simple isometries on a strongly bolic space. Indeed, the Heisenberg group with integer coefficients $H_{3}(\mathbb{Z)}$ is nilpotent, has $\mathbb{Z}$ as its center, but does not virtually split. Thus, to generalize the proof of Section \ref{section : bolicité forte pour gp relativement hyperbolique} to the non-uniform lattices of $SU(n,1),Sp(n,1), F_{4(-20)}$, it would be necessary to construct non-semi-simple actions of the parabolic subgroups.

\section{Hyperbolicity and relatively hyperbolic groups} \label{section : groupe relativement hyperbolic}

In this section, we recall the notions of hyperbolicity and relative hyperbolicity used in this work.

\subsection{Gromov hyperbolicity}

Let $(X,d)$ be metric space and $x,y \in X$, a geodesic from $x$ to $y$ is a map $\phi: I \mapsto X$, where $I=[a,b]$ is an interval of $\mathbb{R}$ such that $\phi$ is an isometry, $\phi(a)=x$ and $\phi(b)=y$. We will often denote by $[x,y]$ a geodesic between $x$ and $y$, even if it is not necessarily unique.

A metric space $(X,d)$ is said to be geodesic if for all $x,y \in X$, there exists a geodesic from $x$ to $y$.  Let assume that $(X,d)$ is geodesic, for all $x,y,z \in X$ a geodesic triangle is the union of three geodesics $[x,y]\cup[y,z]\cup[z,y]$ between those three points. We will often denote by $[x,y,z]$ a geodesic triangle between $x,y$ and $z$ even if it is not necessarily unique.

\begin{definition}\label{definition: hyperbolicite triangle fin}

A geodesic metric space $(X,d)$ is \emph{$\delta$-hyperbolic} for some $\delta \geq 0$ if all geodesic triangles $[x,y,z]$ are $\delta$-thin, i.e. every side of the triangle is contained in the $\delta$-neighbourhood of the union of the other two sides.\\

A geodesic metric space $X$ is said to be \emph{hyperbolic} if there exists $\delta \geq 0$ such that $X$ is $\delta$-hyperbolic.
\end{definition}

The assumption of being geodesic is not mandatory to define hyperbolicity.
We recall the definition of the Gromov-product, let $X$ be a metric space and $x,y,z \in X$, $$(x,z)_{y}=\frac{1}{2}(d(x,y)+d(y,z)-d(x,z)).$$
Note that the Gromov products are $1$-Lipschitz in the three variables.

\begin{definition}\label{definition: hyperbolicite condition des quatres points}

Let $\delta>0$, a metric space $X$ is \emph{$\delta$-hyperbolic} in the sense of Gromov if for all four points $x_{0},x_{1},x_{2},x_{3} \in X$,
$$ (x_{1},x_{3})_{x_0}\ge \min((x_{1},x_{2})_{x_{0}},(x_{2},x_{3})_{x_{0}}) -\delta . $$
\end{definition}

According to \cite[Proposition 21.]{GhysDelaHarpe}, for a geodesic metric space $(X,d)$,
if triangles in $X$ are $\delta$-thin then it is $8\delta$-hyperbolic in the sense of Gromov and conversely if $X$ is $\delta$-hyperbolic in the sense of Gromov, triangles are $4\delta$-thin. From now on, we will speak almost exclusively of geodesic spaces. 

Note that the Gromov product between three points has a particular interpretation in a geodesic hyperbolic space. For all $x,y,z \in X$ in a hyperbolic geodesic metric space, $(x,y)_{z}$ is roughly the distance from $z$ to a geodesic between $x$ and $y$.

\begin{lem} \label{lemma : DrutuKapovichcomparaisonproduitdeGromovet distanceàunegeodesique}\cite[Lemma 11.22]{DrutuKapovichGeometricGroupTheory}

Let $X$ be a geodesic $\delta$-hyperbolic space. For every $x,y,z \in X$ and every geodesic $[x,y]$, we have :
$$(x|y)_{z}\le d(z,[x,y]) \le (x|y)_{z}+2\delta . $$
\end{lem}

In Section \ref{section : bolicité forte pour gp hyperbolique} and Section \ref{section : bolicité forte pour gp relativement hyperbolique}, we will also use the following lemma about geodesics in hyperbolic space.

\begin{lem} \cite[Lemma 1.15]{BridsonHaefliger} \label{lemme de Bridson} 

Let $X$ be a geodesic space such that geodesic triangles are $\delta$-thin. Let $\gamma,\gamma' : [0,T] \mapsto X$ be geodesics with $\gamma(0)=\gamma'(0)$. If $d(\gamma(t_{0}),im(\gamma'))\le K$, for some $K>0$ and $t_{0} \in [0,T]$, then $d(\gamma(t),\gamma'(t))\le 2\delta$ for all $t \le t_{0}-K-\delta$.

\end{lem}

\subsection{Visual pseudo-metric in hyperbolic spaces}\label{subsubsection: angles dans espaces hyperboliques}
Here we describe a less conventional property of hyperbolic spaces, which will be used in Section \ref{section : bolicité forte pour gp hyperbolique}. To metrize $\partial X$, the metric traditionally used is the visual metric, which is a metric comparable to $e^{-\varepsilon(x|y)_{o}}$, where $o$ is a base point and $\varepsilon$ is a parameter. The result we will mention here is a generalization of \cite{AlvarezLafforgueQuotienthyperbolic} and allows to construct a visual metric on an entire hyperbolic space, not just on its boundary.
This notion of angle generalizes the angle of hyperbolic manifolds coming from their CAT(0) geometry.

\begin{proposition}\cite[Proposition 3.2.]{anglehyperbolicHaettelChatterjiDahmaniLecureux}\label{Proposition: angle for hyperbolic metric space}

Fix a $\delta$-hyperbolic geodesic space $X$, and choose $\varepsilon, D > 0$ such that $0\leq \varepsilon \leq \frac{log(2)}{\delta+D}$. Let $\alpha=De^{-2D\varepsilon}$ and $\beta= \frac{8}{\varepsilon}$. For any $a\in X$, there exists a pseudo-distance $d^{a}_{\varepsilon}$ on $X$ such that :
\begin{itemize}
\item $d^{a}_{\varepsilon}(x,y) \leq \beta e^{-\varepsilon (x|y)_{a}}$ for all $x,y \in X$,

\item $\alpha e^{-\varepsilon (x|y)_{a}}\leq d^{a}_{\varepsilon}(x,y)$ for every $x,y \in X$ with $d(x,y) \geq 2D$,

\item for all $g \in Isom(X)$, for all $x,y \in X$, $d^{g.a}_{\varepsilon}(g.x,g.y)=d^{a}_{\varepsilon}(x,y)$.

\end{itemize}
\end{proposition}

\subsection{Relatively hyperbolic groups}\label{subsection: groupes relativement hyperboliques}

In this subsection, we will describe some useful tools to work on relatively hyperbolic groups.

\subsubsection{Fine graphs and relatively hyperbolic groups}

In this paragraph, we give a definition of relatively hyperbolic groups owing to Farb \cite{Farbrelhyp} and Bowditch \cite{Bowditchrelhyp}. The definition is based on a graph constructed from the Cayley graph, the coned-off graph.

\begin{definition}\label{definition: coned-off}(Coned-off graph)

Let $G$ be a finitely generated group and $H_{1},...,H_{k}$ subgroups of $G$. We consider $Cay(G,S)$ a Cayley graph of $G$ with respect to a finite generating set $S$. The coned-off of $G$ relatively to the subgroups $H_{1},...,H_{k}$ is the graph constructed as follows : for each $i$ and for each left coset $gH_{i}$ of $H_{i}$, add a vertex $\widehat{gH_{i}}$ and for each $h \in H_{i}$, add an edge starting at $\widehat{gH_{i}}$ and ending at $gh$. We denote this graph by $\widehat{Cay}(G,S)$.\end{definition}

\begin{figure}[!ht]
    \centering
   \includegraphics[scale=0.6]{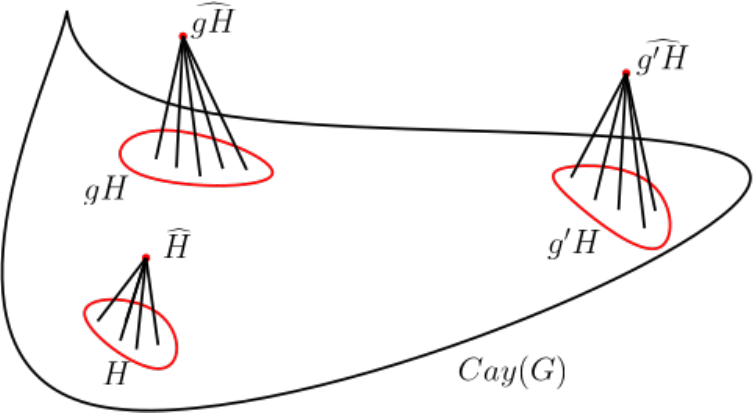}
   \caption{Coned-off graph}
  \label{Coned-off graph}
  \end{figure}

Let us remark that, when the subgroups $H_{i}$ are infinite, the coned-off graph is locally infinite, see Figure \ref{Coned-off graph}.

Informaly speaking, if the subgroups $H_{i}$, $i \in \{1,...,n\}$ and their cosets are the only obstructions of $G$ to be hyperbolic, then this graph must be hyperbolic because each translate of $H_{i}$ has a diameter $2$ for the metric of the new graph. However, in the definition of relatively hyperbolic, we require an additional combinatorial condition on the coned-off graph, called \emph{uniformly finesse}.

\begin{definition}{(Fine graph)} \label{definition: graph fin} \cite[Proposition 2.1.]{Bowditchrelhyp}

 A graph is \emph{fine} if for every edge $e$, for all $L\ge0$, the set of simple simplicial loops of length at most $L$ which contain $e$ is finite. \\

 It is \emph{uniformly fine} if this set has cardinality bounded above by a function depending only on $L$. We denote by $\varphi$ this function of uniform finesse.

\end{definition}

The followings graph are examples of uniformly fine graphs.

\begin{ex}
\begin{itemize}
    \item Any locally finite graph is fine.
    
    \item Trees are fine, since they have no loops.
\end{itemize}
\end{ex}

We remark that finesse is not invariant by quasi-isometry. In fact, $\widehat{Cay}(G,S)$ is quasi-isometric to $Cay(G,S\cup H_{1}\cup...\cup H_{k})$ which is not uniformly fine. Thanks to this quasi-isometry, $Cay(G,S\cup H_{1}\cup...\cup H_{k})$ is hyperbolic if and only if $\widehat{Cay}(G,S)$ is hyperbolic. Therefore, $Cay(G,S\cup H_{1}\cup...\cup H_{k})$ is sometimes used as a model of the coned-off graph, for example, in \cite{OsinRelativelyHyperbolic}.

Let us give now Bowditch's definition \cite[Theorem 7.10.]{Bowditchrelhyp} of relative hyperbolicity.

\begin{definition} \label{definition: groupe relativement hyperbolique}
 
A pair $(G,{H_{1},...,H_{k}})$ of a finitely generated group $G$ and a collection of proper subgroups is relatively hyperbolic, if one (or equivalently every) coned-off Cayley graph $\widehat{Cay}(G)$ over $H_{1},...,H_{k}$ is hyperbolic and uniformly fine.

 \end{definition}

The uniform finesse condition can be replaced by the Bounded Coset Penetration $(BCP)$ defined in \cite{Farbrelhyp}.

When the coned-off Cayley-graph is hyperbolic but not uniformly fine (or does not satisfy the $(BCP)$ property), we refer to weakly relatively hyperbolic groups in \cite{Farbrelhyp}. An interesting example of a weakly relatively hyperbolic group is the mapping class group of a finite type surface. According to Masur and Minsky \cite[Theorem 1.3.]{MasurMinsky}, the coned-off graph of the mapping class groups with respect to a finite collection of stabilizers of curves is hyperbolic, however it is not uniformly fine, if the surface is not a one-punctured torus or a four-times punctured sphere. Moreover, the article \cite[Proposition 10] {AndersonShackletonKennethZAramayonaCriterePasRelHyp} shows that the mapping class group is not relatively hyperbolic with respect to any finite set of infinite-index subgroups. One difficulty of working on relatively hyperbolic groups comes from the fact that these graphs are not locally finite. Angles and cones, which we define in Paragraph \ref{subsubsection: angles and cones}, allow us to solve this difficulty.

\subsubsection{Angles and cones}\label{subsubsection: angles and cones}

In the next two paragraphs, we return to uniformly fine graphs. We begin by discussing very important objects in these graphs: angles and cones. Cones will play the role of neighborhoods in uniformly fine graphs, with the advantage that their cardinality is uniformly controlled. The majority of the results presented here come from \cite{ChatterjiDahmani}, and we prove some other. We start with the definition of angles.

\begin{definition}\label{definition: angle} {(Angle)}

Let $X$ be a graph. Given a vertex $v \in X^{(0)}$ and two unoriented edges $e_{1}=\{v,w\}$ and $e_{2}=\{v,w'\}$ such that  $v \in e_{1}\cap e_{2}$, the \emph{angle} between $e_{2}$ and $e_{1}$ at $v$, denoted by $\measuredangle_{v}(w,w')$, is defined as follows :
$$ \measuredangle_{v}(e_{1},e_{2})=d_{X   \backslash \{ v \} }(w,w').$$

In other words, the \emph{angle} between $e_{1}=\{v,w\}$ and $e_{2}=\{v,w'\}$ at $v$ is the infimum of length of paths from $w$ to $w'$ that avoid $v$.
In particular, the angle could be equal to $+\infty$.

By convention, we will use the following abuse of notation : if $x_{1}, x_{2}$ and $v$ are vertices of $X$, distinct from $v$, we say that $ \measuredangle_{v}(x_{1},x_{2}) > \theta$ if there exist two edges $e_{1}$, $e_{2}$, such that $ v \in e_{1}\cap e_{2} $,  with $e_{i}$ on a geodesic between $v$ and $x_{i}$ $(i=1,2)$ and such that $ \measuredangle_{v}(e_{1},e_{2}) > \theta$. 
We say that $ \measuredangle_{v}(x_{1},x_{2}) \leq \theta$ otherwise, i.e. for all geodesics between $v$ and $x_{i}$ $(i=1,2)$ and starting by the edges $e_{i}$ $(i=1,2)$ such that $v \in e_{1}\cap e_{2}$, then $\measuredangle_{v}(e_{1},e_{2})\leq \theta$.
\end{definition}

\begin{figure}[!ht]
    \centering
   \includegraphics[scale=0.6]{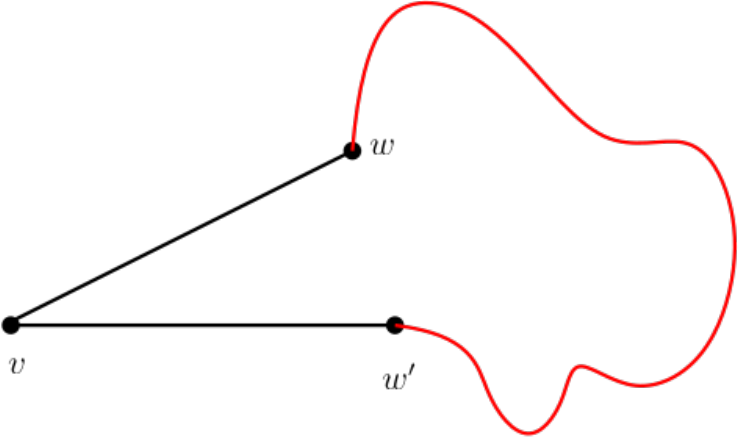}
    \caption{Angles between $e_{1}=\{v,w\}$ and $e_{2}=\{v,w'\}$ seen from $v$}
    \label{ fig: angles }
 \end{figure}

It should be noted that the angle between two points $a$ and $b$ viewed from $c$ depends only on the first edge of the geodesics connecting $a$ to $c$ and, respectively, connecting $b$ to $c$.

 \begin{proposition} \label{proposition: triangle inequality} (Triangle inequality for angles)\\ For all edges $e_{1}$, $e_{2}$ and $e_{3}$ and vertex $v$ such that $ v \in e_{1}\cap e_{2} \cap e_{3}$ , we have:
$$ \measuredangle_{v}(e_{1},e_{3})\le \measuredangle_{v}(e_{1},e_{2}) + \measuredangle_{v}(e_{2},e_{3}) .$$

 Let $a,b,c \in X^{(0)} \setminus \{v\}$, we have also:
 $$ \measuredangle_{v}(a,b)\le \measuredangle_{v}(a,c) + \measuredangle_{v}(c,b) .$$

 \end{proposition}

 \begin{proof}
     The first property comes directly from the triangle inequality in $X \setminus \{v\}$.\\

     To deduce the second inequality with vertices, we just have to come back to the definition of angle and use the first point.

 \end{proof}

In a $\delta$-hyperbolic graph, large angles at a vertex force geodesics to pass through this vertex.

\begin{proposition}\label{proposition: les propirétés des angles}\cite[Proposition 1.2]{ChatterjiDahmani}  
Let $X$ be a $\delta$-hyperbolic graph and $a,b,c$ vertices in $X$, $a,b \neq c$.
\begin{itemize}
    \item if $\measuredangle _{c}(a,b) > 12\delta$, then every geodesic between $a$ and $b$ goes through $c$;

     \item if $e_{1}$, $e_{2}$ are two edges such that $c  \in e_{1}  \cap e_{2}$ and that are on geodesics from $c$ to $a$, then $\measuredangle _{c}(e_{1},e_{2}) \leq 6 \delta$;

     \item for any $\theta \geq 0$, if $\measuredangle _{c}(a,b) > \theta + 12\delta$ then for all choice of edges $e_{1}$, $e_{2}$ at $c$ starting geodesics respectively to $a$ and $b$, $\measuredangle _{c}(e_{1},e_{2}) \geq \theta$.
\end{itemize}
\end{proposition}

Following Chatterji and Dahmani in \cite{ChatterjiDahmani}, let us define cones.

\begin{definition} {(Cone)}\\
Let $X$ be a graph. Let $e$ be an edge
and a number $\theta>0$. 

The cone of parameter $\theta$ around the edge $e$, denoted by $Cone_{\theta}(e)$, is the subset of vertices and edges of $X$, $v$ and $e'$ such that there is a path of length at most $\theta$ from $e$ to $e'$ and for which two consecutive edges make an angle at most $\theta$, i.e.: 
$$\begin{aligned} 
         Cone_{\theta}(e) & =\{e' ~|~ \exists e_{0}=e, e_{1}, ..., e_{n}=e'\\
       & \text{such that} ~ n\le \theta,  \forall i , \exists v_{i}\in e_{i} \cap e_{i+1} \text{and} \measuredangle_{v_{i}}(e_{i},e_{i+1}) \le \theta \} \\
         & \cup \{ v ~|~ \exists e'=\{v,w\} ~\text{and}~ e_{0}=e, e_{1}, ..., e_{n}=e' \\
         & \text{such that}~ n\le \theta,  \forall i , \exists v_{i}\in e_{i} \cap e_{i+1} \text{and} \measuredangle_{v_{i}}(e_{i},e_{i+1}) \le \theta \} .
\end{aligned}$$

\end{definition}

\begin{proposition} \label{cardinal cones}\cite[Proposition 1.10]{ChatterjiDahmani} 
In a fine graph, cones are finite.

If the graph is uniformly fine, for all $e\in  X^{(1)} $ and $\theta > 0$, the cardinality of $Cone_{\theta}(e)$ can be bounded above by a function $f$ of $\theta$, called \emph{function of uniforme finesse} of $X$.

\end{proposition}  

The following property allows to use cones in relatively hyperbolic groups as a sort of finite neighborhoods of geodesics in the definition of thin triangles.

\begin{proposition} \cite[Proposition 1.10]{ChatterjiDahmani} \label{conetriangle} {(Cones and triangles)} 

In a $\delta$-hyperbolic graph, geodesic triangles are conically fine, i.e. for any geodesic triangle $[a,b,c]$, every edge $e$ on $[a,b]$ is contained in a cone of parameter $50\delta$ around an edge $e'$ that is either on $[a,c]$ at the same distance of $a$ than $e$, either on $[b,c]$ at the same distance of $b$ than $e$. 
\end{proposition}

A corollary of this proposition is the fact that in a uniformly fine graph intervals are uniformly finite even if the graph is not locally finite.
Let us recall, that for a simplicial graph $X$ and two vertices $x,x'$, $I(x,x')$ is the set of vertices which belong to geodesic between $x$ and $x'$, in other words $I(x,x')=\{ u \in X ~ | ~ d(x,u)+d(u,x')=d(x,x') \}$.

\begin{corollary}

In a uniformly fine and $\delta$-hyperbolic graph $X$, for all $x,x' \in X$,  the cardinality of $I(x,x')$ is bounded above by a function depending of $d(x,x')$, of $\delta$ and of the function of uniform finesse $f$.  
\end{corollary}

\subsubsection{Distance formula}\label{subsubsection: distance formula}
In this section, we discuss about the distance formula for relatively hyperbolic groups, proved by Sisto in \cite{SistoFormuleDelaDistance}. We give the version presented in \cite{ChatterjiDahmani}: the word metric is equivalent to the coned-off metric to which we add the sum of angles at infinite valence vertices.

\begin{proposition} \cite[Proposition 1.6.]{ChatterjiDahmani} \label{Proposition : formule de la distance}

Let $G$ be a group hyperbolic relative to  $(H_{1},...,H_{k})$, let $d_{G}$ denote the word metric, for a finite generating set.

Let $X_{c}$ denote its coned-off graph with respect to the same generating set, with $d_{c}$ its associated graph metric.

There are $A,B>1$ for which, for all $g,h \in G$, for each choice of geodesic $[g,h]_{c}$ in $X_{c}$, if $\Theta(g,h)$ denotes the sum of angles of edges of $[g,h]_{c}$ at vertices of infinite valence, then :

$$ \frac{1}{A} d_{G }(g,h)- B \leq d_{c}(g,h) + \Theta (g,h) \leq A d_{G}(g,h) +B.$$

\end{proposition}

We will use in Section \ref{section : bolicité forte pour gp relativement hyperbolique} the following reformulation of the distance formula.

\begin{corollary}\label{corollary : reformulation de la formule de la distance}

Let $G$ be a group hyperbolic relative to  $(H_{1},...,H_{k})$, let $d_{G}$ denote the word metric, for a finite generating set.

Let $X_{c}$ denote its coned-off graph with respect to the same generating set, with $d_{c}$ its associated graph metric.

Let $M>0$. For all $g,h \in G$, $\Theta(g,h)$ denotes the sum of angles of edges of a geodesic $[g,h]_{c}$ at vertices of infinite valence and $\Theta_{>M}(g,h)$ denotes the sum of angles greater than $M$ of edges of $[g,h]_{c}$ at vertices of infinite valence.

Then $d_{c}+\Theta_{>M}$ is quasi-isometric to $d_{G}$.
\end{corollary}

\begin{proof}
We will compare $d_{c}+\Theta_{>M}$ with $d_{c}+\Theta$ and conclude with Proposition \ref{Proposition : formule de la distance}.

Obviously, for all $g,h \in G$:
$$d_{c}(g,h)+\Theta_{>M}(g,h) \leq  d_{c}(g,h)+\Theta(g,h) .$$

On the other hand, if $\Theta_{\leq M}(g,h)$ denotes the sum of angles of edges of $[g,h]_{c}$ at vertices of infinite valence lower than $M$, we have:
$$  \Theta_{\leq M}(g,h) \leq M d_{c}(g,h).$$

Therefore:
$$ d_{c}(g,h)+\Theta(g,h) \leq d_{c}(g,h)+ M d_{c}(g,h) + \Theta_{>M}(g,h)  \leq M(d_{c}(g,h) + \Theta_{>M}(g,h)) .$$

Then, according to Proposition \ref{Proposition : formule de la distance}, $d_{c}+\Theta_{>M}$ is quasi-isometric to $d_{G}$.

\end{proof}

The following proposition shows that for two vertices of the coned-off graph, the sum of the angles along two geodesics differs linearly, this will be used in Section \ref{section : bolicité forte pour gp relativement hyperbolique}.

\begin{proposition}\label{proposition: comparaison de la somme des angles le long de deux géodésiques}
Let $G$ be a relatively hyperbolic group and $X_{c}$ one of its coned-off graph. For any $M>0$, there exists $\alpha_{1},\alpha_{2},\beta_{2}>0$ such that for all $x, y \in X_{c}$ and any two geodesics $\gamma$ and $\omega$ between those two points, if we denote by:
\begin{itemize}
\item $\Theta_{\gamma}(x,y)$, $\Theta_{\omega}(x,y)$ the sum of angles along $\gamma$, $\omega$ respectively,

\item $d'_{\gamma}(x,y)=d_{c}(x,y)+\Theta_{\gamma}(x,y)$, $d'_{\omega}(x,y)=d_{c}(x,y)+\Theta_{\omega}(x,y)$,

    \item $\Theta_{>M,\gamma}(x,y)$, $\Theta_{>M,\omega}(x,y)$ the sum of angles greater than $M$ along $\gamma$, $\omega$ respectively,

    \item $d'_{>M,\gamma}(x,y)=d_{c}(x,y)+\Theta_{>M,\gamma}(x,y)$, $d'_{>M,\omega}(x,y)=d_{c}(x,y)+\Theta_{>M,\omega}(x,y)$.
    
\end{itemize}

We have:

\begin{itemize}
    \item  
$ \frac{1}{\alpha_{1}}d'_{\omega}(x,y) \leq d'_{\gamma}(x,y) \leq \alpha_{1} d'_{\omega}(x,y).$

    \item $ \frac{1}{\alpha_{2}}d'_{>M,\omega}(x,y)-\beta_{2} \leq d'_{>M,\gamma}(x,y) \leq \alpha_{2} d'_{>M,\omega}(x,y)+\beta_{2}.$
\end{itemize}

\end{proposition}

\begin{proof}

Let us denote $n:=d_{c}(x,y)$. We order the two geodesics in the following way:
\begin{itemize}
    \item $[x,y]_{c,\gamma}=\{u_{0},u_{1},...,u_{n}\}$ with $u_{0}=x$, $u_{n}=y$ and $d_{c}(u_{i},u_{i+1})=1$, for $0\le i \le n-1$,

    \item $[x,y]_{c,\omega}=\{v_{0},v_{1},...,v_{n}\}$ with $v_{0}=x$, $v_{n}=y$ and $d_{c}(v_{i},v_{i+1})=1$, for $0\le i \le n-1$.
\end{itemize}

 Let $i_{0} \in  \{1,\dots,n-1\}$, there are two cases to consider.

 \begin{itemize}
     \item Either $u_{i_{0}}=v_{i_{0}}$, then according to the second point of Proposition \ref{proposition: les propirétés des angles}, we have:
$$\measuredangle_{u_{i_{0}}}(u_{i_{0}-1},u_{i_{0}+1}) \leq \measuredangle_{v_{i_{0}}}(v_{i_{0}-1},v_{i_{0}+1})+12\delta. $$

     \item Otherwise $ u_{i_{0}}\neq v_{i_{0}} $, then according to the first point of Proposition \ref{proposition: les propirétés des angles}, we have: $ \measuredangle_{v_{i_{0}}}(v_{i_{0}-1},v_{i_{0}+1})\leq 12 \delta$ and $\measuredangle_{u_{i_{0}}}(u_{i_{0}-1},u_{i_{0}+1})\leq 12 \delta$.
 \end{itemize}
    
Then, for all $i_{0} \in  \{1,\dots,n-1\}$, we have:
$$ \measuredangle_{u_{i_{0}}}(u_{i_{0}-1},u_{i_{0}+1}) \leq \measuredangle_{v_{i_{0}}}(v_{i_{0}-1},v_{i_{0}+1})+12\delta, $$
 and therefore: 

 $$ \Theta_{\gamma}(x,y) \leq \Theta_{\omega}(x,y)+12 \delta. $$

Thus:
$$d'_{\gamma}(x,y)\leq 13\delta d'_{\omega}(x,y). $$
Then by symmetry we get:
$$\frac{1}{13\delta} d'_{\omega}(x,y)\leq d'_{\gamma}(x,y)\leq 13\delta d'_{\omega}(x,y), $$

and $\alpha_{1}=13\delta.$

For the second point, we conclude with an application of Corollary \ref{corollary : reformulation de la formule de la distance}.

\end{proof}

\subsubsection{Triangles and angles}\label{subsubsection: triangles and angles}

In this paragraph, we will mention the existence of a particular geodesic triangle between any three points in a fine and hyperbolic graph. Throughout this paragraph, $X$ denotes a $\delta$-fine hyperbolic graph.

\begin{proposition}\cite[Proposition 4.18]{Lajoinie-Dodel}
    
\label{pointloingrandangleexistence}
Let $a,b,c \in X^{(0)}$, we consider the set of points $v$ with the following properties:

\begin{itemize}
    \item $v$ is contained in every geodesic between $a$ and $b$ and between $a$ and $c$,
    \item $\measuredangle_{v}(a,b)> 50\delta$,
    \item $\measuredangle_{v}(a,c)> 50\delta$.
\end{itemize}

If this set is not empty, it contains a unique point at maximal distance of $a$.

\end{proposition}

\begin{definition} \cite[Definition 4.19]{Lajoinie-Dodel}\label{pointloingrandangle}
Let $a,b,c \in X^{(0)}$, according to Proposition \ref{pointloingrandangleexistence}, we define $\tilde{a}$ as the furthest point from $a$ among the points $v$ with the following properties:

\begin{itemize}
    \item $v$ is contained in every geodesic between $a$ and $b$ and between $a$ et $c$,
    \item $\measuredangle_{v}(a,b)> 50\delta$,
    \item $\measuredangle_{v}(a,c)> 50\delta$.
\end{itemize}

If the set of points following these three properties is empty, we set $\tilde{a}=a$.\\

We define $\tilde{b}$ and $\tilde{c}$ in the same way.

\end{definition}

\begin{proposition}\cite[Proposition 4.20]{Lajoinie-Dodel}(see Figure \ref{fig: atilde est plus prÃ¨s de a que ctilde_1})

Let $a,b,c \in X^{(0)}$, we have $d(a,\tilde{a})\leq d(a,\tilde{c})$. In other words, when we travel along a geodesic from $a$ to $c$, we pass through $a$, $\tilde{a}$, $\tilde{c}$ and $c$ in this order.

\end{proposition}

\begin{figure}[!ht]
   \centering
   \includegraphics[scale=0.7]{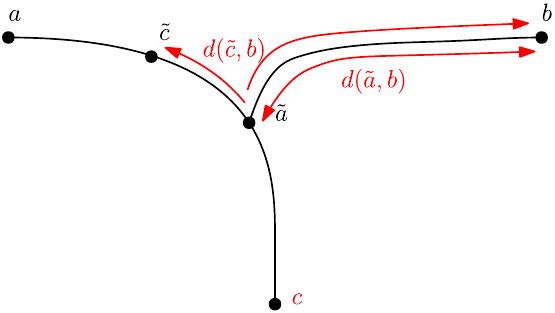}
   \caption{ $\tilde{a}$ is closer to $a$ than $\tilde{c}$ }
    \label{fig: atilde est plus prÃ¨s de a que ctilde_1}
\end{figure}

\newpage
\begin{theo} \cite[Theorem 4.21]{Lajoinie-Dodel}\label{formenormaledestriangles}(see Figure \ref{fig: forme normale triangle})

Let $a,b,c \in X^{(0)}$ and a geodesic $[a,b]$ between $a$ and $b$. There exists a triangle 
$[a,b,c]$ whose side between $a$ and $b$ is 
$[a,b]$ such that :
\begin{itemize}
    \item the sides $[a,b]$ and $[a,c]$ coincide between $a$ and $\tilde{a}$,
    \item the sides $[b,c]$ and $[b,a]$ coincide between $b$ and $\tilde{b}$,
    \item the sides $[c,a]$ and $[c,b]$ coincide between $c$ and $\tilde{c}$,
    \item for each $v \in ]\tilde{a},\tilde{b}[$, if we denote by $e_{1}$ and $e_{2}$ the two different edges of $[\tilde{a},\tilde{b}]$ such that $v \in e_{1} \cap e_{2}$ then $\measuredangle_{v}(e_{1},e_{2}) \leq 100\delta$ and similarly this property is also true for $]\tilde{a},\tilde{c}[$ and $]\tilde{b},\tilde{c}[$.\\

\end{itemize}

\end{theo}

It should be noted in the proof of \cite[Theorem 4.21]{Lajoinie-Dodel} that one can prescribe a side of the triangle.

\begin{figure}[!ht]
   \centering
   \includegraphics[scale=0.4]{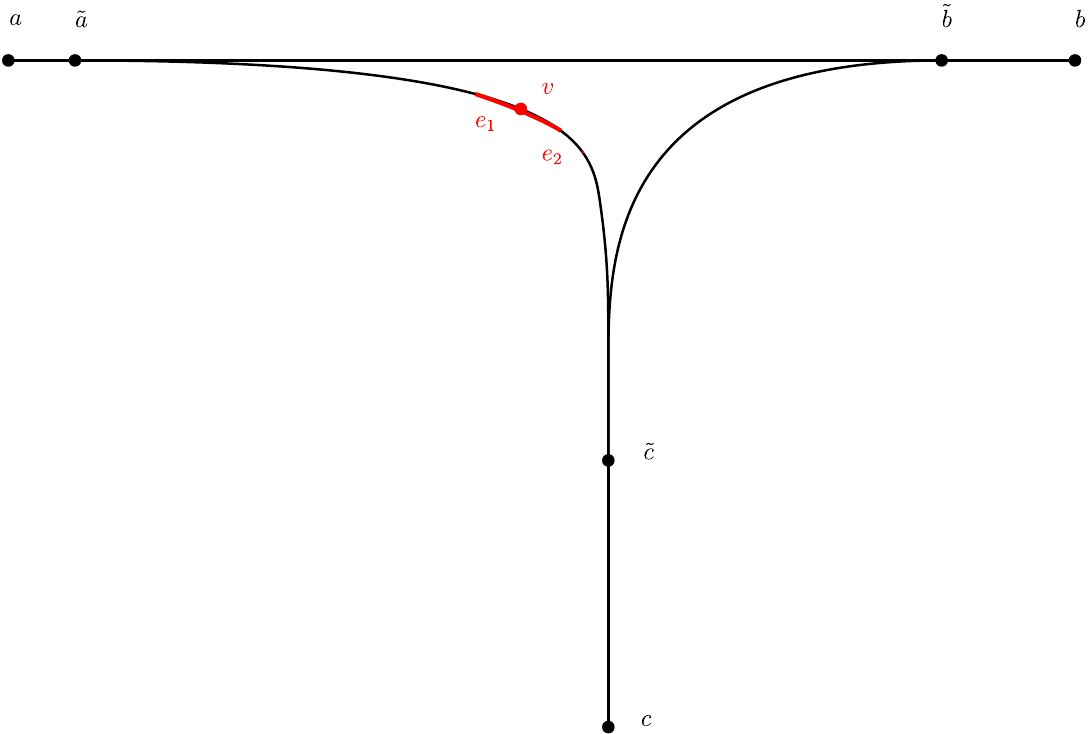}
   \caption{ Particular geodesic triangle between the three points $a$, $b$, $c$ with $\measuredangle_{v}(e_{1},e_{2}) \leq 100\delta$ }
    \label{fig: forme normale triangle}
\end{figure}

\begin{proposition} \label{proposition: angle en a tilde}
Let $a,b,c \in X^{(0)}$ such that $\tilde{a}\neq \tilde{b}\neq \tilde{c}$ and $[a,b,c]$ a triangle as in Theorem \ref{formenormaledestriangles}.
We set $e_{1}$ the edge on $[\tilde{a},\tilde{b}]$ such that $\tilde{a} \in e_{1}$ and $e_{2}$ the edge on $[\tilde{a},\tilde{c}]$ such that $\tilde{a} \in e_{2}$.
We have:
$$ \measuredangle_{\tilde{a}}(e_{1},e_{2})\leq 12 \delta.$$

\end{proposition}

\begin{proof}
The proof is similar to the proof of Proposition \ref{conetriangle} (see \cite{ChatterjiDahmani}).

\end{proof}

\section{Strong Bolicity for hyperbolic groups }\label{section : bolicité forte pour gp hyperbolique}

We will reprove in this section the fact that hyperbolic groups admit a strongly bolic metric. It is one of the two key steps to prove the Baum-Connes conjecture without coefficient, according to Theorem \ref{theo : Baum-Connes=Rapid decay + Strongly Bolic} of Lafforgue \cite{LafforgueBaumConnesConjecture}.

The fact that hyperbolic groups satisfy the rapid decay property was proved by de la Harpe  and Jolissaint in \cite{delaHarpeRD} and \cite{JolissaintRD}.

For the fact that hyperbolic groups admit a strongly bolic metric, the following theorem was proved by Mineyev and Yu \cite{MineyevYuStrongBolicityhypergroups}. This result was reproved by Haïssinsky and Matthieu in \cite{HaissinskyMatthieu} using random walks and the Green metric.

\begin{theo}\cite[Theorem 17]{MineyevYuStrongBolicityhypergroups}\label{Theorem : Strongly Bolic metric for hyperbolic groups }
Every hyperbolic group $G$ admits a metric $\hat{d}$ with the following properties.

\begin{itemize}
    \item $\hat{d}$ is $G$-invariant, i.e. $\hat{d}(g.x,g.y)=\hat{d}(x,y)$, for all $g,x,y \in G$,
    
    \item $\hat{d}$ is quasi-isometric to the word metric,
    
    \item The metric space $(G,\hat{d})$ is weakly geodesic and strongly bolic.
\end{itemize}
    
\end{theo}

To reprove Theorem \cite{MineyevYuStrongBolicityhypergroups}, we will use the following notion of an angle defined in Proposition \ref{Proposition: angle for hyperbolic metric space}.

In any graph $(X,d)$, a projection of a vertex $x$ on a closed subset $F \subset X$ is a vertex $t \in F$ such that $d(x,F)=d(x,t)$.

In the rest of the section, $G$ will denote an hyperbolic group and $X$ will denote a Cayley graph of $G$ over a finite generating set. We set $d$ the metric on $G$,  $\delta$ the hyperbolic constant of $X$, $\tau$ the exponential growth rate of of $X$, i.e. $\tau= \displaystyle \limsup_{n\rightarrow \infty }{\frac{\ln{|B_{X}(x,n)|}}{n}}$ where $|B_{X}(x,n)|$ denotes the number of vertices in $B_{X}(x,n)$. For all $x,y \in X$, $[x,y]$ will denote a geodesic between $x$ and $y$. For all $x,y,z \in  X$, $[x,y,z]$ will denote a geodesic triangle between $x$, $y$ and $z$ i.e. a choice of three geodesics between each of these three points. For every closed set\\

We choose $\eta>0$. We fix $\varepsilon>0$ such that $0 < \varepsilon \leq \frac{log(2)}{\delta+\eta}$. According to Proposition \ref{Proposition: angle for hyperbolic metric space}, for any $a \in X$, there exists a pseudo-metric $d_{\varepsilon}^{a}$ on $X$ comparable to $e^{-\varepsilon (x|y)_{a}}$. Let us choose $p$ such that $p\varepsilon>\tau$. We fix $\tau' \in ( \tau, p \varepsilon)$, so that there exists $C>0$, such that for all $n \in \mathbb{N}$, for all $x \in X$, $|B_{X}(x,n)|\leq C e^{\tau' n}$. \\

We will need the following counting lemmas to prove further results.

\begin{lem} \label{lemma : denombrement de g tel que g.o est proche d'une géodesique}

There exists $C>0$, such that for all $x,y \in X$, for all $n \in \mathbb{N}$, we have:
$$  |\{g\in G  ~ |~ d(g.o,[x,y])= n \}| \leq C d(x,y) e^{\tau' n} . $$

\end{lem}

\begin{proof}

The result is clear with the following inclusion :
$$\{g\in G ~ |~ d(g.o,[x,y])= n \} \subset \bigcup_{x_{i} \in [x,y]} (B(x_{i},n) \cap G.o) .$$

We deduce that :
$$ |\{g\in G | d(g.o,[x,y])= n \}| \leq C d(x,y) e^{\tau' n} .$$
    
\end{proof}

\begin{lem}\label{lemma: somme des exponentielles distance à un point}
There exists a constant $\kappa>0$ such that for all $x \in X$:
$$ \sum_{ a\in X } e^{-\varepsilon pd(a,x)} \leq \kappa. $$

\end{lem}

\begin{proof}

Let $x \in X$:

$$\begin{aligned}
\sum_{ a\in X } e^{-\varepsilon pd(a,x)}&= \displaystyle \sum_{n \in \mathbb{N}} \sum_{\{a \in X |d(a,x)=n\}}  e^{-\varepsilon pn}\\
&\leq \sum_{n \in \mathbb{N}} C e^{\tau' n} e^{-\varepsilon p n }\\
&\leq \sum_{n \in \mathbb{N}} C e^{(\tau'-\varepsilon p) n}\\
&\leq \frac{C}{1-e^{\tau'-\varepsilon p}}~(\text{since } \tau'-\varepsilon p<0).
\end{aligned}$$

With $\kappa=\frac{C}{1-e^{\tau'-\varepsilon p}}$, we get the desired result.
\end{proof}

Here, we define $D$, which will be the strongly bolic metric up to a constant.

\begin{proposition} \label{proposition : D est fini et quasiisom à d}
For all $x,y \in X$,

$$D(x,y):=\sum_{a \in X} d_{\varepsilon}^{a}(x,y)^{p}$$ is finite.

Moreover, there exist $A,B>0$ such that, for all $x,y \in X$:

$$ \frac{1}{A}d(x,y)-B \leq D(x,y) \leq Ad(x,y) + B.$$

\end{proposition}

\begin{proof}
Let $x,y \in X$ :

$$\begin{aligned} 
           D(x,y) & = \sum_{a \in X} d_{\varepsilon}^{a}(x,y)^{p}\\      
            & \leq  \sum_{a \in X} \beta e^{-\varepsilon p (x|y)_{a}} ~(\text{according to the first point of Proposition \ref{Proposition: angle for hyperbolic metric space}})\\
            & \leq \sum_{a \in X} \beta e^{-\varepsilon p (d(a,[x,y])-2\delta)} ~(\text{according to Lemma \ref{lemma : DrutuKapovichcomparaisonproduitdeGromovet distanceàunegeodesique}})\\
            & \leq \beta e^{2 p \delta \varepsilon} \sum_{n \in \mathbb{N}} \displaystyle \sum_{ \{a\in X | d(a,[x,y])=n \}}  e^{ -\varepsilon p d(a,[x,y]) }\\ 
            & \leq \beta e^{2 p \delta \varepsilon} \sum_{n \in \mathbb{N}} e^{ -\varepsilon p n } C d(x,y) e^{\tau' n}~(\text{according to Lemma \ref{lemma : denombrement de g tel que g.o est proche d'une géodesique}})\\
            & \leq  \beta e^{2 p \delta \varepsilon} \frac{C}{1-e^{\tau'-\varepsilon p}} d(x,y) \\
            & < \infty.
           \end{aligned}$$
Therefore $D(x,y)$ is finite and:

$$ D(x,y) \leq  \beta e^{2 p \delta \varepsilon} \kappa d(x,y),$$

with $\kappa =  \frac{C}{1-e^{\tau'-\varepsilon p}} $ like in Lemma \ref{lemma: somme des exponentielles distance à un point}.

For the other inequality, we remark that for all $a \in [x,y]$, $(x|y)_{a}=0$.

Moreover, according to Proposition \ref{Proposition: angle for hyperbolic metric space}, for every $x,y \in X$ with $d(x,y) \geq 2D$, and every $a \in X$, we have that $ d^{a}_{\varepsilon}(x,y) \geq \alpha e^{-\varepsilon (x|y)_{a}} $. Then if $d(x,y) \geq 2D$, and $a \in  [x,y]$, we get $d^{a}_{\varepsilon}(x,y) \geq \alpha$. \\

From this fact, we deduce that for every $x,y \in X$ such that $d(x,y) \geq 2D$, we have that:
$$ D(x,y)   \geq d(x,y) \alpha^{p} .$$

Then, there exists a constant $K>0$ such that for all $x,y \in X$:

$$ D(x,y) \geq \alpha^{p} d(x,y)-K. $$

If we set $A:= \max(\beta e^{2 p \delta \varepsilon}  \kappa , \frac{1}{\alpha^{p}}) $ and $B:= K$, we get the desired inequalities of the proposition.

\end{proof}

The next lemmas will be useful in the rest of the section. The first is an application of the mean value theorem.

\begin{lem} \label{lemme: accroissement fini}

Let $t,s \in \mathbb{R}_{+}$, then for all $p>1$, we have the following:

$$| t^{p}-s^{p} | \leq p \max(t^{p-1},s^{p-1}) | t-s| .$$

\end{lem}

\begin{proof}
 For all $t,s \in \mathbb{R}_{+}$, if $t<s$, according to the mean value theorem there exists $c \in (t,s)$ such that:

 $$ |t^{p}-s^{p} | = p c^{p-1} | t-s| .$$

 Then clearly,

 $$| t^{p}-s^{p} | \leq p \max(t^{p-1},s^{p-1}) | t-s| .$$

\end{proof}
We will now prove some lemmas from hyperbolic geometry.

For every $x,y,z,a \in X$, when $z$ is close to a geodesic between $x$ and $y$, we prove that when the projection of $a$ onto $[x,z]\cup[z,y]$ lies in $[x,z]$, then the projection of $a$ onto $[z,y]$ is close to $z$.

\begin{figure}[!ht]
    \centering
   \includegraphics[scale=0.6]{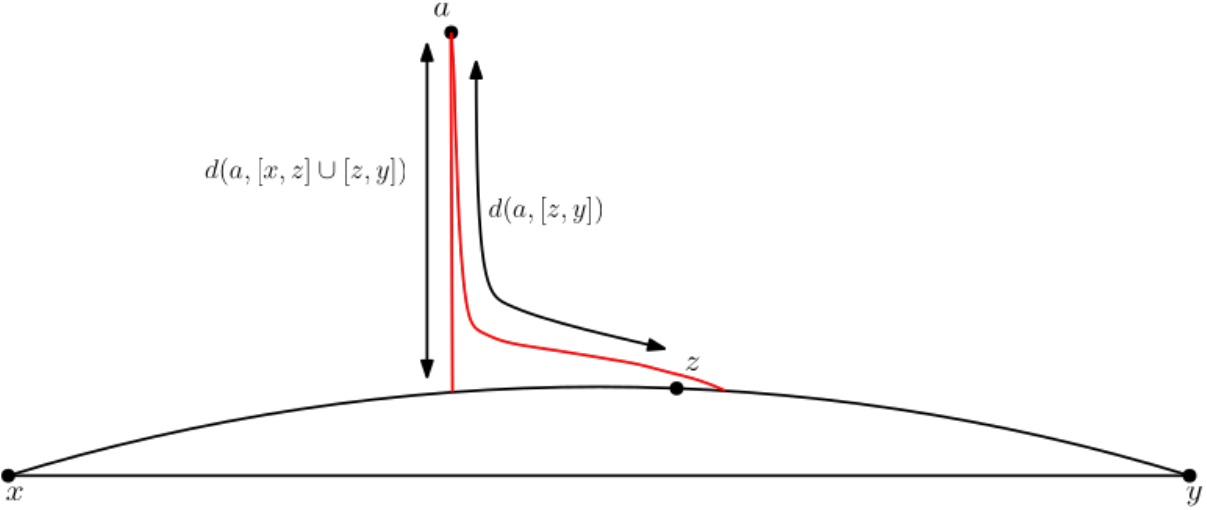}
   \caption{The projection of $a$ lies in $[x,z]$ then $d(a,[z,y])$ is close to $d(a,z)$.}
   \label{ fig: projection sur union de deux geodesiques }
\end{figure}

\begin{lem} \label{lemma : siaprochedegeodentrexetzalorsdistancedeaàzenvironproduitdegromov}
Let $D \geq 0$.
Let $x,y,z \in X$ such that $d(z,[x,y])\leq D$. For every $a \in X$ such that $d(a,[x,z]\cup[z,y])=d(a,[x,z])$, we have:

\begin{itemize}
    \item $d(a,[z,y]) \leq d(a,z) \leq 5\delta + 4D+ d(a,[z,y])$,

    \item $ (z|y)_{a}\leq d(a,z) \leq 7\delta +4D + (z|y)_{a}$.
\end{itemize}

\end{lem}

\begin{proof}
Let $a \in X$ such that $d(a,[x,y]\cup[y,z])=d(a,[x,z])$. We will first prove that:
$$d(a,[z,y]) \leq d(a,z) \leq 5\delta + 4D+ d(a,[z,y]). $$

The inequality on the left is clear. See Figure \ref{ fig: projection sur union de deux geodesiques }, let $u$ be a projection of $a$ on $[x,z]$, let $v$ be a projection of $a$ on $[z,y]$. With the assumption that $d(a,[x,y]\cup[y,z])=d(a,[x,z])$, we know that $d(a,v) \geq d(a,u)$. 

According to \cite[Lemme 11.4]{DrutuKapovichGeometricGroupTheory}, there exists $u',z',v' \in [x,y]$, such that: $d(u,u') \leq 2(D+\delta)$ , $d(z,z') \leq 2(D+\delta)$ , $d(v,v') \leq 2(D+\delta)$ and $u',z',v'$ are arranged in this order on $[x,y]$.

The thin triangle property applied to a geodesic triangle $[a,u',v']$ which contains $z'$, gives the fact that $d(z',[a,u']) \leq \delta$ or $d(z',[a,v']) \leq \delta$. Therefore, we get with another use of the thin triangle property that $d(z,[a,u])\leq 4D+5 \delta $ or $d(z,[a,v])\leq 4D+5 \delta $. Therefore we get:
$$
d(a, z)\leq d(a, v) + 4D+5\delta \leq  d(a, [z, y]) + 4D+5\delta.$$

To conclude, we just have to combine this inequality with the inequalities of Lemma \ref{lemma : DrutuKapovichcomparaisonproduitdeGromovet distanceàunegeodesique}.

\end{proof}

In the following lemma, we show that if a point $a$ projects to $z$ on a geodesic $[x,y]$, then $z$ is close to being a quasi-center of the triangles $[x,z,a]$ and $[y,z,a]$.

\begin{lem}\label{lemme : z est la projection donc c'est un quasi-centre}

Let $a,x,y \in X$ and $z$ a projection of $a$ on $[x,y]$ then:
$$d(a,x) \geq d(a,z)+d(z,x)-24\delta. $$

\end{lem}

\begin{proof}

Let us consider a geodesic triangle $[a,x,z]$. We denote by $t$ a quasi-center of this triangle, $u$ a vertex in $[x,a]$ , $v$ a vertex in $[x,y]$ and $w$ a vertex in $[a,x]$ such that $d(t,u) \leq 4 \delta, d(t,v) \leq 4 \delta, d(t,w) \leq 4\delta$.

We have the following:
$$\begin{aligned} 
           d(a,z)&\leq   d(a,v) ~(\text{since } v \in [x,y])  \\
            &\leq  d(a,w)+8 \delta.
\end{aligned}$$

Therefore:
$$ d(w,z) \leq 8 \delta ~(\text{since } w \in [a,z]).$$

And thus:
$$\begin{aligned} 
           d(a,x)&=   d(a,u)+d(u,x) ~(\text{since } u \in [a,x])  \\
            & \geq  d(a,z)+d(z,x)-2d(u,z)\\
            & \geq  d(a,z)+d(z,x)-24\delta.
\end{aligned}$$

\end{proof}

\begin{figure}[!ht]
   \centering
  \includegraphics[scale=0.5]{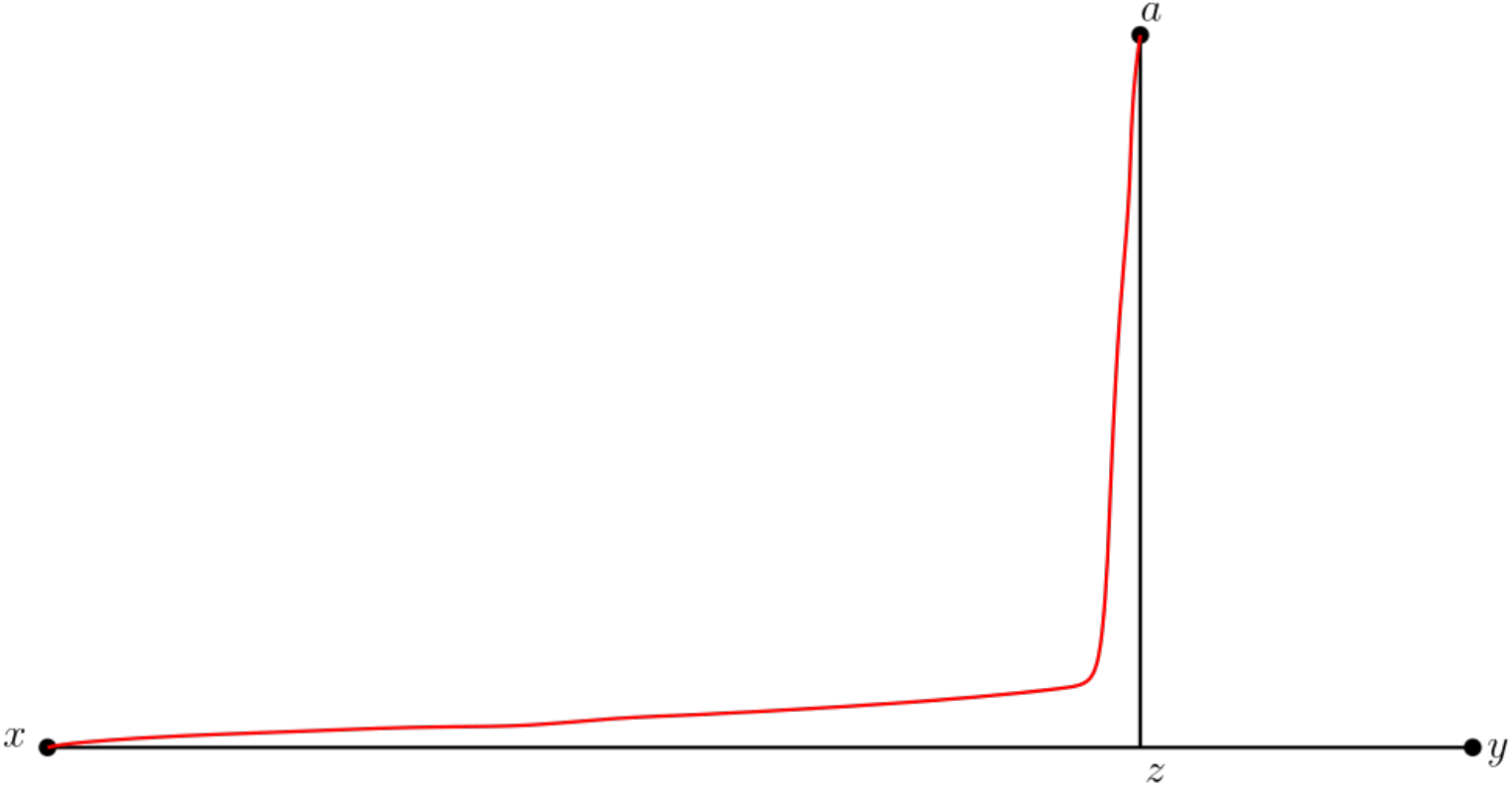}
   \caption{$z$ the projection of $a$ on $[x,y]$ is a quasi-center of the triangle $[x,a,z]$}
    \label{ fig: projection quasicentre }
 \end{figure}

\newpage

\begin{lem}\label{lemme: comparaison projection avec z presque sur la geodesique}
Let $x,y,z \in X$, a geodesic $[x,y]$ between $x$ and $y$ such that $d(z,[x,y])\leq 8 \delta$, and two geodesics $[x,z]$ and $[z,y]$ between $x$, $z$ and $z$, $y$ respectively then:
$$\forall a \in X, d(a,[x,z]\cup[z,y])-11\delta \leq d(a,[x,y])\leq d(a,[x,z]\cup[z,y])+11 \delta.$$
    
\end{lem}

\begin{proof}
According to Lemma \ref{lemme de Bridson}, each vertex of $[x,y]$ is at a distance of at most $11\delta$ from a point on $[x,z]\cup[z,y]$, and vice versa.

Then the inequalities of the lemma follow directly.

\end{proof}

To show that $D$ is almost a metric, we will need an inequality close to the triangular inequality. This proposition will be crucial to prove this and crucial to prove that this metric is weakly geodesic.\\

\begin{figure}[!ht]
   \centering
   \includegraphics[scale=0.5]{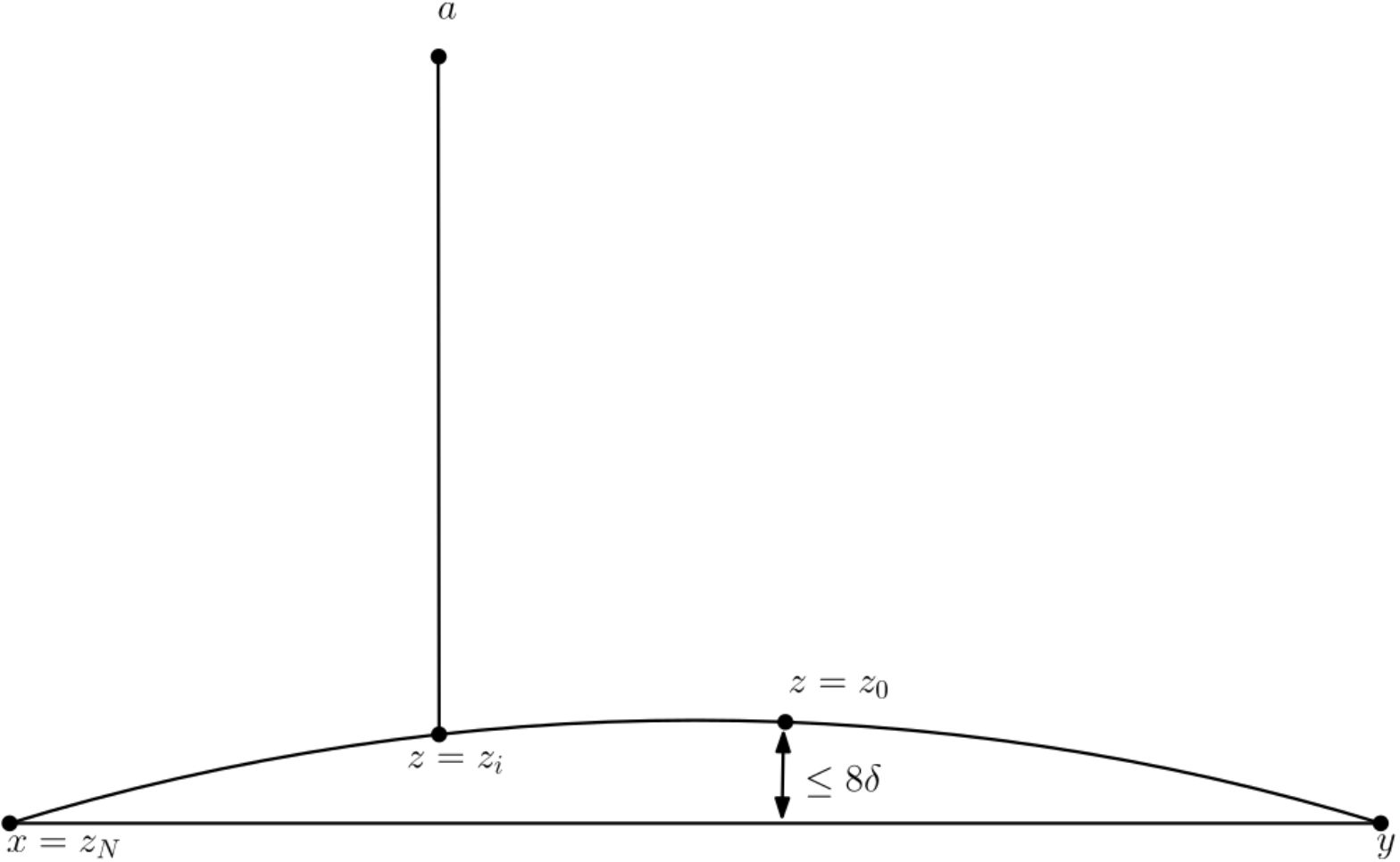}
    \caption{$a$ projects onto $z_{i}$ on $[x,z]\cup[z,y]$}
    \label{fig: proposition 8.9}
 \end{figure}

\newpage

\begin{proposition}\label{proposition : D est faiblement géodésique}
There exists a constant $C>0$ such that for all $x,y \in X$ and $z \in X$, such that $d(z,[x,y])\leq 8\delta$, we have the following:
$$|D(x,y)-D(x,z)-D(z,y)|\leq C.$$

\end{proposition}

\begin{proof}
Let $x,y \in X$ and $z \in X$ such that $d(z,[x,y])\leq 8\delta$. We fix two geodesics $[x,z]$ between $x$ and $z$ and $[z,y]$ between $z$ and $y$. We have:
$$\begin{aligned} 
           &~~|D(x,y)-D(x,z)-D(z,y)|\\
           \leq & \sum_{a \in X} | d_{\varepsilon}^{a}(x,y)^{p} -d_{\varepsilon}^{a}(x,z)^{p} - d_{\varepsilon}^{a}(z,y)^{p}
             | \\
            \leq & \displaystyle \sum_{ \{a\in X | d(a,[x,z]\cup[z,y])=d(a,[x,z]) \}} |d_{\varepsilon}^{a}(x,y)^{p} -d_{\varepsilon}^{a}(x,z)^{p} - d_{\varepsilon}^{a}(z,y)^{p}
             | \\
            & + \displaystyle \sum_{ \{a\in X | d(a,[x,z]\cup[z,y])=d(a ,[z,y]) \}} | d_{\varepsilon}^{a}(x,y)^{p} -d_{\varepsilon}^{a}(x,z)^{p} - d_{\varepsilon}^{a}(z,y)^{p}
             |.
\end{aligned}$$

We will bound from above the two obtained sums, we get the following:

$$\begin{aligned} 
           &~~\sum_{ \{a\in X | d(a,[x,z]\cup[z,y]=d(a,[x,z]) \}} | d_{\varepsilon}^{a}(x,y)^{p} -d_{\varepsilon}^{a}(x,z)^{p} - d_{\varepsilon}^{a}(z,y)^{p}|
           &\\
             \leq & \sum_{ \{a\in X | d(a,[x,z]\cup[z,y])=d(a,[x,z]) \}} | d_{\varepsilon}^{a}(x,y)^{p} -d_{\varepsilon}^{a}(x,z)^{p}| + d_{\varepsilon}^{a}(z,y)^{p} \\
            \leq & \sum_{ \{a\in X | d(a,[x,z]\cup[z,y])=d(a,[x,z]) \}} | d_{\varepsilon}^{a}(x,y)^{p} -d_{\varepsilon}^{a}(x,z)^{p}| \\
            & + \displaystyle \sum_{ \{a\in X | d(a,[x,z]\cup[z,y])=d(a,[x,z]) \}} d_{\varepsilon}^{a}(z,y)^{p}.
\end{aligned}$$

For $a \in X$ such that $d(a,[x,z]\cup[z,y])=d(a,[x,z])$, we have the following:
$$\begin{aligned} 
           d_{\varepsilon}^{a}(z,y)^{p} & \leq  \beta e^{-\varepsilon p (z|y)_{a}}\\      
            & \leq \beta e^{-\varepsilon p( d(a,z)-23\delta )} ~(\text{according to Lemma \ref{lemma : siaprochedegeodentrexetzalorsdistancedeaàzenvironproduitdegromov} }).
           \end{aligned}$$

Therefore:
$$\begin{aligned} 
          \displaystyle \sum_{ \{a\in X | d(a,[x,z]\cup[z,y])=d(a,[x,z]) \}} d_{\varepsilon}^{a}(z,y)^{p} & \leq  \beta e^{3\delta p \varepsilon} \sum_{ \{a\in X | d(a,[x,z]\cup[z,y])=d(a,[x,z]) \}}  e^{-\varepsilon p d(a,z)}\\\      
            & \leq \beta e^{23\delta p \varepsilon} \sum_{ a\in X  }  e^{-\varepsilon p d(a,z)} \\
            & \leq \beta e^{23\delta p \varepsilon} \kappa  ~(\text{according to Lemma \ref{lemma: somme des exponentielles distance à un point} })\\
           & \leq K_{1}.
\end{aligned}$$

with $K_{1}=4\beta e^{23\delta p \varepsilon} \kappa$.\\

We will bound the sum $ \displaystyle \sum_{ \{a\in X | d(a,[x,z]\cup[z,y])=d(a,[x,z]) \}} | d_{\varepsilon}^{a}(x,y)^{p} -d_{\varepsilon}^{a}(x,z)^{p}|$.

Let us denote $N:=d(x,z)$. We order $[x,z]$ in the following sense: $[z,x]=\{z_{0},z_{1},...,z_{N}\}$ with $z_{0}=z$, $z_{N}=x$ and $d(z_{i},z_{i+1})=1$ for $0\le i \le N-1$, see Figure \ref{fig: proposition 8.9}.\\

Let $i \in \{0,...,N \}$, such that $d(a,[x,z]\cup[z,y])=d(a,[x,z])=d(a,z_{i})$.
Thus $(x|z)_{a} \geq d(a,z_{i})-2\delta$ and $(x|y)_{a}\geq d(a,z_{i})-13\delta$ according to Lemmas \ref{lemma : DrutuKapovichcomparaisonproduitdeGromovet distanceàunegeodesique} and \ref{lemme: comparaison projection avec z presque sur la geodesique}.

Moreover, the assumption $d(a,[x,z]\cup[z,y])=d(a,[x,z])=d(a,z_{i})$ implies that we can apply Lemma~\ref{lemme : z est la projection donc c'est un quasi-centre} to $a,z_{i}$ and $z$, therefore:
$$d(a,z) \geq d(a,z_{i})+d(z_{i},z)- 24 \delta \geq d(a,z_{i})+i-24\delta.$$

Then we have the following:
$$\begin{aligned} 
           | d_{\varepsilon}^{a}(x,y)^{p} -d_{\varepsilon}^{a}(x,z)^{p}| & \leq  p \max (d_{\varepsilon}^{a}(x,y)^{p-1},d_{\varepsilon}^{a}(x,z)^{p-1}) | d_{\varepsilon}^{a}(x,y) -d_{\varepsilon}^{a}(x,z) | ~(\text{according to Lemma \ref{lemme: accroissement fini} }) \\      
            & \leq  p \max (d_{\varepsilon}^{a}(x,y)^{p-1},d_{\varepsilon}^{a}(x,z)^{p-1}) d_{\varepsilon}^{a}(y,z) \\
             & \leq p \beta^{2}  \max(e^{- \varepsilon(p-1)(x|y)_{a}}, e^{- \varepsilon(p-1)(x|z)_{a}}) e^{-\varepsilon (y|z)_{a}} \\
             & \leq p \beta^{2} e^{-\varepsilon(p-1) (d(a,z_{i})-13\delta)} e^{-\varepsilon( d(a,z)-2\delta )} \\
             & \leq p \beta^{2} e^{p \varepsilon 13 \delta}e^{\varepsilon13\delta} e^{-\varepsilon(p-1)d(a,z_{i})}  e^{-\varepsilon(d(a,z_{i})+i-26\delta)}\\
            & \leq p \beta^{2} e^{ \varepsilon 13 \delta(p+1)} e^{ -\varepsilon p d(a,z_{i}) } e^{- \varepsilon i }.
            \end{aligned}$$

Therefore:
$$\begin{aligned}
 & \sum_{ \{a\in X | d(a,[x,z]\cup[z,y])=d(a,[x,z]) \}} | d_{\varepsilon}^{a}(x,y)^{p} -d_{\varepsilon}^{a}(x,z)^{p}| \\
& \leq \sum_{i=0}^{N} \sum_{ \{a\in X | d(a,[x,z]\cup[z,y])=d(a,z_{i}) \}} | d_{\varepsilon}^{a}(x,y)^{p} -d_{\varepsilon}^{a}(x,z)^{p}| \\
& \leq \sum_{i=0}^{N} \sum_{ \{a\in X | d(a,[x,z]\cup[z,y])=d(a,z_{i}) \}}  p \beta^{2} e^{ \varepsilon 13 \delta(p+1)} e^{ -\varepsilon p d(a,z_{i}) } e^{- \varepsilon i }\\
& \leq p \beta^{2} e^{ \varepsilon 13 \delta(p+1)}  \sum_{i=0}^{N}  e^{- \varepsilon i } \sum_{a \in X} e^{-\varepsilon p d(a,z_{i})}\\
& \leq p \beta^{2} e^{ \varepsilon 13 \delta(p+1)} \kappa  \sum_{i=0}^{N}  e^{- \varepsilon i } \\
& \leq  K_{2}.
\end{aligned}$$

with $K_{2}= p \beta^{2} e^{ \varepsilon 13 \delta(p+1)} \kappa \frac{1}{1-e^{-\varepsilon}}$.

To conclude, we have:
$$\begin{aligned} 
           &~~|D(x,y)-D(x,z)-D(z,y)|\\\
           &\\
             \leq & \sum_{a \in X} | d_{\varepsilon}^{a}(x,y)^{p} -d_{\varepsilon}^{a}(x,z)^{p} - d_{\varepsilon}^{a}(z,y)^{p}
             | \\
            \leq & \displaystyle \sum_{ \{a\in X | d(a,[x,z]\cup[z,y])=d(a,[x,z]) \}} |d_{\varepsilon}^{a}(x,y)^{p} -d_{\varepsilon}^{a}(x,z)^{p} - d_{\varepsilon}^{a}(z,y)^{p}
             | \\
             +&  \displaystyle \sum_{ \{a\in X | d(a,[x,z]\cup[z,y])=d(a ,[z,y]) \}} | d_{\varepsilon}^{a}(x,y)^{p} -d_{\varepsilon}^{a}(x,z)^{p} - d_{\varepsilon}^{a}(z,y)^{p}
             | \\
              \leq & 2 K_{1}+2K_{2}.
\end{aligned}$$

and with $C:=2 K_{1}+2K_{2}$ we get the desired result.
\end{proof}

The next proposition shows that $D$ satisfies a coarse triangular inequality.

\begin{figure}[!ht]
   \centering
  \includegraphics[scale=0.4]{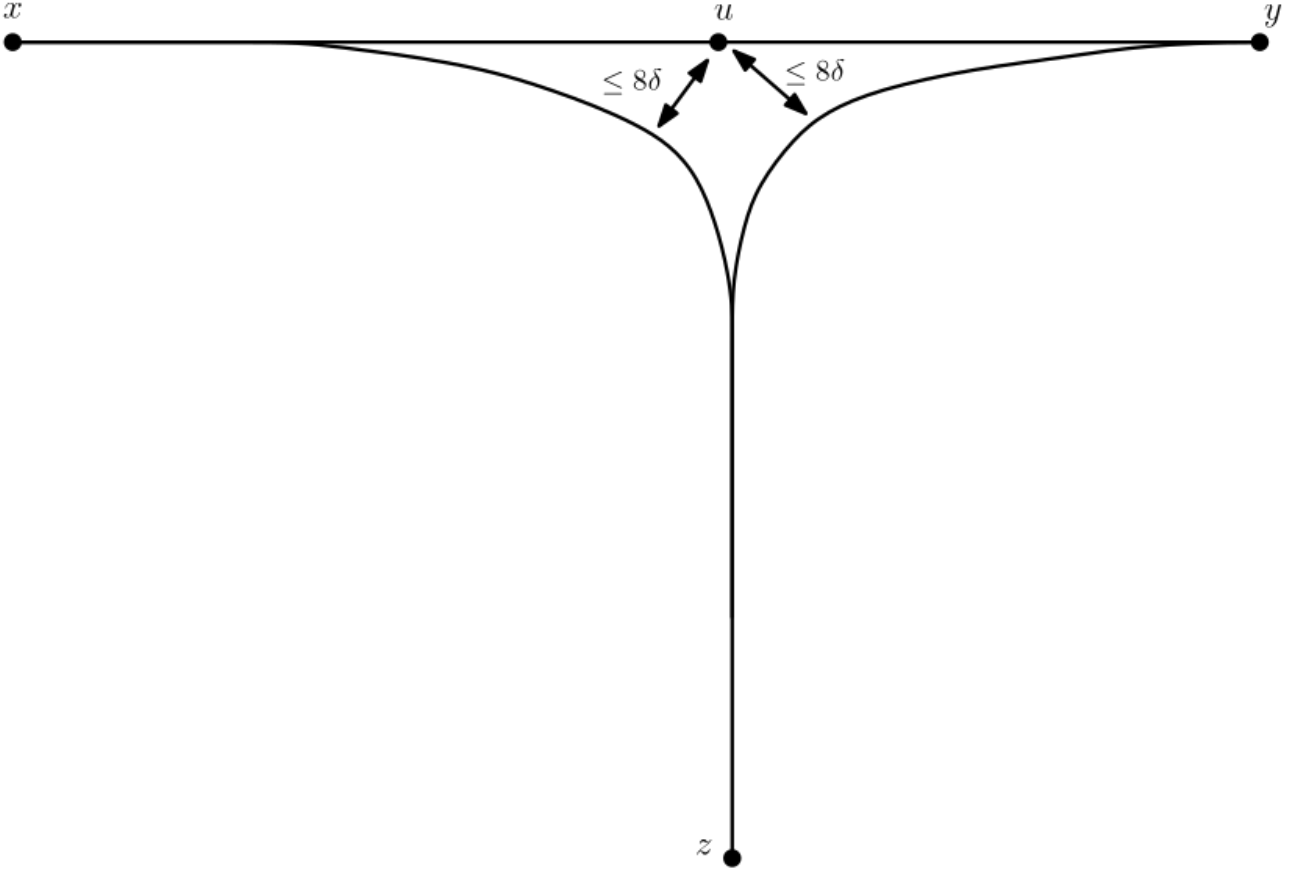}
  \caption{$u$ is at distance $8\delta$ from the three sides of the triangle}
   \label{fig: quasi-centre sur un côté}
 \end{figure}

\begin{proposition}\label{proposition : inegalite triangulaire grossiere}

There exists a constant $C'>0$ such that, for all $x,y,z \in X$, we have the following:
$$D(x,y)\leq D(x,z)+D(z,y) + C'.$$

\end{proposition}

\begin{proof}

Consider $[x,y,z]$ a geodesic triangle between $x,y ,z$.
According to hyperbolicity, every triangle admits a $4\delta$-quasi-center therefore, there exists $u \in [x,y]$ such that:
$$ d(u,[x,z]) \leq 8\delta, d(u,[y,z]) \leq 8\delta,$$

see Figure \ref{fig: quasi-centre sur un côté}.

Then:
$$\begin{aligned} 
           & D(x,y)-D(x,z)+D(z,y)\\  \leq &D(x,u)+D(u,y)+C\\
            - &D(x,u)-D(u,z)+C-D(z,u)-D(u,y)+C ~(\text{according to Proposition \ref{proposition : D est faiblement géodésique} })   \\      
            \leq & 3C.
            \end{aligned}$$
            
with $C':= 3C+ 2(A8\delta+B) $, we get the desired result.

\end{proof}

\begin{definition}\label{definition : metrique fortement bolique groupe hyperbolique}

For all $\varepsilon > \tau$, we define the function $\hat{d}: X \times X \rightarrow \mathbb{R}_{+}$ by:

$$
\hat{d}(x,y) = \left\{
    \begin{array}{ll}
        D(x,y)+C' & \mbox{if} ~ x \ne y \\
        0 & \mbox{if} ~ x=y.
    \end{array}
\right.
$$

where $C'$ is the constant of Proposition \ref{proposition : inegalite triangulaire grossiere}.

\end{definition}

\begin{proposition}

The function $\hat{d}$ of Definition \ref{definition : metrique fortement bolique groupe hyperbolique} is a metric on $X$.
    
\end{proposition}

\begin{proof}
By definition, $\hat{d}$ is symmetric and $\hat{d}(x,y)=0$ if and only if $x=y$.

The triangle inequality is a direct consequence of Proposition \ref{proposition : inegalite triangulaire grossiere}.
    
\end{proof}

The rest of the section will be dedicated to a proof of the following Theorem.

\begin{theo}\label{Theorem : metric fortmement bolic pour les groupes relativement hyperoblique}

The metric $\hat{d}$ satisfies the following properties.
\begin{itemize}
    \item $\hat{d}$ is $G$-invariant, i.e. $\hat{d}(g.x,g.y)=\hat{d}(x,y)$, for all $g,x,y \in G$,
    
    \item $\hat{d}$ is uniformly locally finite and quasi-isometric to the word metric,
    
    \item The metric space $(G,\hat{d})$ is weakly geodesic and strongly bolic.

\end{itemize}

\end{theo}

We already have all the material to prove the first two points.

\begin{proposition}\label{proposition : dchapeau est invariante et quasi-isom à la métrique des mots}
$\hat{d}$ is $G$-invariant, uniformly locally finite and quasi-isometric to the word metric.

\end{proposition}

\begin{proof}
The fact that $\hat{d}$ is $G$-invariant comes directly from the definition of $D$ and $\hat{d}$ and the fact that the family pseudo-distance $d_{\varepsilon}^{a}$ are $G$-equivariant according to Proposition \ref{Proposition: angle for hyperbolic metric space}.

According to Proposition \ref{proposition : D est fini et quasiisom à d}, there exist $A,B>0$ such that for all $x,y \in X$:
$$  \frac{1}{A}d(x,y)-B \leq D(x,y) \leq Ad(x,y) + B. $$

Therefore, for all $x,y \in X$:
$$  \frac{1}{A}d(x,y)-B+C' \leq \hat{d}(x,y) \leq Ad(x,y) + B+C'. $$

Thus:

$$  \frac{1}{A}d(x,y)-B-C' \leq \hat{d}(x,y) \leq Ad(x,y) + B+C' ,$$

and $\hat{d}$ is quasi-isometric to $d$.

Since $d$ is uniformly locally finite, $\hat{d}$ is also uniformly locally finite.
\end{proof}

\begin{proposition}\label{proposition : d chapeau est faiblement geodesique}

The metric space $(X,\hat{d})$ is weakly geodesic.
\end{proposition}

\begin{proof}

Applying Proposition \ref{proposition : D est faiblement géodésique} and \ref{proposition : dchapeau est invariante et quasi-isom à la métrique des mots}, the proof of the fact that $(X,\hat{d})$ is weakly geodesic, follows easily and it is the same proof as in \cite{MineyevYuStrongBolicityhypergroups}[Proposition~14]. Let us recall it for the reader's convenience.\\

Let $x,y \in X$ and $z \in [x,y]$, according to Proposition \ref{proposition : D est faiblement géodésique}, we have:
$$\hat{d}(x,z)+\hat{d}(z,y)-\hat{d}(x,y)\leq C+C',$$

where $C$ is the constant of Proposition \ref{proposition : D est faiblement géodésique} and $C'$ the constant of Definition~\ref{definition : metrique fortement bolique groupe hyperbolique}.

It follows that:
$$ \hat{d}(x,z) \leq \hat{d}(x,y) + C+ C'  $$

hence the image of the map:

$$ \hat{d}(x,.):[x,y]\rightarrow ]0,+\infty[  $$
is contained in $[0,\hat{d}(x,y)+C+C']$. Moreover, the image contains $0$ and $\hat{d}(x,y)$.\\

By Proposition \ref{proposition : dchapeau est invariante et quasi-isom à la métrique des mots}, there exists $M\geq 0$ such that, we have:

$$ |\hat{d}(x,z)-\hat{d}(x,z')| \leq M $$

when $d(z,z')=1$. This together with the fact that $[x,y]$ is a geodesic gives the fact that the image of $[x,y]$ by the map $\hat{d}(x,.)$ is $M$-dense in $[0,\hat{d}(x,y)]$, i.e. for every $t \in [0,\hat{d}(x,y)] $, there exists $ z \in [x,y]$ such that :

$$|\hat{d}(x,z)-t| \leq M.$$

Then $-\hat{d}(x,z) \leq M-t$.

By Proposition \ref{proposition : D est faiblement géodésique}, we get the following:

$$ \hat{d}(z,y) \leq \hat{d}(x,y)-t+M+C+C'  .$$

Then $(X,\hat{d})$ is weakly $\eta$-geodesic for $\eta = M+C+C'$.

\end{proof}

Since hyperbolicity is invariant by quasi-isometry for weakly geodesic metric spaces (see \cite{Lafforgue}), we get as a corollary that $(X,\hat{d})$ is hyperbolic.\\

From this, we deduce that $(X,\hat{d})$ verifies the weak-$B2'$ condition.

\begin{corollary} \label{corollary : delta B2}

There exists $\delta_{3}\geq 0$ such that $(X,\hat{d})$ is $\delta_{3}$-$B2'$.

\end{corollary}

\begin{proof}
According to \cite{BucherKarlssonDefBolicSapces}[Proposition 11], hyperbolic metric spaces which admit a midpoint map are bolic, therefore they satisfy $\delta_{3}$-$B2'$ for some $\delta_{3} \geq 0$.

$(X,\hat{d})$ is weakly
geodesic and admits a midpoint map according to \ref{proposition : midpoint}, therefore it satisfies $\delta_{3}$-$B2'$ for some $\delta_{3} \geq 0$.

\end{proof}

The last thing to prove is the fact that $(X,\hat{d})$ satisfies the strong-$B1$ condition.

The next lemma will be used to compare the projections of a point $a$ onto the geodesics determined by four points in the $B1$ configuration. We state the inequalities directly in terms of the Gromov product.

\begin{figure}[!ht]
   \centering
   \includegraphics[scale=0.4]{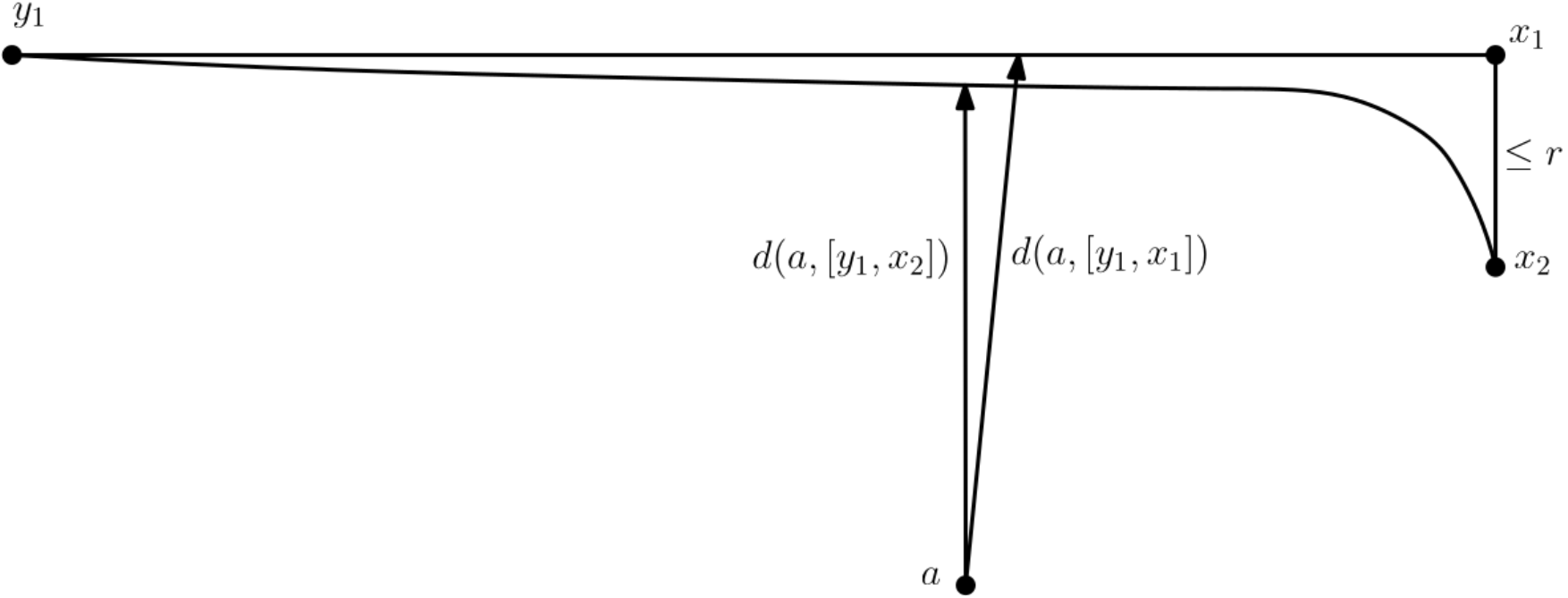}
  \caption{Projection of $a$ on $[x_{1},y_{1}]$ and on $[x_{2},y_{1}]$}
   \label{ fig: projection sur deux geodesiques en situation de delta-b1 }
\end{figure}
\newpage

\begin{lem}\label{lemma: comparaison des produits de gromov}

Consider $r\geq 0$ and $x_{1}$, $x_{2}$, $y_{1}$, $y_{2} \in X$ such that $d(x_{1},x_{2}),d(y_{1},y_{2}) \leq~r$. 

Then for every $a \in X$, we have:

\begin{itemize}

    \item $ (x_{1}|y_{1})_{a}-r \leq (x_{2}|y_{1})_{a} \leq (x_{1}|y_{1})_{a} +r, $

    \item  $ (x_{1}|y_{1})_{a}-r \leq (x_{1}|y_{2})_{a} \leq (x_{1}|y_{1})_{a} +r, $

    \item $ (x_{1}|y_{1})_{a}-2r \leq (x_{2}|y_{2})_{a} \leq (x_{1}|y_{1})_{a} +2r .$
 \end{itemize}

\end{lem}

\begin{proof}

For the first point, we have on one side:
$$(x_{2}|y_{1})_{a} \leq (x_{1}|y_{1})_{a}+ |(x_{2}|y_{1})_{a}-(x_{1}|y_{1})_{a}| \leq (x_{1}|y_{1})_{a} +r.$$

And on the other side:
$$ (x_{2}|y_{1})_{a} \geq (x_{1}|y_{1})_{a}- |(x_{2}|y_{1})_{a}-(x_{1}|y_{1})_{a}| \geq (x_{1}|y_{1})_{a} -r. $$

The second point is direct by symmetry of the role of $x_{1}, x_{2}$ and $y_{1},y_{2}$.

The third point follows directly by applying the argument of the first point twice.
    
\end{proof}

\begin{lem}\label{lemma : comparaison des projections}

 Consider $r\geq0$ and $x_{1}$, $x_{2}$, $y_{1}$, $y_{2} \in X$ such that $d(x_{1},x_{2}),d(y_{1},y_{2}) \leq~r$.

 Then for every $a \in X$, we have:

\begin{itemize}
     \item $ d(a,[x_{1},y_{1}])-(r+\delta) \leq d(a,[x_{2},y_{1}])\leq d(a,[x_{1},y_{1}])+(r+\delta), $

    \item $ d(a,[x_{1},y_{1}])-(r+\delta) \leq d(a,[x_{1},y_{2}])\leq d(a,[x_{1},y_{1}])+(r+\delta), $

    \item  $ d(a,[x_{1},y_{1}])-2(r+\delta) \leq d(a,[x_{2},y_{2}])\leq d(a,[x_{1},y_{1}])+2(r+\delta).$
    
 \end{itemize}
  
\end{lem}

\begin{proof}
For the first point.
 Consider a geodesic triangle $[x_{1},x_{2},y_{1}]$ and a vertex $u \in [x_{1},y_{1}]$ such that $d(a,u)=d(a,[x_{1},y_{1}])$. According to $\delta$-thinness, there exists $v \in [x_{2},y_{1}]$ such that $d(u,v)\leq r+\delta.$
 Indeed, there exists $w \in [x_{1},x_{2}] \cup [x_{2},y_{1}]$ such that $d(u,w) \leq \delta$. If $w \in [x_{2},y_{1}] $, then we take $v=w$, in the other case $w$ is at a distance less than $r$ from $y_{1} $, and then $d(u,y_{1})\leq r+\delta$.

 Therefore:
$$  d(a,[x_{2},y_{1}]) \leq d(a,v)\leq d(a,u) +d(u,v)\leq d(a,[x_{1},y_{1}])+(r+\delta).$$

By the symmetry of the roles of $[x_{1},y_{1}]$ and $[x_{2},y_{1}]$, we get the reverse inequality.\\

The second point is clear by symmetry of the roles of $[x_{1},y_{2}]$ and $[x_{2},y_{1}]$.\\

For the third point, the same proof applied to the triangle $[x_{1},x_{2},y_{2}]$ gives that:
$$ d(a,[x_{1},y_{2}])-(r+\delta) \leq d(a,[x_{2},y_{2}])\leq d(a,[x_{1},y_{2}])+(r+\delta),$$

and combining it with the second point, we get that:
$$ d(a,[x_{1},y_{1}])-2(r+\delta) \leq d(a,[x_{2},y_{2}])\leq d(a,[x_{1},y_{1}])+2(r+\delta), $$

which is the third point.

 \end{proof}

\begin{theo}\label{theorem : dchapeau est fortement B1}
 The metric space $(X,\hat{d})$ satisfies strong-$B1$.

\end{theo}

\begin{proof}
 According to Proposition \ref{proposition : dchapeau est invariante et quasi-isom à la métrique des mots}, the fact that $(X,\hat{d})$ satisfies strong-$B1$ is equivalent to the fact that for all $r>0$ and $\eta >0$, there exists $R=R(\eta,r)\geq 0$ such that for all $x_{1}$, $x_{2}$, $y_{1}$, $y_{2} \in X$, with $d(x_{1},y_{1}),d(x_{1},y_{2}),d(x_{2},y_{1}),d(x_{2},y_{2}) \geq R$ and $d(x_{1},x_{2}),d(y_{1},y_{2}) \leq r$, we have :

$$  |\hat{d}(x_{1},y_{1})+\hat{d}(x_{2},y_{2})-\hat{d}(x_{1},y_{2})-\hat{d}(x_{2},y_{1})| \leq \eta .$$ 
We will choose $R\in \mathbb{N}$ even and large enough later.
We denote by $m$ a vertex in $[x_{1},y_{1}]$ such that $d(x_{1},m) \geq \frac{R}{2}$ and $d(m,y_{1}) \geq \frac{R}{2}$.  For all $a \in X$, $\pi(a)$ denotes the projection of $a$ onto $[x_{1},y_{1}]$ that is closest to $x_{1}$. More precisely, this means that $ \pi(a) \in [x_{1},y_{1}]$, $d(a,[x_{1},y_{1}])=d(a,\pi(a))$, for all $ u \in [x_{1},y_{1}]$, such that $d(a,[x_{1},y_{1}])=d(a,u)$, we have: $d(x_{1},u) \geq d(x_{1}, \pi(a))$. The fact that $\pi(a)$ is the projection of $a$ closest to $x_{1}$ is used only to define the projection of $a$ in $[x_{1},y_{1}]$ without ambiguity, see Figure \ref{ fig: situation B1 et projection sur x1,y1 }.

\begin{figure}[!ht]
  \centering
  \includegraphics[scale=0.3]{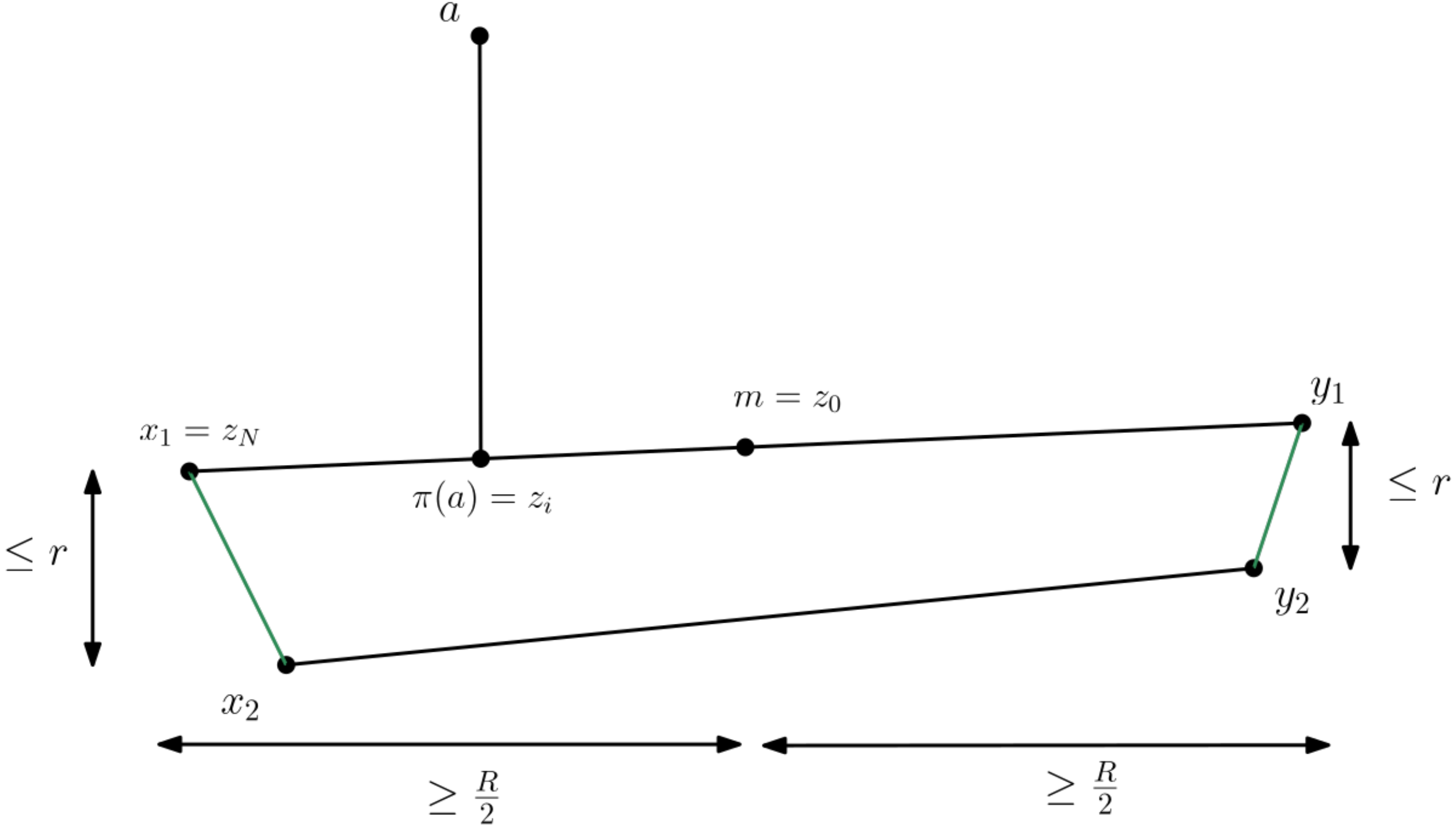}
   \caption{$\pi(a)$ is the projection of $a$ on $[x_{1},y_{1}]$}
   \label{ fig: situation B1 et projection sur x1,y1 }
 \end{figure}

\newpage
We fix $\eta,r>0$.  Consider $R>0$ and four points $x_{1}$, $x_{2}$, $y_{1}$, $y_{2} \in X$, with $d(x_{1},y_{1}),d(x_{1},y_{2}),d(x_{2},y_{1}),d(x_{2},y_{2}) \geq R$ and $d(x_{1},x_{2}),d(y_{1},y_{2}) \leq r$. In particular, we remark that $\hat{d}(x_{1},y_{1}),\hat{d}(x_{1},y_{2}),\hat{d}(x_{2},y_{1}),\hat{d}(x_{2},y_{2}) > 0$. Therefore, for all $i,j \in \{1,2\}$, $\hat{d}(x_{i},y_{j})=D(x_{i},y_{j})+C'$, we have the following:
$$|\hat{d}(x_{1},y_{1})+\hat{d}(x_{2},y_{2})-\hat{d}(x_{1},y_{2})-\hat{d}(x_{2},y_{1})|=|D(x_{1},y_{1})+D(x_{2},y_{2})-D(x_{1},y_{2})-D(x_{2},y_{1})|.$$

Then, we have:
$$\begin{aligned} 
           &~~ |D(x_{1},y_{1})+D(x_{2},y_{2})-D(x_{1},y_{2})-D(x_{2},y_{1})|\\
             \leq & \sum_{a \in X} |d_{\varepsilon}^{a}(x_{1},y_{1})^{p}+d_{\varepsilon}^{a}(x_{2},y_{2})^{p}-d_{\varepsilon}^{a}(x_{1},y_{2})^{p}-d_{\varepsilon}^{a}(x_{2},y_{1})^{p}| \\
            \leq & \displaystyle \sum_{ \{a\in X | \pi(a) \in [x_{1},m] \} } | d_{\varepsilon}^{a}(x_{1},y_{1}^{p})+d_{\varepsilon}^{a}(x_{2},y_{2})^{p}-d_{\varepsilon}^{a}(x_{1},y_{2})^{p}-d_{\varepsilon}^{a}(x_{2},y_{1})^{p}| \\
             +  &\displaystyle \sum_{ \{a\in X | \pi(a) \in [m,y_{1}] \}} | d_{\varepsilon}^{a}(x_{1},y_{1})^{p}+d_{\varepsilon}^{a}(x_{2},y_{2})^{p}-d_{\varepsilon}^{a}(x_{1},y_{2})^{p}-d_{\varepsilon}^{a}(x_{2},y_{1})^{p}|
           .      
\end{aligned}$$

The goal now is to bound from above by some constant the two obtained sums, when $R$ is great enough:

$$\begin{aligned}
 \displaystyle \sum_{ \{a\in X | \pi(a) \in [m,y_{1}]  \} } | d_{\varepsilon}^{a}(x_{1},y_{1})^{p}+d_{\varepsilon}^{a}(x_{2},y_{2})^{p}-d_{\varepsilon}^{a}(x_{1},y_{2})^{p}-d_{\varepsilon}^{a}(x_{2},y_{1})^{p}| & \\
             \leq \displaystyle \sum_{ \{a\in X |\pi(a) \in [m,y_{1}]  \} } | d_{\varepsilon}^{a}(x_{1},y_{1})^{p}-d_{\varepsilon}^{a}(x_{2},y_{1})^{p} | +  | d_{\varepsilon}^{a}(x_{2},y_{2})^{p}-d_{\varepsilon}^{a}(x_{1},y_{2})^{p} | & \\
             \leq \displaystyle \sum_{ \{a\in X | \pi(a) \in [m,y_{1}] \} } p \max(d_{\varepsilon}^{a}(x_{1},y_{1})^{p-1},d_{\varepsilon}^{a}(x_{2},y_{1})^{p-1})d_{\varepsilon}^{a}(x_{1},x_{2})
&\\
+ \sum_{ \{a\in X | \pi(a) \in [m,y_{1}] \} } p \max(d_{\varepsilon}^{a}(x_{2},y_{2})^{p-1},d_{\varepsilon}^{a}(x_{1},y_{2})^{p-1})d_{\varepsilon}^{a}(x_{1},x_{2}).
\end{aligned}$$

We will start to bound from above the sum:

$$ \displaystyle \sum_{ \{a\in X | \pi(a) \in [m,y_{1}] \} } p \max(d_{\varepsilon}^{a}(x_{1},y_{1})^{p-1},d_{\varepsilon}^{a}(x_{2},y_{1})^{p-1})d_{\varepsilon}^{a}(x_{1},x_{2}) .$$

Let us denote $N:=d(m,y_{1})$. We order $[m,y_{1}]$ in the following sense: $[m,y_{1}]=\{z_{0},z_{1},...,z_{N}\}$ with $z_{0}=m$, $z_{N}=y_{1}$ and $d(z_{i},z_{i+1})=1$ for $0\le i \le N-1$, see Figure \ref{fig: situation B1 projection du cote de x1,x2}.

 \begin{figure}[!ht]
   \centering
    \includegraphics[scale=0.5]{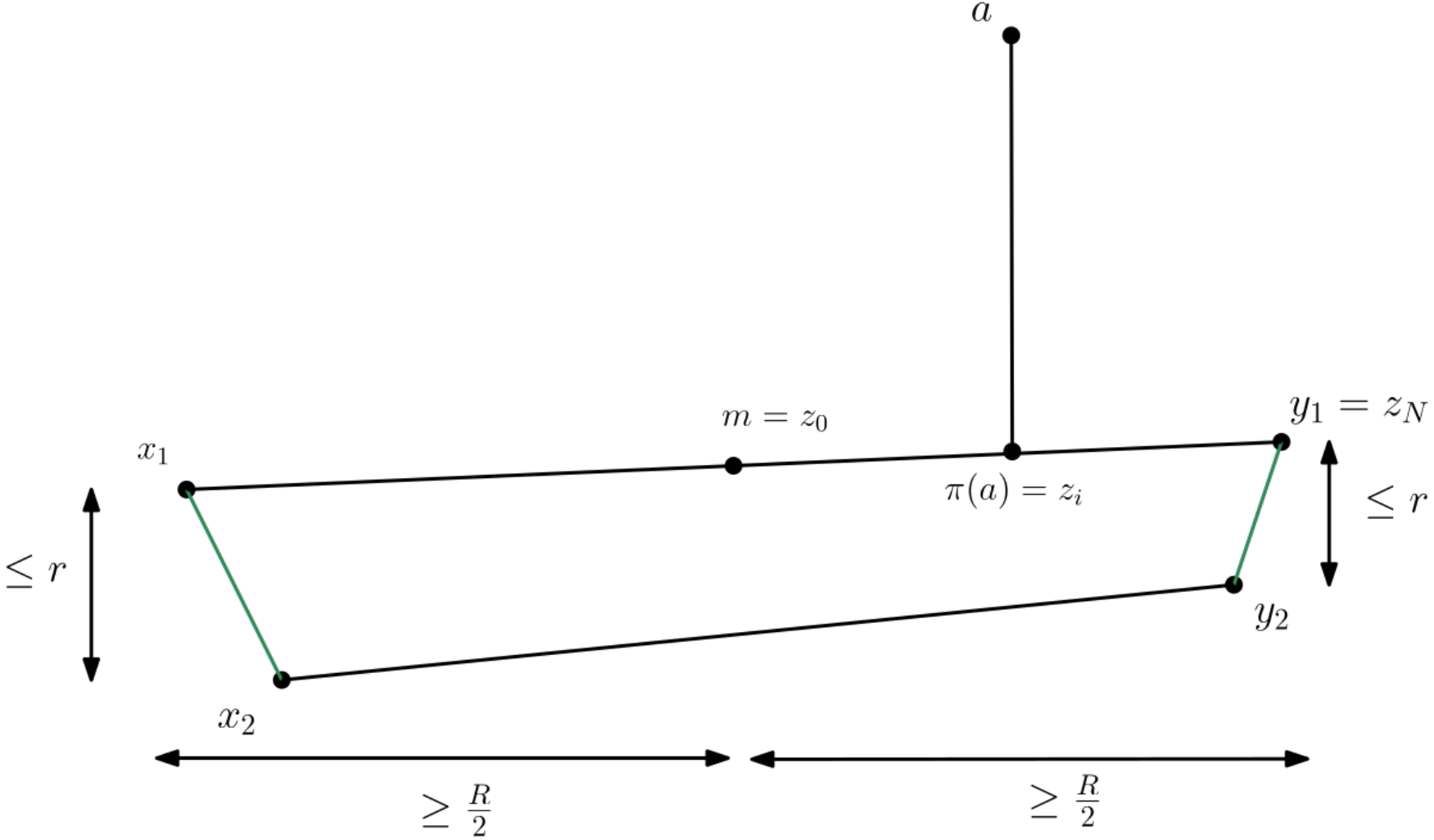}
    \caption{$\pi(a)=z_{i} \in [m,y_{1}]$}
  \label{fig: situation B1 projection du cote de x1,x2}
 \end{figure}

\newpage

According to the definition of $z_{i}$, we have $d(x_{1},z_{i})\geq \frac{R}{2}+i.$

Let $a \in X$ such that $\pi(a)=z_{i}$, then we have the following:
$$\begin{aligned}
 &\max(d_{\varepsilon}^{a}(x_{1},y_{1})^{p-1},d_{\varepsilon}^{a}(x_{2},y_{1})^{p-1})  \\
 &\leq \beta \max(e^{- \varepsilon(p-1)(x_{1}|y_{1})_{a}}, e^{- \varepsilon(p-1)(x_{2}|y_{1})_{a}}) \\
             &\leq  \beta \max(e^{-\varepsilon(p-1) (d(a,z_{i})-2\delta)}, e^{-\varepsilon(p-1) (d(a,z_{i})-2\delta-r) }) ~(\text{according to Lemma \ref{lemma : DrutuKapovichcomparaisonproduitdeGromovet distanceàunegeodesique} and Lemma \ref{lemma: comparaison des produits de gromov} }) \\
& \leq \beta e^{\varepsilon(p-1)(2\delta+r)}e^{-\varepsilon(p-1)d(a,z_{i})}.
\end{aligned}$$

To summarize, we have:
\begin{equation}\label{equation : controle de da(x1,y1) et da(x2,y1) quand a se projette sur zi}
  \max(d_{\varepsilon}^{a}(x_{1},y_{1})^{p-1},d_{\varepsilon}^{a}(x_{2},y_{1})^{p-1})  \leq \beta e^{\varepsilon(p-1)(2\delta+r)}e^{-\varepsilon(p-1)d(a,z_{i})} .
\end{equation}

Still in the case where $\pi(a)=z_{i}$, according to Lemma \ref{lemme : z est la projection donc c'est un quasi-centre}, we have:
$$ d(a,x_{1})\geq d(a,z_{i})+d(z_{i},x_{1})-24\delta \geq d(a,z_{i})+\frac{R}{2}+i-24\delta.$$

Moreover, we remark according to the triangular inequality that:
$$(x_{1},x_{2})_{a} \geq d(a,x_{1})-d(x_{1},x_{2}).$$

Therefore, we get the following:
$$\begin{aligned}
d_{\varepsilon}^{a}(x_{1},x_{2})\leq & \beta e^{-\varepsilon(x_{1},x_{2})_{a}}\\
\leq & \beta e^{- \varepsilon(d(a,x_{1})-d(x_{1},x_{2}))}\\
\leq & \beta e^{-\varepsilon(d(a,z_{i})+\frac{R}{2}+i-r-24\delta)}\\
\leq & \beta e^{\varepsilon(r+24\delta)}e^{-\varepsilon\frac{R}{2}} e^{-\varepsilon i} e^{-\varepsilon d(a,z_{i})}.
\end{aligned}$$

To summarize, we get:
\begin{equation}\label{equation: controle de da(x1,x2) quand a se projette sur zi}
d_{\varepsilon}^{a}(x_{1},x_{2})\leq \beta e^{\varepsilon(r+24\delta)}e^{-\varepsilon\frac{R}{2}} e^{-\varepsilon i} e^{-\varepsilon d(a,z_{i})}. 
\end{equation}

Putting this together, we can bound the sum:
$$\begin{aligned}
 &\displaystyle \sum_{ \{a\in X | \pi(a) \in [m,y_{1}] \} } p \max(d_{\varepsilon}^{a}(x_{1},y_{1})^{p-1},d_{\varepsilon}^{a}(x_{2},y_{1})^{p-1})d_{\varepsilon}^{a}(x_{1},x_{2})  \\
 \leq  & p \displaystyle \sum_{i=0}^{N} \sum_{ \{a\in X | \pi(a)=z_{i} \}} \max(d_{\varepsilon}^{a}(x_{1},y_{1})^{p-1},d_{\varepsilon}^{a}(x_{2},y_{1})^{p-1})d_{\varepsilon}^{a}(x_{1},x_{2}) \\
             \leq & p \beta^{2} \displaystyle \sum_{i=0}^{N} \sum_{ \{a\in X | \pi(a)=z_{i} \}}  e^{\varepsilon(p-1)(2\delta+r)}e^{-\varepsilon(p-1)d(a,z_{i})}   e^{-\varepsilon(d(a,z_{i})+\frac{R}{2}+i-r-24\delta)} ~(\text{according to Equations \ref{equation : controle de da(x1,y1) et da(x2,y1) quand a se projette sur zi} and \ref{equation: controle de da(x1,x2) quand a se projette sur zi}}) \\
\leq  & p \beta^{2} e^{\varepsilon(pr+2p\delta+22\delta)}  e^{-\varepsilon \frac{R}{2}} \displaystyle \sum_{i=0}^{N} \sum_{ \{a\in X | \pi(a)=z_{i} \}} e^{-\varepsilon pd(a,z_{i})} e^{-\varepsilon i} \\
\leq  & p \beta^{2}e^{\varepsilon(pr+2p\delta+22\delta)}  e^{-\varepsilon \frac{R}{2}} \displaystyle \sum_{i=0}^{N} e^{-\varepsilon i} \sum_{ a\in X } e^{-\varepsilon pd(a,z_{i})} \\
\leq  & p \beta^{2}e^{\varepsilon(pr+2p\delta+22\delta)}  e^{-\varepsilon \frac{R}{2}} \kappa  \displaystyle \sum_{i=0}^{N} e^{-\varepsilon i} ~(\text{according to Lemma \ref{lemma: somme des exponentielles distance à un point}}) \\
\leq & p \beta^{2} e^{\varepsilon(pr+2p\delta+22\delta)}  e^{-\varepsilon \frac{R}{2}} \kappa \frac{1-e^{-\varepsilon(N+1)}}{1-e^{-\varepsilon}}\\
\leq & p \beta^{2} e^{\varepsilon(pr+2p\delta+22\delta)}  e^{-\varepsilon \frac{R}{2}} \kappa \frac{1}{1-e^{-\varepsilon}}.
\end{aligned}$$

Therefore for $R \in \mathbb{N}$ large enough and even, $e^{-\varepsilon \frac{R}{2}}$ is arbitrarily small and the sum is also arbitrarily small.

Now, we will bound from above the sum:

$$  \sum_{ \{a\in X | \pi(a) \in [m,y_{1}] \} } p \max(d_{\varepsilon}^{a}(x_{2},y_{2})^{p-1},d_{\varepsilon}^{a}(x_{1},y_{2})^{p-1})d_{\varepsilon}^{a}(x_{1},x_{2}).$$

For all $a$ such that $\pi(a)=z_{i}$, we still have according to Equation \ref{equation: controle de da(x1,x2) quand a se projette sur zi}:
$$d_{\varepsilon}^{a}(x_{1},x_{2})\leq \beta e^{\varepsilon(r+24\delta)}e^{-\varepsilon\frac{R}{2}} e^{-\varepsilon i} e^{-\varepsilon d(a,z_{i})}. $$

For all $a$ such that $\pi(a)=z_{i}$, we have:
$$\begin{aligned}
 &\max(d_{\varepsilon}^{a}(x_{2},y_{2})^{p-1},d_{\varepsilon}^{a}(x_{1},y_{2})^{p-1})  \\
 &\leq \beta \max(e^{- \varepsilon(p-1)(x_{2}|y_{2})_{a}}, e^{- \varepsilon(p-1)(x_{1}|y_{2})_{a}}) \\
             &\leq  \beta \max(e^{-\varepsilon(p-1) (d(a,z_{i})-2\delta-2r)}, e^{-\varepsilon(p-1) (d(a,z_{i})-2\delta-r) }) ~(\text{according to Lemma \ref{lemma : DrutuKapovichcomparaisonproduitdeGromovet distanceàunegeodesique} and Lemma \ref{lemma: comparaison des produits de gromov} })\\
& \leq \beta e^{\varepsilon(p-1)(2\delta+2r)}e^{-\varepsilon(p-1)d(a,z_{i})}.
\end{aligned}$$

To summarize, we have:
\begin{equation}\label{equation: controlee de da(x2,y2)}
  \max(d_{\varepsilon}^{a}(x_{2},y_{2})^{p-1},d_{\varepsilon}^{a}(x_{1},y_{2})^{p-1})\leq \beta e^{\varepsilon(p-1)(2\delta+2r)}e^{-\varepsilon(p-1)d(a,z_{i})}.   
\end{equation}

Putting this together we can bound the sum:
$$\begin{aligned}
 &\displaystyle \sum_{ \{a\in X | \pi(a) \in [m,y_{1}] \} } p \max(d_{\varepsilon}^{a}(x_{2},y_{2})^{p-1},d_{\varepsilon}^{a}(x_{1},y_{2})^{p-1})d_{\varepsilon}^{a}(x_{1},x_{2})  \\
 \leq  & p \displaystyle \sum_{i=0}^{N} \sum_{ \{a\in X | \pi(a)=z_{i} \}} \max(d_{\varepsilon}^{a}(x_{2},y_{2})^{p-1},d_{\varepsilon}^{a}(x_{1},y_{2})^{p-1})d_{\varepsilon}^{a}(x_{1},x_{2}) \\
             \leq & p \beta^{2} \displaystyle \sum_{i=0}^{N} \sum_{ \{a\in X | \pi(a)=z_{i} \}}  e^{\varepsilon(p-1)(2\delta+2r)}e^{-\varepsilon(p-1)d(a,z_{i})} e^{-\varepsilon(d(a,z_{i})+\frac{R}{2}+i-r-24\delta)} ~(\text{according to Equations \ref{equation: controle de da(x1,x2) quand a se projette sur zi} and \ref{equation: controlee de da(x2,y2)}}) \\
\leq  & p \beta^{2} e^{\varepsilon(p2r+2p\delta+22\delta)}  e^{-\varepsilon \frac{R}{2}} \displaystyle \sum_{i=0}^{N} \sum_{ \{a\in X | \pi(a)=z_{i} \}} e^{-\varepsilon pd(a,z_{i})} e^{-\varepsilon i} \\
\leq  & p \beta^{2}e^{\varepsilon(p2r+2p\delta+22\delta)}  e^{-\varepsilon \frac{R}{2}} \displaystyle \sum_{i=0}^{N} e^{-\varepsilon i} \sum_{ a\in X } e^{-\varepsilon pd(a,z_{i})} \\
\leq  & p \beta^{2}e^{\varepsilon(p2r+2p\delta+22\delta)}  e^{-\varepsilon \frac{R}{2}} \kappa  \displaystyle \sum_{i=0}^{N} e^{-\varepsilon i} ~(\text{according to Lemma \ref{lemma: somme des exponentielles distance à un point}}) \\
\leq & p \beta^{2} e^{\varepsilon(p2r+2p\delta+22\delta)}  e^{-\varepsilon \frac{R}{2}} \kappa \frac{1-e^{-\varepsilon(N+1)}}{1-e^{-\varepsilon}}\\
\leq & p \beta^{2} e^{\varepsilon(p2r+2p\delta+22\delta)} e^{-\varepsilon \frac{R}{2}} \kappa \frac{1}{1-e^{-\varepsilon}}.
\end{aligned}$$

Like in the previous case, the sum is arbitrarily small when $R$ is arbitrarily big.

To conclude the proof of the Theorem, we need to bound from above by $\eta$ the sum:

$$\displaystyle \sum_{ \{a\in X | \pi(a) \in [x_{1},m] \} } | d_{\varepsilon}^{a}(x_{1},y_{1}^{p})+d_{\varepsilon}^{a}(x_{2},y_{2})^{p}-d_{\varepsilon}^{a}(x_{1},y_{2})^{p}-d_{\varepsilon}^{a}(x_{2},y_{1})^{p}|.$$

Even if the definition of $\pi(a)$ is not symmetric in $x_{1},x_{2},y_{1},y_{2}$, in the previous case, we just use the fact that $\pi(a)$ is a closest point projection on the geodesic $[x_{1},y_{1}]$ to bound the sum for $a \in X$ such that $\pi(a) \in [m,y_{1}]$, therefore the same proof will work in the case where $\pi(a) \in [x_{1},m]$ and the proof is complete.
\end{proof}

We summarize the results of this section with the proof of Theorem \ref{Theorem : metric fortmement bolic pour les groupes relativement hyperoblique}.

\begin{proof} of Theorem \ref{Theorem : metric fortmement bolic pour les groupes relativement hyperoblique}

Every hyperbolic group admits a metric $\hat{d}$ which is defined in Definition \ref{definition : metrique fortement bolique groupe hyperbolique}.

This metric is $G$-invariant, uniformly locally finite and quasi-isometric to the word metric according to the Proposition \ref{proposition : dchapeau est invariante et quasi-isom à la métrique des mots}.

According to Proposition \ref{proposition : d chapeau est faiblement geodesique}, the metric space $(X,\hat{d})$ is weakly geodesic.

According to Corollary \ref{corollary : delta B2}, $(X,\hat{d})$ verifies $\delta-B2$ for some $\delta$ and according to Theorem \ref{theorem : dchapeau est fortement B1}, $(X,\hat{d})$ satisfies strong-$B1$, thus $(X,\hat{d})$ is strongly bolic.

\end{proof}

\newpage

\section{Strong Bolicity for relatively hyperbolic groups }\label{section : bolicité forte pour gp relativement hyperbolique}

In this section, we prove the following theorem.

\begin{theo}\label{theorem: bolicite forte pour relativement hyperboliques avec paraboliques cat(0)}

Every \emph{relatively hyperbolic group with $CAT(0)$ parabolics} $G$ admits a metric $\hat{d}$ with the following properties.

\begin{itemize}
    \item $\hat{d}$ is $G$-invariant, i.e. $\hat{d}(g.x,g.y)=\hat{d}(x,y)$, for all $g,x,y \in G$,
    
    \item $\hat{d}$ is uniformly locally finite and quasi-isometric to the word metric,
    
    \item the metric space $(G,\hat{d})$ is weakly geodesic and strongly bolic.
\end{itemize}

\end{theo}

More precisely, we prove this theorem for a broader class of relatively hyperbolic groups, defined in Section \ref{subsection: assumptions on parabolic}. We will also show that it includes also virtually abelian parabolics according to \cite[Proposition 8.1]{minasyan}.

\begin{corollary}\label{corollary: Baum-Connes pour relativement hyperboliques avec paraboliques}
Groups hyperbolic relatively to $CAT(0)$ parabolics with the $(RD)$ property and their subgroups satisfy the Baum-Connes conjecture.

\end{corollary}

More precisely, we obtain the following corollary for families of $CAT(0)$ groups that have $(RD)$ property.

\begin{corollary}\label{corollary: Baum-Connes pour relativement hyperbolique en disant lesquels satisfont RD}
Groups that are hyperbolic relatively to following classes of groups and their subgroups:
\begin{itemize}
    \item virtually abelian,

    \item cocompaclty cubulated,

    \item Coxeter groups,
\end{itemize}

satisfy the Baum-Connes conjecture.

\end{corollary}

\begin{corollary}\label{Corollay: Baum-Connes pour variétés hyperboliques réelles}

Fundamental groups of complete real hyperbolic manifolds admits a strongly bolic metric.

\end{corollary}

\textbf{Structure of Section \ref{section : bolicité forte pour gp relativement hyperbolique}.}
\begin{itemize}

\item Subsection \ref{subsection: notations} is dedicated to notations.

\item In Subsection \ref{subsection: assumptions on parabolic}, we precisely define the class of admissible parabolics subgroups that interest us and show that this class includes $CAT(0)$ groups and virtually abelian groups.

\item In Subsection \ref{subsection: masks}, we discuss a notion of random representatives on the parabolics, namely masks, defined by Dahmani and Chatterji in \cite{ChatterjiDahmani}. Most of the properties cited come from \cite{ChatterjiDahmani}, while some are new.

\item In Subsection \ref{subsection: definition de da ,control}, we define for all vertices of the coned-off graph, a pseudo-metric, related to the masks. These pseudo-metrics will play, in the definition of strongly bolic metric for relatively hyperbolic groups, the role that angles of Proposition \ref{Proposition: angle for hyperbolic metric space} played in the definition of the strongly bolic metric for hyperbolic groups. We prove an important control result for these pseudo-metrics.

\item In Subsection \ref{subsection: comparaison angles et masques}, we prove that for points of infinite valence, the pseudo-metrics are comparable to the angles of Definition \ref{definition: angle} in the coned-off graph. 

\item Subsection \ref{subsection: la metrique fortment bolique} is dedicated to the definition of the metric $\hat{d}$, which will be the strongly bolic metric. In this subsection, we prove that this is a metric.

\item In Subsection \ref{subsection: metric comparable to the group metric}, we prove that $\hat{d}$ is quasi-isometric to the group metric.

\item In Subsection \ref{subsection: faiblement geodesique}, we prove that the metric space $(G,\hat{d})$ is weakly-geodesic.

\item In Subsection \ref{subsection: metrique bolic verifie B2'}, we prove that $\hat{d}$ satisfies the weak-$B2'$ condition.

\item In Subsection \ref{subsection: metrique bolic verifie B1}, we prove that $\hat{d}$ satisfies the strong-$B1$ condition.

\item Subsection \ref{subsection: conclusion fortement bolic} is a conclusion, we summarize all the previous results. We combined Theorem \ref{Theorem : metric fortmement bolic pour les groupes relativement hyperoblique} and Theorem \ref{theo : Baum-Connes=Rapid decay + Strongly Bolic} to prove the Baum-Connes conjecture for relatively hyperbolic with finitely generated $CAT(0)$ with $(RD)$ property and virtually abelian parabolics.

\end{itemize}

\subsection{Notations} \label{subsection: notations}
In the whole section, $G$ will denote a group hyperbolic relatively to proper subgroups $P_{1},...,P_{n}$. We assume further that all $P_{1},...,P_{n}$ are infinite and finitely generated.\\

Let $d_{G}$ denote a word-distance on $G$ with respect to a finite generating set. Let
$X_{c}$ denote a coned-off graph, defined in Definition \ref{definition: coned-off} of $G$ relative to $P_{1},...,P_{n}$ with the same generating set. Let $d_{c}$ will denote the graph metric of $X_{c}$, and for all $x,y \in X_{c}$, $[x,y]_{c}$ will denote a geodesic for $d_{c}$ between $x$ and $y$. We denote by $X_{c}^{\infty}$ the set of vertices of infinite valence of $X_{c}$. We will denote by $I_{c}(x,y)=\{z \in X_{c} ~|~ d_{c}(x,y)=d_{c}(x,z)+d_{c}(z,y) \} $ the geodesic interval between $x$ and $y$ in $X_{c}$. Let us denote $\delta$ the hyperbolicity constant of $X_{c}$ and we assume that $\delta$ is an integer. In what follows, since parabolics subgroups are assumed to be infinite, we overlook the distinction between $G$ and $X_{c} \backslash X_{c}^{\infty}$.\

For all $x,y \in X_{c}$, we denote $\Theta'(x,y)$ the minimal sum of angles at infinite valence vertices on a geodesic $[x,y]_{c}$ and for all $M>0$, $\Theta'_{>M}(x,y)$ the minimal sum of angles at infinite valence vertices on a geodesic $[x,y]_{c}$ greater than $M$.  We define $d': X_{c} \times X_{c} \rightarrow \mathbb{N}$ as $d'(x,y)=d_{c}(x,y)+\Theta'(x,y)$. We recall that its restriction to $G$ is a metric quasi-isometric to the word metric according to Proposition \ref{Proposition : formule de la distance}. For all $M>0$, we define, in the same way $d'_{>M}: X_{c} \times X_{c} \rightarrow \mathbb{N}$ as $d'_{>M}(x,y)=d_{c}(x,y)+\Theta'_{>M}(x,y)$. According to Corollary \ref{corollary : reformulation de la formule de la distance}, on $G$, $d'_{M}$ is also quasi-isometric to $d_{G}$.   \\

Let $H$ be a finitely generated group with a fixed generating set, let $C>0$ be a constant, we denote by $\text{Proba}_{C}(H)$ the set of finitely supported probability measures on $H$ with support of diameter at most $C$. We observe that $H$ embeds into $\text{Proba}_{C}(H)$ through Dirac measures.

We define the norm $|| \eta ||$ of a measure on a discrete space $X$ as its $\ell^{1}$-norm $|| \eta ||:=\sum_{x\in X}| \eta(x)|$. Moreover, we set $|| \eta \Delta \eta' ||:=|| \eta - \eta' ||$.\\

For all $K>0$, $x \in \mathbb{R}_{+}$, we set:
$$
[x]_{K}:= \left\{
    \begin{array}{ll}
        0 \mbox{ if } x \leq K \\
        x  \mbox{ if } x>K.
    \end{array}
\right.
$$

\subsection{Assumptions on parabolics}\label{subsection: assumptions on parabolic}

In this section, we describe the parabolic subgroups for which our result is applicable. We begin by stating the assumptions we require for the parabolics. We then show that $CAT(0)$ groups satisfy these assumptions.

The following two properties are strengthening of the conditions of the weakly-$B2$' property. We need to assume that the metric defined on the parabolic subgroups satisfies these two properties in order to establish the weakly-$B2'$ property for the relatively hyperbolic group.

The first property, which we call $\eta$-peakless, means that the distance from a point to a geodesic reaches its maximum at the ends of the geodesic, up to a constant $\eta$.

\begin{definition}\label{definition: peakless} ($\eta_{2}$-peakless) Let $\eta_{1}\geq0$ and $\eta_{2}\geq0$.
Let $(X,d)$ be a $\eta_{1}$-weakly geodesic metric space, we say that $(X,d)$ is $\eta_{2}$-\emph{peakless}, if there exists $\eta_{2}>0$ such that, for all $x,y,z \in X$, there exists a $\eta_{1}$-weak geodesic $\sigma: [0,1] \rightarrow X$ from $x$ to $y$, such that for $t \in [0, 1]$, we have:
$$ d(z,\sigma(t))\leq \max(d(x,z),d(y,z))+\eta_{2}.$$

\end{definition}

The second property, named $\eta$-strongly convex, is a non-symmetric reinforcement, up to an additive constant, of classical uniform convexity. In fact, given three points $x$, $y$, and $z$, rather than requiring that $z$ be closer to the midpoint of the geodesic between $x$ and $y$ than to $x$ and $y$, we require here that $z$ be closer to every point on the geodesic that lies sufficiently far from both $x$ and $y$.

\begin{definition}\label{definition: strongly convex}($\eta_{2}$-strongly-convex) Let $\eta_{1}\geq0$ and $\eta_{2}\geq 0$.
Let $(X,d)$ be a $\eta_{1}$-weakly geodesic metric space, we say that $(X,d)$ is $\eta_{2}$-\emph{strongly-convex}, if for all $0<\eta<\frac{1}{2}$, for  $\varepsilon \in (0,1)$, there exists $\delta \in (0,1)$ such for all $x,y,z \in X$, with $d(x,z) \leq N, d(z,y) \leq N$ and $d(x,y)>\varepsilon N$, there exists a $\eta_{1}$-weak geodesic $\sigma: [0,1] \rightarrow X$ from $x$ to $y$, such that for all $t \in [\eta, 1-\eta]$, we have:
$$ d(z,\sigma(t))\leq N(1-\delta)+\eta_{2}.$$
\end{definition}

We remark directly from the definition that a geodesic metric space $(X,d)$, which is  $\eta_{1}$-strongly-convex and $\eta_{2}$-peakless for some $\eta_{1}\geq0$ and $\eta_{2}\geq0$, satisfies weakly-$B2'$.

\begin{definition}\label{definition: lesbons parabolic}

A finitely generated group $P$ is said to be \emph{admissible}, if for all $C>0$, $\text{Proba}_{C}(P)$ admits a $P$-invariant pseudo-metric, denoted $d_{b}$, such that:

\begin{itemize}
    \item $d_{b}$ is weakly geodesic,

     \item  $d_{b}$ is $\eta$-peakless and $\eta$-strongly-convex for some $\eta>0$ and for the same choice of weak geodesic,
     
    \item $d_{b}$ verifies strong-$B1$, weakly-$B2'$,

     \item the inclusion map from $(P,d_{P})$ to $(\text{Proba}_{C}(P),d_{b})$ via the Dirac masses is a quasi-isometry,

     \item for all $\mu,\nu \in \text{Proba}_{C}(P) $ with  $\text{supp}(\mu) \cap \text{supp}(\nu) \neq \emptyset$, we have: $ d_{b}(\mu,\nu) \leq 2C || \mu \Delta \nu || $. 
\end{itemize}

\end{definition}

The rest of Section \ref{section : bolicité forte pour gp relativement hyperbolique} will be dedicated to the proof of the following theorem. We will provide the proof precisely in Subsection \ref{subsection: conclusion fortement bolic}.

\begin{theo}\label{theo: theoremimportant bolicite forte}

Every \emph{group $G$ hyperbolic relatively to admissible subgroups} admits a metric $\hat{d}$ with the following properties.

\begin{itemize}
    \item $\hat{d}$ is $G$-invariant, i.e. $\hat{d}(g.x,g.y)=\hat{d}(x,y)$, for all $g,x,y \in G$,
    
    \item $\hat{d}$ is uniformly locally finite and quasi-isometric to the word metric,
    
    \item the metric space $(G,\hat{d})$ is weakly geodesic and strongly bolic.
\end{itemize}
\end{theo}

We show in the remainder of this section that CAT(0) groups are admissible.

We prove here that $(\mathbb{R}^{n},\ell^{2})$ is $0$-peakless. In what follows, 
$||.||$ will denote the Euclidean norm and $d_{\ell^{2}}$ the Euclidean distance.

\begin{lem}\label{lemme: peakless euclidien}
For all $x,y,z \in \mathbb{R}^{n}$, for all $t \in [0,1]$, we have:

$$d_{\ell^{2}}(z,(1-t)x+ty)\leq \max(d_{\ell^{2}}(z,x),d_{\ell^{2}}(z,y)).$$

\end{lem}

\begin{proof}
    By applying a translation, we can suppose that $z=0$.
Therefore, it is equivalent to prove that for all $t \in [0,1]$, we have:
$$||(1-t)x+ty || \leq \max(||x||,||y||). $$

This is true for every norm, then this proves the lemma.
\end{proof}

As a corollary $CAT(0)$ spaces are $0$-peakless.

\begin{corollary}\label{corollary: espace cat(0) sont peakless}

$CAT(0)$ spaces are peakless.
    
\end{corollary}

\begin{proof}
Let $(X,d)$ a $CAT(0)$ space. Let $x,y,z \in X$ and we consider $\tilde{x},\tilde{y},\tilde{z} \in (\mathbb{R}^{2},d_{\ell^{2}})$ a comparison triangle.

Let $p$ be a point on the geodesic between $x$ and $y$ and $\tilde{p}$ a comparison point, i.e. a point on the geodesic between $\tilde{x}, \tilde{y}$ such that $d_{\ell^{2}}(\tilde {p},\tilde{x})=d(p,x)$ and $d_{\ell^{2}}(\tilde {p},\tilde{y})=d(p,y)$.

We have:
$$\begin{aligned}
d(z,p) &\leq d_{\ell^{2}}(\tilde{z},\tilde{p})~\text{(according to the $CAT(0)$ inequality)}\\
&  \leq \max(d_{\ell^{2}}(\tilde{z},\tilde{x}),d_{\ell^{2}}(\tilde{z},\tilde{y})) ~\text{(according to Lemma \ref{lemme: peakless euclidien})}\\
& \leq \max(d_{\ell^{2}}(z,x),d_{\ell^{2}}(z,y))~\text{(according to the definition of a comparison triangle)}.
\end{aligned}$$

\end{proof}

We prove here that $(\mathbb{R}^{n},\ell^{2})$ is $0$-strongly-convex.
\begin{lem}\label{lemme: strongly convex euclidien}
For all $\eta \in (0,1)$, for all $ \varepsilon \in (0,1)$, there exists $ \delta \in (0, 1)$ such that for all $x,y,z \in \mathbb{R}^{n}$ with  $d_{\ell^{2}}(x,y) \geq \varepsilon \max(d_{\ell^{2}}(x,z),d_{\ell^{2}}(z,y))$, we have for all $t \in (\eta,1-\eta)$:
$$d_{\ell^{2}}(z,(1-t)x+ty)\leq \max(d_{\ell^{2}}(x,z),d_{\ell^{2}}(z,y))(1-\delta). $$
\end{lem}

\begin{proof}
Let $ \varepsilon \in (0,1)$, by applying a translation we assume that $z=0$. Let us assume that $||x||\leq ||y|| $ and $ ||x-y||\geq \varepsilon ||y|| $.

We will use the following variation of the median identity, for all $t\in (0,1)$:
$$||(1-t)x+ty||^{2}+t(1-t)||x-y||^{2}=(1-t)||x||^{2}+t||y||^{2}.$$

We get the following:
$$\begin{aligned}
||(1-t)x+ty||^{2}&=(1-t)||x||^{2}+t||y||^{2}-t(1-t)||x-y||^{2}\\
& \leq (1-t) ||y||^{2}+t||y||^{2}-\varepsilon t(1-t)||y||^{2}\\
& \leq (1- \varepsilon t(1-t))||y||^{2}\\
& \leq (1-\varepsilon \eta (1-\eta)) ||y||^{2}.
\end{aligned}$$

Therefore:
$$ ||(1-t)x+ty|| \leq \sqrt{1-\varepsilon \eta (1-\eta)} ||y||,$$

and since $\sqrt{1-\varepsilon \eta (1-\eta)} \in (0,1)$, we set:
 $$\delta:= 1- \sqrt{1-\varepsilon \eta (1-\eta)} ,$$

 and we get the desired result.

\end{proof}

As a corollary, $CAT(0)$ spaces are strongly-convex.

\begin{corollary}\label{corollary: les espaces Cat(0) sont fortement convex}
$CAT(0)$ spaces are strongly-convex.
    
\end{corollary}

\begin{proof}

Let $\eta \in (0,1)$, $\varepsilon \in (0,1)$ and $x,y,z \in X$ such that $d(x,y) \geq \varepsilon \max (d(x,z),d(z,y)$.
We consider $\tilde{x},\tilde{y},\tilde{z} \in (\mathbb{R}^{2},d_{\ell^{2}})$ a comparison triangle for $x,y,z$ and we remark that this points satisfy the assumptions of Lemma \ref{lemme: strongly convex euclidien}.
Let $\sigma$ the geodesic in $X$ between $x$ and $y$ and $\tilde{\sigma}$ the geodesic between $\tilde{x}$ and $\tilde{y}$. For all $t \in [0,1]$, by definition of a comparison point, a comparison point for $\sigma(t)$ is $\tilde{\sigma}(t)$.

According to Lemma \ref{lemme: strongly convex euclidien}, there exists $\delta \in (0,1)$ such that for all $t \in (\eta,1-\eta)$, we have:
$$d_{\ell^{2}}(\tilde{z},\tilde{\sigma}(t))\leq \max(d_{\ell^{2}}(\tilde{x},\tilde{z}),d_{\ell^{2}}(\tilde{z},\tilde{y}))(1-\delta). $$

Therefore, we have:
$$\begin{aligned}
d(z,\sigma(t)) &\leq d_{\ell^{2}}(\tilde{z},\tilde{\sigma}(t))~\text{(according to the $CAT(0)$ inequality)}\\
&  \leq \max(d_{\ell^{2}}(\tilde{x},\tilde{z}),d_{\ell^{2}}(\tilde{z},\tilde{y}))(1-\delta) ~\text{(according to the previous remark)}\\
& \leq \max(d_{\ell^{2}}(x,z),d_{\ell^{2}}(z,y))(1-\delta) ~\text{(according to the definition of a comparison triangle)}.
\end{aligned}$$

\end{proof}

In order to prove that $CAT(0)$ groups are admissible, we will use barycenter of measures on $CAT(0)$ spaces.

\begin{proposition}\cite[Proposition 4.3.]{Sturmbarycentre}

Let $(X,d)$ be a $CAT(0)$ metric space and $\mu:= \sum_{x \in X}\mu_{x} \delta_{x}$ a finitely supported measure on $X$, with $\sum_{x \in X}\mu_{x}=1$ and the family of positive reals $(\mu_{x})_{x\in X}$ is almost zero.\\

The function $z \mapsto \sum_{x \in X} \mu_{x} d^{2}(z,x)$ admits a unique minimum, called its barycenter.

\end{proposition}

To verify the last condition for admissible groups in the CAT(0) setting, we will use the $L^{1}$ Wasserstein metric, which we define here.

\begin{definition}\label{definition : metrique de wasserstein}\cite[Definition 4.1.]{Sturmbarycentre}

Let $(X,d)$ be a $CAT(0)$ space, let $\mu= \sum_{x \in X}\mu_{x} \delta_{x}$ and $\nu = \sum_{x \in X} \nu_{x} \delta_{x}$ be two finitely supported probability measures on $X$, where $(\mu_{x})_{x\in X}$ and $(\nu_{x})_{x\in X } $ are almost zero families of reals numbers such that $ \sum_{x \in X}\mu_{x}= \sum_{x \in X} \nu_{x}=1 $.\\

A coupling $\Pi$ between $\mu$ and $\nu$ is probability measure on $X^{2}$ such that the marginals are $ \mu$ and $\nu$, i.e $\Pi=\sum_{x,y \in X } \Pi_{x,y} \delta_{(x,y)}$ where
for all $x \in X$, $\sum_{y \in X}\Pi_{x,y}=\mu_{x}$ and for all $y \in X$, $\sum_{x \in X}\Pi_{x,y}=\nu_{y}$.
The cost of a transport plan $\Pi$ denoted by $c(\Pi)$ is defined as:
$c(\Pi):=\sum_{x,y \in X } \Pi_{x,y} d(x,y).$\\

The Wassertein distance between $ \mu $ and $\nu$ is defined as follow:
$$ d^{W}(p,q)=\inf \{  c(\Pi) : \Pi \text{ is a coupling between }\mu \text{ and } \nu \}.$$
    
\end{definition}

The following result, due to \cite{Sturmbarycentre}, allows us to control the difference between barycenters in terms of the Wasserstein distance.

\begin{theo}\label{theo: inegalite avec Wasserstein} \cite[Theorem 6.3]{Sturmbarycentre}

Let $(X,d)$ be a $CAT(0)$ metric space.

Let $\mu$, $\nu$ two finitely supported measures on $(X,d)$.
We have:
$$d(Bar(\mu),Bar(\nu))\leq d^{W}(\mu,\nu).$$

\end{theo}

\begin{lem}\label{lemme: mesuresintersectentdifbarycentre}
Let $(X,d)$ be a $CAT(0)$ metric space. Let $C>0$. Let $\mu, \nu$ two finitely supported probability measure on $X$ of support diameter bounded by $C$ and let assume further that $\text{supp}(\mu) \cap \text{supp}(\nu)\neq \emptyset$.

Then:
$$ d (Bar(\mu),Bar(\nu)) \leq 2 C|| \mu \Delta \nu ||. $$

\end{lem}

\begin{proof}
Let $\mu$ and $\nu$ be two finitely supported measure of support of diameter bounded by $C$. We assume that $\text{supp}(\mu) \cap \text{supp}(\nu)\neq \emptyset$.
Therefore, we can write:
\begin{itemize}
    \item $\mu=\sum_{i=1}^{p} \mu_{i} \delta_{x_{i}} + \sum_{i>p}^{n}\mu_{i}\delta_{x_i}$ with $\sum_{i=1}^{p} \mu_{i}+\sum_{i>p}^{n}\mu_{i}=1$,

    \item $\nu=\sum_{i=1}^{p} \mu_{i} \delta_{x_{i}} + \sum_{i>p}^{m}\nu_{i}\delta_{y_i}$ with $\sum_{i=1}^{p} \mu_{i}+\sum_{i>p}^{m}\nu_{i}=1$,
\end{itemize}
if we allow the $x_{i}$ and the $y_{i}$ to not all be distinct.

We set the following transport plan on $X\times X$:
$$ \Pi= \sum_{i=1}^{p}\mu_{i}\delta_{(x_{i},x_{i})}+\frac{1}{1-\sum_{i=1}^{p}\mu_{i}}\sum_{i>p}^{n}\sum_{j>p}^{m}\mu_{i}\nu_{j}\delta_{(x_{i},y_{j})}.$$

An easy computation shows that $\Pi$ is in fact a transport plan between $\mu$ and $\nu$.

We set $a:= \sum_{i=1}^{p}\mu_{i}$.

An easy computation shows that:
$$ || \mu \Delta \nu ||=2-a$$

 With another easy computation, we get:

$$c(\Pi)\leq 2C \frac{1}{1-a}\sum_{i>p}^{n}\sum_{j>p}^{m}\mu_{i}\nu_{j} \leq 2C(1-a)\leq 2C(2-a),$$
thus $c(\Pi) \leq 2C|| \mu \Delta \nu ||$.

According to Theorem \ref{theo: inegalite avec Wasserstein} and Definition \ref{definition : metrique de wasserstein}, we get that:
$$ d (Bar(\mu),Bar(\nu)) \leq 2 C|| \mu \Delta \nu ||. $$

\end{proof}

We show here that $CAT(0)$ groups are admissible.

\begin{proposition}\label{proposition: lesgroupes Cat(0) sont admissibles}
$CAT(0)$ groups are admissible.
\end{proposition}

\begin{proof}
Let $P$ be a $CAT(0)$ group. By definition, there exists $(X,d)$ a $CAT(0)$ metric space such that $P$ acts properly and cocompactly by isometries on it. Moreover, there exists $\varphi$ a $P$-invariant quasi-isometry between $P$ equipped with the word metric and $(X,d)$.\\

Let $C>0$, let $ \mu,\nu \in \text{Proba}_{C}(P)$.
We define $\text{Proba}_{C}(X)$, the set of finitely supported probability measure of diameter bounded by $C$ on $X$. We define $\varphi^{\star}$, the map between $\text{Proba}_{C}(P)$ and $\text{Proba}_{C}(X)$ defined by $\varphi$ and by extension on the dirac masses.

We set:
$$ d_{b}(\mu,\nu)=d(\text{Bar}(\varphi^{\star}(\mu)),\text{Bar}(\varphi^{\star}(\nu))).$$

We will prove the $d_{b}$ satisfies all the assumption of Definition \ref{definition: lesbons parabolic}. $d_{b}$ is clearly a pseudo-metric on $\text{Proba}_{C}(P)$.\\

As $P$ acts by isometries on $(X,d)$, $d_{b}$ is clearly $P$ invariant.\\

As $\varphi$ is a quasi-isometry between $P$ and $(X,d)$. There exists $\eta>0$ such that for all $x \in X$, there exists $h \in P$ such that $d(x,\varphi(h))\leq \eta$.
We will show that $d_{b}$ is $\eta$-weakly geodesic.

Let $\mu, \nu\in \text{Proba}_{C}(P) $. We denote by $ \gamma$ the unique geodesic between $\text{Bar}(\varphi^{\star}(\mu))$ and $\text{Bar}(\varphi^{\star}(\nu))$ in $(X,d)$.

For all $t \in [0,1]$, there exists $\tilde{\gamma}(t)$ such that $d(\gamma(t),\varphi(\tilde{\gamma}(t))) \leq \eta$. Thus $\delta_{\tilde{\gamma}(t)} \in\text{Proba}_{C}(P)$ and we have:  
\begin{itemize}
    \item  $d_{b}(\mu,\delta_{\tilde{\gamma}(t)})=d(\text{Bar}(\varphi^{\star}(\mu)), \gamma(t) ) \leq d(\text{Bar}(\varphi^{\star}(\mu)),  \gamma(t))+\eta \leq t+\eta,  $

    \item $d_{b}(\delta_{\tilde{\gamma}(t)}, \nu)=d( \tilde{\gamma}(t), \text{Bar}(\nu)) \leq \eta+ d( \gamma(t), \text{Bar}(\nu)) \leq d_{b}(\mu,\nu)-t+\eta.$ 
\end{itemize}
Then $d_{b}$ is $\eta$-weakly geodesic.\\

We will show that $d_{b}$ is peakless. 
To begin, according to Corollary \ref{corollary: les espaces Cat(0) sont fortement convex}, we know that
for all $x,y,z \in X$ and $\gamma$ the unique geodesic between $x$ and $y$, we have, for all $t \in [0,1]$:
$$d(z,\gamma(t))\leq \max(d_{\ell^{2}}(z,x),d_{\ell^{2}}(z,y)).$$

Let $\mu, \nu,\lambda \in \text{Proba}_{C}(P)$, let  $\gamma$ be the unique geodesic between $\text{Bar}(\varphi^{\star}(\mu))$ and $\text{Bar}(\varphi^{\star}(\nu))$, let $\sigma$ be a weak geodesic between $\mu$ and $\nu$ defined as before. We have the following:
$$\begin{aligned}
d_{b}(\lambda,\sigma(t))&\leq d(\text{Bar}(\varphi^{\star}(\lambda)), \gamma(t) )+\eta \\ 
&\leq \max(d(\text{Bar}(\varphi^{\star}(\lambda)),\text{Bar}(\varphi^{\star}(\mu))),d(\text{Bar}(\varphi^{\star}(\lambda)),\text{Bar}(\varphi^{\star}(\nu)))+\eta .
\end{aligned}$$

Then $d_{b}$ is $\eta$-peakless.\\

We will show that $d_{b}$ is strongly-convex. Let $\eta \in (0,\frac{1}{2}), \varepsilon \in (0,1)$ and $N \in \mathbb{R}_{+}^{\ast}$. Let $\mu, \nu,\lambda \in \text{Proba}_{C}(P)$, with $d_{b}(\mu,\lambda)\leq N, d_{b}(\lambda,\nu) \leq N$ and $d_{b}(\mu,\nu)\geq \varepsilon N$. Let  $\gamma$ be the unique geodesic between $\text{Bar}(\varphi^{\star}(\mu))$ and $\text{Bar}(\varphi^{\star}(\nu))$, let $\sigma$ be a weak geodesic between $\mu$ and $\nu$ defined as before. According to Corollary \ref{corollary: les espaces Cat(0) sont fortement convex}, there exists $\delta \in (0,1)$ such that  for all $t \in [\eta,1-\eta]$:
$$d(\lambda,\gamma(t))\leq N(1-\delta).  $$

Therefore, for all $t \in [\eta,1-\eta] $, we have:
$$\begin{aligned}
d_{b}(\lambda,\sigma(t))&= d(\text{Bar}(\varphi^{\star}(\lambda)),\sigma(t)) \\
&\leq d(\text{Bar}(\varphi^{\star}(\lambda)), \gamma(t) )+\eta \\ 
&\leq N(1-\delta)+\eta .
\end{aligned}$$
Therefore, $d _{b}$ is $\eta$-strongly-convex.\\

As $d_{b}$ is weakly $\eta$-geodesic, $\eta$-peakless and $\eta$-strongly-convex, it satisfies weakly-$B2'$ according .\\

As $(X,d)$ is a $CAT(0)$ metric space, $(\text{Proba}_{C}(P),d_{b})$ satisfies clearly strongly-$B1$.\\

Let $g,h \in P$.
We have by definition:
$d_{b}(\delta_{g},\delta_{h})=d(\varphi(g),\varphi(h))$.
Therefore as $\varphi$ is quasi-isometry, the inclusion from $(P,d_{P})$ to $(\text{Proba}_{C}(P),d_{b})$ via the Dirac masses is a quasi-isometry

Finally, let $\mu,\nu \in \text{Proba}_{C}(P) $ with  $\text{supp}(\mu) \cap \text{supp}(\nu) \neq \emptyset$. Then $\text{supp}(\varphi^{\star}(\mu)) \cap \text{supp}(\varphi^{\star}(\nu)) \neq \emptyset$, therefore according to Lemma \ref{lemme: mesuresintersectentdifbarycentre}, we get that, we have: $$ d_{b}(\mu,\nu)= d(\text{Bar}(\varphi^{\star}(\mu)),\text{Bar}(\varphi^{\star}(\nu))) \leq 2C || \mu \Delta \nu || .$$
\end{proof}

We deduce now that finitely virtually abelian parabolics are admissible.

\begin{corollary}\label{corollary: virtually abelain groups are admissible}
 Virtually abelian parabolics are admissible.   
\end{corollary}

The rest of the article will be dedicated to the proof of the following Theorem.

\begin{theo}\label{theo: bolicite forte pour relativement hyperbolique avec paraobliques admissible}

Every \emph{relatively hyperbolic group with admissible parabolics} $G$ admits a metric $\hat{d}$ with the following properties.

\begin{itemize}
    \item $\hat{d}$ is $G$-invariant, i.e. $\hat{d}(g.x,g.y)=\hat{d}(x,y)$, for all $g,x,y \in G$,
    
    \item $\hat{d}$ is uniformly locally finite and quasi-isometric to the word metric,
    
    \item the metric space $(G,\hat{d})$ is weakly geodesic and strongly bolic.
\end{itemize}
    
\end{theo}

\subsection{The masks}\label{subsection: masks}

In this subsection,  we discuss an important notion in our proof defined by Dahmani and Chatterji in \cite[Section 2]{ChatterjiDahmani}, namely the masks.\\
To define masks we need to recall the definition of flow for a uniformly locally fine graph defined also in \cite{ChatterjiDahmani}. This concept generalizes the notion of flow introduced by Alvarez and Lafforgue in \cite{AlvarezLafforguehyperboliclp} in the context of uniformly finite hyperbolic graphs.
The flow takes an initial point and a target point and moves the initial point towards the target along "acceptable" paths. In Alvarez-Lafforgue, the acceptable paths are the $\alpha$-geodesics of coarse neighborhoods of geodesics. These are not suitable for uniformly fine graphs because they are not finite. Here, they are replaced by conical $\alpha$-geodesic cones, which behave similarly, but whose cardinality is really well controlled. The final step of the flow is a uniformly bounded support measure contained within the $1$-neighborhood of the target point, called the mask. The idea behind this name is that the support of the mask represents "the set of points that the target point must hide in order to hide the initial point." Masks will also play the role of random representative of cosets.

The most important property of masks is that they converge exponentially fast, i.e., for any pair of nearby points, the difference between their masks is exponentially small in terms of the distance from the source points to the target point.

Most of the results in this section come from \cite{ChatterjiDahmani}; however, we prove certain properties of the masks.
We first recall the definition of $\alpha$-geodesic and conical $\alpha$-geodesic.

\begin{definition}\label{definition : }\cite{ChatterjiDahmani}

Let $X$ be graph, $d$ its graph metric and $\alpha \geq 0$. For any two vertices $x,a \in X$, $t$ belongs to an $\alpha$-geodesic between $x$ and $a$ if :
$$ d(x,t)+d(t,a) \leq d(x,a) + \alpha. $$

\end{definition}

\begin{definition}\label{definition : gedeosic conique}\cite[Definition 1.12.]{ChatterjiDahmani}

Let $X$ be a graph, and $x,a \in X^{(0)}$. For $\rho \geq 0$ we denote by $\mathcal{E}_{a,x}(\rho)$ the set of edges $e$ with $d(a,e)=\rho$ that are contained in a geodesic from $a$ to $x$. We say that $t \in X$ belongs to an $\alpha$-conical-geodesic between $x$ and $a$ if $t$ is on an $\alpha$-geodesic, $d(a,t) \leq d(a,x)$, and :
$$ \displaystyle t \in \bigcap_{e \in \mathcal{E}_{a,x}(d(a,t))} \text{Cone}_{40 \alpha}(e).$$

We denote by $U_{\alpha}[a,x]$ the set of points belonging to an $\alpha$-conical-geodesic between $x$ and $a$.\\

For each $\rho \geq 0$, we denote $\mathcal{S}_{a,x}(\rho)$ the slice of $U_{\alpha}[a,x]$ at distance $\rho$ from $a$, that is the set :
$$\mathcal{S}_{a,x}(\rho) = \{ t \in U_{\alpha}[a,x] ~|~ d(t,a)= \rho \}$$

of points in $U_{\alpha}[a,x]$ at distance $\rho$ from $a$.
    
\end{definition}

In the rest of the text, we set $\alpha=2\delta$.

The following proposition shows that geodesics are included in the conical $\alpha$-geodesics.

\begin{proposition}\label{proposition : propriete des alpha cone geodesic}\cite[Proposition 1.13.(3)]{ChatterjiDahmani}
For every $\rho \geq 0$, for all $ v\in I_{c}(x,a)$ with $d_{c}(v,a)=\rho$, we have:
$$ v \in \mathcal{S}_{a,x}(\rho).$$

\end{proposition}

Let us denote by $\text{Proba}(X_{c})$ the set of finitely supported probability measure supported on the vertex of $X_{c}$.

We recall the definition of the flow step:

$$ 
T: \left\{ 
\begin{array}{ccc}
  \text{Proba}(X_{c})\times X_{c}^{(0)} &  \rightarrow & \text{Proba}(X_{c})\times X_{c}^{(0)} \\
   (\eta,a)  & \mapsto & (T_{a}(\eta),a)
\end{array}
\right. $$

We recall here that $\delta$ is assume to be an integer.
Given $a$ and $x$, we say that the step $T_{a}$ at $x$ is:
\begin{itemize}
    \item \emph{initial} if $d_{c}(a,x)> 5 \delta$ is not a multiple of $5\delta$,
    
    \item \emph{regular} if $d_{c}(a,x)> 5 \delta$ and $d(a,x)$ is a multiple of $5\delta$,
    
    \item \emph {ending} if $a$ is of infinite valence and $1< d_{c}(a,x) \leq 5\delta$, or if $a$ has finite valence, $d_{c}(a,x) \leq 5\delta$ and there exists $u$ such that $\measuredangle_{u}(a,x) > 1000\delta$,
    
    \item \emph{stationary} in all other cases, i.e. in any of the following three cases:
      \begin{itemize}
          \item $x=a$,
          
          \item $a$ is of infinite valence and $x$ is a neighbor of $a$,
          
          \item $a$ is of finite valence, $d_{c}(a,x)\leq 5\delta$ and for all $u$, $\measuredangle_{u}(a,x) \leq 1000\delta$.
      \end{itemize}
\end{itemize}

We denote by $\delta_{x}$ the Dirac mass at $x$ and define the flow step $T$ by defining $T_{a}(\delta_{x})$, for $a,x$ two vertices of $X$ and then extend by linearity on probability measures :
$$ T_{a}\big(\sum_{x} \lambda_{x} \delta_{x}\big)= \sum_{x} \lambda_{x} T_{a}(\delta_{x}) .$$

\begin{definition}\cite[Definition 2.1.]{ChatterjiDahmani}\label{definition : le flow step}
Let $a,x \in X_{c}^{(0)}$. We set $r_{a,x}$ to be the largest integer $r$ such that $5\delta r< d_{c}(a,x)$.
\begin{itemize}
    \item If the step $T_{a}$ at $x$ is either initial or regular, then $T_{a}(\delta_{x})$ is the uniform probability measure supported by the slice $\mathcal{S}_{a,x}=\mathcal{S}_{a,x}(5\delta r_{a,x})$ as in Definition \ref{definition : gedeosic conique},
    
    \item If the step $T_{a}$ at $x$ is ending, and $a$ has infinite valence then $T_{a}(\delta_{x})$ is the uniform probability measure supported by the slice $\mathcal{S}_{a,x}=\mathcal{S}_{a,x}(1)$ at distance $1$ from $a$. If the step is ending and $a$ has finite valence, then $T_{a}(\delta_{x})$ is the Dirac mass supported on the unique $c$ minimizing $d_{c}(a,c)$ such that $\measuredangle_{c}(a,x) > 900\delta$,

    \item If the step $T_{a}$ at $x$ is stationary, then $T_{a}(\delta_{x})=\delta_{x}$ (and $\mathcal{S}_{a,x}=\{x \})$.

\end{itemize}
\end{definition}

See Figure \ref{fig: masque}.

This proposition shows that the flow is eventually constant.

\begin{proposition}\cite[Proposition 2.2.]{ChatterjiDahmani}\label{proposition : le flow est stationnaire}

For all $a,x \in X_{c}^{(0)}$, the flow step from \ref{definition : le flow step} is stationary in $k$: if $k>r_{a,x}+1$, one has $T_{a}^{r_{a,x}+1}(\delta_{x})=T_{a}^{k}(\delta_{x})$.
\end{proposition}

The stationarity of the flow allows to define limits, the masks.

\begin{definition}\cite[Definition 2.3.]{ChatterjiDahmani}\label{defintion : definition du mask}

Given any two vertices $a$ and $x$ of $X_{c}$, the mask of $a$ for $x$ is the probability measure given by:
$$ \mu_{x}(a)=T_{a}^{R(a,x)}(\delta_{x})=\lim_{k \to \infty} T_{a}^{k}(\delta_{x}),$$

where $R(a,x)$ the minimal number of iterations of $T_{a}$ on $\delta_{x}$ so that the stationary step is reached.
\end{definition}

For $x \in X_{c}$, $a \in X_{c}^{\infty}$, $\mu_{x}(a)$ will play the role of a random representative of $x$ on the coset associated to $a$.\\

\begin{lem}\label{lemme: pas regulier du flow de points proches s'interctent}\cite[Lemma 2.6.]{ChatterjiDahmani}
Let $x,y \in X_{c}$, $a \in X_{c}$ and $d_{c}(a,x)=d_{c}(a,y)$.

If $d_{c}(x,y)\leq 8\delta$ and assume for both of them, the flow step is regular. Then we have:
$$ \text{supp}(T_{a}(\delta_{x})) \cap \text{supp}(T_{a}(\delta_{y})) \neq \emptyset . $$

\end{lem}

 The following proposition shows that masks have uniformly bounded diameter, see Figure \ref{fig: masque}.

\begin{proposition}\cite[Proposition 2.4.]{ChatterjiDahmani}\label{proposition : les masks sont uniformément bornés}

For all $x \in X_{c}$, for all $a \in  X_{c}^{ \infty}$, $a \neq x$, the measure $\mu_{x}(a)$ is supported by the $1$-neighborhood of $a$. Moreover, if $e$ is the first edge of a geodesic segment $[a,x]_{c}$, the measure $\mu_{x}(a)$ is supported by $Cone_{160\delta}(e)$.\\

For all $x \in X_{c}$, for all $a \in X_{c} \backslash X_{c}^{\infty}$, $\mu_{x}(a)$ is supported by $Cone_{1160\delta}(a)$.

In particular, the supports of the masks are uniformly bounded. To be more precise, we set $C=\text{diam} (Cone_{1160\delta}(e)$, and we have $ \text{diam}(\mu_{x}(a))\leq \text{diam} (Cone_{1160\delta}(e) \leq C$.
    
\end{proposition}

\begin{figure}[!ht]
    \centering
   \includegraphics[scale=0.5]{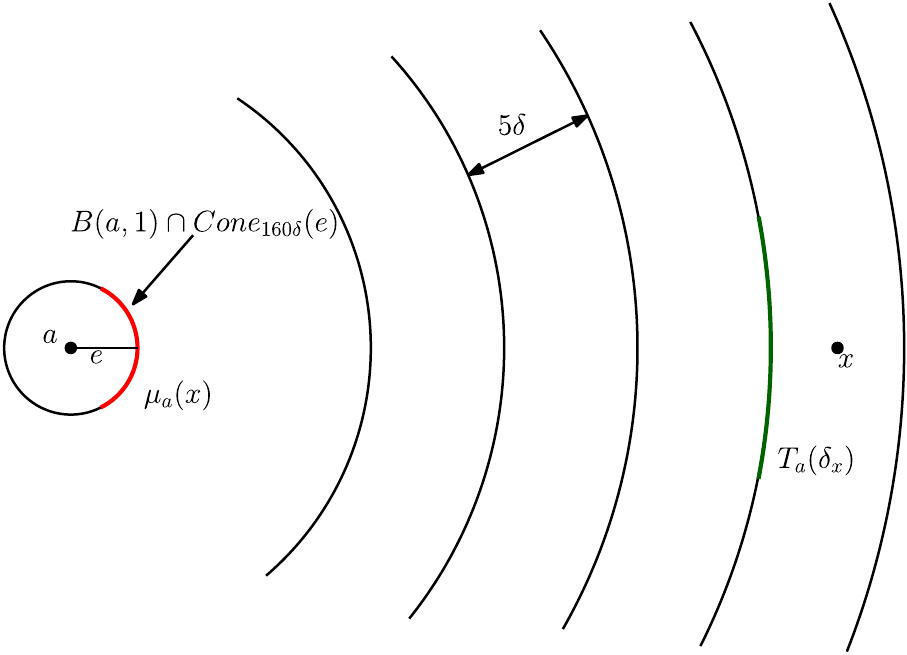}
   \caption{The mask of $x$ on $a$}
   \label{fig: masque}
\end{figure}

In the rest of the section, $C$ will denote the constant of Proposition \ref{proposition : les masks sont uniformément bornés} when we talk about $\text{Proba}_{C}(P_{i})$.

The following proposition says that when the flow passes a large angle, this large angle acts as a check point.

\begin{proposition}\cite[Proposition 2.7.]{ChatterjiDahmani} \label{proposition : grand angle comme des checkpoints}

Let $a,x$ be vertices in $X_{c}$. If there exists $u \in I_{c}(a,x)$ such that $\measuredangle_{u}(a,x)> (2000\delta)^{2}$, then for any $y$ at distance $1$ from $x$, we have that :
$$\mu_{x}(a)=\mu_{y}(a)=\mu_{u}(a).$$
\end{proposition}

The next proposition shows that for two sources of the flow, any point on a geodesic between those two sources has disjoint
masks for each of those sources.

\begin{proposition}\cite[Proposition 2.12.]{ChatterjiDahmani} \label{proposition: non-confluence des masks}

If $d_{c}(x,y)\geq 10\delta$ and if $a \in G$, at distance at least $5\delta$ from both $x$ and $y$, and $a \in [x,y]_{c}$ for some geodesic,$$|| \mu_{x}(a) \Delta \mu_{y}(a) ||=2.$$
    
\end{proposition}

With the following crucial property, we can control in terms of angle the difference between the masks of two points $x$ and $y$ close to each other at a point $a$ far away from them.

\begin{proposition}\cite[Proposition 2.11.]{ChatterjiDahmani}\label{proposition : la difference de deux masks est controlé par l'angle}

There exist two constants $\kappa<1$ and $R_{0}>0$ such that, for all $a \in X_{c}$, for every $x_{1},x_{2} \in X_{c}$ such that $d_{c}(x_{1},x_{2})=1$, if $x$ is among $x_{1},x_{2}$, closer to $a$, then for all geodesic $[x,a]_{c}$, if $\Theta(x,a)$ denotes the sum of angles on the geodesic $[x,a]_{c}$, then if $$d_{c}(a,x)+\Theta(a,x)\geq R_{0},$$ one has:
$$ || \mu_{x_{1}}(a) \Delta \mu_{x_{2}}(a) || \leq \kappa^{d_{c}(a,x)+\Theta(a,x)} . $$

\end{proposition}

The following corollary is a reformulation of Proposition \ref{proposition : la difference de deux masks est controlé par l'angle}.

\begin{corollary}\label{corollary : diffenrece entre deux masks loin nouvelle version}
There exists a constant $\kappa<1$ and $R_{0}>0$ such that, for all $x,y \in X_{c}$ at distance $1$ from each other, for all $a \in X_{c}$ and for all geodesics $[x,a]_{c}$, $[y,a]_{c}$, if $\Theta_{>(2000\delta)^{2}}(x,a)$, $\Theta_{>(2000\delta)^{2}}(y,a)$  denotes respectively the sum of angles greater than $(2000\delta)^{2}$ on the geodesic $[x,a]_{c}$, $[y,a]_{c}$.

For any $R\geq R_{0}$,
if $\Theta_{>(2000\delta)^{2}}(x,a)+d_{c}(x,a) \geq R$ one has:

\begin{itemize}
    \item $ || \mu_{x}(a) \Delta \mu_{y}(a) || \leq \kappa^{R-1} , $

    \item and $ || \mu_{x}(a) \Delta \mu_{y}(a) || \leq \kappa^{\text{min}(\Theta_{>(2000\delta)^{2}}(x,a)+d_{c}(x,a),\Theta_{>(2000\delta)^{2}}(y,a)+d_{c}(y,a))}.$
\end{itemize}

\end{corollary}

\begin{proof}

We fix two geodesics $[x,a]_{c}$ and $[y,a]_{c}$. We set $\Theta(x,a)$ the sum of angles on the geodesic $[x,a]_{c}$, respectively $\Theta(y,a)$ for the geodesic $[y,a]_{c}$. 

If $ \Theta_{(2000\delta)^{2}}(x,a)>0 $, then there exists $u \in I_{c}(x,a)$ such that $\measuredangle_{u}(a,x)> (2000\delta)^{2}$, then according to Proposition \ref{proposition : grand angle comme des checkpoints}:
$$ || \mu_{x}(a) \Delta \mu_{y}(a) ||=0. $$

The same holds for $y$, so let us assume that $\Theta_{>(2000\delta)^{2}}(x,a)=0$ and $\Theta_{>(2000\delta)^{2}}(y,a)=0$.

We take $R_{0}$ as in Proposition \ref{proposition : la difference de deux masks est controlé par l'angle} and $R \geq R_{0}$, therefore we get what follows.

If $x$ is closer to $a$ than $y$, then:
$$ || \mu_{x}(a) \Delta \mu_{y}(a) || \leq \kappa^{d_{c}(a,x)+\Theta(a,x)} \leq \kappa^{d_{c}(a,x)+\Theta_{>(2000\delta)^{2}}(a,x)} \leq \kappa^{R-1}.  $$

If $y$ is closer to $a$ than $x$, then $d_{c}(a,x)+\Theta_{>(2000\delta)^{2}}(a,y)=d_{c}(a,y)+\Theta_{>(2000\delta)^{2}}(a,x)-1 \geq R-1$ and:
$$ || \mu_{x}(a) \Delta \mu_{y}(a) || \leq \kappa^{d_{c}(a,y)+\Theta(a,y)} \leq \kappa^{d_{c}(a,y)+\Theta_{>(2000\delta)^{2}}(a,y)} \leq \kappa^{R-1}.$$

Moreover, we remark that:
$$  || \mu_{x}(a) \Delta \mu_{y}(a) || \leq \kappa^{\text{min}(\Theta_{>(2000\delta)^{2}}(x,a)+d_{c}(x,a),\Theta_{>(2000\delta)^{2}}(y,a)+d_{c}(y,a))}. $$

\end{proof}
The following proposition will be crucial for defining the strongly bolic metric. It shows that for $p$ large enough, for all $x, y \in X_{c}$ the difference of masks of $x$ and $y$ belongs to $l_{p}$.

\begin{proposition}\label{proposition : somme lp des différences de mask converge2}

There exists $p_{0}>1$ such that for all $p>p_{0}$ and for all $x, y \in X_{c}$ :
$$\sum_{a \in X_{c} } || \mu_{x}(a) \Delta \mu_{y}(a) ||^{p}  $$

is convergent.
    
\end{proposition}

\begin{proof}
According to \cite[Corollary 2.13.]{ChatterjiDahmani}\label{proposition : somme lp des différences de mask converge}, the proposition is true for $x\in G$ and $y\in X_{c}$.

Let assume that $ x,y \in X_{c}^{\infty}$. Let $x_{0} \in G$, we have the following:
$$\begin{aligned}
(\sum_{a \in X_{c}}  || \mu_{x}(a) \Delta \mu_{y}(a) ||^{p} )^{\frac{1}{p}}&\leq ( \sum_{a \in X_{c} } || \mu_{x}(a) \Delta \mu_{x_{0}}(a) ||^{p} ) ^{\frac{1}{p}}+ ( \sum_{a \in X_{c} } || \mu_{x_{0}}(a) \Delta \mu_{y}(a) ||^{p} ) ^{\frac{1}{p}}\\
& < \infty (\text{ according to \cite[Corollary 2.13.]{ChatterjiDahmani}}).
\end{aligned}$$

This proves the proposition.    
\end{proof}

The following proposition shows that, for neighbors, there is a uniform bound on this sum. Although this result is used in \cite{ChatterjiDahmani}, it is not stated explicitly; we include a proof here for the reader’s convenience.

\begin{proposition}\label{Proposition : sum des differences pour des voisins}

Let $p_{0}$ as in Proposition \ref{proposition : somme lp des différences de mask converge}. For all $p>p_{0}$, there exists $D>0$, which depends only on $p$ such that for all $x,y \in X_{c}$ with $d_{c}(x,y)=1$, we have:
$$\sum_{a \in X_{c} } || \mu_{x}(a) \Delta \mu_{y}(a) ||^{p}\leq D  .$$

\end{proposition}

\begin{proof}

By Definition \ref{definition: coned-off} of the coned-off graph, two adjacent vertices cannot both have finite valence.
Let assume that $x \in G, y  \in X_{c}$ and that $d_{c}(x,y)=1$. 
For all $a \in X_{c}$:
$$\begin{aligned} 
            || \mu_{x}(a) \Delta \mu_{y}(a) || &=|| x.\mu_{1}(x^{-1}a) \Delta x \mu_{x^{-1}y}(x^{-1}a) ||\\
            & =||  \mu_{1}(x^{-1}a) \Delta  \mu_{x^{-1}y}(x^{-1}a) || .
\end{aligned}$$

Therefore the inequality of the proposition is equivalent to the fact that for all $x \in X_{c}$, such that $d_{c}(1,x)=1$, there exists $D>0$ independent of $x$ such that:

$$\sum_{a \in X_{c} } || \mu_{1}(a) \Delta \mu_{x}(a) ||^{p}\leq D  .$$

Similarly to the proof of Corollary 2.13. in \cite{ChatterjiDahmani}, we consider the set $S_{d'_{(2000\delta)^{2}}}(R,G)= \{ g \in G, d'_{(2000\delta)^{2}}(1,g)=R \}$. Since $d'_{(2000\delta)^{2}}$ is quasi-isometric to $d_{G}$, there exists $A\geq 1$ and $\gamma_{G}>1$ such that $S_{d'_{(2000\delta)^{2}}}(R,G)$ has cardinality less than $A\gamma_{G}^{R}$.
We denote by $S_{d'_{(2000\delta)^{2}}}(R)$ the union of $S_{d'_{(2000\delta)^{2}}}(R,G)$ and of vertices $x \in X_{c}^{\infty}$ such that $d'_{(2000\delta)^{2}}(1,g)=R$. Its cardinality is at most $A'\gamma_{G}^{R}$ for a certain $A'$ related to the number of parabolic subgroups.

$\kappa$ denotes the constant of Proposition \ref{proposition : la difference de deux masks est controlé par l'angle} and Corollary \ref{corollary : diffenrece entre deux masks loin nouvelle version}. Then using Corollary \ref{corollary : diffenrece entre deux masks loin nouvelle version}, for $R$ large enough, we have:

$$ \displaystyle \sum_{a \in S_{d'_{(2000\delta)^{2}}}(R) } || \mu_{1}(a) \Delta \mu_{x}(a) ||^{p} \leq  (\kappa^{R-1})^{p} A' \gamma_{G}^{R}.$$

This inequality with the fact that $p$ is chosen in Proposition \ref{proposition : somme lp des différences de mask converge} such that moreover $ \kappa^{p} \leq \frac{1}{\gamma_{G}}$, proves that the sum $\sum_{a \in X_{c} } || \mu_{1}(a) \Delta \mu_{x}(a) ||^{p}$ converges and does not depend on $x$.

\end{proof}

The next lemma will be useful to describe the vertices of infinite valence, for which masks of neighbors do not intersect. We recall that for all $x,y \in X_{c}$, $$I_{c}(x,y):=\{z \in X_{c} ~|~ d_{c}(x,y)=d_{c}(x,z)+d_{c}(z,y) \}.$$

 \begin{lem}\label{proposition : les points a distance un du parabolique sont dans le support du masque}
For all $x \in X_{c}$ and all $a \in X_{c}^{\infty}$ with $a \neq x $, for all $u \in I_{c}(x,a)$ such that $d_{c}(a,u)=1$, we have:
$$ u \in \text{supp}(\mu_{x}(a)). $$

 \end{lem}

 \begin{proof}

 There are four cases to consider:
\begin{enumerate}
 \item If $d_{c}(a,x)=1$, $x$ is the only vertex in $I_{c}(x,a)$ such that $d_{c}(a,x)=1$, so $u=x$. The first step is stationary, thus $\mu_{a}(x)=\delta_{x}$ and $x\in \text{supp}(\mu_{x}(a)).$

\item If $1 < d_{c}(a,x) \leq 5 \delta$, then the first step of the flow is ending. Therefore $T_{a}(\delta_{x})$ is the uniform probability measure supported by the slice $\mathcal{S}_{a,x}=\mathcal{S}_{a,x}(1)$ and according to Proposition \ref{proposition : propriete des alpha cone geodesic}, for all $u \in I_{c}(x,a)$ such that $d_{c}(a,u)=1$, we have $u \in \mathcal{S}_{a,x}(1)$ so:
$$ u \in \text{supp}(T_{a}(\delta_{x}))  .$$

For the next steps the flow is stationary, therefore for all $u \in I_{c}(x,a)$ such that $d_{c}(a,u)=1$, we have:
$$ u \in \text{supp}(\mu_{x}(a)).$$

Let us assume now that $d_{c}(a,x) > 5 \delta$. We set $r_{a,x}$ as in Definition \ref{definition : le flow step} to be the largest integer $r$ such that $5\delta r<d_{c}(a,x)$.

     \item If the first step of the flow is regular. We can show by induction that for all $k \in \{0,...,r_{a,x}-1\}$, for all $u \in I_{c}(x,a)$ such that $d_{c}(a,u)=(r_{a,x}-k)5\delta$, we have:
$$ u \in \text{supp}(T_{a}^{k+1}(\delta_{x})).$$

For the initialization, $T_{a}(\delta_{x})$ is the uniform probability measure supported by the slice $\mathcal{S}_{a,x}=\mathcal{S}_{a,x}(5\delta r_{a,x})$. Thereby for all $u \in I_{c}(x,a)$, $d_{c}(a,u)=5\delta r_{a,x}$, we have:
$$ u\in \text{supp}(T_{a}^{1}(\delta_{x})). $$

Let us assume that the results hold for some $k \in \{0,...,r_{a,x}-2\}$. Then for all $v \in I_{c}(a,x)$ with $d_{c}(a,v)=(r_{a,x}-k)5\delta$, we have:
$$ v \in \text{supp}(T_{a}^{k+1}(\delta_{x})). $$

A new step of the flow gives the fact that, for all $v \in I_{c}(a,x)$ with $d_{c}(a,v)=(r_{a,x}-k)5\delta$, for all $u \in I_{c}(a,v)$ with $d_{c}(a,u)=(r_{a,x}-(k+1))5\delta$:
$$u \in \text{supp}(T_{a}^{1}(\delta_{v})) \subset \text{supp}(T_{a}^{k+2}(\delta_{x})) . $$

To conclude the proof of the induction, we just have to remark that the set of such $u$ with $d_{c}(a,u)=(r_{a,x}-(k+1))5\delta$ and such that there exists $v \in I_{c}(a,x)$ with $d_{c}(a,v)=(r_{a,x}-k)5\delta$ with $ u \in I_{c}(a,v) $ is exactly the set of $u \in I_{c}(x,a)$ such that $d_{c}(a,u)=(r_{a,x}-(k+1))5\delta.$

Then, for all $v \in I_{c}(a,x)$ with $d_{c}(a,v)=5\delta$, we have:
$$v \in \text{supp}(T_{a}^{r_{a,x}}(\delta_{x})) ,$$
and therefore $\text{supp}(\mu_{v}(a)) \subset \text{supp}(\mu_{x}(a)) $ for such $v$.

A new use of the first point of the proof gives the fact that for all $v \in I_{c}(a,x)$ with $d_{c}(a,v)=5\delta$, for all $u \in I_{c}(v,a)$ such that $d_{c}(a,u)=1$ :
$$ u \in \text{supp}(\mu_{v}(a)).   $$

To conclude, we only have to remark that the set of such $u$ is exactly the set of $u \in I_{c}(x,a)$ such that $d_{c}(a,u)=1$.

\item If the first step of the flow is initial, then $T_{a}(\delta_{x})$ is the uniform probability measure supported by the slice $\mathcal{S}_{a,x}=\mathcal{S}_{a,x}(5\delta r_{a,x})$, therefore, for all $v \in I_{c}(x,a)$ such that $d_{c}(a,v)=5\delta r_{a,x}$, we have:
$$ v \in \text{supp}(T_{a}(\delta_{x})), $$

and in the limit we get: $$ \text{supp}(\mu_{v}(a)) \subset \text{supp}(\mu_{x}(a)) . $$

The previous point gives the fact that for all $v \in I_{c}(v,a)$ such that $d_{c}(a,v)=5\delta r_{a,x}$, and for all $u \in I_{c}(v,a)$ such that $d_{c}(a,u)=1$, we have:
$$ u \in \text{supp}(\mu_{v}(a)).  $$

To conclude, we have to remark that the set of such $u$ is exactly the set of $u \in I_{c}(x,a)$ such that $d_{c}(a,u)=1$.
     
\end{enumerate}
\end{proof}

Here, we provide a description of the vertices for which the masks of neighboring points do not intersect.

\begin{proposition} \label{proposition : les masques s'intersectent sauf pour un}
Let $x,y \in X_{c}$ with $d_{c}(x,y)=1$, for all $a \in X_{c}^{\infty}$ with:
$$\text{supp}(\mu_{x}(a)) \cap \text{supp}(\mu_{y}(a)) = \emptyset, $$

we have:
$$a  \in Cone_{(2000\delta)^{2}}(\{x,y\}). $$
    
\end{proposition}

\begin{proof}

Let $a \in X_{c}^{\infty}$.
If $a=x$ or $a=y$, the result is clear, so we assume that $a \neq x, a\neq y$.

If $d_{c}(x,a)\neq 
d_{c}(y,a)$, we can assume that $d_{c}(x,a)>d_{c}(y,a)$. Since $d_{c}(x,y)=1$, we know that $y\in I_{c}(x,a)$. According to Proposition \ref{proposition : les points a distance un du parabolique sont dans le support du masque}, for all $u \in I_{c}(y,a)$ and $d_{c}(a,u)=1$, we have:
$$ u \in \text{supp}(\mu_{y}(a)). $$

Moreover, $ u \in I_{c}(y,a) \subset I_{c}(x,a) $, so $ u \in \text{supp}(\mu_{x}(a)) $ and then:
$$ \text{supp}(\mu_{x}(a)) \cap \text{supp}(\mu_{y}(a)) \neq \emptyset. $$\\

Therefore, we can assume that $d_{c}(x,a)=d_{c}(y,a)$. We can also assume that the angles between $a$ and $x$, and between $a$ and $y$ are bounded above by $(2000\delta)^{2}$. Otherwise, a use of Proposition \ref{proposition : grand angle comme des checkpoints} will give that $\text{supp}(\mu_{x}(a))=\text{supp}(\mu_{y}(a))$. Therefore, to prove the proposition, we just have to prove that:
$$ d_{c}(a,x) \leq (2000 \delta)^{2}.$$
We will, in fact, prove that $d_{c}(a,x)\leq 10\delta$.

Let us assume that $d_{c}(a,x) > 10 \delta$.  We will show that $\text{supp}(\mu_{x}(a)) \cap \text{supp}(\mu_{y}(a)) \neq \emptyset$.  We have two cases to consider.

\begin{itemize}
    \item If the flow step for $x$ and $y$ is regular. Since $d_{c}(x,y)=1\leq 8\delta $ and $d_{c}(x,a)=d_{c}(y,a)$, we can apply Lemma \ref{lemme: pas regulier du flow de points proches s'interctent} to see that:
$$ \text{supp}(T_{a}(\delta_{x})) \cap \text{supp}(T_{a}(\delta_{y})) \neq \emptyset , $$

     thus we deduce that $\text{supp}(\mu_{x}(a)) \cap \text{supp}(\mu_{y}(a)) \neq \emptyset$.

    \item If the flow step for $x$ and $y$ is initial. Since $d_{c}(a,x) > 10 \delta$, there exists $k>1$ such that $T_{a}(\delta_{x})$ is the uniform probability measure supported by the slice $\mathcal{S}_{a,x}=\mathcal{S}_{a,x}(5k \delta)$ and $T_{a}(\delta_{y})$ is the uniform probability measure supported by the slice $\mathcal{S}_{a,y}=\mathcal{S}_{a,y}(5 k\delta)$.

    Let us consider $x'\in I_{c}(x,a)$ and $y' \in I_{c}(y,a)$ such that $d_{c}(x',a)=d_{c}(y',a)=5k\delta$, we know that $x' \in \text{supp}(T_{a}(\delta_{x})) $ and that $y' \in \text{supp}(T_{a}(\delta_{y}))$. Moreover, by $\delta$-hyperbolicity, we remark that $d_{c}(x',y')\leq 2\delta +1 \leq 8\delta$. Since $k>1 $, the flow step for $x'$ and $y'$ is regular, then a new use of Lemma \ref{lemme: pas regulier du flow de points proches s'interctent} give that:
$$ \text{supp}(T_{a}(\delta_{x'})) \cap \text{supp}(T_{a}(\delta_{y'})) \neq \emptyset  .$$

    Since $\text{supp}(T_{a}(\delta_{x'})) \cap \text{supp}(T_{a}(\delta_{y'})) \subset \text{supp}(T_{a}^{2}(\delta_{x})) \cap \text{supp}(T_{a}^{2}(\delta_{y}))$, we conclude that:
$$\text{supp}(\mu_{x}(a)) \cap \text{supp}(\mu_{y}(a)) \neq \emptyset.$$

\end{itemize}

Therefore, we know that $d_{c}(a,x) \leq 10 \delta \leq (2000\delta)^{2}$ and thus $a \in Cone_{(2000\delta)^{2}}(\{x,y\}) $.

\end{proof}

With the following corollary, we count the number of masks that do not intersect for each pair of points.

\begin{corollary}\label{corollary : cardinal des points paraboliques avec distance entre les barycentres des masks nest pas bien controlee}

Let us fix $x,y \in X_{c}$. Let us assume that $d_{c}(x,y)=n$.

Consider a geodesic $[ x,y]_{c}$ between $x$ and $y$ in $X_{c}$.

We order $[x,y]_{c}$ in the following sense: $[x,y]_{c}=\{x_{0},x_{1},...,x_{n}\}$ with $x_{0}=x$, $x_{n}=y$ and $d_{c}(x_{i},x_{i+1})=1$ for $0\le i \le n-1$.

We set:
$$\mathcal{A}:= \{ a \in X_{c}^{\infty} ~ | ~ \exists i_{0} \in  \{0,...,n-1\} ~ \text{such that} ~  \text{supp}(\mu_{x_{i_{0}}}(a)) \cap \text{supp}(\mu_{x_{i_{0}+1}}(a)) = \emptyset  \}.$$

Then $\mathcal{A} \subset \text{Cone}_{(2000\delta)^{2}}([x,y]_{c})$ and therefore there exists a constant $L>0$, which depends on the hyperbolicity constant and on the function of uniform finesse of $X_{c}$ only, such that :
$$ |\mathcal{A}| \leq L d_{c}(x,y).   $$

\end{corollary}

\begin{proof}

Let $a \in \mathcal{A}$, then by definition there exists $i_{0} \in  \{0,...,n-1\}$ such that: 
$$ \text{supp}(\mu_{x_{i_{0}}}(a)) \cap \text{supp}(\mu_{x_{i_{0}+1}}(a)) = \emptyset. $$

According to Proposition \ref{proposition : les masques s'intersectent sauf pour un}, we have:
$$a \in \text{Cone}_{(2000\delta)^{2}}(\{x_{i_{0}},x_{i_{0}+1}\}) .$$

According to Proposition \ref{cardinal cones}, there exists $L>0$ such that, for all $e \in X_{c}^{(1)}$:
$$|\text{Cone}_{(2000\delta)^{2}}(e)| \leq L.$$

Thereby :
$$ \mathcal {A} \subset \bigcup \limits_{i=0}^{n-1} \text{Cone}_{(2000\delta)^{2}}(\{x_{i},x_{i+1}\})=\text{Cone}_{(2000\delta)^{2}}([x,y]_{c}) .$$

And to conclude :
$$ |\mathcal{A}| \leq | \bigcup \limits_{i=0}^{n-1} \text{Cone}_{(2000\delta)^{2}}(\{x_{i},x_{i+1}\}) | \leq L n \leq L d_{c}(x,y). $$

\end{proof}

\subsection{Definition of pseudo-metric associated to every vertex of the coned-off graph}\label{subsection: definition de da ,control}
In the rest of the text, $G$ will denote a group hyperbolic relative to admissible subgroups \ref{definition: lesbons parabolic}. Let $X_{c}$ denote one of its coned-off graphs. Let $d_{b}$ denote the strongly bolic pseudo-distance in the set of probability measures with support of diameter at most $C$, without distinction in the parabolic $P_{i}$, where $C$ is the constant of Proposition \ref{proposition : les masks sont uniformément bornés}.
Let $P_{i}$ be a parabolic subgroup,  $g \in G$, $ \mu, \nu \in \text{Proba}_{C}(gP_{i})$, we set:
$$d_{b}(\mu,\nu):=d_{b}(g^{-1}\mu,g^{-1}\nu).$$
This is well defined since $d_{b}$ is $P_{i}$-invariant.

For all $ a \in X_{c}$, we will define a pseudo-metric, which will serve as a crucial tool in the construction of the strongly bolic metric.

\begin{definition}\label{definition: da}

Let $x,y \in G$ we set:

\begin{itemize}
    \item when $a \in X_{c}^{\infty}$, $d_{a}(x,y):=d_{b}(\mu_{x}(a),\mu_{y}(a))$, where $d_{b}$ is the strongly bolic metric on the associated parabolic,

    \item when $a \in X_{c} \backslash X_{c}^{\infty} $, $d_{a}(x,y):= || \mu_{x}(a) \Delta \mu_{y}(a) ||$.
\end{itemize}

\end{definition}

In particular, for all $a \in X_{c}$, the associated function $d_{a}$ is a pseudo-metric.\\

The following lemma will be used to prove the crucial result that allows us to have precise control on $d_{a}$.

We recall that for all $x,y \in X_{c}$, for all $M>0$,
$\Theta'(x,y)$ denotes the minimal sum of angles on a geodesic between $x$ and $y$ and $\Theta'_{>M}(x,y)$ denotes the minimal sum of angles greater than $M$ on a geodesic between $x$ and $y$. 

 \begin{figure}[!ht]
   \centering
   \includegraphics[scale=0.4]{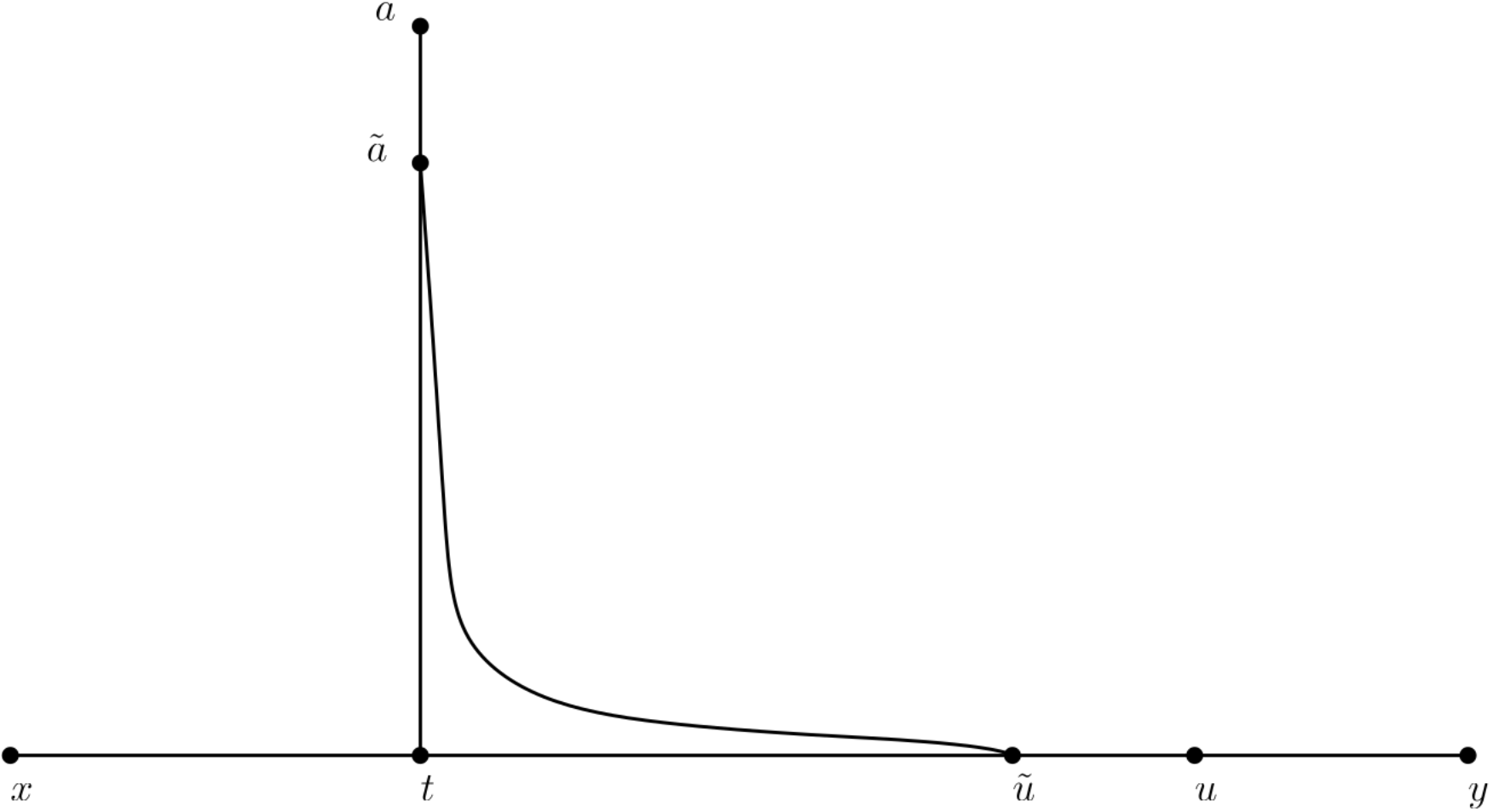}
  \caption{Angles on the sides of triangles $[\tilde{a},t,\tilde{u}]$ are bounded by $100\delta$.}
   \label{fig: somme des angles plus petite vers la projection}
\end{figure}

\begin{lem}\label{lem: somme des angles vers la projection}
Let $[x,y]_{c}$ be a geodesic in $X_{c}$, $u \in [x,y]_{c}$ and $a  \in X_{c}$, and $M>100 \delta$. Let $t \in X_{c}$ be a projection of $a$ on $[x,y]_{c}$.

There exist two geodesics $[u,a]_{c}$, $[t,a]_{c}$ such that:
$$\Theta_{>M}(a,u) \geq \Theta_{>M}(a,t)-12\delta, $$ where $\Theta_{>M}(a,u),\Theta_{>M}(a,t)$ denotes the sum of angles greater than $M$ on $[a,u]_{c}$, $[a,t]_{c}$ respectively.
In particular, we have:
$$\Theta_{>M}(a,u)\geq \Theta_{>M}'(a,t)-12\delta,$$

\end{lem}

\begin{proof}
We can assume that $u \in [y,t]_{c}$. We choose a geodesic triangle $[u,a,t]_{c}$ as in Theorem \ref{formenormaledestriangles} in which the triangle $[\tilde{u},\tilde{a},\tilde{t}]_{c}$ fits nicely and is such that $[u,t]_{c} \subset [x,y]_{c}$, see Figure \ref{fig: somme des angles plus petite vers la projection}.

To begin, let us note that $\tilde{t} \in [y,t]_{c} \subset [x,y]_{c} $ and $\tilde{t} \in [t,a]_{c} $ then $\tilde{t}=t$ since $t$ is a projection of $a$ on $[x,y]_{c}$.

In the case, where $t,\tilde{a},\tilde{u}$ are not all different, we have $[a,t]_{c}\subset [a,u]_{c}$ and the result follows directly.

Let us assume that $t,\tilde{a},\tilde{u}$ are all different.
Let us denote, $\Theta_{>M}(a,u)$ the sum of angles greater than $M$ on $[a,u]_{c}$. In the case where $a \neq \tilde{a}$, we denote by $e_{1}$ the edge of $[a,\tilde{a}]_{c}$ such that $\tilde{a}\in e_{1}$. We denote by $e_{2}$ the edge of $[\tilde{a},t]_{c}$ such that $\tilde{a} \in e_{2}$ and $e_{2}'$ the edge of $[\tilde{a},\tilde{u}]_{c}$ such that $\tilde{a} \in e_{2}'$.
In the same way, we denote by $f_{1}$ the edge of $[\tilde{a},\tilde{u}]_{c}$ such that $\tilde{u}\in f_{1}$ and in the case where $\tilde{u}\neq u$, $f_{2}$ the edge of $[\tilde{u},u]_{c}$ with $\tilde{u}\in f_{2}$.
According to Theorem \ref{formenormaledestriangles}, the angles along $[\tilde{a},\tilde{u}]$ are bounded by $
100\delta$, and
therefore:
\begin{itemize}
    \item $\Theta_{>M}(a,u)=\Theta_{>M}(a,\tilde{a})+[\measuredangle_{\tilde{a}}(e_{1},e_{2}')]_{>M}+[\measuredangle_{\tilde{u}}(f_{1},f_{2})]_{>M}+\Theta_{>M}(\tilde{u},u)$,

    \item and  $\Theta_{>M}(a,t)=\Theta_{>M}(a,\tilde{a})+[\measuredangle_{\tilde{a}}(e_{1},e_{2})]_{>M}$,
\end{itemize}

with the convention that the angles $\measuredangle_{\tilde{a}}(e_{1},e_{2}'), \measuredangle_{\tilde{a}}(e_{1},e_{2}), \measuredangle_{\tilde{u}}(f_{1},f_{2})$ are equal to zero when $e_{1}$, respectively $f_{1}$ are not defined.

According to the triangular inequality and to Proposition \ref{proposition: angle en a tilde}, we have:
$$ |\measuredangle_{\tilde{a}}(e_{1},e_{2}')-\measuredangle_{\tilde{a}}(e_{1},e_{2})|\leq \measuredangle_{\tilde{a}}(e_{2},e_{2}')\leq 12 \delta.$$

This gives the first inequality of the proposition, the last point follows directly.

\end{proof}

The following counting lemma will be useful later on.

\begin{lem}\label{lemme: lemme de comptage avec d'deuxmile}

There exists a constant $A\geq 1$ and $\gamma_{G}>1$, such that for all $x \in G$, for all integer $n$:
$$|\{a \in X_{c}~|~d'_{(2000\delta)^{2}}(x,a)\leq n  \}| \leq A \gamma_{G}^{n}.$$
\end{lem}

\begin{proof}

The argument is similar to the proof of Proposition \ref{Proposition : sum des differences pour des voisins} combined with Corollary \ref{corollary : reformulation de la formule de la distance}.
    
\end{proof}

This lemma shows that if a point $a$ belongs to a cone around a geodesic, the sum of the angles between the point $a$ and its projection onto the geodesic is uniformly bounded.

\begin{lem}\label{lemma: si a dans cone alors proche de projection dans le groupe}
Let $x,y \in X_{c}$ and $a \in X_{c}$, and $t \in X_{c}$ a projection of $a$ on $[x,y]_{c}$, there exists an integer $C_{0}>0$ such that $a \in \text{Cone}_{(2000\delta)^{2}}([x,y]_{c}) $ implies that:
$$d'_{(2000\delta)^{2}}(a,t) \leq C_{0}.$$
    
\end{lem}

\begin{proof}

Since $a \in \text{Cone}_{(2000\delta)^{2}}([x,y]_{c}) $, we know that $d_{c}(a,t) \leq (2000\delta)^{2}. $

Moreover, there exists a point $u \in [x,y]_{c} $ and a geodesic from $u$ to $a$ such that all angles in this geodesic are bounded by $(2000\delta)^{2}$.

Therefore, according to Lemma \ref{lem: somme des angles vers la projection}, we get that:
$$ \Theta_{(2000\delta)^{2}}'(t,a)\leq 12\delta.$$

Hence:
$$d'_{(2000\delta)^{2}}(a,t) \leq (2000\delta)^{2}+12\delta,$$

and with $C_{0}=(2000\delta)^{2}+12\delta$ we get the desired result.

\end{proof}

With this proposition, we can control the difference between two consecutive masks on a geodesic toward a point $a$ using the distance from $a$ to this geodesic.

\begin{proposition}\label{proposition: controle mask et projection}
Let $x,y \in G$ and $a \in X_{c}$. We choose a geodesic $[x,y]_{c}=\{x_{0},x_{1},...,x_{n}\}$ with $x_{0}=x$, $x_{n}=y$ and $d_{c}(x_{i},x_{i+1})=1$ for $0\le i \le n-1$, with $n=d_{c}(x,y)$. Let $i_{0} \in [\![ 0,n ]\!]$ such that $x_{i_{0}}$ is projection of $a$ on $[x,y]_{c}$. If $d'_{(2000\delta)^{2}}(a,x_{i_{0}})\geq R_{0}$, with $R_{0}$ the constant of Corollary \ref{corollary : diffenrece entre deux masks loin nouvelle version}, we have the following,for all $i \in [\![ 0,n ]\!]$:

$$|| \mu_{x_{i}}(a) \Delta \mu_{x_{i+1}}(a) || \leq \kappa^{-24\delta} \kappa^{d'_{(2000\delta)^{2}}(a,t)+|i-i_{0}|},$$

    where $\kappa$ denotes the constant of Proposition \ref{proposition : la difference de deux masks est controlé par l'angle}.
\end{proposition}

\begin{proof}
Let $i \in [\![ 0,i_{0}-1 ]\!]$.
If $d'_{(2000\delta)^{2}}(a,x_{i_{0}})\geq R_{0}$, according to Lemma \ref{lem: somme des angles vers la projection}, there are two geodesics from $a$ to $x_{i}$ and from $a$ to $x_{i+1}$ respectively, such that the sums of angles greater than $(2000\delta)^{2}$ on these geodesics are equal to $\Theta'_{(2000\delta)^{2}}(a,x_{i_{0}})$.\\

Therefore, according to the second point of Corollary \ref{corollary : diffenrece entre deux masks loin nouvelle version}, we have the following:
$$|| \mu_{x_{i}}(a) \Delta \mu_{x_{i+1}}(a) || \leq \kappa^{\Theta'_{(2000\delta)^{2}}(a,x_{i_{0}})+\min(d_{c}(a,x_{i}),d_{c}(a,x_{i+1}))}.$$

According to Lemma \ref{lemme : z est la projection donc c'est un quasi-centre}, since $x_{i_{0}}$ is the projection of $a$ on $[x,y]_{c}$, we have the following:
\begin{itemize}
    \item $ d_{c}(a,x_{i})\geq d_{c}(a,x_{i_{0}})+d_{c}(x_{i_{0}},x_{i})-24\delta \geq d(a,x_{i_{0}})+i_{0}-i-24\delta$,

    \item $d_{c}(a,x_{i+1})\geq d_{c}(a,x_{i_{0}})+d_{c}(x_{i_{0}},x_{i+1})-24\delta \geq d_{c}(a,x_{i_{0}})+i_{0}-(i+1)-24\delta.$
\end{itemize}

Thus $\min(d_{c}(a,x_{i}),d_{c}(a,x_{i+1}))  \geq d_{c}(a,x_{i_{0}})+i_{0}-i-24\delta$ and we have:

$$|| \mu_{x_{i}}(a) \Delta \mu_{x_{i+1}}(a) || \leq \kappa^{-24\delta} \kappa^{-d'_{(2000\delta)^{2}}(a,t)-(i_{0}-i)}.$$

The same proof applies when $i \in [\![i_{0},n-1 ]\!]$.
\end{proof}

The following theorem is the most important result of the section. For all $x,y \in G$ and for all $a\in X_{c}^{\infty}$ sufficiently far from $[x,y]_{c}$, it allows us to control $d_{a}(x,y)$ by a decreasing exponential of the distance in the group from $a$ to the geodesic $[x,y]_{c}$.

\begin{theo}\label{theo: theo stylé sur les barycentres et les masques}
    
Let $x, y \in G$ and $a \in X_{c}$, we choose a geodesic $[x,y]_{c}$ between $x$ and $y$.
Let $t \in X_{c}$ be a projection of $a$ on $[x,y]_{c}$. There exists $C_{1}>0$ and $K_{1}>0$ such that $d'_{(2000\delta)^{2}}(a,t) > C_{1}$ implies that:

$$d_{a}(x,y)\leq K_{1} \kappa^{d'_{(2000\delta)^{2}}(a,t)} ,$$

 where $\kappa$ denotes the constant of Proposition \ref{proposition : la difference de deux masks est controlé par l'angle}.

\end{theo}

\begin{proof}

Let us denote $n:=d_{c}(x,y)$. We order $[ x,y]_{c}$ in the following sense: $[x,y]_{c}=\{x_{0},x_{1},...,x_{n}\}$ with $x_{0}=x$, $x_{n}=y$ and $d_{c}(x_{i},x_{i+1})=1$ for $0\le i \le n-1$. Let $i_{0} \in [\![ 0,n ]\!] $ such that $t=x_{i_{0}}$.  We will start to prove that for all $a  \in X_{c}$ such that for $d'_{(2000\delta)^{2}}(a,t) > C_{0}$, with $C_{0}$ the constant of Lemma \ref{lemma: si a dans cone alors proche de projection dans le groupe}, we have:
$$ d_{a}(x,y) \leq C(|| \mu_{x_{0}}(a) \Delta \mu_{x_{1}}(a) ||+...+|| \mu_{x_{n-1}}(a) \Delta \mu_{x_{n}}(a) ||),$$
with $C$ the constant of Proposition \ref{proposition : les masks sont uniformément bornés}.

If $a$ is of finite valence, this is clear. Indeed:
$$\begin{aligned}d_{a}(x,y)& = || \mu_{x}(a) \Delta \mu_{y}(a) ||\\
& \leq || \mu_{x_{0}}(a) \Delta \mu_{x_{1}}(a) ||+...+|| \mu_{x_{n-1}}(a) \Delta \mu_{x_{n}}(a) ||\\
&\leq  C(|| \mu_{x_{0}}(a) \Delta \mu_{x_{1}}(a) ||+...+|| \mu_{x_{n-1}}(a) \Delta \mu_{x_{n}}(a) || )~(\text{since} ~ C \geq 1 ).
\end{aligned}$$

Assume that $a \in X_{c}^{\infty}$, if $d'_{(2000\delta)^{2}}(a,t) > C_{0} $, we know according to Lemma \ref{lemma: si a dans cone alors proche de projection dans le groupe}, that $a \notin \text{Cone}_{(2000\delta)^{2}}([x,y]_{c})$. Therefore, according to Corollary \ref{corollary : cardinal des points paraboliques avec distance entre les barycentres des masks nest pas bien controlee}, for all $i \in [\![ 0,n-1 ]\!]$, $\text{supp}(\mu_{x_{i}}(a)) \cap \text{supp}(\mu_{x_{i+1}}(a)) \neq \emptyset$ and according to Definition \ref{definition: lesbons parabolic}, we have:
$$ d_{a}(x_{i},x_{i+1}) \leq  C || \mu_{x_{i}}(a) \Delta \mu_{x_{i+1}}(a) ||.$$

Hence:
$$\begin{aligned}d_{a}(x,y)& \leq d_{a}(x_{0},x_{1})+...+d_{a}(x_{n-1},x_{n})\\
&\leq  C (|| \mu_{x_{0}}(a) \Delta \mu_{x_{1}}(a) ||+...+|| \mu_{x_{n-1}}(a) \Delta \mu_{x_{n}}(a) || ).
\end{aligned}$$

Let us assume further that $d'_{(2000\delta)^{2}}(a,t) > R_{0}$, the constant of Proposition \ref{proposition: controle mask et projection}. We can apply Proposition \ref{proposition: controle mask et projection} to get what follows.
$$\begin{aligned}
d_{a}(x,y)
& \leq C (\displaystyle \sum_{i=0}^{n-1} || \mu_{x_{i}}(a) \Delta \mu_{x_{i+1}}(a) || ~(\text{since} ~ C \geq 1 ) \\
& \leq C \kappa^{-24\delta} \sum_{i=0}^{n-1}  \kappa^{d'_{(2000\delta)^{2}}(a,t)+|i-i_{0}|} ~(\text{Proposition \ref{proposition: controle mask et projection}})\\
&  \leq C \kappa^{-24\delta} \kappa^{d'_{(2000\delta)^{2}}(a,t)} (  \sum_{i=0}^{i_{0}}\kappa^{i}+\sum_{i=1}^{n-i_{0}-1}\kappa^{i})\\
& \leq C\kappa^{-24\delta} \kappa^{d'_{(2000\delta)^{2}}(a,t)} \frac{2}{1-\kappa}.
\end{aligned}$$

Then with $C_{1}=\max(C_{0},R_{0})$ and $K_{1}= C\kappa^{-24\delta} \frac{2}{1-\kappa} $, we get the desired result.
\end{proof}

\subsection{Comparison between angles and masks at vertices of infinite valence}\label{subsection: comparaison angles et masques}

In this section, we show that for points $x,y$ in the coned-off graph, for all $a \in X_{c}^{\infty}$, $d_{a}$ is quasi-isometric to the angle between $x$ and $y$ as seen from $a$.

\begin{proposition}\label{proposition: l'angle est borné par la distance bolic}

There exist $A,B>0$ such that for all $x,y \in X_{c}$ and $a \in X_{c}^{\infty}$, $a \neq x,y$:
$$ \frac{1}{A} d_{a}(x,y)-B  \leq \measuredangle_{a}(x,y) \leq  Ad_{a}(x,y)+B.$$

\end{proposition}

\begin{proof}

We will start to prove that for all $u,v \in X_{c}$ and $a \in X_{c}^{\infty}$ such that $d_{c}(u,a)=1$ and $d_{c}(v,a)=1$, there exist
$A>0$ such that :
$$ \frac{1}{A} d_{b}(\delta_{u},\delta_{v})  \leq \measuredangle_{a}(x,y) \leq  Ad_{b}(\delta_{u},\delta_{v}).$$

 If we denote $gH$ the parabolic coset associated to $a$ for $g \in G$ and $H$ a subgroup of $G$, we denote by
$d_{H}$ the graph distance in $H$.
Therefore $ d_{H}(u,v)$ is the length of the shortest path in the $1$-neighborhood of $a$ between $u$, $v$ in $X_{c}$.

We set the two edges $e=\{a,u\}$ and $f=\{a,v \}$. This path avoids $a$, then by definition of angles:
$$ \measuredangle_{a}(e,f) \leq  d_{H}(u,v). $$

Moreover, according to Definition \ref{definition: lesbons parabolic}, there exists $\beta>1$, such that :
$$ \frac{1}{\beta} d_{b}(\delta_{u},\delta_{v})-\beta \leq d_{H}(u,v) \leq \beta d_{b}(\delta_{u},\delta_{v})+\beta . $$

Therefore,
$$ \measuredangle_{a}(e,f) \leq \beta d_{b}(\delta_{u},\delta_{v})+\beta.  $$

For the converse inequality, assume that $\measuredangle_{a}(u,v)=n$ and denote by $u_{1},...,u_{n}$ the vertices on a shortest path from $u$ to $v$ that avoids $a$. Let us denote $\tilde{u_{i}}$ the projection of each $u_{i}$ on $gH$. Then looking at the triangle $[\tilde{u_{i}},u_{i},\tilde{u_{i+1}}]$, we get a path of length smaller than $8 \delta +2$, which avoids $a$, then $  \measuredangle_{a}(\tilde{u_{i}},\tilde{u_{i+1}}) \leq 8\delta +2 $.
Using the fact that $X_{c}$ is uniformly fine, there are finitely many transitions $\tilde{u_{i}}^{-1}\tilde{u_{i+1}}$. Thus, there exists $K>0$, such that $ d_{H}(\tilde{u_{i}},\tilde{u_{i+1}})\leq K$. Thus, $ d_{H}(u,v) \leq n K \leq \measuredangle_{a}(u,v) K .$

With the first point, we conclude that there exists $A>1$ such that:
$$ \frac{1}{A} d_{b}(\delta_{u},\delta_{v})-A \leq \measuredangle_{a}(u,v) \leq  Ad_{b}(\delta_{u},\delta_{v})+A.$$

In particular, the inequalities apply for all $x,y \in X_{c}$, for all $a \in X_{c}^{\infty}$ and for all $u \in I_{c}(x,a)$ and $ v \in I_{c}(y,a) $ such that $d_{c}(u,a)=d_{c}(v,a)=1$.

Moreover, according to Definition \ref{definition : le flow step}, $\delta_{u}=\mu_{u}(a)$ since the first of the flow is stationary and according to Lemma \ref{proposition : les points a distance un du parabolique sont dans le support du masque}, $u \in \text{supp}(\mu_{x}(a))$. Therefore, by Definition \ref{definition: lesbons parabolic}, we have:
$$ d_{b}(\delta_{u},\mu_{x}(a)) \leq 2C .$$

In the same way, we have $d_{b}(\delta_{v},\mu_{y}(a)) \leq 2C$, then for all $u \in I_{c}(x,a)$ and $ v \in I_{c}(y,a) $ such that $d_{c}(u,a)=d_{c}(v,a)=1$, we have:
$$ \frac{1}{A} d_{b}(\mu_{x}(a)\mu_{y}(a))- \frac{1}{A} 2C-A \leq \measuredangle_{a}(u,v) \leq  Ad_{b}(\mu_{x}(a)\mu_{y}(a))+A2C+A.$$

Since this inequality is true, for all $u \in I_{c}(x,a)$ and $ v \in I_{c}(y,a) $ such that $d_{c}(u,a)=d_{c}(v,a)=1$, we have:
$$ \frac{1}{A} d_{a}(x,y)- B \leq \measuredangle_{a}(x,y) \leq  Ad_{a}(x,y)+B,$$

 with for some $B$ and $d_{a}(x,y)=d_{b}(\mu_{x}(a),\mu_{y}(a))$ by Definition \ref{definition: da}.

\end{proof}

\subsection{The strongly bolic metric}\label{subsection: la metrique fortment bolique}

In this subsection, we will define the strongly bolic metric and prove that it is indeed a metric.

We need some preliminary lemmas before defining it.

\begin{lem}\label{lem: angle vers projection cone}
Let $M>6\delta$.
Let $x,y \in X_{c}$, a geodesic $[x,y]_{c}$ between $x$ and $y$ and $z \in Cone_{M}([x,y]_{c})$.
Then for all $v \in X_{c}$, $v\neq x,y,z$, we have:
\begin{itemize}
    \item $ \measuredangle_{v}(x,z)\leq 2M$, 
    \item or $ \measuredangle_{v}(z,y)\leq 2M.$
\end{itemize}

\end{lem}

\begin{proof}

We will prove by contradiction that there exists no $v $ such that:
\begin{itemize}
    \item $ \measuredangle_{v}(x,z) > 2M$,

    \item $ \measuredangle_{v}(z,y) > 2M.$ 
\end{itemize}

In the case, where $z \in [x,y]_{c}$, if there exists a such $v$ this means according to Proposition \ref{Proposition: angle for hyperbolic metric space} that $v \in [x,z]_{c} \cap [z,y]_{c}$, which is not possible. 

Suppose that $z \notin [x,y]_{c}$, we denote by $t$ a projection of $z$ on $[x,y]_{c}$. Since $z  \in Cone_{M}([x,y]_{c}) $, for all $v \neq t$, we have:
$$ \measuredangle_{v}(z,t) \leq M,$$
moreover we have:
$$ \measuredangle_{t}(x,z) \leq M,  \measuredangle_{t}(z,y) \leq M.$$

This implies directly that $t\neq v$. We apply the triangle inequality two times and we get:
\begin{itemize}
    \item $2M< \measuredangle_{v}(x,z)\leq \measuredangle_{v}(x,t)+\measuredangle_{v}(t,z),$

    \item $2M< \measuredangle_{v}(y,z)\leq \measuredangle_{v}(y,t)+\measuredangle_{v}(t,y)$.
\end{itemize}

Therefore, we find again according to Proposition \ref{Proposition: angle for hyperbolic metric space} that $v \in [x,t]_{c}\cap[t,y]_{c}$, which is a contradiction.
    
\end{proof}

The following lemma is an analogue of Lemma \ref{lemma : siaprochedegeodentrexetzalorsdistancedeaàzenvironproduitdegromov} in the coned-off graph for the coned-off distance plus the sum of angles along a geodesic.

\begin{figure}[!ht]
    \centering
    \includegraphics[scale=0.45]{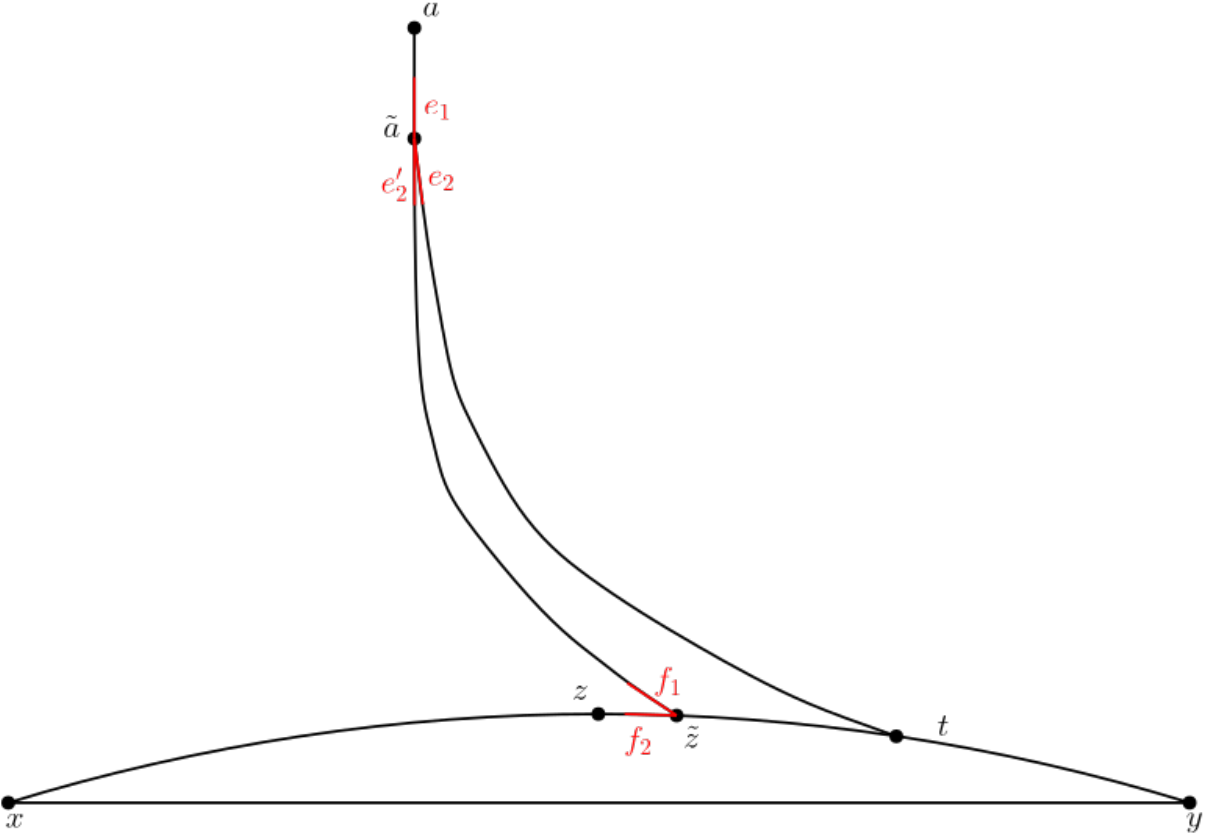}
   \caption{The projection of $a$ on $[x,z]_{c}\cup[z,y]_{c}$ lies in $[x,z]$, $t$ is the projection of $a$ on $[z,y]_{c}$.}
  \label{fig: distance dans le groupe à la projection controle par la distance de t}
\end{figure}

\newpage

\begin{lem}\label{lemme: comparer les angles avec une projection}
 Let $x,y \in X_{c}$ , we choose $[x,y]_{c}$ a geodesic between $x$ and $y$. Let $z \in Cone_{26\delta+2}([x,y]_{c})$, we choose two geodesics $[x,z]_{c}$ and $[z,y]_{c}$ between $x,z$ and $z,y$ respectively. Let $a \in X_{c}$ such that $d_{c}(a,[x,z]_{c}\cup[z,y]_{c})=d_{c}(a,[x,z]_{c})$.
 We denote by $t \in X_{c}$ a projection of $a$ on $[z,y]_{c}$. Then there exists a constant $V\geq 0$ independent of $x,y,z,a$ such that:
$$d'_{(2000\delta)^{2}}(a,t) \geq \frac{1}{\alpha_{2}^{2}} d'_{(2000\delta)^{2}}(a,z)-V-\measuredangle_{t}(a,z),$$
\end{lem}

where $\alpha_{2}^{2}$ is the constant defined in Proposition \ref{proposition: comparaison de la somme des angles le long de deux géodésiques}.

\begin{proof}
We choose $x,y,z,a,t \in X_{c}$ under the same assumptions as in the lemma. The result is clear when $z = t$, so we suppose that $z \neq t$.
We know according to Lemma \ref{lemma : siaprochedegeodentrexetzalorsdistancedeaàzenvironproduitdegromov} with the right constant that there exists $V>0$, that $d_{c}(a,t)\geq d_{c}(a,z)-V$.
We consider a triangle $[a,z,t]_{c}$ as in Theorem \ref{formenormaledestriangles}. We can choose the side $[z,t]_{c}$ so that $[z,t]_{c} \subset [z,y]_{c}$. We note that as $t$ is the projection of $a$ on $[z,y]_{c}$, $\tilde{t}=t$ (see Figure \ref{fig: distance dans le groupe à la projection controle par la distance de t}). We denote by $\Theta_{>(2000\delta)^{2}}(a,t)$, $\Theta_{>(2000\delta)^{2}}(a,z)$ the sum of angles greater than $2000\delta^{2}$ on $[a,t]_{c}$,$[a,z]_{c}$ respectively. There are two cases to consider.\\

In the first case, we assume that $ \tilde{a},\tilde{z},t$ are all different. If $a\neq \tilde{a}$, we denote by $e_{1}$ the edge on $[a,\tilde{a}]_{c}$ such that $\tilde{a} \in e_{1}$ and $e_{2},e'_{2}$ the edges of $[\tilde{a},t]_{c}$, $[\tilde{a},z]_{c}$ respectively such that $\tilde{a} \in e_{2}\cap e_{2}'$. If $z\neq \tilde{z}$, we denote by $f_{1},f_{2}$ the edges of $[a,\tilde{z}]_{c}$, $[z,\tilde{z}]_{c}$ respectively such that $\tilde{z}\in f_{1}\cap f_{2}$, see Figure \ref{fig: distance dans le groupe à la projection controle par la distance de t}. We consider $\measuredangle_{\tilde{a}}(e_{1},e_{2}), \measuredangle_{\tilde{a}}(e_{1},e'_{2}), \measuredangle_{\tilde{z}}(f_{1},f_{2})$ with the convention that the angles are equal to zero in the cases where the edges are not defined.

Then according to Theorem \ref{formenormaledestriangles}, we have:
\begin{itemize}
    \item $\Theta_{>(2000\delta)^{2}}(a,t)=\Theta_{>(2000\delta)^{2}}(a,\tilde{a})+ [\measuredangle_{\tilde{a}}(e_{1},e_{2})]_{(2000\delta)^{2}}$,

    \item $\Theta_{>(2000\delta)^{2}}(a,z)=\Theta_{>(2000\delta)^{2}}(a,\tilde{a})+ [\measuredangle_{\tilde{a}}(e_{1},e'_{2})]_{(2000\delta)^{2}} + [\measuredangle_{\tilde{z}}(f_{1},f_{2})]_{(2000\delta)^{2}}+\Theta_{>(2000\delta)^{2}}(\tilde{z},z) $.
\end{itemize}

To begin, let us remark that:
$$ |\measuredangle_{\tilde{a}}(e_{1},e_{2})- \measuredangle_{\tilde{a}}(e_{1},e'_{2})|\leq \measuredangle_{\tilde{a}}(e_{2},e'_{2})\leq 12\delta,$$
according to Proposition \ref{proposition: angle en a tilde}.

We will show now by contradiction that:
$$ [\measuredangle_{\tilde{z}}(f_{1},f_{2})]_{(2000\delta)^{2}}+\Theta_{>(2000\delta)^{2}}(\tilde{z},z) = 0. $$

Therefore, we assume that:
\begin{itemize}
\item there exists $v \in ]z,\tilde{z}[$, such that $ \measuredangle_{v}(z,a)> (2000\delta)^{2}$.
    \item or $\measuredangle_{\tilde{z}}(f_{1},f_{2}) >(2000\delta)^{2} $. 
\end{itemize}

We denote by $u$ a projection of $a$ on $[x,z]_{c}\cup [z,y]_{c}$, let us recall that $u \in [x,z]_{c}$.
Assume that we are in the first case, then there exists $v \in ]z,\tilde{z}[$ such that:
$$\measuredangle_{v}(z,y)> (2000\delta)^2.$$
Since $u$ is a projection of $z$ on $[x,z]_{c}\cup [z,y]_{c}$, we also have the following:
$$ \measuredangle_{v}(a,u) \leq 12 \delta. $$
By the triangular inequality, we get the following:
$$\measuredangle_{v}(u,z) \geq 2000\delta -12\delta.$$
Since $u \in [x,z]_{c}$, we have:
$$ \measuredangle_{v}(x,z) \geq 2000\delta -12\delta.$$

Then we get a contradiction with Lemma \ref{lem: angle vers projection cone}.\\

In the case where $\measuredangle_{\tilde{z}}(f_{1},f_{2}) >(2000\delta)^{2} $, we have $\measuredangle_{\tilde{z}}(z,a) >(2000\delta)^{2}$.
By definition of $\tilde{z}$, we have $\measuredangle_{\tilde{z}}(z,y) >(2000 \delta)^{2}-12\delta$ according to Proposition \ref{proposition: angle en a tilde}.  Since $u$ is the projection of $[x,z]_{c}\cup [z,y]_{c}$, we have:
$$ \measuredangle_{\tilde{z}}(a,u) \leq 12 \delta. $$

Therefore:
$$ \measuredangle_{\tilde{z}}(z,u) \geq 2000\delta -12\delta,$$
and since $u\in [x,z]_{c}$, we get,
$$\measuredangle_{\tilde{z}}(x,u) \geq 2000\delta -12\delta $$
then we get a contradiction according to Lemma \ref{lem: angle vers projection cone}.

We conclude by applying Proposition \ref{proposition: comparaison de la somme des angles le long de deux géodésiques}.

In the second case, the triangle $[z,a,t]_{c}$ is a tripod, since $t$ is a projection of $a$ on $[z,y]_{c}$, we have $t \in [a,z]_{c}$. Therefore:
$$\Theta_{>(2000\delta)^{2}}(a,z)=\Theta_{>(2000\delta)^{2}}(a,t)+ [\measuredangle_{t}(z,a)]_{(2000\delta)^{2}}.$$

We conclude again by an application of Proposition \ref{proposition: comparaison de la somme des angles le long de deux géodésiques}.

\end{proof}

This lemma allows us to control the angle $\measuredangle_{t}(a,z)$ from the Lemma \ref{lemme: comparer les angles avec une projection}.

\begin{lem}\label{lem: cas ou langle avec une projection est grand}

 Let $x,y \in X_{c}$ , we choose $[x,y]_{c}$ a geodesic between $x$ and $y$. Let $z \in Cone_{26\delta+2}([x,y]_{c})$, we choose two geodesics $[x,z]_{c}$ and $[z,y]_{c}$ between $x,z$ and $z,y$ respectively. Let $a \in X_{c}$ such that $d_{c}(a,[x,z]_{c}\cup[z,y]_{c})=d_{c}(a,[x,z]_{c})$.
 We denote by $t \in X_{c}$ a projection of $a$ on $[z,y]_{c}$. If $\measuredangle_{t}(z,a) \geq (2000\delta)^{2}+12\delta$ then:
$$ | \theta(d_{a}(x,y))- \theta(d_{a}(x,z))-\theta(d_{a}(z,y))|=0.$$
\end{lem}

\begin{proof}
We will use Proposition \ref{proposition : grand angle comme des checkpoints} to prove that in this case:
$$ \mu_{x}(a)=\mu_{z}(a).$$

According to \ref{proposition : grand angle comme des checkpoints}, we have $\mu_{z}(a)=\mu_{t}(a)$.

We denote by $u \in [x,z]_{c}$ a projection of $a$ on $[x,z]_{c}\cup[z,y]_{c}$. Therefore, we have:
$$ \measuredangle_{t}(a,u) \leq 12 \delta. $$

Therefore:
$$ \measuredangle_{t}(z,u) \geq (2000\delta)^{2},$$
and since $u\in [x,z]_{c}$, we get,
$$\measuredangle_{t}(x,u) \geq (2000\delta)^{2}.$$

Then another use of Proposition \ref{proposition : grand angle comme des checkpoints} gives:
$$\mu_{x}(a)=\mu_{t}(a).$$

Therefore, we finally have:
$$\begin{aligned} 
            &| \theta(d_{a}(x,y))- \theta(d_{a}(x,z))-\theta(d_{a}(z,y))|\\
            =&| \theta(d_{b}(\mu_{t}(a),\mu_{y}(a)))
            - \theta(d_{b}(\mu_{t}(a),\mu_{y}(a)))-\theta(d_{b}(\mu_{t}(a),\mu_{t}(a)))|\\
            =& 0 .
\end{aligned}$$

\end{proof}

The following lemma is an analog of Lemma \ref{lemme : z est la projection donc c'est un quasi-centre} in the coned-off graph for the coned-off distance plus the sum of angles along a geodesic.

\begin{lem}\label{lemme: d'deuxmiles projection}

There exists a constant $W\geq 0$ such for all $x,y\in X_{c}$, every geodesic $[x,y]_{c}$, for all $a \in X_{c}$, let $t \in X_{c}$ be a projection of $a$ on $[x,y]_{c}$, for all $z \in [x,y]_{c} $ we have:
$$d'_{(2000\delta)^{2}}(a,z) \geq \frac{1}{\alpha_{2}^{2}}( d'_{(2000\delta)^{2}}(a,t)+d_{c}(t,z))-W, $$

where $\alpha_{2}$ is a constant defined in Proposition \ref{proposition: comparaison de la somme des angles le long de deux géodésiques}.

\end{lem}

\begin{proof}
According to Lemma \ref{lemme : z est la projection donc c'est un quasi-centre}, we have:
$$d_{c}(a,z)\geq d_{c}(a,t)+d_{c}(t,z)-24\delta.$$

Then, according to Lemma \ref{lem: somme des angles vers la projection}, there exist two geodesics $[a,t]_{c}$, $[a,z]_{c}$ such that:
$$\Theta_{(2000\delta)^{2}}(z,a)\geq \Theta_{(2000\delta)^{2}}(t,a)-12\delta, $$
where $\Theta_{(2000\delta)^{2}}(z,a)$ and $\Theta_{(2000\delta)^{2}}(t,a)$ denotes the sum of angles greater than $(2000\delta)^{2}$ on these two geodesics.

Then, we conclude by an application of Proposition \ref{proposition: comparaison de la somme des angles le long de deux géodésiques}.

\end{proof}

The following lemma is an analogue of Lemma \ref{lemme: comparaison projection avec z presque sur la geodesique} in the framework of relatively hyperbolic groups.

\begin{lem}\label{lemme: comparaison projection avec z presque sur la geodesique cas relativement hyperbolique}
Let be $x,y \in X_{c}$, we choose $[x,y]_{c}$ a geodesic between $x$ and $y$. Let be a vertex $z \in cone_{26\delta+2}([x,y]_{c})$. We choose two geodesics between $[x,z]_{c}$ and $[z,y]_{c}$ between $x,z$ and $z,y$ respectively.

Let $a \in X_{c}$, $u \in X_{c}$ a projection of $a$ on $[x,y]_{c}$ and $v \in X_{c}$ a projection of $a$ on $[x,z]_{c}\cup[z,y]_{c}$. There exists $Z>0 $ which depends only on $\delta$, such that:
$$ d'_{(2000\delta)^{2}}(a,u)\geq \frac{1}{\alpha_{2}^{2}} d'_{(2000\delta)^{2}}(a,v)-Z-\measuredangle_{u}(a,v) .$$

\end{lem}

\begin{proof}

According to Lemma \ref{lemme: comparaison projection avec z presque sur la geodesique} with the appropriate constant, there exists $Z>0$ which depends on $\delta$ only such that:
$$ d_{c}(a,u)\geq d_{c}(a,v)-Z.$$

To conclude the proof of the lemma, we need to prove for all $w \in X_{c}$, $w\neq ,v,u$ that:
\begin{equation*}\label{equation: angle entre u et v sont petits}
  \measuredangle_{w}(u,v)\leq 300 \delta
\end{equation*}

First let us prove this fact.

Let $t$ denote a projection of $z$ on $[x,y]_{c}$.

We start by proving by contradiction that:
$$ \measuredangle_{z}(u,v) \leq 300 \delta.$$
Let assume that:
$$ \measuredangle_{z}(u,v) > 300 \delta.$$
Since $v$ is the projection of $a$ on $[x,z]_{c}\cup[z,y]_{c}$, we have that:
$$\measuredangle_{z}(a,v)\leq 12\delta.$$
Let us assume that $v \in [x,z]_{c}$ without loss of generality, then we get:
$$ \measuredangle_{z}(u,x) > 300 \delta,$$
this implies that $z \in [x,u]_{c}\subset[x,y]_{c}$.
To conclude, according to the triangle inequality, we have:
$$\measuredangle_{z}(a,u)\geq \measuredangle_{z}(u,v)-\measuredangle_{z}(a,v)>12 \delta,$$
which contradicts the fact that $u$ is a projection of $a$ on $[x,y]_{c}$.
With the same idea, we have also that:
$$\measuredangle_{t}(u,v)\leq 300\delta.$$

When $w \neq z,t$, we will bound the angles between $u$ and $t$, the angles between $v$ and $z$ and conclude thanks to the triangle inequality and the fact that the angles between $z$ and $t$ are bounded.

Let assume that there exists $w \neq v, w\neq z$ such that:
$$\measuredangle_{w}(v,z)>100\delta.$$

According to Proposition \ref{Proposition: angle for hyperbolic metric space}, this implies that $w \in [x,z]_{c}\cup[z,y]_{c}$. Moreover, since $v$ is the projection of $a$ on $[x,z]_{c}\cup [z,y]_{c}$, we have that:
$$\measuredangle_{w}(a,v)\leq 12\delta.$$

Let assume without loss of generality that $v \in [x,z]_{c}$, then we get that:
$$\measuredangle_{w}(x,z)>100\delta.$$

Thus:
$$\measuredangle_{w}(x,t)\geq \measuredangle_{w}(x,z)-\measuredangle_{w}(x,t)\geq 100\delta-(26\delta+2)>12 \delta. $$
Therefore $w \in [x,t]_{c}\subset [x,y]_{c}$.

However:
$$\measuredangle_{w}(a,u)\geq \measuredangle_{w}(u,v)-\measuredangle_{w}(a,v)>12 \delta,$$

which contradicts the fact that $u$ is a projection of $a$ on $[x,y]_{c}$.
In the same way every angle between $u$ and $t$ is bounded by $100\delta$.
Thus for all $w \neq z,t$, we have:
$$\measuredangle_{w}(u,v)\leq \measuredangle_{w}(u,t)+\measuredangle_{w}(t,z)+\measuredangle_{w}(z,v)\leq 300\delta.$$

We now show that this fact allows us to conclude.
Let us consider a triangle $[a,u,v]$ as in Theorem \ref{formenormaledestriangles}.
Let assume first that this triangle is a tripod.
If $u \in [a,v]_{c}$, we have:
$$\Theta_{>(2000\delta)^{2}}(a,v)=\Theta_{>(2000\delta)^{2}}(a,u)+ [\measuredangle_{u}(a,v)]_{(2000\delta)^{2}} +\Theta_{>(2000\delta)^{2}}(u,v) = \Theta_{>(2000\delta)^{2}}(a,u)+ [\measuredangle_{u}(a,v)]_{(2000\delta)^{2}}, $$
according to Inequality \ref{equation: angle entre u et v sont petits} and we conclude according to Proposition \ref{proposition: comparaison de la somme des angles le long de deux géodésiques}.

If $u \notin [a,v]_{c}$, we denote by $w$ the center of the tripod and we have:
\begin{itemize}
    \item $\Theta_{>(2000\delta)^{2}}(a,v)=\Theta_{>(2000\delta)^{2}}(a,w)+ [\measuredangle_{w}(a,v)]_{(2000\delta)^{2}} +\Theta_{>(2000\delta)^{2}}(w,v)$

    \item $\Theta_{>(2000\delta)^{2}}(a,u)=\Theta_{>(2000\delta)^{2}}(a,w)+ [\measuredangle_{w}(a,u)]_{(2000\delta)^{2}} +\Theta_{>(2000\delta)^{2}}(w,u) $
\end{itemize}

According to Inequality \ref{equation: angle entre u et v sont petits}, we have $\Theta_{>(2000\delta)^{2}}(w,v)=0$. To control $[\measuredangle_{w}(a,v)]_{(2000\delta)^{2}}$, we apply the triangular inequality to get:
$$|\measuredangle_{w}(a,v)-\measuredangle_{w}(a,u)|\leq \measuredangle_{w}(u,v),$$

and we use again Inequality \ref{equation: angle entre u et v sont petits}.

Let assume now that $[a,u,v]$ is not a tripod, then $\tilde{a},\tilde{u},\tilde{v}$ are all different. If $a\neq \tilde{a}$, we denote by $e_{1}$ the edge on $[a,\tilde{a}]_{c}$ such that $\tilde{a} \in e_{1}$ and $e_{2},e'_{2}$ the edges of $[\tilde{a},v]_{c}$, $[\tilde{a},u]_{c}$ respectively such that $\tilde{a} \in e_{2}\cap e_{2}'$. If $v\neq \tilde{v}$, we denote by $f_{1},f_{2}$ the edges of $[a,\tilde{v}]_{c}$, $[v,\tilde{v}]_{c}$ respectively such that $\tilde{v}\in f_{1}\cap f_{2}$. We consider $\measuredangle_{\tilde{a}}(e_{1},e_{2}), \measuredangle_{\tilde{a}}(e_{1},e'_{2}), \measuredangle_{\tilde{v}}(f_{1},f_{2})$ with the convention that the angles are equal to zero in the cases where the edges are not defined. 
We get:

\begin{itemize}
    \item 
$\Theta_{>(2000\delta)^{2}}(a,u)\geq \Theta_{>(2000\delta)^{2}}(a,\tilde{a})+ [\measuredangle_{\tilde{a}}(e_{1},e'_{2})]_{(2000\delta)^{2}}$,

    \item $\Theta_{>(2000\delta)^{2}}(a,v)=\Theta_{>(2000\delta)^{2}}(a,\tilde{a})+ [\measuredangle_{\tilde{a}}(e_{1},e_{2})]_{(2000\delta)^{2}} + [\measuredangle_{\tilde{v}}(f_{1},f_{2})]_{(2000\delta)^{2}}+\Theta_{>(2000\delta)^{2}}(\tilde{v},v) $.

\end{itemize}

According to Inequality \ref{equation: angle entre u et v sont petits}, we have:
$$ \Theta_{>(2000\delta)^{2}}(\tilde{v},v)=0 .$$

According to Proposition \ref{proposition: angle en a tilde}, $|\measuredangle_{\tilde{a}}(e_{1},e'_{2})-\measuredangle_{\tilde{a}}(e_{1},e_{2})|$ is bounded.
and finally according to Proposition \ref{proposition: angle en a tilde}, $\measuredangle_{\tilde{v}}(f_{1},f_{2})$ is close to $\measuredangle_{\tilde{v}}(v,u)$, which is bounded according to Equation \ref{equation: angle entre u et v sont petits}.

We conclude by an application of Proposition \ref{proposition: comparaison de la somme des angles le long de deux géodésiques}.

\end{proof}

This lemma allows us to control the angle $\measuredangle_{u}(v,a)$ from the Lemma \ref{lemme: comparaison projection avec z presque sur la geodesique cas relativement hyperbolique}.
\begin{lem}\label{lem: cas ou langle entre les deux projections est grand}

Let be $x,y \in X_{c}$, we choose $[x,y]_{c}$ a geodesic between $x$ and $y$. Let be a vertex $z \in cone_{26\delta+2}([x,y]_{c})$. We choose two geodesics between $[x,z]_{c}$ and $[z,y]_{c}$ between $x,z$ and $z,y$ respectively.

Let $a \in X_{c}$, $u \in X_{c}$ a projection of $a$ on $[x,y]_{c}$ and $v \in X_{c}$ a projection of $a$ on $[x,z]_{c}\cup[z,y]_{c}$. 
If $\measuredangle_{u}(a,v)\geq (2000\delta)^{2}+12\delta$ then:
$$ | \theta(d_{a}(x,y))- \theta(d_{a}(x,z))-\theta(d_{a}(z,y))|=0.$$

\end{lem}

\begin{proof}
We assume whiteout loss of generality that $v \in [x,z]_{c}$. We will use Proposition \ref{proposition : grand angle comme des checkpoints} to prove that in this case:
$$ \mu_{x}(a)=\mu_{z}(a).$$

According to the triangular inequality for angles, we have:
$$ \measuredangle_{u}(a,x)\geq \measuredangle_{u}(a,v)-\measuredangle_{u}(v,x).$$

In the case where $\measuredangle_{u}(v,z)\geq 12\delta$, this implies that $u \in [v,x]_{c}$ and then the case that $\measuredangle_{u}(a,v)\geq (2000\delta)^{2}+12\delta$ is in contradiction with the fact that $v$ is the projection of $a$ on $[x,z]_{c}\cup[z,y]_{c}$. Thus:
$$\measuredangle_{u}(a,x)\geq (2000\delta)^{2}.$$

According to Proposition \ref{proposition : grand angle comme des checkpoints}, we get that $\mu_{x}(a)=\mu_{u}(a)$.

The same idea applies to prove that $\mu_{z}(a)=\mu_{u}(a)$.

Therefore, we finally have:
$$\begin{aligned} 
            &| \theta(d_{a}(x,y))- \theta(d_{a}(x,z))-\theta(d_{a}(z,y))|\\
            =&| \theta(d_{b}(\mu_{u}(a),\mu_{y}(a)))
            - \theta(d_{b}(\mu_{u}(a),\mu_{y}(a)))-\theta(d_{b}(\mu_{u}(a),\mu_{u}(a)))|\\
            =& 0 .
\end{aligned}$$

\end{proof}

\begin{proposition}\label{proposition: nouveau p}

There exists a real number $P\geq p_{0} >1$, where $p_{0}$ is the real number of Proposition \ref{proposition : somme lp des différences de mask converge} and a constant $\sigma>0$ such that for all $x \in G$:

$$ \sum_{ a\in X_{c} } \kappa^{\frac{P}{\alpha_{2}^{4}}d'_{(2000\delta)^{2}}(a,x)} \leq \sigma. $$

where $\kappa$ is the constant of Proposition \ref{proposition : la difference de deux masks est controlé par l'angle}, $ \alpha_{2}$ the constant of Proposition \ref{proposition: comparaison de la somme des angles le long de deux géodésiques}.
    
\end{proposition}

\begin{proof}
We fix $x \in X_{c}$.

According to Lemma \ref{lemme: lemme de comptage avec d'deuxmile}, there exist  $A\geq 1 $ and $\gamma_{G}>1$ independent of $x$ such that for all $n \in \mathbb{N}$ the set, $S_{d'_{(2000\delta)^{2}}}(n)=\{a \in X_{c}, d'_{(2000\delta)^{2}}(a,x)=n\}$ has cardinality less than $A\gamma_{G}^{n}$ according to Lemma \ref{lemme: lemme de comptage avec d'deuxmile}. Thus, for all $n\in \mathbb{N}$, for all $P \in \mathbb{R}$:
$$ \displaystyle \sum_{a \in S_{d'_{(2000\delta)^{2}}}(n) } \kappa^{\frac{P}{\alpha_{2}^{4}}d'_{(2000\delta)^{2}}(a,x)} \leq  \kappa^{\frac{P}{\alpha_{2}^{4}}n} A \gamma_{G}^{n}.$$
 We choose $P$ such that $\kappa^{\frac{P}{\alpha_{2}^{4}}} \leq \frac{1}{\gamma_{G}}$, therefore the sum converges and does not depend on $x$.
\end{proof}

We define here a function from $\mathbb{R}_{+}$ to $\mathbb{R}_{+}$ which is equal to the identity up to a multiplicative constant for large values and equal to $t\mapsto t^{P}$  for small values. The strongly bolic metric will be defined as the sum of the $d_{a}$, for all $a \in X_{c}$, to which we apply this function. The idea is that when the contribution of a parabolic is small, the metric will be similar to the one in the hyperbolic case. When the contribution of a parabolic is large, we will rely on the strong bolicity of $d_{a}$.

\begin{definition}\label{definition: la fonction theta} 

Let $t_{0}$ be a constant with $t_{0}\geq A ((2000\delta)^{2}+18\delta)+AB+2$, where $A,B$ the constants of Proposition \ref{proposition: l'angle est borné par la distance bolic}.

We define $ \theta: \mathbb{R}_{+}\rightarrow\mathbb{R}_{+} $ as:
$$
\theta(t) = \left\{
    \begin{array}{ll}
        t^{P} & \mbox{if} ~ t \leq t_{0} \\
        t_{0}^{P-1}t & \mbox{if} ~ t\geq t_{0}.
    \end{array}
\right.
$$
 where $P$ is the real number of Proposition \ref{proposition: nouveau p}.
\end{definition}

Let us check that $\theta$ is a lipschitz function.

\begin{lem}\label{lem: theta est lipschtizienne}

$\theta$ is $Pt_{0}^{P-1}$-lipschitz.

\end{lem}

\begin{proof}
$\theta$ is differentiable everywhere except at $t_{0}$ and $|\theta'|\leq Pt_{0}^{P-1}$.

$\theta$ is clearly $Pt_{0}^{P-1}$-lipschitz on $[0,t_{0}]$ and on $[t_{0},\infty[$.

Let $s<t_{0}<t$.
$$\begin{aligned} \theta(t)-\theta(s)& = \theta(t)-\theta(t_{0})+\theta(t_{0})-\theta(s)\\
& \leq Pt_{0}^{P-1} (t-s).
\end{aligned}$$

\end{proof}

We notice here that for all $a \in X_{c}$, $x,y \in G$ when the value $d_{a}(x,y)$ is large then
$a$ has infinite valence.

\begin{lem}\label{lemme: theta est lidentite alors a est dans toute geodesique entre x et y}

Let $x,y \in G$ and $a \in X_{c}$ such that $d_{a}(x,y) \geq t_{0}$, then:
\begin{itemize}
    \item $a \in X_{c}^{\infty}$,

    \item $\measuredangle_{a}(x,y) \geq (2000\delta)^{2}+18\delta $ and in particular every geodesic between $x$ and $y$ goes through $a$.
\end{itemize}
\end{lem}

\begin{proof}
Let $x,y \in G$ and $a \in X_{c}$, $d_{a}(x,y) \geq t_{0} \geq A ((2000\delta)^{2}+18\delta)+AB$ with $A,B$ the constant of Proposition \ref{proposition: l'angle est borné par la distance bolic}. By definition, for all $a \in G$, for all $x,y \in G$, $d_{a}(x,y)\leq2$. As $t_{0}\geq 2$, $a \in X_{c}^{\infty}$. 

Moreover, according to Proposition \ref{proposition: l'angle est borné par la distance bolic}, we have $\measuredangle_{a}(x,y) \geq (2000\delta)^{2}+18\delta $.

Finally, this fact implies that $a$ belongs to every geodesic between $x$ and $y$ according to Proposition \ref{proposition: les propirétés des angles}.

\end{proof}

\begin{lem}\label{lemme: la somme lp des da est finie}

For all $x,y \in G$, the sum

$$\sum_{a \in X_{c}} d_{a}(x,y)^{P},$$

is finite.
\end{lem}

\begin{proof}

According to Proposition \ref{proposition : somme lp des différences de mask converge}, for all $x,y \in X_{c}$, there is a finite number of $a \in X_{c}$ such that :
$$ || \mu_{x}(a) \Delta \mu_{y}(a) ||=2 .$$

Otherwise the sum $\displaystyle \sum_{a \in X_{c} } || \mu_{x}(a) \Delta \mu_{y}(a) ||^{p_{0}}$ would diverge, with $p_{0}$ the integer of Proposition \ref{proposition : somme lp des différences de mask converge}.

Then the set $ \mathcal{E}_{x,y}:= \{ a \in X_{c} ~ | ~|| \mu_{x}(a) \Delta \mu_{y}(a) ||=2 \} $ is finite.\\

Moreover, for all $a \in X_{c} \backslash \mathcal{E}_{x,y} $, we have $\text{supp}(\mu_{x}(a)) \cap \text{supp}(\mu_{y}(a)) \neq \emptyset$, then according to Definition \ref{definition: lesbons parabolic}:
$$ d_{a}(x,y)\leq C || \mu_{x}(a) \Delta \mu_{y}(a) ||.  $$

Therefore :
$$\begin{aligned} 
            \displaystyle \sum_{a \in X_{c}} d_{a}(x,y)^{P} &=\sum_{a \in \mathcal{E}_{x,y}} d_{a}(x,y)^{P} + \sum_{a \in X_{c} \backslash \mathcal{E}_{x,y}} d_{a}(x,y)^{P}\\
            & \leq \sum_{a \in \mathcal{E}_{x,y}} d_{a}(x,y)^{P} + C^{P} \sum_{a \in X_{c} \backslash \mathcal{E}_{x,y}} || \mu_{x}(a) \Delta \mu_{y}(a) ||^{P} ~(\text{according to the previous remark}) \\
           & \leq \displaystyle \sum_{a \in \mathcal{E}_{x,y}} d_{a}(x,y)^{P} + C^{P} \sum_{a \in X_{c}} || \mu_{x}(a) \Delta \mu_{y}(a) ||^{P}\\
            & < \infty ~(\text{according to Proposition \ref{proposition : somme lp des différences de mask converge} and the fact that } P\geq p_{0})  .
\end{aligned}$$

This proves the lemma.
    
\end{proof}

Just as in the case of hyperbolic groups, we define here a function which will be the metric up to a constant.

\begin{proposition} \label{proposition: D est fini et quasiisom à metric du groupe groupe rel hyp}
For all $x,y \in G$,

$$D(x,y):=\sum_{a \in X_{c}} \theta(d_{a}(x,y))$$ is finite.

\end{proposition}

\begin{proof}
Let $x,y \in G$, let $[x,y]_{c}$ be a geodesic between $x$ and $y$. Let $a \in X_{c}$ such that $d_{a}(x,y)\geq t_{0}$, then according to Lemma \ref{lemme: theta est lidentite alors a est dans toute geodesique entre x et y}, we have $a \in [x,y]_{c}$.

Therefore we have the following:
$$\begin{aligned} 
            \displaystyle  D(x,y) &=\sum_{a \in [x,y]_{c}} \theta(d_{a}(x,y)) +  \sum_{a \in X_{c}\backslash [x,y]_{c} } \theta(d_{a}(x,y))\\
            & = \sum_{a \in [x,y]_{c}} \theta(d_{a}(x,y)) + \sum_{a \in X_{c}\backslash [x,y]_{c} }  d_{a}(x,y)^{P} ~(\text{according to the previous remark}) \\
           & \leq \displaystyle \sum_{a \in [x,y]_{c}} \theta(d_{a}(x,y)) + \sum_{a \in X_{c}}d_{a}(x,y)^{P}\\
            & < \infty ~(\text{according to Lemma \ref{lemme: la somme lp des da est finie} and the fact that } [x,y]_{c} \text{ is finite.)}  
\end{aligned}$$
\end{proof}

The following property is crucial, as it will allow us to show that $D$ satisfies a coarse triangular inequality. It will also be used to prove that the strongly bolic metric is weakly geodesic. The outline of the proof is very similar to the proof of Proposition \ref{proposition : D est faiblement géodésique} in the case of hyperbolic groups, with some additional technical difficulties. The next proposition applies in particular for two vertices $x,y\in G$ and every $z \in G\cap I_{c}(x,y)$.\\

\begin{figure}[!ht]
   \centering
   \includegraphics[scale=0.5]{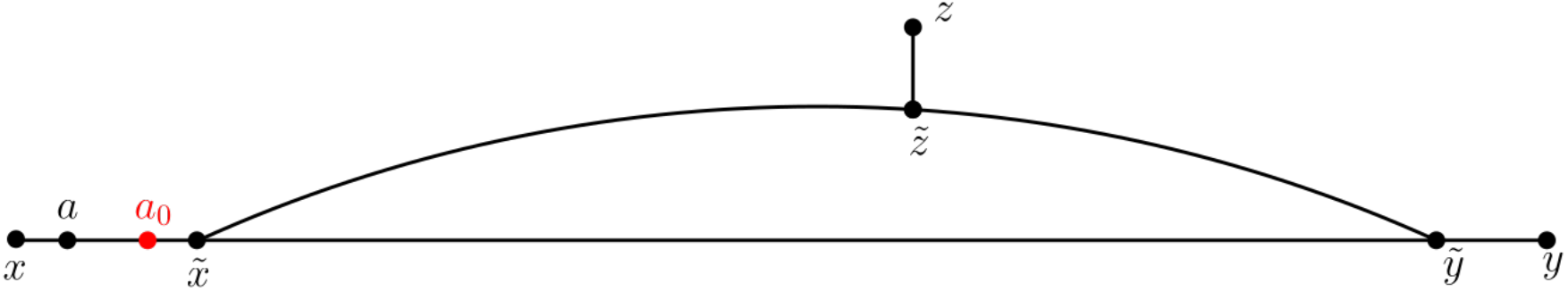}
   \caption{If $a_{0} \in \mathcal{B}$, then $a_{0}$ is a checkpoint and $ \mu_{z}(a)=\mu_{y}(a)$.}
   \label{fig: figure 1 egalite triangulaire relativement hyperbolique}
\end{figure}

\begin{figure}[!ht]
   \centering
   \includegraphics[scale=0.5]{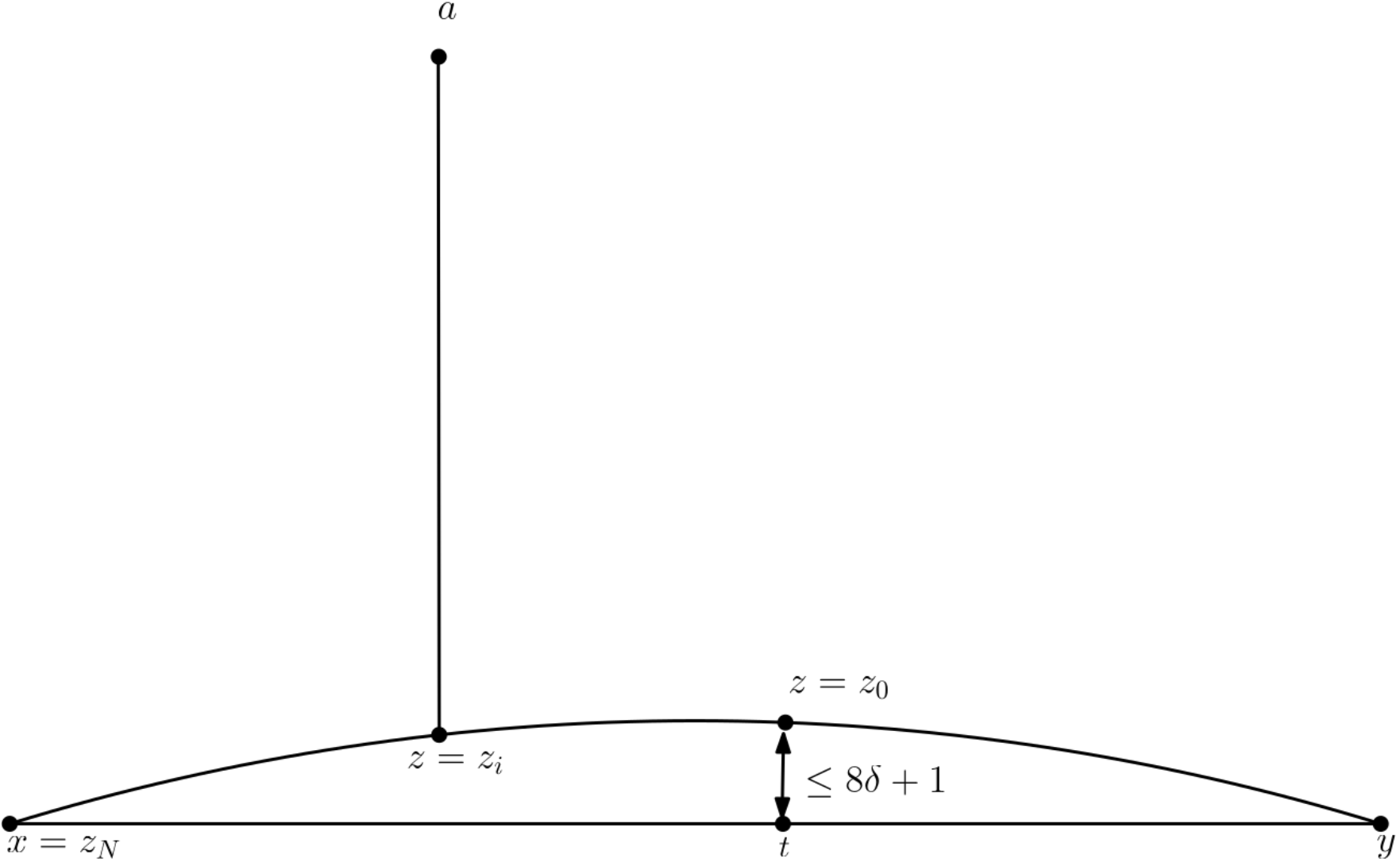}
  \caption{$a$ projects onto $z_{i}$ on $[x,z]_{c}\cup[z,y]_{c}$}
   \label{fig: figure 2 proposition 9.44}
\end{figure}

\newpage

\begin{proposition}\label{proposition: inegalite clutch inegalite triangulaire et faiblement geodesique}
There exists $\nu>0$, such that for all $x,y \in G$ and all geodesic $[x,y]_{c}$ between $x $ and $y$ and for all $z \in  G$ such that $z \in Cone_{26\delta+2}([x,y]_{c})$, we have:
$$|D(x,y)-D(x,z)-D(z,y)|\leq \nu.$$
\end{proposition}

\begin{proof}

Let $x,y \in G$ and $z \in G$ under the assumption of the proposition. We will show that there exists $\nu>0$ such that:
$$|D(x,y)-D(x,z)-D(z,y)|\leq \sum_{a \in X_{c}} | \theta(d_{a}(x,y))- \theta(d_{a}(x,z))-\theta(d_{a}(z,y))|\leq \nu .$$

Let us remark first that as $z \in Cone_{26\delta+2}([x,y]_{c})$, according to Lemma \ref{lem: angle vers projection cone}, for all $a \in X_{c}$, we have:
\begin{equation*}\label{equation: pas grand des deux cotés de z}
d_{a}(x,z) < t_{0} \text{ or } d_{a}(z,y) < t_{0}, 
\end{equation*}

with $t_{0}$ the constant of Definition \ref{definition: la fonction theta}.\\

Therefore for all $a \in X_{c}$, assuming without loss of generality that $d_{a}(z,y) < t_{0}$, we have:
$$\begin{aligned}
| \theta(d_{a}(x,y))- \theta(d_{a}(x,z))-\theta(d_{a}(z,y))|& \leq Pt_{0}^{P-1} |d_{a}(x,y)-d_{a}(x,z)|+\theta(d_{a}(z,y))~(\text{according to Lemma \ref{lem: theta est lipschtizienne}})\\
            & \leq Pt_{0}^{P-1} d_{a}(z,y)+d_{a}(z,y)^{P}\\
            & \leq (P+1) t_{0}^{P},\\
\end{aligned}$$

with $P$ the real number of Proposition \ref{proposition: nouveau p}.\\

Let us denote by $\mathcal{B}$ the set of $a \in X_{c}$ such that:

\begin{itemize}
    \item $d_{a}(x,y) \geq t_{0}$,

    \item or $d_{a}(x,z) \geq t_{0}$,

    \item or $d_{a}(z,y)\geq t_{0}$.

\end{itemize}

We will start to show that:
$$\sum_{a \in \mathcal{B}} | \theta(d_{a}(x,y))- \theta(d_{a}(x,z))-\theta(d_{a}(z,y))| \leq 4t_{0}^{P}(P+1).$$\\
We consider a triangle $[x,y,z]_{c}$ as in Theorem \ref{formenormaledestriangles}.

According to Equation \ref{equation: pas grand des deux cotés de z}, we have that:
\begin{itemize}
    \item $\mathcal{B}\cap [z,\tilde{z}[ =\emptyset$,

    \item and $\tilde{z} \in \mathcal{B}$ implies that $\tilde{z}=\tilde{x}=\tilde{y}$.

\end{itemize}

Therefore, we have:
$$ \mathcal{B}\subset ([x,\tilde{x}]_{c}\cup[\tilde{y},y]_{c}) \cap X_{c}^{\infty}.$$

If $\mathcal{B}\cap[x,\tilde{x}[_{c}\neq \emptyset$, we define $a_{0}$ as the point in $\mathcal{B}\cap [x,\tilde{x}[_{c}$ closest to $\tilde{x}$, see Figure \ref{fig: figure 1 egalite triangulaire relativement hyperbolique}. We will show that, for all $a \in \mathcal{B}\cap[x,\tilde{x}[_{c}, a \neq a_{0}$, we have:
$$\mu_{z}(a)=\mu_{a_{0}}(a) \text{ and } \mu_{y}(a)=\mu_{a_{0}}(a)  ,$$
and this will imply that:
$$\begin{aligned} 
            &| \theta(d_{a}(x,y))- \theta(d_{a}(x,z))-\theta(d_{a}(z,y))|\\
            =&| \theta(d_{b}(\mu_{x}(a),\mu_{a_{0}}(a)))
            - \theta(d_{b}(\mu_{x}(a),\mu_{a_{0}}(a)))-\theta(d_{b}(\mu_{a_{0}}(a),\mu_{a_{0}}(a)))|\\
            =& 0 .
\end{aligned}$$

To do so, we will show that, for such $a$, we have $\measuredangle_{a_{0}}(a,z)\geq (2000\delta)^{2}$, and $\measuredangle_{a_{0}}(a,y)\geq (2000\delta)^{2}$ and we will apply Proposition \ref{proposition : grand angle comme des checkpoints}. Since $a_{0} \in [x,\tilde{x}[_{c}$, we know that $a_{0} \notin [z,y]_{c}$, then according to Lemma \ref{lemme: theta est lidentite alors a est dans toute geodesique entre x et y}, we know that $d_{a_{0}}(z,y)< t_{0}$. Then $a_{0}\in \mathcal{B}$ implies that $d_{a_{0}}(x,y)\geq t_{0}$ or $d_{a_{0}}(z,y)\geq t_{0}$.\\

In the case where $d_{a_{0}}(x,y) \geq t_{0}$, we know according to Lemma \ref{lemme: theta est lidentite alors a est dans toute geodesique entre x et y}, that $\measuredangle_{a_{0}}(x,y)\geq (2000\delta)^{2}+6\delta$. As $a \in I_{c}(x,a_{0})$, $\measuredangle_{a_{0}}(a,y)\geq (2000\delta)^{2}+12\delta>(2000\delta)^{2}$ according to the second point of Proposition \ref{proposition: les propirétés des angles}.
Thus, we deduce that:
$$\measuredangle_{a_{0}}(a,z)\geq \measuredangle_{a_{0}}(a,y)-\measuredangle_{a_{0}}(y,z)\geq (2000\delta)^{2} ,$$
since $a_{0} \notin [y,z]_{c}$. We proceed in the same way for the case $d_{a_{0}}(x,z) \geq t_{0}$.

Therefore, when $\mathcal{B}\cap[x,\tilde{x}[_{c} \neq \emptyset$:
$$\begin{aligned} 
            &\displaystyle \sum_{a \in \mathcal{B}\cap[x,\tilde{x}]_{c}} | \theta(d_{a}(x,y))- \theta(d_{a}(x,z))-\theta(d_{a}(z,y))|\\ &\leq | \theta(d_{a_{0}}(x,y))- \theta(d_{a_{0}}(x,z))-\theta(d_{a_{0}}(z,y))| + | \theta(d_{\tilde{x}}(x,y))- \theta(d_{\tilde{x}}(x,z))-\theta(d_{\tilde{x}}(z,y))|\\
            & \leq 2 (P+1)t_{0}^{P}.\\
\end{aligned}$$

The inequality is trivially true when $\mathcal{B}\cap[x,\tilde{x}]_{c} = \emptyset$, with the same computation for $\mathcal{B}\cap[\tilde{y},y]_{c}$, we deduce that:
$$\begin{aligned} 
            &\displaystyle \sum_{a \in \mathcal{B}} | \theta(d_{a}(x,y))- \theta(d_{a}(x,z))-\theta(d_{a}(z,y))| \\ &= \sum_{a \in \mathcal{B}\cap[x,\tilde{x}]_{c}} | \theta(d_{a}(x,y))- \theta(d_{a}(x,z))-\theta(d_{a}(z,y))|\\
            & + \sum_{a \in \mathcal{B}\cap[\tilde{y},y]_{c}} | \theta(d_{a}(x,y))- \theta(d_{a}(x,z))-\theta(d_{a}(z,y))|\\
            & \leq 4t_{0}^{P}(P+1).\\
\end{aligned}$$

Now, we will show that there exists a constant $\nu_{1}$ such that:
$$\sum_{a \in X_{c}\backslash \mathcal{B}} | d_{a}(x,y)^{P}- d_{a}(x,z)^{P}-d_{a}(z,y)^{P}| \leq \nu_{1}.$$

We have :
$$\begin{aligned} 
           & \sum_{a \in X_{c} \backslash \mathcal{B}} | d_{a}(x,y)^{P} -d_{a}(x,z)^{P} - d_{a}(z,y)^{P}| \\
            \leq & \displaystyle \sum_{ \{a\in X_{c} \backslash \mathcal{B} | d_{c}(a,[x,z]_{c}\cup[z,y]_{c})=d_{c}(a,[x,z]_{c}) \}} |d_{a}(x,y)^{P} -d_{a}(x,z)^{P} - d_{a}(z,y)^{P}
             | \\
            & + \displaystyle \sum_{ \{a\in X_{c} \backslash \mathcal{B} | d_{c}(a,[x,z]_{c}\cup[z,y]_{c})=d_{c}(a ,[z,y]_{c}) \}} | d_{a}(x,y)^{P} -d_{a}(x,z)^{P} - d_{a}(z,y)^{P}
             |.
\end{aligned}$$

We will bound from above the first sum, the same idea will apply to the second sum.
For all $a \in X_{c}$ such that $d_{c}(a,[x,z]_{c}\cup[z,y]_{c})=d_{c}(a,[x,z]_{c})$, we denote by $t$ a projection of $a$ $[z,y]_{c}$. We can suppose that $\measuredangle_{t}(a,z) < (2000\delta)^{2}+12\delta$, otherwise, according to Lemma \ref{lem: cas ou langle avec une projection est grand}, we have:
$$| d_{a}(x,y)^{P} -d_{a}(x,z)^{P} - d_{a}(z,y)^{P}|=0.$$

Then according to Lemma \ref{lemme: comparer les angles avec une projection}, there exists $V>1$, independant of $a$, (up to a change of $V$), such that:
\begin{equation}\label{equation: distance à projection sur z,y en fonction de distance à z}
d'_{(2000\delta)^{2}}(a,t) \geq \frac{1}{\alpha_{2}^{2}} d'_{(2000\delta)^{2}}(a,z)-V.
\end{equation}

In the same way, with $u$ a projection of $a$ on $[x,y]_{c}$ and $v$ a projection on $[x,z]_{c} \cup[z,y]_{c}$, we can always suppose according to Lemma \ref{lem: cas ou langle entre les deux projections est grand}, that $\measuredangle_{u}(a,v)\leq (2000\delta){2}+12\delta$.
According to the Lemma \ref{lemme: comparer les angles avec une projection}, there exists a constant $Z>1$ (up to a change of $Z$), such that:
\begin{equation}\label{equation: distance à projection sur x,y en fonction de distance à proj sur xz et zy}
 d'_{(2000\delta)^{2}}(a,u)\geq \frac{1}{\alpha_{2}^{2}} d'_{(2000\delta)^{2}}(a,v)-Z.
\end{equation}

We obtain the following:
$$\begin{aligned} 
           &~~\sum_{ \{a\in X_{c} \backslash \mathcal{B} | d_{c}(a,[x,z]_{c}\cup[z,y]_{c})=d_{c}(a,[x,z]_{c}) \}} | d_{a}(x,y)^{P} -d_{a}(x,z)^{P} - d_{a}(z,y)^{P}|
           &\\
             \leq & \sum_{ \{a\in X_{c} \backslash \mathcal{B} | d_{c}(a,[x,z]_{c}\cup[z,y]_{c})=d_{c}(a,[x,z]_{c}) \}} | d_{a}(x,y)^{P} -d_{a}(x,z)^{P}| + d_{a}(z,y)^{P} \\
            \leq & \sum_{ \{a\in X_{c} \backslash \mathcal{B} | d_{c}(a,[x,z]_{c}\cup[z,y]_{c})=d_{c}(a,[x,z]_{c}) \}} | d_{a}(x,y)^{P} -d_{a}(x,z)^{P}| \\
            & + \displaystyle \sum_{ \{a\in X_{c} \backslash \mathcal{B} | d_{c}(a,[x,z]_{c}\cup[z,y]_{c})=d_{c}(a,[x,z]_{c}) \}} d_{a}(z,y)^{P}.
\end{aligned}$$

 To start with, we bound from above the sum:
$$ \displaystyle \sum_{ \{a\in X_{c} \backslash \mathcal{B} | d_{c}(a,[x,z]_{c}\cup[z,y]_{c})=d_{c}(a,[x,z]_{c}) \}} d_{a}(z,y)^{P}.$$

Let $a \in X_{c}$ such that $d_{c}(a,[x,z]_{c}\cup[z,y]_{c})=d_{c}(a,[x,z]_{c})$. Let us denote $t \in  X_{c}$ a projection of $a$ on $[z,y]_{c}$. We have as noted in Inequality \ref{equation: distance à projection sur z,y en fonction de distance à z}:
$$d'_{(2000\delta)^{2}}(a,u) \geq \frac{1}{\alpha_{2}^{2}} d'_{(2000\delta)^{2}}(a,z)-V.$$

Then there exists a constant $C_{2}>0$ such that:
$$d'_{(2000\delta)^{2}}(a,z) \geq C_{2} \implies d'_{(2000\delta)^{2}}(a,u) \geq C_{1}.$$

Moreover, according to Corollary \ref{corollary : reformulation de la formule de la distance}, as $z \in G$, the set $\{a \in X_{c}~|~ d'_{(2000\delta)^{2}}(a,z) \leq C_{2} \}$ has cardinality less than $A \gamma_{G}^{C_{2}}$.

Thus, according to Theorem \ref{theo: theo stylé sur les barycentres et les masques} and Lemma \ref{lemme: comparer les angles avec une projection}, for all $a \in X_{c}$ with $d_{c}(a,[x,z]_{c}\cup[z,y]_{c})=d_{c}(a,[x,z]_{c})$ and $d'_{(2000\delta)^{2}}(a,z) \geq C_{2}$, we have:
\begin{equation}\label{equation: contre de da(z,y)}
d_{a}(z,y)^{P} \leq \kappa^{\frac{P}{\alpha_{2}^{2}}d'_{(2000\delta)^{2}}(a,z)-PV}
\end{equation}

Therefore, according to Lemma \ref{lemme: theta est lidentite alors a est dans toute geodesique entre x et y}, for all $a \in X_{c}$ such that $d_{c}(a,[x,z]_{c}\cup[z,y]_{c})=d_{c}(a,[x,z]_{c})$, we get the following: 
\begin{align}\label{equation: bornage de dazy}
           &~~\displaystyle \sum_{ \{a\in X_{c} \backslash \mathcal{B} | d_{c}(a,[x,z]_{c}\cup[z,y]_{c})=d_{c}(a,[x,z]_{c}) \}} d_{a}(z,y)^{P}
            \notag \\
             \leq & \displaystyle \sum_{ \{a\in X_{c} | d_{c}(a,[x,z]_{c}\cup[z,y]_{c})=d_{c}(a,[x,z]_{c}) \}} d_{a}(z,y)^{P} \notag \\
            \leq & A \gamma_{G}^{C_{2}} t_{0}^{P} + \displaystyle \sum_{ a\in X_{c}}\kappa^{\frac{P}{\alpha_{2}^{2}}d'_{(2000\delta)^{2}}(a,z)-PV} \notag \\
            \leq & A \gamma_{G}^{C_{2}} t_{0}^{P} + \kappa^{-PV} \displaystyle \sum_{ a\in X_{c}}\kappa^{\frac{P}{\alpha_{2}^{2}}d'_{(2000\delta)^{2}}(a,z)} \notag \\
            \leq   & A \gamma_{G}^{C_{2}} t_{0}^{P} + \kappa^{-PV} \sigma ~(\text{according to Proposition \ref{proposition: nouveau p}}) .
\end{align}

Now, we will bound from above the sum:
$$\sum_{ \{a\in X_{c} \backslash \mathcal{B} | d_{c}(a,[x,z]_{c}\cup[z,y]_{c})=d_{c}(a,[x,z]_{c}) \}} | d_{a}(x,y)^{P} -d_{a}(x,z)^{P}|.$$

Let us denote $N:=d_{c}(x,z)$. We order $[z,x]_{c}$ in the following sense: $[z,x]_{c}=\{z_{0},z_{1},...,z_{N}\}$ with $z_{0}=z$, $z_{N}=x$ and $d_{c}(z_{i},z_{i+1})=1$ for $0\le i \le N-1$, see Figure \ref{fig: figure 2 proposition 9.44}.

Let us remark that, for all $a \in X_{c}$ with $d_{c}(a,[x,z]_{c}\cup[z,y]_{c})=d_{c}(a,z_{i})$, we have:
$$\begin{aligned} 
           d'_{(2000\delta)^{2}}(a,z) & \geq  d_{c}(a,z) \\      
            & \geq d_{c}(a,z_{i})+d_{c}(z_{i},z)-24\delta \\
             & \geq d_{c}(a,z_{i})+i-24\delta\\
             & \geq i-24\delta.
            \end{aligned}$$

and also according to Lemma \ref{lemme: d'deuxmiles projection}, on the geodesic $[x,z]_{c}$, we have:

$$  d'_{(2000\delta)^{2}}(a,z) \geq \frac{1}{\alpha_{2}^{2}}(i+ d'_{(2000\delta)^{2}}(a,z_{i}))-W .$$

Moreover, for all $i \in \{0, \dots, N \}$, for all $a \in X_{c}$, there exists a constant $C_{3}\geq C_{1}$ such that for all $a \in X_{c}$, $d_{c}(a,z_{i})=d_{c}(a,[x,z]_{c}\cup[z,y]_{c})$ and $d'_{(2000\delta)^{2}}(a,z_{i})\geq C_{3}$, we have according to Theorem \ref{theo: theo stylé sur les barycentres et les masques} and to Inequality \ref{equation: distance à projection sur x,y en fonction de distance à proj sur xz et zy}:
\begin{equation}\label{equation: controle de daxy et daxz}
 d_{a}(x,y)\leq \kappa^{\frac{1}{\alpha_{2}^{2}} d'_{(2000\delta)^{2}}(a,z_{i})-Z} \text{ and }d_{a}(x,z)\leq \kappa^{d'_{(2000\delta)^{2}}(a,z_{i})}.
\end{equation}

Therefore:
\begin{align*}
 & \sum_{ \{a\in X_{c} \backslash \mathcal{B} | d_{c}(a,[x,z]_{c}\cup[z,y]_{c})=d_{c}(a,[x,z]_{c}) \}} | d_{a}(x,y)^{P} -d_{a}(x,z)^{P}| \notag \\
 & \leq \sum_{i=0}^{N} \sum_{ \{a\in X_{c} \backslash \mathcal{B} | d_{c}(a,[x,z]_{c}\cup[z,y]_{c})=d_{c}(a,z_{i}) \}} | d_{a}(x,y)^{P} -d_{a}(x,z)^{P}| \notag \\
& \leq P \sum_{i=0}^{N} \sum_{ \{a\in X_{c} \backslash \mathcal{B} | d_{c}(a,[x,z]_{c}\cup[z,y]_{c})=d_{c}(a,z_{i}) \}} \max(d_{a}(x,y)^{P-1},d_{a}(x,z)^{P-1})d_{a}(z,y) ~(\text{according to the mean value theorem}) \notag \\
& \leq PA \gamma_{G}^{C_{3}} t_{0}^{P} \\
&+P  \sum_{i=0}^{N} \sum_{ \{a\in X_{c} \backslash \mathcal{B} | d_{c}(a,[x,z]_{c}\cup[z,y]_{c})=d_{c}(a,z_{i}) \}} \max(d_{a}(x,y)^{P-1},d_{a}(x,z)^{P-1})\kappa^{\frac{1}{\alpha_{2}^{2}}d'_{(2000\delta)^{2}}(a,z)-V} ~(\text{according to Inequality \ref{equation: bornage de dazy}}) \notag \\
& \leq PA \gamma_{G}^{C_{3}} t_{0}^{P}+ P \kappa^{-V} t_{0}^{P-1}A\gamma_{G}^{C_{1}} \sum_{i=0}^{N} \kappa^{\frac{i}{\alpha_{2}^{2}}} +P \kappa^{\frac{-W}{V}-(P-1)Z-V} \sum_{i=0}^{N} \kappa^{\frac{i}{\alpha_{2}^{4}}} \sum_{a \in X_{c}} \kappa^{(\frac{P-1}{\alpha_{2}^{2}}+\frac{1}{\alpha_{2}^{4}})d'_{(2000\delta)^{2}}(a,z_{i})} ~(\text{according to Inequalities \ref{equation: controle de daxy et daxz}})\notag \\
& \leq  PA \gamma_{G}^{C_{3}} t_{0}^{P}+ P \kappa^{-V} t_{0}^{P-1}A\gamma_{G}^{C_{1}} \frac{1-{\kappa^{ \frac{N+1}{\alpha_{2}^{2}}}}}{1-\kappa^{\frac{1}{\alpha_{2}^{2}}}}+ P \kappa^{\frac{-W}{V}-(P-1)Z-V} \sum_{i=0}^{N} \kappa^{\frac{i}{\alpha_{2}^{4}}} \sum_{a \in X_{c}} \kappa^{\frac{P}{\alpha_{2}^{4}}d'_{(2000\delta)^{2}}(a,z_{i})} ~(\text{according to Proposition \ref{proposition: nouveau p}}) \notag  \\
& \leq  PA \gamma_{G}^{C_{2}} t_{0}^{P}+ P\kappa^{-V} t_{0}^{P-1}A\gamma_{G}^{C_{3}} \frac{1}{1-\kappa^{\frac{1}{\alpha_{2}^{2}}}}+P \kappa^{\frac{-W}{\alpha_{2}^{2}}-(P-1)Z-V} \sigma \sum_{i=0}^{N} \kappa^{\frac{i}{\alpha_{2}^{4}}} \notag  \\
& \leq  PA \gamma_{G}^{C_{3}} t_{0}^{P}+ P \kappa^{-V} t_{0}^{P-1}A\gamma_{G}^{C_{1}} \frac{1}{1-\kappa^{\frac{1}{\alpha_{2}^{2}}}}+P \kappa^{\frac{-W}{\alpha_{2}^{2}}-(P-1)Z-V} \sigma \frac{1-{\kappa^{ \frac{N+1}{\alpha_{2}^{4}}}}}{1-\kappa^{\frac{1}{\alpha_{2}^{4}}}} \notag \\
& \leq  PA \gamma_{G}^{C_{3}} t_{0}^{P}+ P \kappa^{-V} t_{0}^{P-1}A\gamma_{G}^{C_{1}} \frac{1}{1-\kappa^{\frac{1}{\alpha_{2}^{2}}}}+P \kappa^{\frac{-W}{\alpha_{2}^{2}}-(P-1)Z-V} \sigma \frac{1}{1-\kappa^{\frac{1}{\alpha_{2}^{4}}}}.
\end{align*}

Therefore, combined with Inequality \ref{equation: bornage de dazy}, the sum:

$$\sum_{a \in X_{c} \backslash \mathcal{B}} | d_{a}(x,y)^{P} -d_{a}(x,z)^{P} - d_{a}(z,y)^{P}|,$$

is bounded by a constant independent of $x,y$ and $z$. This concludes the proof of the proposition.
\end{proof}

The next proposition shows that for two neighboring points of the same parabolic $a$, $D$ is comparable to $d_{a}$.

\begin{proposition}\label{proposition: pour deux sommets voisins de a un pt parabolique D est presque da}

There exists a constant $\Lambda >0$ such that for all $x,y \in G$ and $a_{0} \in X_{c}^{\infty}$, with $d_{c}(x,a_{0})=d_{c}(a_{0},y)=1$, we have:

$$|D(x,y)-t_{0}^{P-1}d_{a_{0}}(x,y)|\leq \Lambda.$$
\end{proposition}

\begin{proof}

To begin, we remark that:
$$|\theta(d_{a_{0}}(x,y))-t_{0}^{P-1}d_{a_{0}}(x,y)|\leq 2t_{0}^{P}.$$

In fact, if $\theta(d_{a_{0}}(x,y))=t_{0}^{P-1}d_{a_{0}}(x,y)$, it is trivially true. When $ \theta(d_{a_{0}}(x,y))=d_{a_{0}}(x,y)^{P}$, then:
$$|\theta(d_{a_{0}}(x,y))-t_{0}^{P-1}d_{a_{0}}(x,y)|\leq d_{a_{0}}(x,y)^{P}+t_{0}^{P-1}d_{a_{0}}(x,y)\leq 2t_{0}^{P}.  $$

Moreover, according to Lemma \ref{lemme: theta est lidentite alors a est dans toute geodesique entre x et y}, for all $a \in X_{c},a \neq a_{0}$, we have: $ \theta(d_{a}(x,y))=d_{a}(x,y)^{P}$. Then:
$$|D(x,y)-t_{0}^{P-1}d_{a_{0}}(x,y)|\leq 2t_{0}^{P}+\sum_{a\in X_{c}, a\neq a_{0}}d_{a}(x,y)^{P}.$$

We will bound the last sum from above by a constant.
Let us choose a geodesic $[x,y]_{c}$ between $x$ and $y$. According to Corollary \ref{corollary : cardinal des points paraboliques avec distance entre les barycentres des masks nest pas bien controlee}, if $a \notin Cone_{(2000\delta)^{2}}([x,y]_{c})$, we have, for all $a\neq a_{0}$:
$$d_{a}(x,y)\leq d_{a}(x,a_{0})+d_{a}(a_{0},y) \leq C(|| \mu_{x}(a) \Delta \mu_{a_{0}}(a) ||+|| \mu_{a_{0}}(a) \Delta \mu_{y}(a) ||). $$

As $d_{c}(x,y)\leq2$, there exists a constant $\lambda>0$ independent of $x$ and $y$ such that:
$$|Cone_{(2000\delta)^{2}}([x,y]_{c})|\leq \lambda.$$

Then:
$$\begin{aligned} 
           &~~\sum_{a\in X_{c}, a\neq a_{0}}d_{a}(x,y)^{P}
           &\\
             \leq & \lambda t_{0}^{P}+\sum_{a \in X_{c}}C^{P}(|| \mu_{x}(a) \Delta \mu_{a_{0}}(a) ||+|| \mu_{a_{0}}(a) \Delta \mu_{y}(a) ||)^{P} \\
            \leq & \lambda t_{0}^{P} +C^{P} [(\sum_{ a\in X_{c}}  || \mu_{x}(a) \Delta \mu_{a_{0}}(a) ||^{P})^{\frac{1}{P}} + (\sum_{ a\in X_{c}}  || \mu_{a_{0}}(a) \Delta \mu_{y}(a) ||^{P})^{\frac{1}{P}}]^{P}\\
            \leq & \lambda t_{0}^{P}+C^{P}2^{\frac{1}{P}}D~(\text{according to Proposition \ref{Proposition : sum des differences pour des voisins}}).
\end{aligned}$$

With $\Lambda:=t_{0}^{P}(\lambda+2)+C^{P}2^{\frac{1}{p'}}D$, we get the desired inequality.
\end{proof}

To conclude this subsection, it remains to show that $D$ satisfies a coarse triangular inequality. The next lemma will be used to prove it.

\begin{figure}[!ht]
   \centering
    \includegraphics[scale=0.4]{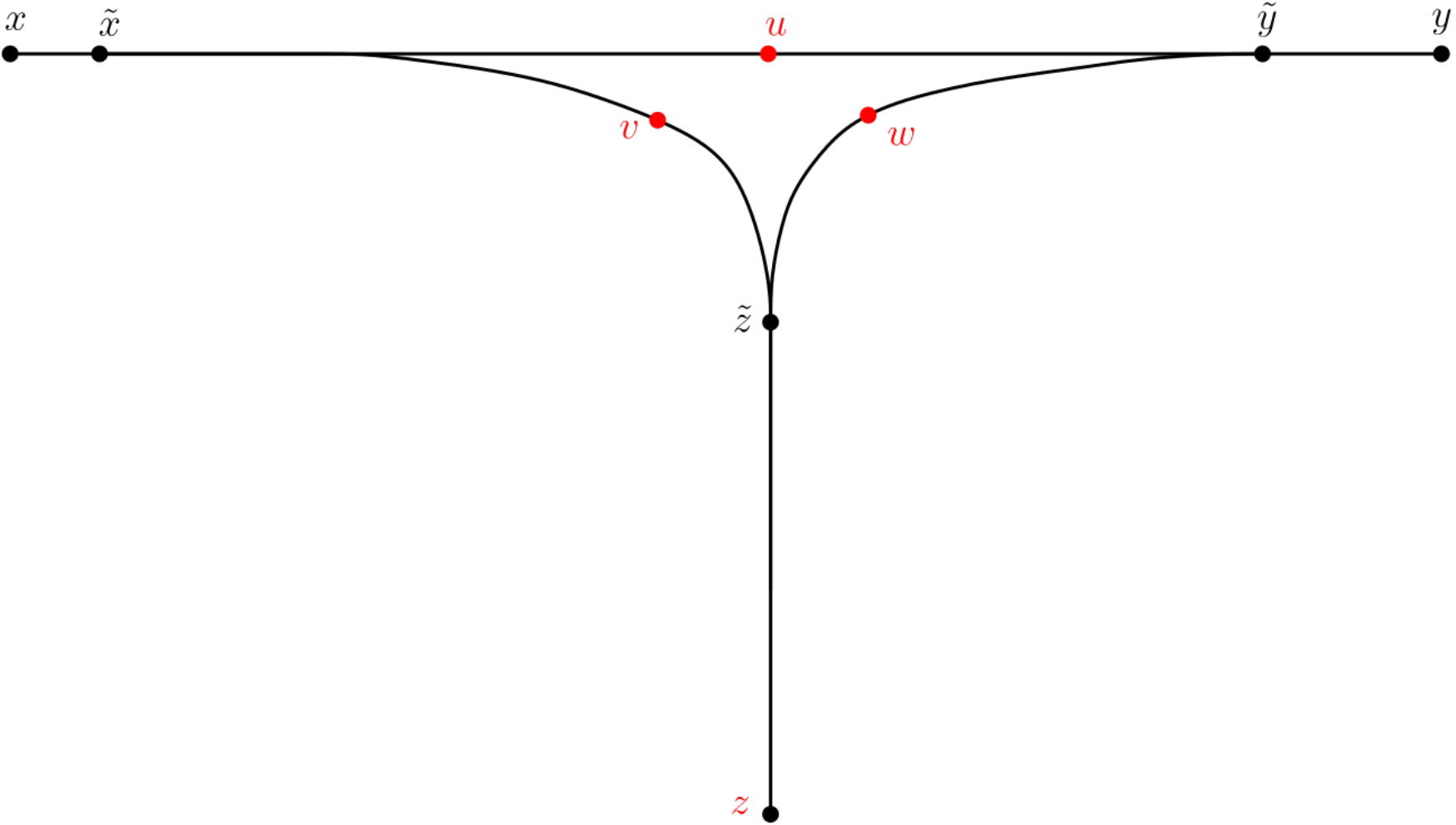}
    \caption{Angles between $u$ and $v$ and between $u$ and $w$ are bounded}
   \label{fig: lemme des points sur les geodesiques avec angles controlés}
 \end{figure}

\newpage
\begin{lem}\label{lem: quasimilieu a angle bornée}
Let us consider $x,y,z \in X_{c}$ such that $I_{c}(x,y)\cap I_{c}(y,z)  \cap I_{c}(z,x)=\emptyset $. Let $[x,y,z]_{c}$ be a triangle between $x,y $ and $z$. There exists $u \in [x,y]_{c} \cap G$ such that:

 \begin{itemize}
     \item $u \in Cone_{26+2}([x,z]_{c})$,

     \item $u \in Cone_{26+2}([z,y]_{c})$.
 \end{itemize}

\end{lem}

\begin{proof}

Since $I_{c}(x,y)\cap I_{c}(y,z)  \cap I_{c}(z,x)=\emptyset $ are all different, a triangle $[x,y,z]_{c}$ is not a tripod. According to $\delta$-hyperbolicity of $X_{c}$ and \cite[Proposition 21]{GhysDelaHarpe}, $[x,y,z]_{c}$ admits a $4\delta$-quasicenter. Up to choosing neighbors, we can find:
\begin{itemize}
    \item $u \in G\cap[x,y]_{c}$,

    \item $v \in [z,x]_{c}$,

    \item and $w \in [z,y]_{c} $,

    \item such that $d_{c}(u,v),d_{c}(u,w),d_{c}(v,w)\leq 8\delta+1$,
\end{itemize}
see Figure \ref{fig: lemme des points sur les geodesiques avec angles controlés}.

Moreover, we could choose $u,v,w$ such that:
\begin{itemize}
    \item $ |d_{c}(u,x)-d_{c}(v,x)|\leq 1$,

    \item $|d_{c}(u,y)-d_{c}(w,y)|\leq 1$,

    \item $|d_{c}(v,z)-d_{c}(w,z)|\leq 1.$
\end{itemize}
We will show that every angle between $u$ and $v$ is bounded, the proof is similar to the proof of \cite[Proposition 1.10]{ChatterjiDahmani}. Let us denote by $\sigma_{0}$ a geodesic between $u$ and $v$.
Let us denote by $\sigma_{1}$ the subsegment of $[x,y]_{c}$ obtained from $u$ and extended by $3\delta$ in the direction of $x$, unless it reaches $x$ before. (Otherwise, following the geodesics $[x,y]_{c}$ and $[x,z]_{c}$ provides a path of length at most $14\delta+2$). Denote $\sigma_{1}^{0}=u$ its initial point and $\sigma_{1}^{1}$ its endpoint. By hyperbolicity, there exists a segment at least $2\delta$ long, $\sigma_{2}$ from $\sigma_{1}^{1}$ to $[x,z]_{c}$. Let us denote $\sigma_{3}$, the subsegment of $[x,z]_{c}$ that closes the loop $\sigma_{0} \sigma_{1} \sigma_{2} \sigma_{3}$. By the triangular inequality, $\sigma_{3}$ is at least of length $13\delta+1$. This gives a loop of length $26\delta+2$, then all the angles between $u$ and $v$ are bounded by $26\delta+2$.
 The same proof applies to prove that $\measuredangle_{v}(x,u)\leq 26\delta+2$,$\measuredangle_{v}(u,y)\leq 26\delta +2$.
We conclude by symmetry of the roles of $u,v,w$.

\end{proof}

We now prove that $D$ satisfies a coarse triangular inequality.

 \begin{figure}[!ht]
    \centering
   \includegraphics[scale=0.5]{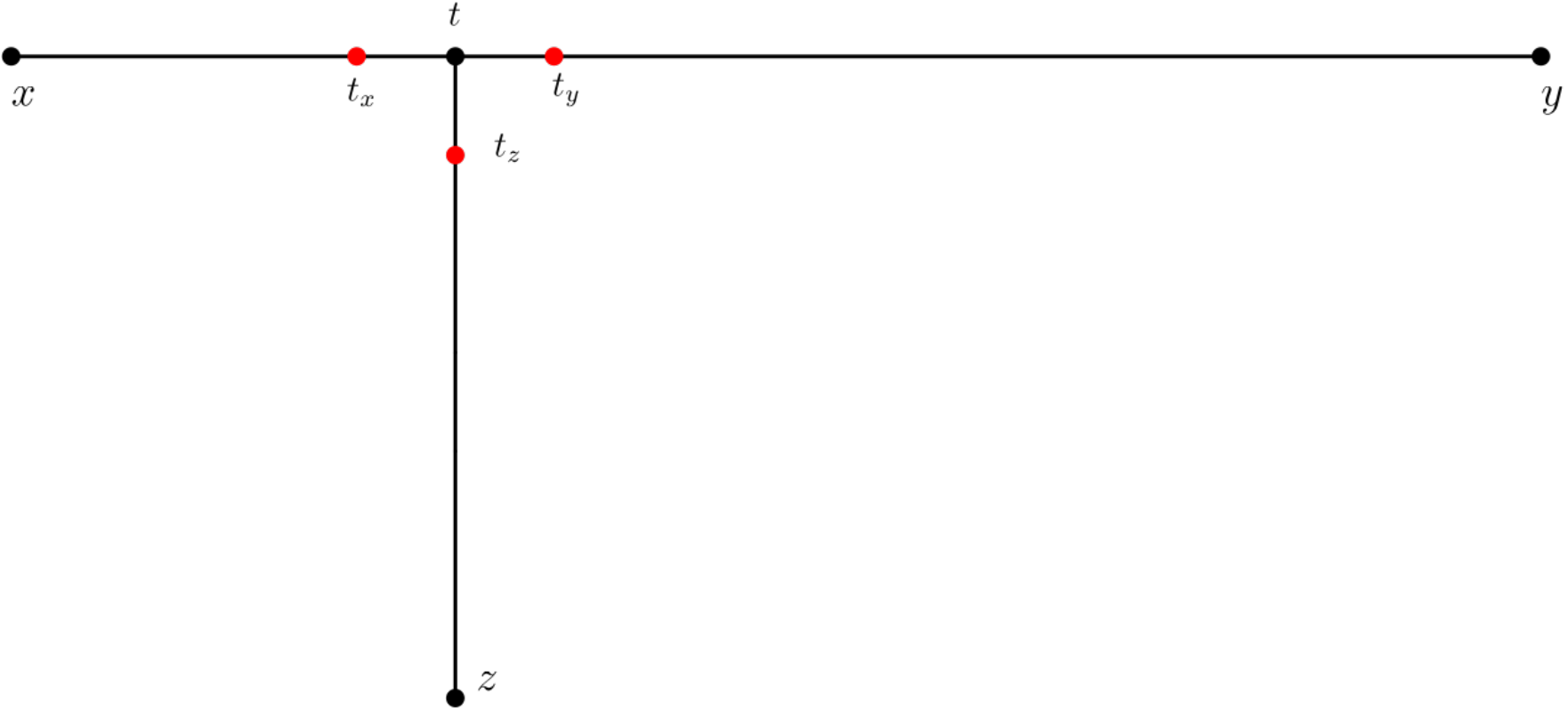}
   \caption{Case where $\tilde{x},\tilde{y},\tilde{z}$ are not all different and $t \in X_{c}^{\infty}$.}
   \label{fig: inégalité tirangulaire t de valence infini}
\end{figure}

\newpage

\begin{proposition}\label{proposition: inegalite triangulaire grossiere}
There exists a constant $C_{\text{triangle}}>0$ such that for all $x,y,z \in G$:
$$D(x,y) \leq D(x,z)+D(z,y)+C_{\text{triangle}}.$$
\end{proposition}

\begin{proof}
Let $x,y,z \in X_{c}$. We consider a triangle $[x,y,z]_{c}$. There are two cases to consider wether $I_{c}(x,y)\cap I_{c}(y,z)  \cap I_{c}(z,x)$ is empty or not.

\begin{enumerate}
\item Let us assume that $I_{c}(x,y)\cap I_{c}(y,z)  \cap I_{c}(z,x)\neq \emptyset$. There exists a vertex $t$ belonging to the three intervals.

\begin{itemize}
    \item In the case where $t \in G$, we can apply Proposition \ref{proposition: inegalite clutch inegalite triangulaire et faiblement geodesique}.
    Then we get:
 $$\begin{aligned} 
           &D(x,y)-D(x,z)-D(z,y)
           &\\
             \leq & D(x,t)+D(y,t)+\nu\\
             -&D(x,t)-D(t,z)+\nu-D(y,t)-D(t,z)+\nu ~(\text{according to Proposition \ref{proposition: inegalite clutch inegalite triangulaire et faiblement geodesique}}) \\
            \leq & 3\nu .
\end{aligned}$$

\item In the case where $t \in X_{c}^{\infty}$, we will apply Proposition \ref{proposition: pour deux sommets voisins de a un pt parabolique D est presque da} and the fact that for all $a \in X_{c}$, $d_{a}$ verifies the triangular inequality. We denote $t_{x}$ as the neighbor of $t$ on $[x,t]_{c}$, $t_{y}$ as the neighbor on $[t,y]_{c}$, and $t_{z}$ as the neighbor on $[z,t]_{c}$, see Figure \ref{fig: inégalité tirangulaire t de valence infini}.
Then, we get:
$$\begin{aligned} 
          & D(x,y)-D(x,z)-D(z,y)\\
             \leq & D(x,t_{x})+D(t_{x},t_{y})+D(t_{y},y)+2\nu\\
             -&D(x,t_{x})-D(t_{x},t_{z})-D(t_{z},z)+2\nu\\
             -&D(y,t_{y})-D(t_{y},t_{z})-D(t_{z},z)+2\nu ~(\text{according to Proposition \ref{proposition: inegalite clutch inegalite triangulaire et faiblement geodesique} applied three times}) \\
            \leq & D(t_{x},t_{y})-D(t_{y},t_{z})-D(t_{z},t_{x})+6\nu \\
            \leq & t_{0}^{p'-1}d_{a}(t_{x},t_{y})+\Lambda-t_{0}^{p'-1}d_{a}(t_{x},t_{z})\\
            +&\Lambda-t_{0}^{p'-1}d_{a}(t_{y},t_{z})+\Lambda+6\nu~(\text{according to Proposition \ref{proposition: pour deux sommets voisins de a un pt parabolique D est presque da}})
            \\ \leq & \Lambda t_{0}^{p'-1}(d_{a}(t_{x},t_{y})-d_{a}(t_{x},t_{z})-d_{a}(t_{z},t_{y}))+6\nu\\
            \leq & 6\nu~(\text{according to the triangular inequality for}~d_{a}).
            \end{aligned}$$
\end{itemize}

\item Let us now assume that $I_{c}(x,y)\cap I_{c}(y,z)  \cap I_{c}(z,x)= \emptyset$. According to Lemma \ref{lem: quasimilieu a angle bornée}, see Figure \ref{fig: lemme des points sur les geodesiques avec angles controlés}, there exist: $u \in G\cap[x,y]_{c},v \in [x,z]_{c}$ such that:

 \begin{itemize}
     \item $u \in Cone_{26+2}([x,z]_{c})$,

     \item $u \in Cone_{26+2}([z,y]_{c})$.
 \end{itemize}

According to Proposition \ref{proposition: inegalite clutch inegalite triangulaire et faiblement geodesique}, we have:
\begin{itemize}
    \item $|D(x,y)-D(x,u)-D(u,y)|\leq \nu$,

    \item $|D(x,z)-D(x,u)-D(u,y)|\leq \nu$,

    \item $|D(z,y)-D(z,u)-D(u,y)|\leq \nu$
    
\end{itemize}

and thus:
$$\begin{aligned} 
           &D(x,y)-D(x,z)-D(z,y)
           &\\
             \leq & D(x,u)+D(u,y)+\nu\\
             -&D(x,u)-D(u,z)+\nu-D(y,u)-D(u,z)+\nu ~(\text{according to Proposition \ref{proposition: inegalite clutch inegalite triangulaire et faiblement geodesique}}) \\
            \leq & 3\nu .
\end{aligned}$$

With $C_{\text{triangle}}=6\nu$, we get the result.

\end{enumerate}

\end{proof}

\begin{definition}\label{definition : metrique fortement bolique groupe relativement hyperbolique}

We define the function $\hat{d}: G \times G \rightarrow \mathbb{R}_{+}$ by:

$$
\hat{d}(x,y) = \left\{
    \begin{array}{ll}
        D(x,y)+C_{\text{triangle}}& \mbox{if} ~ x \ne y \\
        0 & \mbox{if} ~ x=y.
    \end{array}
\right.
$$

where $C_{\text{triangle}}$ is the constant of Proposition \ref{proposition: inegalite triangulaire grossiere}.

\end{definition}

\begin{proposition}\label{proposition: la metrique fortement bolique est une metrique}

The function $\hat{d}$ of Definition \ref{definition : metrique fortement bolique groupe relativement hyperbolique} is a metric on $X_{c}$.
    
\end{proposition}

\begin{proof}
By definition, $\hat{d}$ is symmetric and $\hat{d}(x,y)=0$ if and only if $x=y$.

The triangle inequality is a direct consequence of Proposition \ref{proposition: inegalite triangulaire grossiere}.
    
\end{proof}

The rest of the article will be dedicated to the proof of

\subsection{The bolic metric is comparable to the group metric}\label{subsection: metric comparable to the group metric}

In this section, we will show that the metric $\hat{d}$ is coarsely smaller than the group metric, that is, there exists an increasing and unbounded function $\rho_{1}$ such that for all $x,y\in G$:
$$\rho_{1}(d_{G}(x,y)) \leq \hat{d}(x,y).$$

For the reverse inequality, since $d$ is $G$-invariant 
$G$ is finitely generated, we directly know that $\hat{d}$ is smaller than $G$ with linear control.

The fact that the strongly bolic metric is coarsely bigger than the group metric will be used to show that this metric is uniformly locally finite.
The other inequality will be used to prove that the strongly bolic metric verifies the strongly-$B1$ condition.

In Subsection \ref{subsection: faiblement geodesique}, we will use the fact that these two metrics are weakly geodesic to show that these two metrics are in fact quasi-isometric.

\begin{proposition}\label{Proposition : la distance parabolique est plus petite que la distance dans le groupe}

There exists $M>0$, such that for all $x,y \in G$:
$$\hat{d}(x,y)\leq M d_{G}(x,y).  $$

\end{proposition}

\begin{proof}
We set $S$ a generating of $G$ associated to $d_{G}$.
We set $M:=\max_{s\in S}{\hat{d}(1,s)}$.

Let $x,y \in G$. Let $n:=d_{G}(x,y)$. Then there exists $s_{1},...,s_{n} \in S$ such that $x^{-1}y=s_{1},...,s_{n}$ and we have:
$$\begin{aligned}
\hat{d}(x,y)&=\hat{d}(1,x^{-1}y) \\
& \leq \hat{d}(1,s_{1})+\hat{d}(1,s_{2})+...+\hat{d}(1,s_{n})\\
& \leq M d_{G}(x,y).
\end{aligned}$$

\end{proof}

We prove now the reverse inequality.

\begin{proposition}\label{proposition: la distance dans le groupe est plus petite que la distance bolic}

There exists an increasing and unbounded function $\rho_{1}$ such that, for all $x,y \in G$:
$$ \rho_{1}(d_{G}(x,y)) \leq \hat{d}(x,y). $$

\end{proposition}

\begin{proof}

Let $a \in X_{c}^{\infty}$, according to Proposition \ref{proposition: l'angle est borné par la distance bolic},
$$ \measuredangle_{a}(x,y)\rightarrow \infty \Rightarrow d_{a}(x,y) \rightarrow \infty .$$

We fix a geodesic $[x,y]_{c}$. In the construction of the coned-off graph, two vertices of infinite valence cannot be adjacent, therefore:
$$|[x,y]_{c}\cap G| \geq \frac{1}{2}d_{c}(x,y). $$

Moreover, according to Proposition \ref{proposition: non-confluence des masks}, if $d_{c}(x,y)\geq 10 \delta$, for all $a \in [x,y]_{c}$ at distance at least $5\delta$ from $x$ and $y$, we have $||\mu_{x}(a)\Delta\mu_{y}(a)||=2$.
Thus, if we suppose that $d_{c}(x,y)\geq 10 \delta$, we have:
$$\begin{aligned} 
           D(x,y) \geq & \sum_{a \in [x,y]_{c}\cap G} \theta(d_{a}(x,y)),\\
            \geq &\sum_{a \in [x,y]_{c}\cap G}||\mu_{x}(a)\Delta\mu_{y}(a)||^{p}~(\text{according to Definition \ref{definition: da}}),\\
            \geq & \frac{1}{2}d_{c}(x,y)2^{p}-10\delta ~(\text{according to Proposition \ref{proposition: non-confluence des masks}}).
\end{aligned}$$

Then:
$$d_{c}(x,y)\rightarrow \infty \Rightarrow D(x,y)\rightarrow \infty.$$

To conclude, we use Theorem \ref{Proposition : formule de la distance}, the distance formula and Definition \ref{definition : metrique fortement bolique groupe relativement hyperbolique}, to show that there exists an increasing and unbounded function $\rho_{1}$, such that for all $x,y \in G$:
$$\rho_{1}(d_{G}(x,y)) \leq  \hat{d}(x,y).  $$
\end{proof}

\subsection{Weak geodesicity and quasi-isometry with $d_{G}$}\label{subsection: faiblement geodesique}

In this section, we will show that the metric $\hat{d}$ is weakly geodesic. We will use this fact and Subsection \ref{subsection: metric comparable to the group metric} to show that $\hat{d}$ and $d_{G}$ are, in fact, quasi-isometric.

Recall that a metric space $(X,d)$ is said to be weakly geodesic if there exists $\eta \geq 0 $ such that, for every pair of points $x,y  \in X$ and every $t \in [0,d(x,y)]$, there exists a point $z \in X$ such that $d(x,z) \leq t + \eta$ and $d(z,y)\leq d(x,y)-t+\eta$.

Recall that the parabolic subgroups of $G$ are admissible as in Definition \ref{definition: lesbons parabolic}. Therefore, we have assumed that for all parabolic subgroups $P$, for all $C>0$, $\text{Proba}_{C}(P)$ admits a strongly bolic metric and these metrics are weakly geodesic. Thus, by taking the maximum of the weak geodesicity constants, we can assume that there exists $\varepsilon \geq 0$ such that $d_{b}$ is $\varepsilon$-weakly geodesic, where $d_{b}$ denotes the associated strongly bolic metric, without distinction on the parabolic.\\

According to Proposition \ref{proposition: inegalite clutch inegalite triangulaire et faiblement geodesique}, for any pair of vertices $x,y \in X_{c}$, and for any geodesic $[x,y]_{c}$ between these two points, a vertex $z\in G\cap [x,y]_{c}$ satisfies the equality case of the triangle inequality for $\hat{d}$ up to a bounded error.

However, this is not sufficient to show the weak-geodesicity of $\hat{d}$, since two adjacent vertices $u$ and $v$ of a parabolic vertex $a$ can be at an arbitrarily large distance for $\hat{d}$. In this case, we will use the weak-geodesicity of $d_{b}$. The projection of the weak geodesic between $u$ and $v$ on the $1$-neighborhood of $a$ will allow us to construct a path in $X_{c}$ between $u$ and $v$ for which the steps are bounded for $\hat{d}$ and whose vertices satisfy the equality case of the triangle inequality up to a bounded error.

\begin{proposition}\label{faiblement geodesique pour deux points voisins dun parabolique}

For all $a \in X_{c}^{\infty}$, for all $x,y\in G$ with $d_{c}(x,a)=d_{c}(a,y)=1$, for all $t \in [0,d_{a}(x,y)]$, there exists $u \in G$ with $d_{c}(u,a)=1$ such that:

\begin{itemize}
    \item $d_{a}(x,u)\leq t+4C+\varepsilon$,

    \item $d_{a}(u,y) \leq d_{a}(x,y)-t+4C+\varepsilon$,

    \item and $\measuredangle_{a}(x,u)\geq 12 \delta$,  $\measuredangle_{a}(u,y)\geq 12 \delta$,
\end{itemize}

where $\varepsilon$ is the weak geodesic constant of $d_{b}$ and $C$ is the constant of Proposition \ref{proposition : les masks sont uniformément bornés} .
\end{proposition}

\begin{proof}
After applying a translation if needed, we denote by $P$ the parabolic associated to $a$.
Let $x, y \in G$ under the assumptions of the proposition. Recall that $d_{a}(x,y)=d_{b}(\delta_{x},\delta_{y})$. According to the $\varepsilon$-weak geodesicity of $d_{b}$, for $t \in [0,d_{a}(x,y)]$, there exists $\lambda\in \text{Proba}_{C}(P) $ such that:

\begin{itemize}
    \item $d_{b}(\delta_{x},\lambda)\leq t+\varepsilon$,

    \item $d_{b}(\lambda,\delta_{y})\leq  d_{a}(x,y)-t+\varepsilon$.
\end{itemize}

Furthermore, by adjusting $\varepsilon$ if necessary, we can assume that:
$$d_{b}(\delta_{x},\lambda)\geq A12\delta-AB+2C, ~d_{b}(\lambda,\delta_{y})\geq A12\delta-AB+2C,$$

where $A, B$ are the constants of Proposition \ref{proposition: l'angle est borné par la distance bolic}.

Let $u \in  \text{supp}(\lambda)$, then $u \in G$ and according to the last point of Definition \ref{definition: lesbons parabolic}, we have:
$$ d_{b}(\lambda,\delta_{u})\leq 4C.$$

Therefore, we have:
\begin{itemize}
    \item $d_{b}(\delta_{x},\delta_{u})\leq t+4C+\varepsilon$,

    \item $d_{b}(\delta_{u},\delta_{y})\leq  d_{a}(x,y)-t+4C+\varepsilon$,

    \item $ d_{b}(\delta_{x},\delta_{u})\geq A12\delta-AB, ~d_{b}(u,\delta_{y})\geq A12\delta-AB. $
\end{itemize}

According to Proposition \ref{proposition: l'angle est borné par la distance bolic}, the third point gives the following:
$$ \measuredangle_{a}(x,u)\geq 12 \delta,  ~\measuredangle_{a}(u,y)\geq 12 \delta.$$

\end{proof}

The next definition describes the weak geodesic in a parabolic or equivalently in the $1$-neighborhood of a vertex of infinite valency.

\begin{definition}\label{definition: chemin dnas les paraboliques} ( Privileged weak geodesic in a parabolic)

Let $a \in X_{c} ^{\infty}$ and $P$ the associated parabolic subgroup, for all $x,y \in G$ with $d_{c}(x,a)=d_{c}(a,y)=1$, we call a \emph{privileged weak geodesic}, denoted by $[x,y]_{d_{a}}$, a set of vertices in the $1$-neighborhood of $a$ with the following properties:
\begin{itemize}
    \item $x,y\in  [x,y]_{d_{a}}$,

    \item $[x,y]_{d_{a}}$ is a $(4C+\varepsilon)$-weak geodesic,

    \item for all $u \in [x,y]_{d_{a}}, u\neq x,y $,
$$ \measuredangle_{a}(x,u)\geq 12 \delta,  ~\measuredangle_{a}(u,y)\geq 12 \delta.$$

\item $[x,y]_{d_{a}}$ is $\eta$-peakless and $\eta$-strongly-convex for some $\eta\geq 0$.
\end{itemize}

\end{definition}
According to Proposition \ref{faiblement geodesique pour deux points voisins dun parabolique}, for all $a \in X_{c}^{\infty}$, for all $x,y \in G$ with $d_{c}(x,a)=d_{c}(a,y)=1$, there exists such a set of vertices $[x,y]_{d_{a}}$.

\begin{definition}\label{definition: systeme de geodesic faible}
Let $x,y \in G$, such that $d_{c}(x,y)=n$. We consider a geodesic $[x,y]_{c}$ between those two vertices ordered as follows $[x,y]_{c}=\{x_{0},x_{1},...,x_{n}\}$ with $
x_{0}=x$, $x_{n}=y$ and $d_{c}(x_{i},x_{i+1})=1$ for $0\le i \le n-1$.

We set, (even if its depend on the choice of a geodesic): $$[x,y]_{\hat{d}}:=([x,y]_{c}\cap G) \cup \bigcup_{x_{i}\in X_{c}^{\infty}}[x_{i-1},x_{i+1}]_{d_{x_{i}}}.$$

\end{definition}

\begin{figure}[!ht]
   \centering
  \includegraphics[scale=0.8]{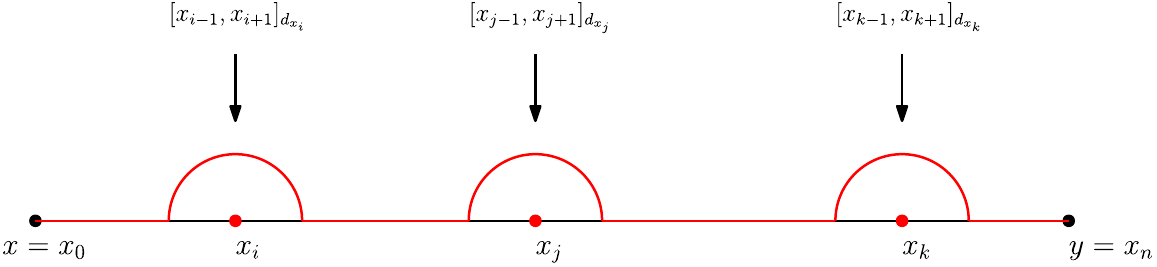}
  \caption{Weak geodesic between $x$ and $y$ with $3$ points of infinite valency}
   \label{ fig: geodesic faible}
\end{figure}

\newpage

\begin{proposition}\label{proposition: faiblement geodesique}
The metric space $(G,\hat{d})$ is weakly geodesic.
\end{proposition}

\begin{proof}
We will show that the set defined in Definition \ref{definition: systeme de geodesic faible} is a weak geodesic between $x$ and $y$.

Let $x,y \in G$, such that $d_{c}(x,y)=n$. We consider a geodesic $[x,y]_{c}$ between those two vertices ordered as follows $[x,y]_{c}=\{x_{0},x_{1},...,x_{n}\}$ with $
x_{0}=x$, $x_{n}=y$ and $d_{c}(x_{i},x_{i+1})=1$ for $0\le i \le n-1$. We consider the associated set $[x,y]_{\hat{d}}:=([x,y]_{c}\cap G) \cup \bigcup_{x_{i}\in X_{c}^{\infty}}[x_{i-1},x_{i+1}]_{d_{x_{i}}}.$

Let $z \in  [x,y]_{c}\cap G $, according to Proposition \ref{proposition: inegalite clutch inegalite triangulaire et faiblement geodesique}, we have:
$$\hat{d}(x,z)+\hat{d}(z,y)-\hat{d}(x,y)\leq \nu + 2C_{triangle}.$$

Let $i_{0} \in \{1,...,n-1\}$ such that $x_{i_{0}}\in X_{c}^{\infty}$. Let $z \in [x_{i_{0}-1},x_{i_{0}+1}]_{d_{x_{i_{0}}}}$.
According to Definition \ref{definition: chemin dnas les paraboliques}, we have:
$$\measuredangle_{x_{i_{0}}}(x_{i_{0}-1},z)\geq 12 \delta,  ~\measuredangle_{x_{i_{0}}}(z,x_{i_{0}+1})\geq 12 \delta.$$

Therefore, since $x_{i_{0}-1} \in I_{c}(x,x_{i_{0}})$ and $x_{i_{0}+1} \in I_{c}(x_{i_{0}},y)$, we get that:
$$\measuredangle_{x_{i_{0}}}(x,z)\geq 12 \delta,  ~\measuredangle_{x_{i_{0}}}(z,y)\geq 12 \delta.$$

Thus, according to Proposition \ref{proposition: les propirétés des angles}, we get that:
\begin{itemize}
    \item $x_{i_{0}-1}\in I_{c}(x,z)$,

    \item and $x_{i_{0}+1}\in I_{c}(z,y).$
\end{itemize}

We get:
$$\begin{aligned}
    &~~~\hat{d}(x,z)+\hat{d}(z,y)-\hat{d}(x,y)\\
     &\leq D(x,x_{i_{0}-1})+ D(x_{i_{0}-1},z)+D(z,x_{i_{0}+1})+D(x_{i_{0}+1},y) \\
     &-D(x,x_{i_{0}-1})-D(x_{i_{0}-1},x_{i_{0}+1})-D(x_{i_{0}+1},y)+ 4\nu +2C_{triangle}~\text{(according to Proposition \ref{proposition: inegalite clutch inegalite triangulaire et faiblement geodesique})}\\
     &\leq  D(x_{i_{0}-1},z)+D(z,x_{i_{0}+1})-D(x_{i_{0}-1},x_{i_{0}+1})+ 4\nu +2C_{triangle}\\
     &\leq  t_{0}^{P-1}(d_{x_{i_{0}}}(x_{i_{0}-1},z)+d_{x_{i_{0}}}(z,x_{i_{0}+1})-d_{x_{i_{0}}}(x_{i_{0}-1},x_{i_{0}+1}))\\
     &+3\Lambda+ 4\nu +2C_{triangle}~ \text{(according to Proposition \ref{proposition: pour deux sommets voisins de a un pt parabolique D est presque da})}\\
      &\leq t_{0}^{P-1}(4C+\varepsilon)+3\Lambda+ 4\nu +2C_{triangle}~\text{(according to Definition \ref{definition: chemin dnas les paraboliques})}.
\end{aligned}$$

Then for all $z \in [x,y]_{\hat{d}}$, we have:
$$ \hat{d}(x,z)+\hat{d}(z,y)-\hat{d}(x,y) \leq  t_{0}^{P-1}(2C+\varepsilon)+3\Lambda+ 4\nu +2C_{triangle}. $$

Hence, the image of map:
$$\hat{d}(x,.): [x,y]_{\hat{d}} \rightarrow [0,\infty) $$
is contained in $[0,\hat{d}(x,y)+t_{0}^{P-1}(2C+\varepsilon)+3\Lambda+ 4\nu +2C_{triangle}]$, it contains $0$ and $\hat{d}(x,y)$.

Let $z \in  [x,y]_{\hat{d}}$.
If $z \in [x,y]_{\hat{d}} - (\bigcup_{x_{i}\in X_{c}^{\infty}}[x_{i-1},x_{i+1}]_{d_{x_{i}}})$, there exists $z' \in [x,y]_{c}\cap G $ with $d_{c}(z,z')=1$.
According to Proposition \ref{Proposition : la distance parabolique est plus petite que la distance dans le groupe}, there exists $\tau>0$ such that for all $z,z' \in G$ with $d_{c}(z,z')=1$, we have:
$$|\hat{d}(x,z)-\hat{d}(x,z')|\leq \hat{d}(z,z')\leq \tau. $$

If $z \in \bigcup_{x_{i}\in X_{c}^{\infty}}[x_{i-1},x_{i+1}]_{d_{x_{i}}}$, there exists $i_{0}\in \{1,...,n\}$, such that $x_{i_{0}}\in [x,y]_{c}\cap X_{c}^{\infty}$ and $z \in [x_{i_{0}-1},x_{i_{0}+1}]_{d_{x_{i_{0}}}}$. According to Definition \ref{definition: chemin dnas les paraboliques} of $[x_{i_{0}-1},x_{i_{0}+1}]_{d_{x_{i_{0}}}}$, there exists $z' \in [x_{i_{0}-1},x_{i_{0}+1}]_{d_{x_{i_{0}}}} $ such that:
$$ |d_{x_{i_{0}}}(x_{i_{0}-1},z)-d_{x_{i_{0}}}(x_{i_{0}-1},z')|\leq  4C+\varepsilon .$$

Therefore:
$$\begin{aligned}
|\hat{d}(x,z)-\hat{d}(x,z')|&\leq|\hat{d}(x_{i_{0}},z)-\hat{d}(x_{i_{0}},z')|+2\nu\\
&\leq t_{0}^{P-1}|d_{x_{i_{0}}}(x_{i_{0}},z)-d_{x_{i_{0}}}(x_{i_{0}},z')|+2\Lambda+2\nu\\
&\leq t_{0}^{P-1}2C+\varepsilon+2\Lambda+2\nu.
\end{aligned}$$

Thus, the image of the map:
$$ \hat{d}(x,.): [x,y]_{\hat{d}} \rightarrow [0,\hat{d}(x,y)+t_{0}^{P-1}(4C+\varepsilon)+3\Lambda+ 4\nu +2C_{triangle}], $$

is $\max(\tau,t_{0}^{P-1}2K+\varepsilon+2\Lambda+2\nu)$-dense in $[0,\hat{d}(x,y)]$, i.e. for every $t \in [0,\hat{d}(x,y)] $, there exists $z \in [x,y]_{\hat{d}} $ such that:
$$ |\hat{d}(x,z)-t|\leq \max(\tau,t_{0}^{P-1}4C+\varepsilon+2\Lambda+2\nu).  $$

It follows that $ \hat{d}(x,z)\leq t +  \max(\tau,t_{0}^{P-1}4
C+\varepsilon+2\Lambda+2\nu)$, and according to the previous discussion, we have:
$$\begin{aligned}
\hat{d}(z,y)&\leq \hat{d}(x,y)-\hat{d}(x,z)+t_{0}^{P-1}(4C+\varepsilon)+3\Lambda+ 4\nu +2C_{tri}\\
& \leq \hat{d}(x,y)-t+\max(\tau,t_{0}^{P-1}4C+\varepsilon+2\Lambda+2\nu)+t_{0}^{P-1}(4C+\varepsilon)+3\Lambda+ 4\nu +2C_{triangle}.
\end{aligned}$$

To conclude $(G,\hat{d})$ is $\eta$-weakly geodesic for $$\eta:=\max(\tau,t_{0}^{P-1}4C+\varepsilon+2\Lambda+2\nu)+t_{0}^{P-1}(4C+\varepsilon)+3\Lambda+ 4\nu +2C_{triangle}.$$ 

\end{proof}

The fact that $(G,\hat{d})$ is weakly geodesic allows us to improve the coarse equivalence between $\hat{d}$ and $d_{G}$ into a quasi-isometry.

\begin{proposition}\label{proposition: quasi-isometry}
$(G,\hat{d})$ and $(G,d_{G})$ are quasi-isometric.
\end{proposition}

\begin{proof}
According to Proposition \ref{Proposition : la distance parabolique est plus petite que la distance dans le groupe}, there exists $M>0$ such that for all $x,y \in G$, we have:
$$\hat{d}(x,y)\leq M d_{G}(x,y).$$

For the converse inequality, we use that $\hat{d}$ is $\varepsilon$-weakly geodesic for some $\varepsilon >0$. Let us consider a weak geodesic $\gamma =\{w_{0},w_{1},...,w_{m}\}$ between $x$ and $y$, with $
w_{0}=x$, $w_{m}=y$. We assume further that there exists $\beta>0$ such that $\hat{d}(w_{i},w_{i+1})\in [\beta,2\beta] $ for $0\le i \le m-1$.
We remark that:
$$ \hat{d}(x,y)\geq \sum_{i=0}^{m}\hat{d}(w_{i},w_{i+1})-m\varepsilon\geq m(\beta-\varepsilon).$$

Then: $$ m \leq \frac{1}{\beta-\varepsilon}\hat{d}(x,y).$$

Moreover, according to Proposition \ref{proposition: la distance dans le groupe est plus petite que la distance bolic}, there exists $\alpha_{2}\geq 0$, such that for all $u,v \in G$ with $\hat{d}(u,v)\leq 2\beta$, we have:

$$d_{G}(u,v)\leq \alpha_{2}.$$

Thus:
$$d_{G}(x,y)\leq \sum_{i=0}^{m}d_{G}(w_{i},w_{i+1})\leq m \alpha_{2} \leq \frac{\alpha_{2}}{\beta-\varepsilon}\hat{d}(x,y).$$

\end{proof}

\subsection{$\hat{d}$ verifies weak-$B2'$}\label{subsection: metrique bolic verifie B2'}
In this subsection, we will show that $\hat{d}$ verifies weak-$B2'$. We recall, according to Definitions \ref{definition: lesbons parabolic} and \ref{definition: chemin dnas les paraboliques}, there exists $\eta >0$, such that for all $a \in X_{c}^{\infty}$, $d_{a}$ satisfies $\eta-B2'$, and for all $x,y \in X_{c}$ with $d_{c}(x,a)=d_{c}(a,y)=1$, there exists $[x,y]_{d_{a}}$ a sequence of vertices in the $1$-neighborhood of $a$, such that $[x,y]_{d_{a}}$ is a $\eta$-weak geodesic, $\eta$-peakless and $\eta$-weakly convex. 

\begin{proposition}\label{proposition: stabilite geod faible}
 Let $x_{1},y_{1} \in G$ and $x_{2},y_{2} \in G$ such that $x_{2},y_{2} \in I_{c}(x_{1},y_{1})$. There exists a geodesic between $x_{1}$ and $y_{1}$ and a geodesic between $x_{2}$ and $y_{2}$ in $X_{c}$, such that the associated weak geodesic $[x_{1},y_{1}]_{ \hat{d}}$ and $[x_{2},y_{2}]_{ \hat{d}}$ satisfy: $$[x_{2},y_{2}]_{ \hat{d}}\subset [x_{1},y_{1}]_{ \hat{d}}.$$
\end{proposition}

\begin{proof}

Let $x_{1},y_{1} \in G$, such that $d_{c}(x_{1},y_{1})=n$, $x_{2},y_{2} \in G$ such that $d_{c}(x_{1},x_{2})=l \in \mathbb{N}$ and $ d_{c}(x_{1},y_{2})=k $. We consider a geodesic $[x_{1},y_{1}]_{c}$ between those two vertices which contains $x_{2},y_{2}$ between those two vertices ordered as follows $[x_{1},y_{1}]_{c}=\{z_{0},...,z_{l},...,z_{k},...,z_{n}\}$ with $
z_{0}=x_{1}$, $z_{l}=x_{2}$, $z_{k}=y_{2}$ $z_{n}=y_{1}$ and $d_{c}(z_{i},z_{i+1})=1$ for $0\le i \le n-1$.

We set the associated weak geodesic sequence as in Definition \ref{definition: systeme de geodesic faible}:
$$[x_{1},y_{1}]_{\hat{d}}=([x_{1},y_{1}]_{c}\cap G) \cup \bigcup_{z_{i}\in X_{c}^{\infty}}[z_{i-1},z_{i+1}]_{d_{z_{i}}},$$
with $[z_{i-1},z_{i+1}]_{d_{z_{i}}}$ a weak geodesic sequence as in Definition \ref{definition: chemin dnas les paraboliques}.\\

Then $[x_{2},y_{2}]_{\hat{d}}=([x_{2},y_{2}]_{c}\cap G) \cup \displaystyle \bigcup_{z_{i}\in X_{c}^{\infty}, i \in \{k,...,l\}}[z_{i-1},z_{i+1}]_{d_{z_{i}}}$ is a weak geodesic sequence, according to Proposition \ref{proposition: faiblement geodesique} and is included in $[x_{1},y_{1}]_{\hat{d}}$.
    
\end{proof}

\begin{lem}\label{lemme: lemme pour geodesique faible}
Let $x,y \in G$, we denote by $[x,y]_{c}$ a geodesic between $x$ and $y$ and $[x,y]_{\hat{d}}$ the weak geodesic sequence associated to $[x,y]_{c}$ as in Definition \ref{definition: systeme de geodesic faible}. Let $z \in G$, we set $t \in X_{c}$ a projection of $z$ on $[x,y]_{c}$. We set also $t_{x} \in [x,t]_{c}$ such that $d_{c}(t,t_{x})=1$ and $t_{z}\in [t,z]_{c}$ such that $d_{c}(t,t_{z})=1$.
We have the following:
\begin{itemize}
    \item if $t \in G$, then $\hat{d}(z,x)\geq\hat{d}(z,t)+\hat{d}(t,x)-\nu-C_{triangle} $,

    \item otherwise $t_{x},t_{z} \in G$ and $\hat{d}(z,x)\geq\hat{d}(z,t_{z})+\hat{d}(t_{z},t_{x})+\hat{d}(t_{x},x)-\nu-C_{triangle} $,
\end{itemize}
where $\nu$ is the constant of Proposition \ref{proposition: inegalite clutch inegalite triangulaire et faiblement geodesique} and $C_{triangle}$ is the constant of Proposition \ref{proposition : inegalite triangulaire grossiere}.
    
\end{lem}

\begin{proof}

If $t \in G$, there are two cases to consider.

If $\measuredangle_{t}(x,z) > 12 \delta$, then $t$ belongs to every geodesic between $x$ and $z$,
then $\hat{d}(z,x)\geq\hat{d}(z,t)+\hat{d}(t,x)-\nu-C_{triangle} $ according to Proposition \ref{proposition: inegalite clutch inegalite triangulaire et faiblement geodesique}.
Otherwise, we set a the triangle $[x,z,t]$.  We notice that angles between $t$ and its projection on $[x,z]_{c}$ are bounded by $26 \delta+2$, then $t \in cone_{26\delta+2}([x,z]_{c})$, then it satisfies again the assumptions of Proposition \ref{proposition: inegalite clutch inegalite triangulaire et faiblement geodesique}.

If $t \notin G$, we apply the same method to prove that $t_{x},t_{z} \in  cone_{26\delta+2}([x,z]_{c})$ and we conclude according to Proposition \ref{proposition: inegalite clutch inegalite triangulaire et faiblement geodesique}.

\end{proof}

We prove here that $\hat{d}$ satisfies weak-$B2'$. We conjecture the stronger property that $\hat{d}$ is also $\eta$-strongly convex for some $\eta \geq0$.

\begin{proposition}\label{proposition: dchapeau verifie faiblement b2'}

$(G,\hat{d})$ verifies weak $B2'$, i.e. there exists $\eta_{1}\geq 0$ and a map $m: X \times X \rightarrow X$, such that:

    \begin{itemize}

    \item $m$ is a $\eta_{1}$-middle point map, that is for all $x,y \in X$,
    $$|2\hat{d}(x,m(x,y))-\hat{d}(x,y)|\leq 2 \eta_{1},$$
$$|2\hat{d}(y,m(x,y))-\hat{d}(x,y)|\leq 2 \eta_{1},$$

    \item for all $x,y,z \in X$,
    $$ \hat{d}(m(x,y),z) \le \max(\hat{d}(x,z),\hat{d}(y,z)) + 2 \eta_{1}, $$

    \item for every $p \in \mathbb{R}_{+}$, there exists $N(p) \in \mathbb{R}_{+}$ such that for all $N \geq N(p) $, with $\hat{d}(x,z) \leq N$, $\hat{d}(y,z) \leq N$ and $\hat{d}(x,y)>N$, we have:
    $$ \hat{d}(m(x,y),z) \leq N-p. $$

    \end{itemize}

\end{proposition}

\begin{proof}
According to Proposition \ref{proposition: faiblement geodesique} and \ref{proposition : midpoint} , for all $x,y \in G$, there exists $m \in [x,y]_{\hat{d}}$ such that:
$$|2\hat{d}(x,m)-\hat{d}(x,y)|\leq 2 \eta,$$
$$|2\hat{d}(y,m)-\hat{d}(x,y)|\leq 2 \eta.$$

We will start to prove the first point, i.e there exists $\eta_{1}\geq 0$ such that for all $x,y,z \in G$,
\begin{equation}\label{equation: premiere condition de faiblement b2'}
\hat{d}(m,z) \le \max(\hat{d}(x,z),\hat{d}(y,z)) + 2 \eta_{1}.    
\end{equation}

We denote by $t$ a projection of $z$ on $[x,y]_{c}$.
\begin{enumerate}
    \item   \begin{figure}[!ht]
   \centering
   \includegraphics[scale=0.3]{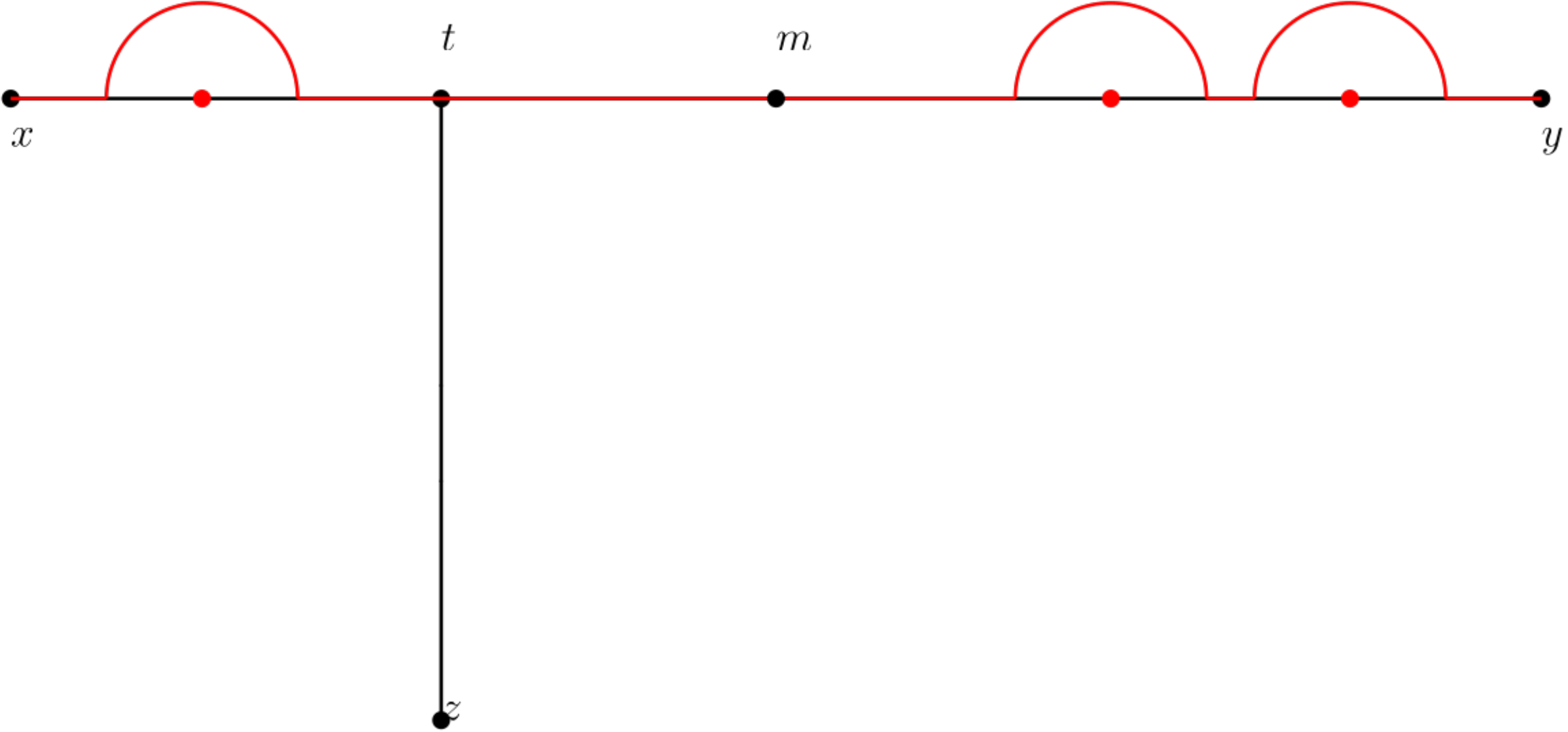}
  \caption{Case where $t$ as finite valence}
  \label{ fig: t as finite valence}
\end{figure}

    To begin, let us assume that $t \in [x,y]_{\hat{d}}$ see Figure \ref{ fig: t as finite valence}. In this case $t \in G$ and we choose the right geodesics such that the associated weak-geodesic of Definition \ref{definition: systeme de geodesic faible} satisfies $[x,y]_{\hat{d}}=[x,t]_{\hat{d}}\cup[t,y]_{\hat{d}}$, according to Proposition \ref{proposition: stabilite geod faible}. By the symmetry of the roles of $x$ and $y$, we can assume that $m \in [t,y]_{\hat{d}}$. In this case, we remark that $\hat{d}(z,y)\geq \hat{d}(z,t)+\hat{d}(t,y)-\nu-C_{triangle} $ according to Lemma \ref{lemme: lemme pour geodesique faible}. Thus:
$$\begin{aligned}
\hat{d}(z,y)&\geq \hat{d}(z,t)+\hat{d}(t,y)-\nu-C_{triangle}\\
           & \geq \hat{d}(z,t)+\hat{d}(t,m)+\hat{d}(m,y)-\nu-C_{triangle}-\eta\\
           & \geq \hat{d}(z,m)-\nu-C_{triangle}-\eta.
\end{aligned}$$

    \item We assume in this case that $t \notin [x,y]_{\hat{d}}$, i.e. $t \in [x,y]_{c}\cap X_{c}^{\infty}$. We set $t_{x}$ the vertex of $[x,t]_{c}$ such that $d_{c}(t_{x},t)=1$, $t_{y}$ the vertex of $[t,y]_{c}$ such that $d_{c}(t_{y},t)=1 $ and $t_{z}$ the vertex of $[t,z]_{c}$ such that $d_{c}(t,t_{z})=1$. 
\begin{itemize}
    \item   \begin{figure}[!ht]
   \centering
   \includegraphics[scale=0.5]{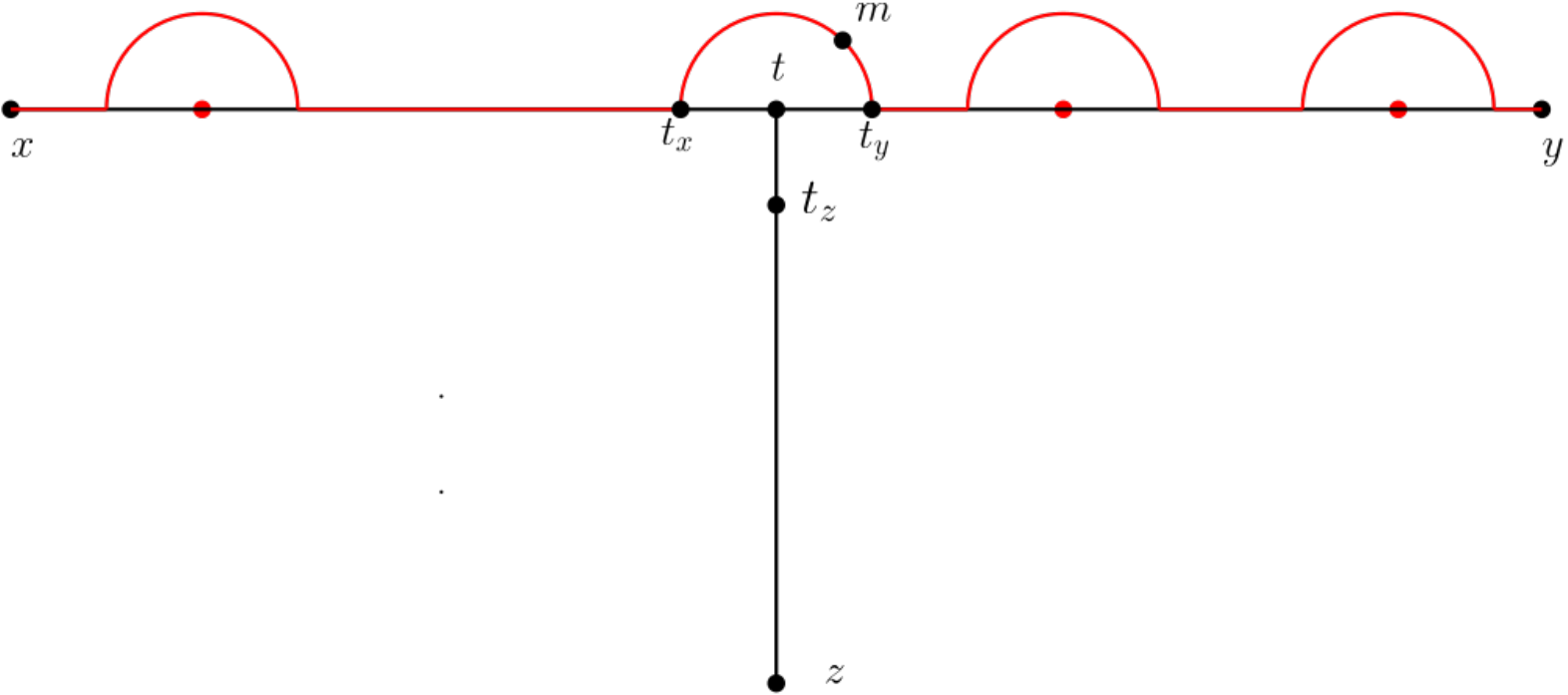}
   \caption{Case where $t$ as infinite valence and $m$ is close to $t$}
  \label{ fig: t as infinite valence m far from t}
\end{figure}

    Let us assume that $d_{c}(m,t)=1$ and $m \notin[x,y]_{c}$ see Figure \ref{ fig: t as infinite valence m far from t}. Equivalently this means that $m \in [t_{x},t_{y}]_{\hat{d}}.$

In this case, as $d_{t}$ is peakless according to Definition \ref{definition: lesbons parabolic}, we have:
$$d_{t}(t_{z},m) \leq \max(d_{t}(t_{z},t_{x}),d_{t}(t_{z},t_{y}))+\eta.$$

Let us assume without loss of generality that: $$ d_{t}(t_{z},m)\leq d_{t}(t_{y},t_{z}) + \eta. $$

Thus:
$$\begin{aligned}
&~~~\hat{d}(z,y)\\
& \geq \hat{d}(z,t_{z})+\hat{d}(t_{z},t_{y})+\hat{d}(t_{y},y)-2\nu-2C_{triangle}\\ &~~~\text{(according to Lemma \ref{lemme: lemme pour geodesique faible} applied two times)}\\
           & \geq \hat{d}(z,t_{z})+t_{0}^{P-1}d_{t}(t_{z},t_{y})+\hat{d}(t_{y},y)-2\nu-2C_{triangle}-\Lambda\\ &~~~\text{(according to Proposition \ref{proposition: pour deux sommets voisins de a un pt parabolique D est presque da})} \\
           & \geq \hat{d}(z,t_{z})+t_{0}^{P-1}d_{t}(m,t_{z})+\hat{d}(t_{y},y)-2\nu-2C_{triangle}-\Lambda-2t_{0}^{P-1}\eta\\
            & \geq \hat{d}(z,t_{z})+\hat{d}(m,t_{z})+\hat{d}(t_{y},y)-2\nu-2C_{triangle}-2\Lambda-2t_{0}^{P-1}\eta \\
            &~~~\text{(according to Proposition \ref{proposition: pour deux sommets voisins de a un pt parabolique D est presque da})}\\
            & \geq \hat{d}(m,z)  -2\nu-2C_{triangle}-2\Lambda-2t_{0}^{P-1}\eta.
\end{aligned}$$

\item \begin{figure}[!ht]
   \centering
   \includegraphics[scale=0.5]{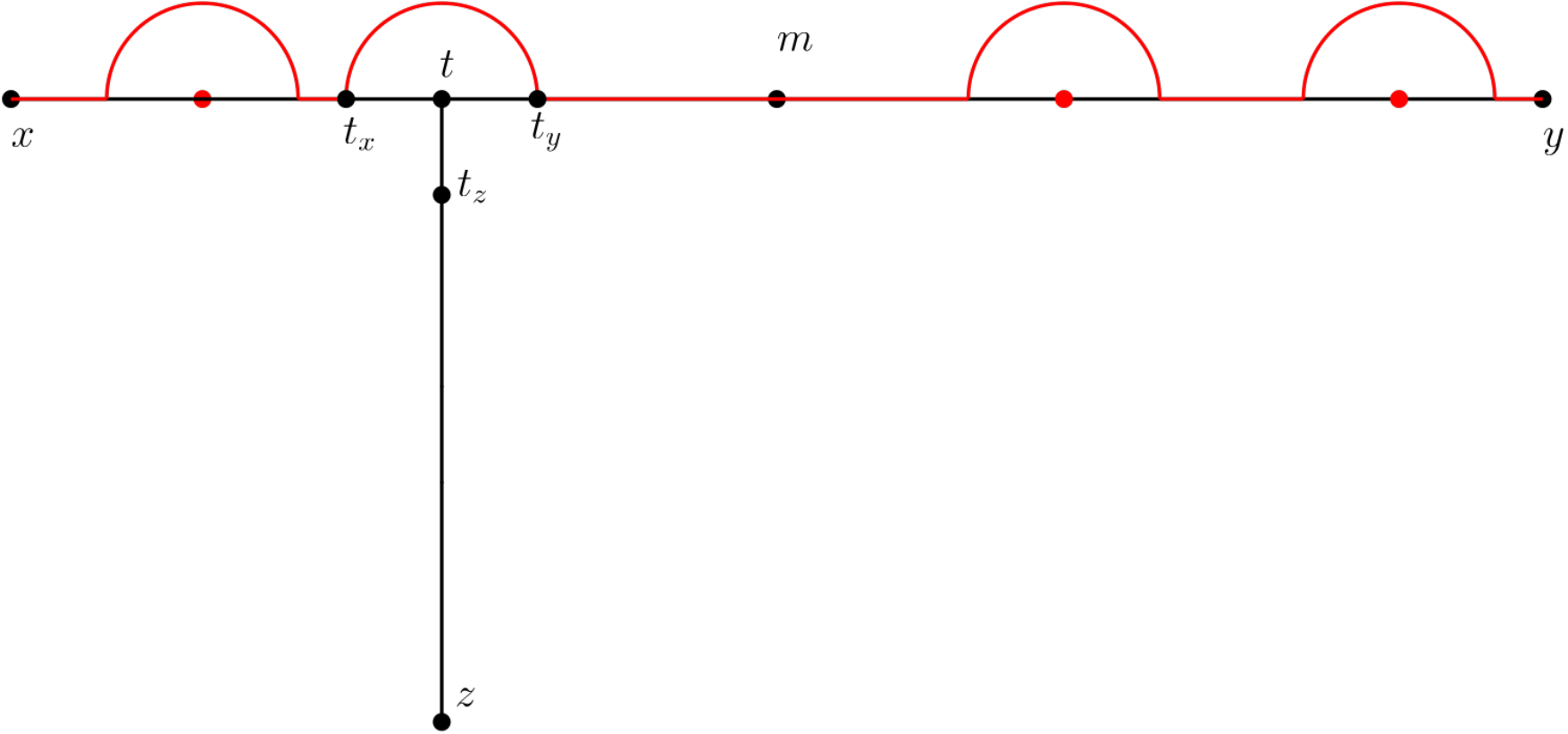}
   \caption{Case where $t$ as infinite valence and $m$ far from $t$}
    \label{fig: t as infinite valence m close to t}
\end{figure}

     Let us assume that $d_{c}(m,t)> 1$ or $m \in [x,y]_{c}$ in this case see Figure \ref{fig: t as infinite valence m close to t}. We can assume by the symmetry of the roles between $x$ and $y$ that $m \in [t_{y},y]_{\hat{d}}$. In this case we apply Lemma \ref{lemme: lemme pour geodesique faible} to $t_{y}$ and we get $\hat{d}(z,y)\geq \hat{d}(z,t_{y})+\hat{d}(t_{y},y)-\nu-C_{triangle} $. Thus:
$$\begin{aligned}
\hat{d}(z,y)&\geq \hat{d}(z,t_{y})+\hat{d}(t_{y},y)-\nu-C_{triangle}\\
           & \geq \hat{d}(z,t_{y})+\hat{d}(t_{y},m)+\hat{d}(m,y)-\nu-C_{triangle}-\eta\\
           & \geq \hat{d}(z,m)-\nu-C_{triangle}-\eta.
\end{aligned}$$
\end{itemize}

To conclude, we have:
$$\hat{d}(z,m)\leq \max(\hat{d}(x,z),\hat{d}(y,z))+ 2\nu+2C_{triangle}+2\Lambda+2t_{0}^{P-1}\eta. $$

Then $\hat{d}$ satisfies Inequality \ref{equation: premiere condition de faiblement b2'} for some $\eta_{1}>0$.
\end{enumerate}

We will now show that for every $p \in \mathbb{R}_{+}$, there exists $N(p) \in \mathbb{R}_{+}$ such that for all $N \geq N(p) $, with $\hat{d}(x,z) \leq N$, $\hat{d}(y,z) \leq N$ and $\hat{d}(x,y)>N$, we have :
    $$ \hat{d}(m,z) \leq N-p. $$

We will apply the same disjunction as the previous point. We still denote by $t$ a projection of $z$ on $[x,y]_{c}$.

\begin{enumerate}
    \item Let us assume in this case that $t \in [x,y]_{\hat{d}}$, i.e $t \in G$. We assume without loss of generalities that $m \in [t,y]_{\hat{d}}$. Thus, according to a previous computation:
$$
\hat{d}(z,y) \geq \hat{d}(z,m)+\hat{d}(m,y)-\nu-C_{triangle}-\eta.$$

By definition of $m$, we have:
$$\hat{d}(y,m)\geq \frac{N}{2}-\eta. $$

Then:
$$\hat{d}(z,m)\leq  \hat{d}(z,y)-\hat{d}(m,y)+\nu+C_{triangle}+\eta \leq N-\frac{N}{2}+ \nu+C_{triangle}+2\eta. $$
Then for every $p \in \mathbb{R}_{+}$, if we take $N$ large enough, we have:
$$\hat{d}(z,m)\leq N-p.$$

\item Let us assume in this case that $t \notin [x,y]_{\hat{d}}$.
We set $t_{x}$ the vertex of $[x,t]_{c}$ such that $d_{c}(t_{x},t)=1$, $t_{y}$ the vertex of $[t,y]_{c}$ such that $d_{c}(t_{y},t)=1 $ and $t_{z}$ the vertex of $[t,z]_{c}$ such that $d_{c}(t,t_{z})=1$.

\begin{itemize}
    \item Let us assume that $d_{c}(m,t)>1$ or $m \in [x,y]_{c}$ in this case. We can assume that $m \in [t_{y},y]_{\hat{d}}$. This case is similar as the previous one. According to a previous computation, we have:
$$\hat{d}(z,y)\geq \hat{d}(z,m)+\hat{d}(m,y)-\nu-C_{triangle}-\eta.$$
We still have:
$$\hat{d}(y,m)\geq \frac{N}{2}-\eta. $$
Then again:
$$\hat{d}(z,m)\leq  N-\frac{N}{2}+ \nu+C_{triangle}+2\eta. $$
Then for every $p \in \mathbb{R}_{+}$, if we take $N$ large enough, we have:
$$\hat{d}(z,m)\leq N-p.$$

    \item Let us assume that $d_{c}(m,t)=1$ and $m \notin[x,y]_{c}$. In this case, we have $m \in [t_{x},t_{y}]_{\hat{d}}$. We rewrite $\sigma:[0,1]\rightarrow [0,\hat{d}(x,y)]$ the weak geodesic $[t_{x},t_{y}]_{d_{t}}$. Let $p>0$.
Let us assume further that $\hat{d}(t_{z},t_{y})\geq \hat{d}(t_{z},t_{x})$,
then according to the fact that $\hat{d}$ satisfies Inequality $\ref{equation: premiere condition de faiblement b2'}$ for some $\eta_{1}>0$, we have $\hat{d}(t_{z},m)\leq \hat{d}(t_{z},t_{y}) +\eta_{1} $.

If we assume that $\hat{d}(y,t_{y})>p+2C_{triangle}+2\nu$+$\eta_{1}$.

We have:
$$\begin{aligned}
\hat{d}(z,m)&\leq \hat{d}(z,t_{z})+\hat{d}(t_{z},m)\\
& \leq \hat{d}(z,t_{z})+\hat{d}(t_{z},t_{y})+\eta_{1} \\
& \leq \hat{d}(z,y)-\hat{d}(y,t_{y})+2 C_{triangle} +2 \nu+\eta_{1}\\
& \leq N-p.
\end{aligned}$$

If we assume that $ \hat{d}(x,t_{x}) > \hat{d}(t_{x},m)+p+C_{triangle}+\nu  $, we have:
$$\begin{aligned}
\hat{d}(z,m)&\leq \hat{d}(z,t_{x})+\hat{d}(t_{x},m)\\
& \leq \hat{d}(z,x)-\hat{d}(x,t_{x})+C_{triangle}+\nu+\hat{d}(t_{x},m)\\
& \leq N-p.
\end{aligned}$$

Let us assume now that $\hat{d}(y,t_{y})\leq p+2C_{triangle}+2\nu$ and $ \hat{d}(x,t_{x}) \leq  \hat{d}(t_{x},m)+p+C_{triangle}+\nu  $. We will use in the case the fact that $d_{t}$ is $\eta_{2}$-strongly convex for some $\eta_{2}>0$.

By definition of $m$, we have $\hat{d}(y,m)\geq \frac{N}{2}-\eta$.

Then:
$$\hat{d}(t_{y},m)\geq \frac{N}{2}-\eta- \hat{d}(y,t_{y}) \geq \frac{N}{2}-\eta -p -2C_{triangle}-2\nu .$$

For the same reason, $\hat{d}(x,m)\geq \frac{N}{2}-\eta$.
Thus:
$$2 \hat{d}(t_{x},m)+p+C_{triangle}+\nu \geq \hat{d}(x,t_{x})+\hat{d}(t_{x},m)\geq \frac{N}{2}-\eta ,$$
then:
$$\hat{d}(t_{x},m) \geq \frac{N}{4}-\frac{C_{triangle}}{2}-\frac{\nu}{2}-\frac{p}{2}-\frac{\eta}{2}.$$

Then for $N$ large enough, there exists $0<\eta< \frac{1}{2}$ such that $m=\sigma(t)$, $t \in [\eta,1-\eta]$.

Moreover:
$$\begin{aligned}
\hat{d}(t_{x},t_{y})&\geq \hat{d}(t_{x},m)+\hat{d}(m,t_{y})-C_{triangle}-\nu\\
& \geq \frac{3}{4}N-\frac{3}{2}(\eta+p)-\frac{5}{2}(C_{triangle}+\nu).
\end{aligned}$$

Therefore, for $N$ large enough, for all $\varepsilon \in (0,\frac{3}{4})$, we have:
$$\hat{d}(t_{x},t_{y})\geq \varepsilon (N+2\nu+2C_{triangle}).$$

In this case:
$$\begin{aligned}
\hat{d}(z,y)
 \geq \hat{d}(z,t_{z})+\hat{d}(t_{z},t_{y})
+\hat{d}(t_{y},y)-2\nu-2C_{triangle}.
\end{aligned}$$

Therefore:
$$ \hat{d}(t_{z},t_{y}) \leq N+2\nu+2C_{triangle} $$
Then:
$$ \hat{d}(t_{x},t_{y}) \geq \varepsilon \hat{d}(t_{z},t_{x}).$$

Then, according to the fact that $d_{t}$ is  $\eta_{2}$-peakless, there exists $\delta \in (0,1)$ such that we have:
$$\hat{d}(t_{z},m)\leq (1-\delta) \hat{d}(t_{z},t_{y})+\eta_{2}.$$

Moreover:
$$\begin{aligned}
2\hat{d}(t_{z},t_{y})\geq& \hat{d}(t_{x},t_{z})+\hat{d}(t_{z},t_{y})\\
 \geq & \hat{d}(t_{x},t_{y})\\
 \geq & \frac{\varepsilon}{2}N.
\end{aligned}$$

Thus, for $N$ large enough, we have:
$$\hat{d}(t_{z},m)\leq \hat{d}(t_{z},t_{y})-p.$$

To conclude, we remark that:
$$\begin{aligned}
\hat{d}(z,m)\leq & \hat{d}(z,t_{z})+\hat{d}(t_{z},m)\\
\leq &  \hat{d}(z,t_{z})+\hat{d}(t_{z},t_{y})+\hat{d}(t_{y},y)-p\\
\leq & N-p.
\end{aligned}$$

This proves finally that $\hat{d}$ satisfies weakly-$B2'$.

\end{itemize}
    
\end{enumerate}

\end{proof}

\subsection{$\hat{d}$ verifies strong-$B1$}\label{subsection: metrique bolic verifie B1}

In this subsection, we prove that $\hat{d}$ verifies strong-$B1$. We need some preliminary lemmas.

This lemma corresponds to Lemma \ref{lemma : comparaison des projections} in the relatively hyperbolic context, as it takes angles into account.

\begin{figure}[!ht]
   \centering
   \includegraphics[scale=0.5]{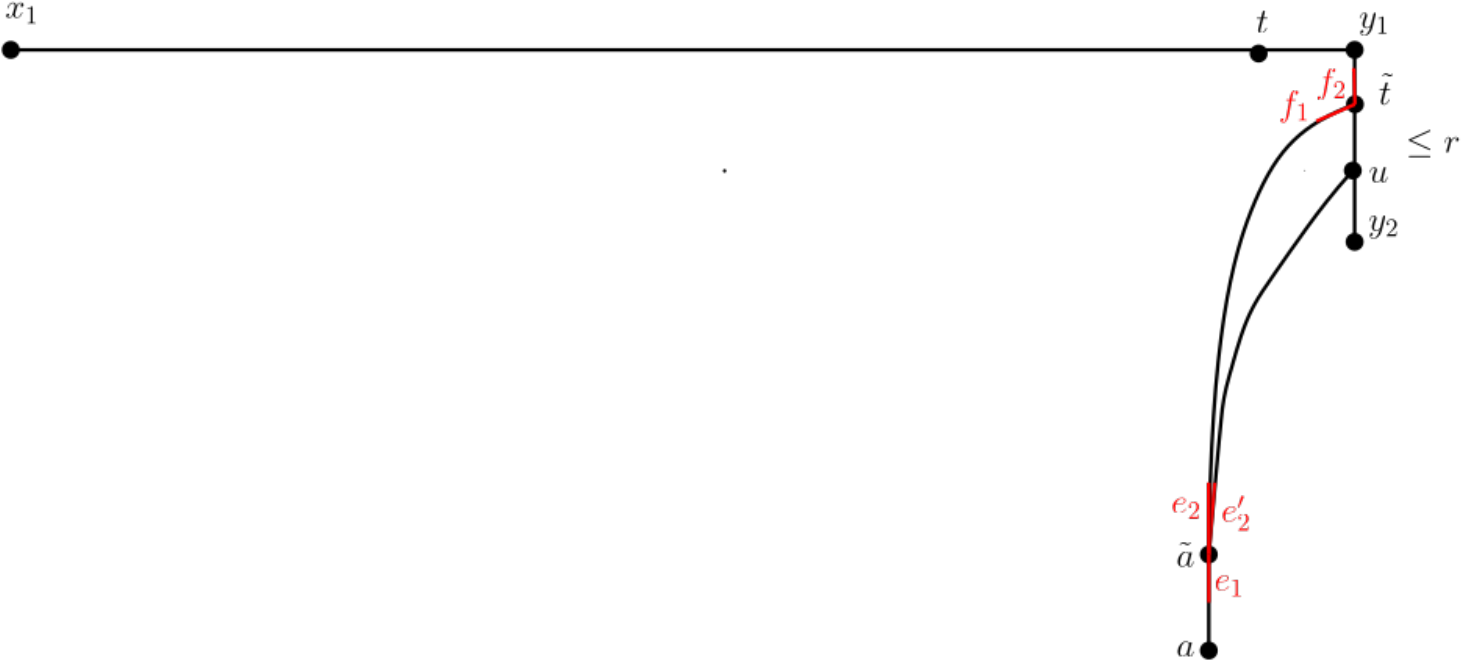}
   \caption{Comparison of projections onto the geodesics 
$[x_{1},y_{1}]_{c}$ and 
$[x_{1},y_{2}]_{c}$ when the angle at $y_{1}$ is large.}
 \label{fig: lemme projections sur deux geod angles}
\end{figure}

\newpage
\begin{lem} \label{lemma: projection sur geodesic proche hyperbolicite relative}
Let $r\geq 0$.
Let $x_{1},x_{2},y_{1},y_{2}\in G$ such that $d'(x_{1},x_{2}),d'(y_{1},y_{2})\leq r$.  Let $a \in X_{c}$ and let us denote:
\begin{itemize}
    \item $t \in [x_{1},y_{1}]_{c}$ a projection of $a$ on $[x_{1},y_{1}]_{c}$,

    \item $u \in [x_{1},y_{2}]_{c}$ a projection of $a$ on $[x_{1},y_{2}]_{c}$,
    
    \item $v \in [x_{2},y_{1}]_{c}$ a projection of $a$ on $[x_{2},y_{1}]_{c}$,
    
    \item $w \in [x_{2},y_{2}]_{c}$ a projection of $a$ on $[x_{2},y_{2}]_{c}$.
\end{itemize}

There exists $ \lambda_{1} \geq0$, which depends on $r$ and $\delta$ only, such that:

\begin{itemize}
     \item $ d'_{(2000\delta)^{2}}(a,u)\geq \frac{1}{\alpha_{2}^{2}}d'_{(2000\delta)^{2}}(a,t)-\measuredangle_{u}(a,t)-\lambda_{1}, $

    \item $ d'_{(2000\delta)^{2}}(a,v)\geq  \frac{1}{\alpha_{2}^{2}}d'_{(2000\delta)^{2}}(a,t)-\measuredangle_{v}(a,t)-\lambda_{1}, $

    \item  $ d'_{(2000\delta)^{2}}(a,w)\geq \frac{1}{\alpha_{2}^{2}}d'_{(2000\delta)^{2}}(a,t)-\measuredangle_{w}(a,t)-\lambda_{1},$
\end{itemize}

where $\alpha_{2}$ is the constant of Proposition \ref{proposition: comparaison de la somme des angles le long de deux géodésiques}.

\end{lem}
 
 \begin{proof}
We consider the geodesics $[x_{1},y_{1}]_{c}$ and $[x_{1},y_{2}]_{c}$. 
According to Lemma \ref{lemma : comparaison des projections}, we have that:
$$d_{c}(a,u) \geq d_{c}(a,t)-(r+\delta) .$$

To conclude we need to control the sum of angles greater than $(2000\delta)^{2}$  between $a$ and $u$ by the sum of angles greater than $(2000\delta)^{2}$ between $a$ and $t$.\\

\begin{enumerate}
\item  If $\measuredangle_{y_{1}}(x_{1},y_{2})\leq \max(r,12 \delta)$, we have $y_{2}\in cone_{\max(r,12\delta)}([x_{1},y_{1}]_{c})$. In this case, we apply the same proof as in Lemma \ref{lemme: comparaison projection avec z presque sur la geodesique cas relativement hyperbolique}, with $x_{1}$ playing the role of $x$, $y_{1}$ the one $y$ and $y_{2}$ the one of $z$. We get a constant $\lambda_{1}$ which depends on $r$ and $\delta$ only such that:
$$d'_{(2000\delta)^{2}}(a,u)\geq \frac{1}{\alpha_{2}^{2}}d'_{(2000\delta)^{2}}(a,t)-\measuredangle_{u}(a,t)-\lambda_{1}, $$

where $\alpha_{2}$ is the constant of Proposition \ref{proposition: comparaison de la somme des angles le long de deux géodésiques}.

\item It remains the case where $\measuredangle_{y_{1}}(x_{1},y_{2}) > \max(r,12 \delta)$, in this case $y_{1}\in [x_{1},y_{2}]_{c}$ according to Proposition \ref{proposition: les propirétés des angles} and $t \notin I_{c}(a,u)$. We consider a triangle $[a,u,t]$ as in Theorem \ref{formenormaledestriangles} with the side $[t,u]_{c}\subset[x_{1},y_{2}]_{c}$. As $u$ is a projection of $a$ on $[x_{1},y_{2}]_{c}$, $\tilde{u}=u$.

\begin{itemize}
    \item In the case where $[a,u,t]$ is a tripod, we have $u \in [a,t]_{c}$, therefore:
$$ \Theta_{>(2000\delta)^{2}}(a,t)= \Theta_{>(2000\delta)^{2}}(a,u) + \measuredangle_{t}(a,u)+\Theta_{>(2000\delta)^{2}}(u,t). $$

For all $w \in [u,t]_{c}$, we have the following.
If $w \in [y_{1},y_{2}]_{c}$, then $\measuredangle_{w}(u,t)\leq r$.
If $w \in [x_{1},y_{1}]_{c}$, we will show by contradiction that $\measuredangle_{w}(u,t)< (2000\delta)^{2}$. Let us assume then that $\measuredangle_{w}(u,t)\geq (2000\delta)^{2}$. Since $w \in [x_{1},y_{1}]$ and $t$ is a projection of $a$ on $[x_{1},y_{1}]_{c}$, we have $\measuredangle_{w}(a,t)\leq 12\delta$ according to Proposition \ref{proposition: les propirétés des angles}. This implies that:
$$\measuredangle_{w}(a,u)\geq \measuredangle_{w}(u,t)-\measuredangle_{w}(t,u)>12\delta, $$
which is in contradiction with the fact that $w \in [x_{1},y_{2}]_{c}$ and that $u$ is a projection of $a$ on $[x_{1},y_{2}]_{c}$. Thus:
$$\Theta_{>(2000\delta)^{2}}(u,t) \leq r^{2}.$$

\item Let assume now that $[a,u,t]$ is not a tripod, then $\tilde{a},u,\tilde{t}$ are all different (see Figure \ref{fig: lemme projections sur deux geod angles}). If $a\neq \tilde{a}$, we denote by $e_{1}$ the edge on $[a,\tilde{a}]_{c}$ such that $\tilde{a} \in e_{1}$ and $e_{2},e'_{2}$ the edges of $[\tilde{a},t]_{c}$, $[\tilde{a},u]_{c}$ respectively such that $\tilde{a} \in e_{2}\cap e_{2}'$. If $t\neq \tilde{t}$, we denote by $f_{1},f_{2}$ the edges of $[a,\tilde{t}]_{c}$, $[t,\tilde{t}]_{c}$ respectively such that $\tilde{t}\in f_{1}\cap f_{2}$. We consider $\measuredangle_{\tilde{a}}(e_{1},e_{2}), \measuredangle_{\tilde{a}}(e_{1},e'_{2}), \measuredangle_{\tilde{t}}(f_{1},f_{2})$ with the convention that the angles are equal to zero in the cases where the edges are not defined. 
We get:

\begin{itemize}
    \item 
$\Theta_{>(2000\delta)^{2}}(a,u)\geq \Theta_{>(2000\delta)^{2}}(a,\tilde{a})+ [\measuredangle_{\tilde{a}}(e_{1},e'_{2})]_{(2000\delta)^{2}}$,

    \item $\Theta_{>(2000\delta)^{2}}(a,t)=\Theta_{>(2000\delta)^{2}}(a,\tilde{a})+ [\measuredangle_{\tilde{a}}(e_{1},e_{2})]_{(2000\delta)^{2}} + [\measuredangle_{\tilde{t}}(f_{1},f_{2})]_{(2000\delta)^{2}}+\Theta_{>(2000\delta)^{2}}(\tilde{t},t) $.

\end{itemize}

\end{itemize}

In the same way as the previous case, we have $\Theta_{>(2000\delta)^{2}}(\tilde{t},t)\leq r^{2}$. According to Proposition \ref{proposition: angle en a tilde}, we have:
$$|\measuredangle_{\tilde{a}}(e_{1},e'_{2})-\measuredangle_{\tilde{a}}(e_{1},e_{2})|\leq 12\delta.$$
According to Proposition \ref{proposition: angle en a tilde}
again, $\measuredangle_{\tilde{t}}(f_{1},f_{2})$ is $12\delta$ closed to $\measuredangle_{\tilde{t}}(t,u)$ which is smaller than $r$.

We conclude with an application of Proposition \ref{proposition: comparaison de la somme des angles le long de deux géodésiques}.

\end{enumerate}
\end{proof}

The following lemma allows us to handle the case where the angle between the two projections is large in Lemma \ref{lemma: projection sur geodesic proche hyperbolicite relative}.

\begin{lem}\label{lem: projection sur geodesic proche hyperbolicite relative avec grand angle}

Let $r\geq 0$.
Let $x_{1},x_{2},y_{1},y_{2}\in G$ such that $d'(x_{1},x_{2}),d'(y_{1},y_{2})\leq r$.  Let $a \in X_{c}$, we set:
\begin{itemize}
    \item let $t \in [x_{1},y_{1}]_{c}$ a projection of $a$ on $[x_{1},y_{1}]_{c}$,

    \item $u \in [x_{1},y_{2}]_{c}$ a projection of $a$ on $[x_{1},y_{2}]_{c}$,
    
    \item $v \in [x_{2},y_{1}]_{c}$ a projection of $a$ on $[x_{2},y_{1}]_{c}$,
    
    \item $w \in [x_{2},y_{2}]_{c}$ a projection of $a$ on $[x_{2},y_{2}]_{c}$.
\end{itemize}

If one of the following conditions holds:

\begin{itemize}
    \item $\measuredangle_{u}(a,t)\geq (2000\delta)^{2}+12\delta+\max(r,12\delta)$,
    
    \item $\measuredangle_{v}(a,t)\geq (2000\delta)^{2}+12\delta+r+\max(r,12\delta)$,
    
    \item $\measuredangle_{w}(a,t)\geq (2000\delta)^{2}+12\delta+r+\max(r,12\delta)$,
\end{itemize}

    we have:
    $$ |\theta(d_{a}(x_{1},y_{1}))+\theta(d_{a}(x_{2},y_{2}))- \theta(d_{a}(x_{1},y_{2}))-\theta(d_{a}(x_{2},y_{1}))|=0. $$
\end{lem}

\begin{proof}
Let us assume that $\measuredangle_{u}(a,t)\geq (2000\delta)^{2}+r$, the other cases will follow by symmetry between $x_{1},x_{2},y_{1},y_{2}$. We will prove using Proposition \ref{proposition : grand angle comme des checkpoints} that:
$$\mu_{x_{1}}(a)=\mu_{y_{1}}(a)=\mu_{y_{2}}(a).$$

Let us remark first by contradiction that $\measuredangle_{u}(t,x_{1})\leq 12\delta$. Indeed, if we have $\measuredangle_{u}(t,x_{1})> 12\delta$, this implies that $u \in [x_{1},t]_{c}\subset [x_{1},y_{2}]_{c}$, in contradiction with the fact that $\measuredangle_{u}(a,t)\geq (2000\delta)^{2}+r\geq 12\delta$ and the fact that $t$ is a projection of $a$ on $[x_{1},y_{1}]_{c}$. 

Thus we deduce that:
$$\measuredangle_{u}(a,x_{1})\geq \measuredangle_{u}(a,t)-\measuredangle_{u}(t,x_{1})\geq (2000\delta)^{2}+\max(r,12\delta) \geq(2000\delta)^{2},$$

which implies that $\mu_{x_{1}}(a)=\mu_{u}(a)$ according to Proposition \ref{proposition : grand angle comme des checkpoints}.

In the same way, one can prove that $\measuredangle_{u}(a,y_{2})\geq (2000\delta)^{2}+r$ and that $\mu_{y_{2}}(a)=\mu_{u}(a)$.

From this we deduce that:
$$\measuredangle_{u}(a,y_{1})\geq \measuredangle_{u}(a,y_{2})-\measuredangle_{u}(y_{1},y_{2})\geq (2000\delta)^{2}, $$
since $d'(y_{1},y_{2})\leq r$. Therefore, $\mu_{y_{1}}(a)=\mu_{u}(a)$.

We finally have $\mu_{x_{1}}(a)=\mu_{y_{1}}(a)=\mu_{y_{2}}(a)$ and this implies that:
$$ |\theta(d_{a}(x_{1},y_{1}))+\theta(d_{a}(x_{2},y_{2}))- \theta(d_{a}(x_{1},y_{2}))-\theta(d_{a}(x_{2},y_{1}))|=0. $$

\end{proof}

 \begin{lem}\label{lemme: projection plus angle sur un point qcq pour petite geodesique}
Let $r \geq 0$, $y_{1},y_{2} \in G$ such that $d'(y_{1},y_{2})\leq r$. Let $a \in X_{c}$ and $t \in X_{c}$ a projection of $a$ on $[y_{1},y_{2}]_{c}$, then there exists $\lambda_{2} \geq 0$ which depends only on $r$ and $\delta$ such that:
$$d'_{(2000\delta)^{2}}(a,t)\geq \frac{1}{\alpha_{2}^{2}}d'_{(2000\delta)^{2}}(a,y_{1})-\measuredangle_{t}(a,y_{1})-\lambda_{2},$$
with $\alpha_{2}$ the constant of Proposition \ref{proposition: comparaison de la somme des angles le long de deux géodésiques}.

 \end{lem}

\begin{proof}
To begin we remark that:
$$d_{c}(a,t)\geq d_{c}(a,y_{1})-r. $$
We need now to control the difference between large angles from $a$ to $t$ and large angles from $a$ to $y_{1}$.
In the case where $t \in [a,y_{1}]_{c}$, we have:
$$\Theta_{>(2000\delta)^{2}}(a,y_{1})=\Theta_{>(2000\delta)^{2}}(a,t)+[\measuredangle_{t}(a,y_{1})]_{>(2000\delta)^{2}}+\Theta_{>(2000\delta)^{2}}(t,y_{1}).$$

The fact that $d'(y_{1},y_{2})\leq r$ implies that $\Theta_{>(2000\delta)^{2}}(t,y_{1})\leq r^{2}$ then:
$$\Theta_{>(2000\delta)^{2}}(a,t)\geq \Theta_{>(2000\delta)^{2}}(a,y_{1})-\measuredangle_{t}(a,y_{1})-r^{2}.$$

Proposition \ref{proposition: comparaison de la somme des angles le long de deux géodésiques} gives the desired inequality for some $\lambda_{2}$ which depends on $\delta$ and $r$ only.\\

In the case where $t \notin [y_{1},a]_{c}$, we consider a triangle $[y_{1},t,a]_{c}$ as in Theorem \ref{formenormaledestriangles} where the side $[t,y_{1}]_{c}$ is a subset of $[y_{1},y_{2}]_{c}$.
As $t$ is a projection of $a$ on $[y_{1},y_{2}]_{c}$, we have $t =\tilde{t}$ and $t, \tilde{y_{1}}$ and $a$ are all different.
If $a\neq \tilde{a}$, we denote by $e_{1}$ the edge on $[a,\tilde{a}]_{c}$ such that $\tilde{a} \in e_{1}$ and $e_{2},e'_{2}$ the edges of $[\tilde{a},y_{1}]_{c}$, $[\tilde{a},t]_{c}$ respectively such that $\tilde{a} \in e_{2}\cap e_{2}'$. If $y_{1}\neq \tilde{y_{1}}$, we denote by $f_{1},f_{2}$ the edges of $[a,\tilde{y_{1}}]_{c}$, $[y_{1},\tilde{y_{1}}]_{c}$ respectively such that $\tilde{y_{1}}\in f_{1}\cap f_{2}$. We consider $\measuredangle_{\tilde{a}}(e_{1},e_{2}), \measuredangle_{\tilde{a}}(e_{1},e'_{2}), \measuredangle_{\tilde{y_{1}}}(f_{1},f_{2})$ with the convention that the angles are equal to zero in the cases where the edges are not defined.  If $a\neq \tilde{a}$, we denote by $e_{1}$ the edge on $[a,\tilde{a}]_{c}$ such that $\tilde{a} \in e_{1}$ and $e_{2},e'_{2}$ the edges of $[\tilde{a},y_{1}]_{c}$, $[\tilde{a},t]_{c}$ respectively such that $\tilde{a} \in e_{2}\cap e_{2}'$. If $y_{1}\neq \tilde{y_{1}}$, we denote by $f_{1},f_{2}$ the edges of $[a,\tilde{y_{1}}]_{c}$, $[y_{1},\tilde{y_{1}}]_{c}$ respectively such that $\tilde{y_{1}}\in f_{1}\cap f_{2}$. We consider $\measuredangle_{\tilde{a}}(e_{1},e_{2}), \measuredangle_{\tilde{a}}(e_{1},e'_{2}), \measuredangle_{\tilde{y_{1}}}(f_{1},f_{2})$ with the convention that the angles are equal to zero in the cases where the edges are not defined.
We get:
\begin{itemize}
    \item $\Theta_{>(2000\delta)^{2}}(a,t)= \Theta_{>(2000\delta)^{2}}(a,\tilde{a})+ [\measuredangle_{\tilde{a}}(e_{1},e'_{2})]_{(2000\delta)^{2}}$,

    \item $\Theta_{>(2000\delta)^{2}}(a,y_{1})=\Theta_{>(2000\delta)^{2}}(a,\tilde{a})+ [\measuredangle_{\tilde{a}}(e_{1},e_{2})]_{(2000\delta)^{2}} + [\measuredangle_{\tilde{y_{1}}}(f_{1},f_{2})]_{(2000\delta)^{2}}+\Theta_{>(2000\delta)^{2}}(\tilde{y_{1}},y_{1}) $.
\end{itemize}

As in the previous case, we have $\Theta_{>(2000\delta)^{2}}(\tilde{y_{1}},y_{1})\leq r^{2}$. Moreover according to Proposition \ref{proposition: angle en a tilde}, we have:
$$ |\measuredangle_{\tilde{a}}(e_{1},e'_{2})-\measuredangle_{\tilde{a}}(e_{1},e_{2})|\leq 12\delta.$$
Again according to Proposition \ref{proposition: angle en a tilde}, the angle $\measuredangle_{\tilde{y_{1}}}(f_{1},f_{2})$ is $12\delta$ closed to $\measuredangle_{\tilde{y_{1}}}(y_{1},t)$ which is smaller than $r$ since $d'(y_{1},y_{2})\leq r$. Thus:
$$ \Theta_{>(2000\delta)^{2}}(a,t)\geq \Theta_{>(2000\delta)^{2}}(a,y_{1})-r^{2}-r-12\delta.  $$

We conclude again with an application of Proposition \ref{proposition: comparaison de la somme des angles le long de deux géodésiques}.

\end{proof}

This Lemma allows us to control the angle $\measuredangle_{t}(a,y_{1})$ of Lemma \ref{lemme: projection plus angle sur un point qcq pour petite geodesique}.

\begin{lem}\label{lem: projection et grand angle pour petite geodesique}
Let $r \geq 0$, $y_{1},y_{2} \in G$ such that $d'(y_{1},y_{2})\leq r$. Let $a \in X_{c}$ and $t \in X_{c}$ a projection of $a$ on $[y_{1},y_{2}]_{c}$, if $ \measuredangle_{t}(a,y_{1})\geq (2000\delta)^{2}+r$, then:
$$d_{a}(y_{1},y_{2})=0.$$

\end{lem}

\begin{proof}

We will prove that if $ \measuredangle_{t}(a,y_{1})\geq (2000\delta)^{2}+r$, we have:
$$\mu_{y_{1}}(a)=\mu_{y_{2}}(a).$$

According to Proposition \ref{proposition : grand angle comme des checkpoints}, we know that:
$$\mu_{y_{1}}(a)=\mu_{t}(a).$$

Moreover:
$$\measuredangle_{t}(a,y_{2})\geq \measuredangle_{t}(a,y_{1})-\measuredangle_{t}(y_{1},y_{2})\geq (2000\delta)^{2},$$

then, we new use of Proposition \ref{proposition : grand angle comme des checkpoints} gives the fact that:
$$\mu_{y_{1}}(a)=\mu_{t}(a),$$

and this implies that $d_{a}(y_{1},y_{2})=0$.

\end{proof}

We now have almost all the ingredients to prove strong-$B1$ for $\hat{d}$. The following proposition simply shows that we can apply strong $B1$ for $d_{a}$ uniformly on $a \in X_{c}^{\infty}$.

\begin{proposition}\label{proposition: fortement B1 dans les paraboliques}
For all $\eta,r>0$, there exists $R=R(\eta,r)\geq 0$ such that for all $a \in X_{c}^{\infty}$, for all $x_{1}, x_{2}, y_{1}, y_{2} \in X$, with $d_{a}(x_{1},y_{1}),d_{a}(x_{1},y_{2}),d_{a}(x_{2},y_{1}),d_{a}(x_{2},y_{2}) \geq R$ and $d_{a}(x_{1},x_{2}),d_{a}(y_{1},y_{2}) \leq r$, we have:

    $$ |d_{a}(x_{1},y_{1})+d_{a}(x_{2},y_{2})-d_{a}(x_{1},y_{2})-d_{a}(x_{2},y_{1})| \leq \eta .$$

\end{proposition}

\begin{proof}

Recall that the group $G$ is hyperbolic relative to a family of admissible subgroups $P_{1},…,P_{n}$. Therefore for each $P_{i}$, for each $C>0$, $ \text{Proba}_{C}(P_{i})$ admits a strongly bolic metric. Thus these metrics must satisfy strong-$B1$. For all $\eta,r>0$, there exists $R_{i}$ such that for all $a \in X_{c}^{\infty}$ associated with $P_{i}$, for all $x_{1}, x_{2}, y_{1}, y_{2} \in X$, with $d_{a}(x_{1},y_{1}),d_{a}(x_{1},y_{2}),d_{a}(x_{2},y_{1}),d_{a}(x_{2},y_{2}) \geq R_{i}$ and $d_{a}(x_{1},x_{1}),d_{a}(x_{2},x_{2}) \leq r$, we have:
$$ |d_{a}(x_{1},y_{1})+d_{a}(x_{2},y_{2})-d_{a}(x_{1},y_{2})-d_{a}(x_{2},y_{1})| \leq \eta .$$
    
Taking $R$ as the maximum of the $R_{i}$, for all $a \in X_{\infty}$, for all $x_{1}, x_{2}, y_{1}, y_{2} \in X$, with $d_{a}(x_{1},y_{1}),d_{a}(x_{1},y_{2}),\\d_{a}(x_{2},y_{1}),d_{a}(x_{2},y_{2}) \geq R$ and $d_{a}(x_{1},x_{1}),d_{a}(x_{2},x_{2}) \leq r$, we have:
$$ |d_{a}(x_{1},y_{1})+d_{a}(x_{2},y_{2})-d_{a}(x_{1},y_{2})-d_{a}(x_{2},y_{1})| \leq \eta .$$
\end{proof}

Here, we will prove the most important result for $\hat{d}$, namely strong-$B1$. According to Proposition \ref{Proposition : formule de la distance} the distance formula, there are two cases to consider when the distance $d(x_{1},y_{1})$ is large: the case where there is a large angle between $x_{1}$ and $y_{1}$ in the coned-off graph, or the case where the distance between $x_{1}$ and $y_{1}$ in the coned-off space is large

\begin{theo}\label{theo: strong b1 relativement hyperbolic}
The metric space $(G,\hat{d})$ satisfies strongly-$B1$.
\end{theo}

\begin{proof}
Let us recall that for all $x,y \in G$, $d'(x,y)=d_{c}(x,y)+\Theta'(x,y)$, where $\Theta'(x,y)$ denote the minimal sum of angles at infinite valence vertices on a geodesic between $[x,y]_{c}$. According to Proposition \ref{proposition: quasi-isometry} and to Proposition \ref{Proposition : formule de la distance} the fact that $(X,\hat{d})$ satisfies strong-$B1$ is equivalent to the fact that for all $r>0$ and $\eta >0$, there exists $R=R(\eta,r)\geq 0$ such that for all $x_{1}$, $x_{2}$, $y_{1}$, $y_{2} \in G$, with $d'(x_{1},y_{1}),d'(x_{1},y_{2}),d'(x_{2},y_{1}),d'(x_{2},y_{2}) \geq R$ and $d'(x_{1},x_{2}),d'(y_{1},y_{2}) \leq r$, we have:
$$  |\hat{d}(x_{1},y_{1})+\hat{d}(x_{2},y_{2})-\hat{d}(x_{1},y_{2})-\hat{d}(x_{2},y_{1})| \leq \eta .$$ 

We will choose $R \in \mathbb{N}$ even and large enough later. Let us assume that:
$$ d'(x_{1},y_{1})\geq R$$
This implies that:
$$d_{c}(x_{1},y_{1})+\Theta(x_{1},y_{1})\geq R.$$
There are two cases to consider.

\begin{enumerate}

\item Let us assume in the first one that:
$$d_{c}(x_{1},y_{1})\leq \sqrt{R}.$$

Thus, if $R\geq 4$, we get:
$$\sqrt{R}\max_{a \in X_{c}^{\infty}}(\measuredangle_{a}(x_{1},y_{1})) \geq \Theta(x_{1},y_{1})\geq \frac{R}{2}.$$

Therefore there exists $a_{0} \in X_{c}^{\infty}$ such that:
$$ \measuredangle_{a_{0}}(x_{1},y_{1}) \geq \frac{\sqrt{R}}{2}.$$

We will show that for $R$ large enough, for all $a \in X_{c}$, $a \neq a_{0}$, we have:

$$|\theta(d_{a}(x_{1},y_{1}))+\theta(d_{a}(x_{2},y_{2})- \theta(d_{a}(x_{1},y_{2}))-\theta(d_{a}(x_{2},y_{1}))|=0.$$

To do so we remark that for all $a \in X_{c}$, $a \neq a_{0}$, we have:
$$ \measuredangle_{a_{0}}(x_{1},a)+\measuredangle_{a_{0}}(a,y_{1})\geq \measuredangle_{a_{0}}(x_{1},y_{1}) \geq \frac{\sqrt{R}}{2 }.$$

Let us assume without loss of generality that:
$$\measuredangle_{a_{0}}(x_{1},a)\geq\frac{\sqrt{R}}{4 }.  $$
Thus we get:
$$ \measuredangle_{a_{0}}(x_{2},a)\geq \measuredangle_{a_{0}}(x_{1},a)-\measuredangle_{a_{0}}(x_{1},x_{2})\geq \frac{\sqrt{R}}{4}-r. $$

Therefore, if we take $R$ large enough such that $\frac{\sqrt{R}}{4}-r\geq (2000\delta)^{2}$, we have:
$$ \measuredangle_{a_{0}}(x_{1},a), \measuredangle_{a_{0}}(x_{2},a) \geq (2000\delta)^{2}.$$

Then according to Proposition \ref{proposition : grand angle comme des checkpoints}, this implies that:
$$\mu_{x_{1}}(a)=\mu_{x_{2}}(a)=\mu_{a_{0}}(a).$$
Thus:
$$\begin{aligned}
 &|\theta(d_{a}(x_{1},y_{1}))+\theta(d_{a}(x_{2},y_{2}))- \theta(d_{a}(x_{1},y_{2}))-\theta(d_{a}(x_{2},y_{1}))|\\
 =&|\theta(d_{b}(\mu_{x_{1}}(a),\mu_{y_{1}}(a)))+\theta(d_{b}(\mu_{x_{1}}(a),\mu_{y_{2}}(a)))- \theta(d_{b}(\mu_{x_{1}}(a),\mu_{y_{1}}(a)))-\theta(d_{a}(\mu_{x_{1}}(a),\mu_{y_{1}}(a)))|\\
 =&0.
\end{aligned}$$

Thus, we get:
$$\begin{aligned}
&|\hat{d}(x_{1},y_{1})+\hat{d}(x_{2},y_{2})-\hat{d}(x_{1},y_{2})-\hat{d}(x_{2},y_{1})| \\
=& |\theta(d_{a_{0}}(x_{1},y_{1}))+\theta(d_{a_{0}}(x_{2},y_{2}))- \theta(d_{a_{0}}(x_{1},y_{2}))-\theta(d_{a_{0}}(x_{2},y_{1}))|.
\end{aligned}$$

We have:
\begin{itemize}
    \item $ \measuredangle_{a_{0}}(x_{1},y_{1}) \geq \frac{\sqrt{R}}{2},$
    \item $ \measuredangle_{a_{0}}(x_{2},y_{1}) \geq \frac{\sqrt{R}}{2}-r,$

    \item $ \measuredangle_{a_{0}}(x_{1},y_{2}) \geq \frac{\sqrt{R}}{2}-r,$
    \item $ \measuredangle_{a_{0}}(x_{2},y_{2}) \geq \frac{\sqrt{R}}{2}-2r.$
\end{itemize}

Thus, according to Proposition \ref{proposition: l'angle est borné par la distance bolic}, if we take $R$ large enough, we have:
$$\begin{aligned} 
&|\hat{d}(x_{1},y_{1})+\hat{d}(x_{2},y_{2})-\hat{d}(x_{1},y_{2})-\hat{d}(x_{2},y_{1})|\\
=&|\theta(d_{a_{0}}(x_{1},y_{1}))+\theta(d_{a_{0}}(x_{2},y_{2}))- \theta(d_{a_{0}}(x_{1},y_{2}))-\theta(d_{a_{0}}(x_{2},y_{1}))|\\
=&t_{0}^{P-1}|d_{a_{0}}(x_{1},y_{1})+d_{a_{0}}(x_{2},y_{2})- d_{a_{0}}(x_{1},y_{2})-d_{a_{0}}(x_{2},y_{1})|.
\end{aligned}$$

We conclude with an application of Proposition \ref{proposition: fortement B1 dans les paraboliques}, we can find $R$ independent of $a_{0}$ such that:
$$|d_{a_{0}}(x_{1},y_{1})+d_{a_{0}}(x_{2},y_{2})- d_{a_{0}}(x_{1},y_{2})-d_{a_{0}}(x_{2},y_{1})| \leq \eta.$$

\item  Let us assume in the second case that $d_{c}(x_{1},y_{1})\geq \sqrt{R}.$ Again, we will choose $R$ large enough later.

Let us denote by $\mathcal{B}$ the set of $a \in X_{c}$ such that:

\begin{itemize}
    \item $d_{a}(x_{1},y_{1})\geq t_{0}$,

    \item or $d_{a}(x_{1},y_{2})\geq t_{0}$,

    \item or $d_{a}(x_{2},y_{1})\geq t_{0}$,

    \item or $ d_{a}(x_{2},y_{2})\geq t_{0}.$

\end{itemize}

We have the following:
$$\begin{aligned}
 &|\hat{d}(x_{1},y_{1})+\hat{d}(x_{2},y_{2})-\hat{d}(x_{1},y_{2})-\hat{d}(x_{2},y_{1})|\\
 & \leq \sum_{a \in \mathcal{B}} | \theta(d_{a}(x_{1},y_{1})+\theta(d_{a}(x_{2},y_{2}))- \theta(d_{a}(x_{1},y_{2}))-\theta(d_{a}(x_{2},y_{1}))|  \\
 &+ \sum_{a \in X_{c}} | d_{a}(x_{1},y_{1})^{P}+d_{a}(x_{2},y_{2})^{P}-d_{a}(x_{1},y_{2})^{P}-d_{a}(x_{2},y_{1})^{P}|. 
\end{aligned}$$

\begin{enumerate}
    \item 

We will start by showing that for $R$ large enough, we have:
$$\sum_{a \in \mathcal{B}} | \theta(d_{a}(x_{1},y_{1})+\theta(d_{a}(x_{2},y_{2}))- \theta(d_{a}(x_{1},y_{2}))-\theta(d_{a}(x_{2},y_{1}))| \leq 4 \eta.$$
We fix four geodesics $[x_{1},y_{1}]_{c},[x_{2},y_{2}]_{c},[x_{1},y_{2}]_{c},[x_{2},y_{1}]_{c}$. According to Lemma \ref{lemme: theta est lidentite alors a est dans toute geodesique entre x et y}, we know that:
$$ \mathcal{B}\subset ([x_{1},y_{1}]_{c}\cup[x_{2},y_{2}]_{c} \cup[x_{1},y_{2}]_{c}\cup[x_{2},y_{1}]_{c})\cap X_{c}^{\infty}.$$

We will show that for $R$ large enough, we have:
$$\sum_{a \in \mathcal{B}\cap[x_{1},y_{1}]_{c}} | \theta(d_{a}(x_{1},y_{1})+\theta(d_{a}(x_{2},y_{2}))- \theta(d_{a}(x_{1},y_{2}))-\theta(d_{a}(x_{2},y_{1}))| \leq \eta.$$
This will implies the same result for $\mathcal{B}\cap[x_{2},y_{1}]_{c},\mathcal{B}\cap[x_{1},y_{2}]_{c},\mathcal{B}\cap[x_{2},y_{2}]_{c}$ by symmetry of the roles of $x_{1},y_{1},x_{2}$ and $y_{2}$. 

 We denote by $m$ a vertex in $[x_{1},y_{1}]_{c}$ such that $d_{c}(x_{1},m) \geq \frac{1}{2}d_{c}(x_{1},y_{1})-1$ and $d_{c}(m,y_{1}) \geq \frac{1}{2}d_{c}(x_{1},y_{1})-1$.
 We remark the following:
 $$\begin{aligned}&\sum_{a \in \mathcal{B}\cap[x_{1},y_{1}]_{c}} | \theta(d_{a}(x_{1},y_{1})+\theta(d_{a}(x_{2},y_{2}))- \theta(d_{a}(x_{1},y_{2}))-\theta(d_{a}(x_{2},y_{1}))| \\
&\leq \sum_{a \in \mathcal{B}\cap[x_{1},y_{1}]_{c},a \in [x_{1},m]_{c}} | \theta(d_{a}(x_{1},y_{1})+\theta(d_{a}(x_{2},y_{2}))- \theta(d_{a}(x_{1},y_{2}))-\theta(d_{a}(x_{2},y_{1}))|\\
&+  \sum_{a \in \mathcal{B}\cap[x_{1},y_{1}]_{c},a \in [m,y_{1}]_{c}} | \theta(d_{a}(x_{1},y_{1})+\theta(d_{a}(x_{2},y_{2}))- \theta(d_{a}(x_{1},y_{2}))-\theta(d_{a}(x_{2},y_{1}))|.
 \end{aligned}$$

 Let us bound from above the first sum, the bound for the second sum will follow by symmetry.

 Let $a \in X_{c}$ such that $a \in [x_{1},m]_{c}$, we have:
$$\begin{aligned}
d_{c}(a,[y_{1},y_{2}]_{c})&\geq d_{c}(a,y_{1})-r~(\text{since } d_{c}(y_{1},y_{2})\leq r)\\
& \geq \frac{1}{2}d_{c}(x_{1},y_{1})-1-r~(\text{since } a \in [x_{1},m]_{c}).
\end{aligned}$$
Thus if $R$ is large enough, we could apply Theorem \ref{theo: theo stylé sur les barycentres et les masques} to $a$ and $[y_{1},y_{2}]_{c}$ and we have:
$$d_{a}(y_{1},y_{2})\leq K_{1} \kappa^{d_{c}(a,[y_{1},y_{2}]_{c})} \leq K_{1} \kappa^{\frac{1}{2}d_{c}(x_{1},y_{1})-1-r},$$

where $K_{1}$ is the constant of Theorem \ref{theo: theo stylé sur les barycentres et les masques}.

Thereby, we get:
$$\begin{aligned}
&\sum_{a \in \mathcal{B}\cap[x_{1},y_{1}]_{c},a \in [x_{1},m]_{c}} | \theta(d_{a}(x_{1},y_{1})+\theta(d_{a}(x_{2},y_{2}))- \theta(d_{a}(x_{1},y_{2}))-\theta(d_{a}(x_{2},y_{1}))| \\
&\leq \sum_{a \in \mathcal{B}\cap[x_{1},y_{1}]_{c},a \in [x_{1},m]_{c}} 2Pt_{0}^{P-1}d_{a}(y_{1},y_{2})~(\text{according to Lemma \ref{lem: theta est lipschtizienne}})\\
&\leq 2Pt_{0}^{P-1} d_{c}(x_{1},y_{1}) \kappa^{\frac{1}{2}d_{c}(x_{1},y_{1})-1-r}\\
& \leq 2Pt_{0}^{P-1}R \kappa^{\frac{1}{2} R -1 -r},
\end{aligned}$$
for $R$ large enough.

The last term converges to $0$ as $R$ tends to positive infinity, thus for $R$ large enough, it is smaller than $\eta$.

To conclude, we proved that:
$$\sum_{a \in \mathcal{B}} | \theta(d_{a}(x_{1},y_{1})+\theta(d_{a}(x_{2},y_{2}))- \theta(d_{a}(x_{1},y_{2}))-\theta(d_{a}(x_{2},y_{1}))| \leq 4 \eta.  $$

\item Now, we will prove that for $R$ large enough, we have:
$$\sum_{a \in X_{c} \backslash \mathcal{B}} | d_{a}(x_{1},y_{1})^{P}+d_{a}(x_{2},y_{2})^{P}-d_{a}(x_{1},y_{2})^{P}-d_{a}(x_{2},y_{1})^{P}|\leq 2 \eta.$$
We will choose $R\in \mathbb{N}$ even and large enough later.
We denote by $m$ a vertex in $[x_{1},y_{1}]_{c}$ such that $d_{c}(x_{1},m) \geq \frac{R}{2}$ and $d_{c}(m,y_{1}) \geq \frac{R}{2}$.  For all $a \in X$, $\pi(a)$ denotes the projection of $a$ onto $[x_{1},y_{1}]_{c}$ that is closest to $x_{1}$. More precisely, this means that $ \pi(a) \in [x_{1},y_{1}]_{c}$, $d_{c}(a,[x_{1},y_{1}]_{c})=d_{c}(a,\pi(a))$, for all $ u \in [x_{1},y_{1}]_{c}$, such that $d_{c}(a,[x_{1},y_{1}]_{c})=d_{c}(a,u)$, we have: $d_{c}(x_{1},u) > d_{c}(x_{1}, \pi(a))$. The fact that $\pi(a)$ is the projection of $a$ closest to $x_{1}$ is used only to define the projection of $a$ in $[x_{1},y_{1}]_{c}$ without ambiguity.

We have the following:
$$\begin{aligned}
 &\sum_{a \in X_{c} \backslash \mathcal{B}} | d_{a}(x_{1},y_{1})^{P}+d_{a}(x_{2},y_{2})^{P}-d_{a}(x_{1},y_{2})^{P}-d_{a}(x_{2},y_{1})^{P}| \\
& \leq \sum_{a \in X_{c} \backslash \mathcal{B}, \pi(a) \in [x_{1},m]_{c}} | d_{a}(x_{1},y_{1})^{P}+d_{a}(x_{2},y_{2})^{P}-d_{a}(x_{1},y_{2})^{P}-d_{a}(x_{2},y_{1})^{P}| \\
 & + \sum_{a \in X_{c} \backslash \mathcal{B}, \pi(a) \in [m,y_{1}]_{c}} | d_{a}(x_{1},y_{1})^{P}+d_{a}(x_{2},y_{2})^{P}-d_{a}(x_{1},y_{2})^{P}-d_{a}(x_{2},y_{1})^{P}|.
\end{aligned}$$

We will bound from above the first sum, the same proof will apply for the second by symmetry.
To do this, we remark that:
$$\begin{aligned}
 &\sum_{a \in X_{c} \backslash \mathcal{B}, \pi(a) \in [x_{1},m]_{c}} | d_{a}(x_{1},y_{1})^{P}+d_{a}(x_{2},y_{2})^{P}-d_{a}(x_{1},y_{2})^{P}-d_{a}(x_{2},y_{1})^{P}|  \\
& \leq \sum_{a \in X_{c} \backslash \mathcal{B}, \pi(a) \in [x_{1},m]_{c}} P \max(d_{a}(x_{1},y_{1})^{P-1},d_{a}(x_{1},y_{2})^{P-1})d_{a}(y_{1},y_{2}) \\
 & + \sum_{a \in X_{c} \backslash \mathcal{B}, \pi(a) \in [x_{1},m]_{c}} P \max(d_{a}(x_{2},y_{1})^{P-1},d_{a}(x_{2},y_{2})^{P-1})d_{a}(y_{1},y_{2}).
\end{aligned}$$

Again, we will bound the first sum. For the second sum, the same proof will apply.
Let us denote $N:=d_{c}(x_{1},m)$. We order $[m,x_{1}]_{c}$ in the following sense: $[m,x_{1}]_{c}=\{z_{0},z_{1},...,z_{N}\}$ with $z_{0}=m$, $z_{N}=x_{1}$ and $d_{c}(z_{i},z_{i+1})=1$ for $0\le i \le N-1$.

\begin{figure}[!ht]
   \centering
   \includegraphics[scale=0.3]{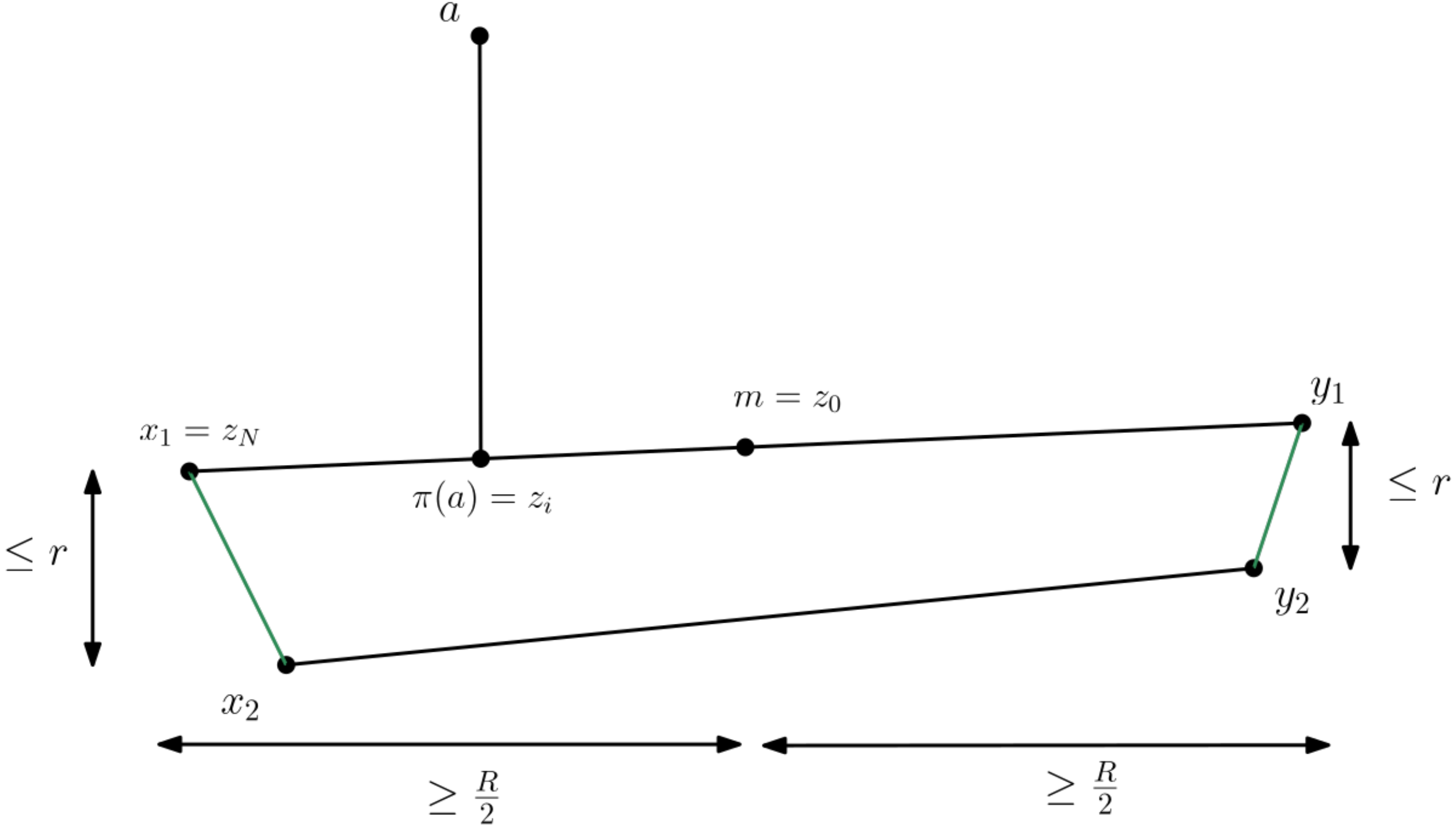}
   \caption{$\pi(a)=z_{i} \in [x_{1},m]$}
   \label{fig: situation B1 projection du cote de x1,x2 relativement hyperbolic}
\end{figure}

If $R$ is large enough, for all $a \in X_{c}$ with $a\in [m,x_{1}]_{c}$, we could apply Theorem \ref{theo: theo stylé sur les barycentres et les masques} to $a$ and $[y_{1},y_{2}]_{c}$, combined with Lemma \ref{lemme: projection plus angle sur un point qcq pour petite geodesique} and Lemma \ref{lem: projection et grand angle pour petite geodesique}, we get:
\begin{equation}\label{equation: control de day1y2}
d_{a}(y_{1},y_{2})\leq K_{1}\kappa^{\frac{1}{\alpha_{2}^{2}}d'_{(2000\delta)^{2}}(a,y_{1})-\omega_{2}},    
\end{equation}
where $\omega_{2}:=\lambda_{2}+(2000\delta)^{2}+r$ with $\lambda_{2}$ the constant of Lemma \ref{lemme: projection plus angle sur un point qcq pour petite geodesique}.

For all $i \in \{1,..., n\}$ and $a \in X_{c}$ such that $\pi(a)=z_{i}$, we have $d_{c}(z_{i},y_{1})\geq i+\frac{R}{2}$, thus:
$$\begin{aligned} 
           d'_{(2000\delta)^{2}}(a,y_{1}) & \geq  d_{c}(a,y_{1})\\      
            & \geq d_{c}(a,z_{i})+d_{c}(z_{i},y_{1})-24\delta \text{(according to Lemma \ref{lemme : z est la projection donc c'est un quasi-centre})} \\
             & \geq d_{c}(a,z_{i})+i+\frac{R}{2}-24\delta\\
             & \geq i+\frac{R}{2}-24\delta.
\end{aligned}$$

We have also according to Lemma \ref{lemme: d'deuxmiles projection} on the geodesic $[x_{1},y_{1}]_{c}$:
$$ d'_{(2000\delta)^{2}}(a,y_{1}) \geq \frac{1}{W}(i+\frac{R}{2}+d'_{(2000\delta)^{2}}(a,z_{i}))-W. $$

Moreover, according to Theorem \ref{theo: theo stylé sur les barycentres et les masques}, for all $i \in \{1,...,n\}$ and $a \in X_{c}$ such that $\pi(a)=z_{i}$ except for a finite number of $a$, we have:
\begin{align}\label{equation: bornage de dax1y1 et dax1y2}
&\max(d_{a}(x_{1},y_{1})^{P-1},d_{a}(x_{1},y_{2})^{P-1}) \notag\\
& \leq K_{1} \max(\kappa^{(P-1)d'_{(2000\delta)^{2}}(a,z_{i})},\kappa^{(P-1)(\frac{1}{\alpha_{2}^{2}}d'_{(2000\delta)^{2}}(a,z_{i})-\omega)}) (\text{according to Lemmata \ref{lemma: projection sur geodesic proche hyperbolicite relative} and \ref{lem: projection sur geodesic proche hyperbolicite relative avec grand angle}}) \notag \\
& \leq K_{1}\kappa^{\frac{P-1}{\alpha_{2}^{2}}d'_{(2000\delta)^{2}}(a,z_{i})-(P-1)\omega_{1}},
\end{align}
where $\omega_{1}:=\lambda_{1}+(2000\delta)^{2}+12\delta+\max(r,12\delta)$ with $\lambda_{1}$ the constant of Lemma \ref{lemma: projection sur geodesic proche hyperbolicite relative}.

Thus, we get:
$$\begin{aligned}
&\sum_{a \in X_{c} \backslash \mathcal{B}, \pi(a) \in [x_{1},m]_{c}} P \max(d_{a}(x_{1},y_{1})^{P-1},d_{a}(x_{1},y_{2})^{P-1})d_{a}(y_{1},y_{2})\\
& \leq P \sum_{i=0}^{N} \sum_{a \in X_{c} \backslash \mathcal{B}, \pi(a)=z_{i}} \max(d_{a}(x_{1},y_{1})^{P-1},d_{a}(x_{1},y_{2})^{P-1})d_{a}(y_{1},y_{2})\\
& \leq P \sum_{i=0}^{N} \sum_{a \in X_{c} \backslash \mathcal{B}, \pi(a)=z_{i}} \max(d_{a}(x_{1},y_{1})^{P-1},d_{a}(x_{1},y_{2})^{P-1})K_{1}\kappa^{\frac{1}{\alpha_{2}^{2}}d'_{(2000\delta)^{2}}(a,y_{1})-\omega_{2}}~(\text{according to Equation \ref{equation: control de day1y2}})\\
& \leq P  K_{1} 2\gamma_{G}^{C_{1}}t_{0}\kappa^{-\frac{24\delta}{\alpha_{2}^{2}}-\omega_{2}} \kappa^{\frac{R}{2\alpha_{2}^{2}}}\sum_{i=0}^{N}\kappa^{\frac{i}{\alpha_{2}^{2}}}~(\text{bound for the $a$ closed to $[x_{1},y_{1}]_{c}$})\\
&+ P \kappa^{\frac{R}{2 \alpha_{2}^{2}}} \kappa^{-(P-1)\omega_{1}-\omega_{2}} 
\sum_{i=0}^{N}\kappa^{\frac{i}{\alpha_{2}^{4}}} \sum_{a \in X_{c} \backslash \mathcal{B}} \kappa^{\frac{P-1}{\alpha_{2}^{2}}d'_{(2000\delta)^{2}}(a,z_{i})} \kappa^{\frac{1}{\alpha_{2}^{4} }d'_{(2000\delta)^{2}}(a,z_{i})}~(\text{according to Equation \ref{equation: bornage de dax1y1 et dax1y2}})\\
&\leq P  K_{1} 2\gamma_{G}^{C_{1}}t_{0}\kappa^{-\frac{24\delta}{\alpha_{2}^{2}}-\omega_{2}} \frac{1}{1-\kappa^{\frac{1}{\alpha_{2}^{2}}}} \kappa^{\frac{R}{2\alpha_{2}^{2}}}~(\text{geometric sum})\\
&+ P \kappa^{\frac{R}{2 \alpha_{2}^{2}}} \kappa^{-(P-1)\omega_{1}-\omega_{2}} 
\kappa^{\frac{R}{2\alpha_{4}}} \sum_{i=0}^{N}\kappa^{\frac{i}{\alpha_{2}^{4}}} \sum_{a \in X_{c}} \kappa^{\frac{P}{\alpha_{2}^{4}}d'_{(2000\delta)^{2}}(a,z_{i})} \\
& \leq P  K_{1} 2\gamma_{G}^{C_{1}}t_{0}\kappa^{-\frac{24\delta}{\alpha_{2}^{2}}-\omega_{2}} \frac{1}{1-\kappa^{\frac{1}{\alpha_{2}^{2}}}} \kappa^{\frac{R}{2\alpha_{2}^{2}}} \\
& + P \kappa^{-(P-1)\omega_{1}-\omega_{2}} 
\kappa^{\frac{R}{2\alpha_{2}^{2}}} \sigma \sum_{i=0}^{N}\kappa^{\frac{i}{\alpha_{2}^{4}}} ~(\text{according to Proposition \ref{proposition: nouveau p}})\\
& \leq P  K_{1} 2\gamma_{G}^{C_{1}}t_{0}\kappa^{-\frac{24\delta}{\alpha_{2}^{2}}-\omega_{2}} \frac{1}{1-\kappa^{\frac{1}{\alpha_{2}^{2}}}} \kappa^{\frac{R}{2\alpha_{2}^{2}}} + P  \kappa^{-(P-1)\omega_{1}-\omega_{2}} 
\sigma \frac{1}{1-\kappa^{\frac{1}{ \alpha_{2}^{4}}}} \kappa^{\frac{R}{\alpha_{2}^{4}}}~(\text{geometric sum}).
\end{aligned}$$

As $R$ goes to infinity the last term goes to zero therefore, for $R$ big enough, we have:
$$\sum_{a \in X_{c} \backslash \mathcal{B}, \pi(a) \in [x_{1},m]_{c}} P \max(d_{a}(x_{1},y_{1})^{P-1},d_{a}(x_{1},y_{2})^{P-1})d_{a}(y_{1},y_{2})\leq \eta.$$

Hence:
$$\sum_{a \in X_{c} \backslash \mathcal{B}} | d_{a}(x_{1},y_{1})^{P}+d_{a}(x_{2},y_{2})^{P}-d_{a}(x_{1},y_{2})^{P}-d_{a}(x_{2},y_{1})^{P}|\leq 2 \eta.$$

This concludes the proof.
\end{enumerate}
\end{enumerate}
    
\end{proof}

\subsection{Conclusion}\label{subsection: conclusion fortement bolic}

We give a proof of Theorem \ref{theo: bolicite forte pour relativement hyperbolique avec paraobliques admissible}.

\begin{proof} of Theorem \ref{theo: bolicite forte pour relativement hyperbolique avec paraobliques admissible}

The function $\hat{d}$ of Definition \ref{definition : metrique fortement bolique groupe relativement hyperbolique} is a metric on $G$ according to Proposition \ref{proposition: la metrique fortement bolique est une metrique}.

By definition $\hat{d}$ is $G$-invariant.

 According to Proposition \ref{proposition: quasi-isometry} $\hat{d}$ is quasi-isometric to the word metric, therefore $\hat{d}$ is uniformly locally finite.

The metric space $(G,\hat{d})$ is weakly geodesic according to Proposition \ref{proposition: faiblement geodesique}.

According to Theorem \ref{theo: strong b1 relativement hyperbolic} $\hat{d}$ satisfies the strong-$B1$ condition and the weak-$B2'$ condition according to Proposition \ref{proposition: dchapeau verifie faiblement b2'}, therefore $\hat{d}$ is strongly bolic.

\end{proof}

As $CAT(0)$ groups are admissible, we deduce Theorem \ref{theorem: bolicite forte pour relativement hyperboliques avec paraboliques cat(0)}.

\begin{proof} of Theorem \ref{theorem: bolicite forte pour relativement hyperboliques avec paraboliques cat(0)}
According to Proposition \ref{proposition: lesgroupes Cat(0) sont admissibles} $CAT(0)$ groups are admissible. We apply Theorem \ref{theo: bolicite forte pour relativement hyperbolique avec paraobliques admissible}.

\end{proof}

According to Lafforgue’s theorem \ref{theo : Baum-Connes=Rapid decay + Strongly Bolic}, we present here our results concerning the Baum-Connes conjecture.

\begin{proof} of Corollary \ref{corollary: Baum-Connes pour relativement hyperboliques avec paraboliques}

Let $H$ be a subgroup of a relatively hyperbolic group $G$ with $CAT(0)$ parabolics with the $(RD)$ property.
 According to Theorem \ref{theo: theoremimportant bolicite forte}, there exists a metric $\hat{d}$ on $G$ such that $(G,\hat{d})$ is weakly geodesic and strongly bolic. According to Proposition \ref{proposition: quasi-isometry}, $\hat{d}$ is quasi-isometric to the word metric on $G$ therefore it is locally uniformly finite and the action of $H$ on $G$ is proper.
Finally, $H$ satisfies the $(RD)$ property according to \cite{ChatterjiRuaneRdreseauxderangUn}.
According to Theorem \ref{theo : Baum-Connes=Rapid decay + Strongly Bolic}, we deduce that $H$ satisfies the Baum-Connes conjecture.

 \end{proof}

 In this corollary, we provide families of $CAT(0)$ groups for which the $(RD)$ property is known.

 \begin{proof} of Corollary \ref{corollary: Baum-Connes pour relativement hyperbolique en disant lesquels satisfont RD}

According to \cite[Proposition 8.1]{minasyan}, virtually abelian groups are $CAT(0)$, therefore they are admissible according to Proposition \ref{proposition: lesgroupes Cat(0) sont admissibles}. They satisfy $(RD)$ property according to Jolissaint \cite{JolissaintRD}.

Cocompactly cubulated groups are $CAT(0)$ by definition, they satisfy $(RD)$ property according to \cite{ChatterjiRuaneRD}.

Coxeter groups are $CAT(0)$ according to \cite{DavisCoxetergroups}, they satisfy $(RD)$ property according to \cite{ChatterjiRuaneRD}.

 \end{proof}

\begin{proof} of Corollary \ref{Corollay: Baum-Connes pour variétés hyperboliques réelles}

According to \cite[Lemma II.10.27]{BridsonHaefliger}, Fundamental groups of real complete hyperbolic manifold are hyperbolic relatively to virtually abelian subgroups, therefore they satisfy the Baum-Connes conjecture according to Corollary \ref{corollary: Baum-Connes pour relativement hyperbolique en disant lesquels satisfont RD}.
\end{proof}

\newpage

\bibliographystyle{alpha}
\bibliography{bibliography}

\end{document}